\pdfoutput=1

\documentclass[10pt]{article}

\usepackage{xspace}
\usepackage{url}
\usepackage{mathtools}
\usepackage{amssymb}
\usepackage{amsthm}
\usepackage{empheq}
\usepackage{latexsym}
\usepackage{enumitem}
\usepackage{eurosym}
\usepackage{dsfont}
\usepackage{appendix}
\usepackage{color} 
\usepackage[unicode]{hyperref}
\usepackage{frcursive}
\usepackage[utf8]{inputenc}
\usepackage[T1]{fontenc}
\usepackage{geometry}
\usepackage{multirow}
\usepackage{todonotes}
\usepackage{lmodern}
\usepackage{anyfontsize}
\usepackage{stmaryrd}
\usepackage{cleveref}
\usepackage[english]{babel}
\usepackage[english=british]{csquotes}
\usepackage{mathtools}
\usepackage{accents}
\usepackage{color}
\usepackage{comment}
\usepackage{hyperref}

\definecolor{red}{rgb}{0.7,0.15,0.15}
\definecolor{green}{rgb}{0,0.5,0}
\definecolor{blue}{rgb}{0,0,0.7}
\hypersetup{colorlinks, linkcolor={red}, citecolor={green}, urlcolor={blue}}

\allowdisplaybreaks

\usepackage{natbib}
\setcitestyle{numbers,open={[},close={]}}

\definecolor{red}{rgb}{0.7,0.15,0.15}
\definecolor{green}{rgb}{0,0.5,0}
\definecolor{blue}{rgb}{0,0,0.7}
\hypersetup{colorlinks, linkcolor={red}, citecolor={green}, urlcolor={blue}}
			
\makeatletter \@addtoreset{equation}{section}

\newtheorem{theorem}{Theorem}[section]
\newtheorem{assumption}[theorem]{Assumption}

\newtheorem{example}[theorem]{Example}

\newtheorem{lemma}[theorem]{Lemma}
\newtheorem{proposition}[theorem]{Proposition}

\newtheorem{definition}[theorem]{Definition}
\newtheorem{remark}[theorem]{Remark}

\DeclareUnicodeCharacter{014D}{\=o}
\newcommand{\smalltext}[1]{\text{\fontsize{4}{4}\selectfont$#1$}}
\newcommand{\tinytext}[1]{\text{\fontsize{3}{3}\selectfont$#1$}}
\newcommand{\stinytext}[1]{\text{\fontsize{2}{2}\selectfont$#1$}}

\newcommand\bC{\mathbf C}

\newcommand\cB{\mathcal B}
\newcommand\cC{\mathcal C}

\newcommand\cE{\mathcal E}
\newcommand\cF{\mathcal F}
\newcommand\cG{\mathcal G}

\newcommand\cI{\mathcal I}
\newcommand\cJ{\mathcal J}

\newcommand\cO{\mathcal O}
\newcommand\cP{\mathcal P}

\newcommand\cT{\mathcal T}
\newcommand\cU{\mathcal U}
\newcommand\cV{\mathcal V}

\def \E{\mathbb{E}}
\def \F{\mathbb{F}}

\def \L{\mathbb{L}}

\def \N{\mathbb{N}}
\def \P{\mathbb{P}}

\def \R{\mathbb{R}}

\def\Uc{\mathcal{U}}

\def\eps{\varepsilon}
\def\d{\mathrm{d}}
\DeclareMathOperator*{\argmax}{arg\,max}

\DeclareMathOperator*{\esssup}{ess\,sup}

\DeclareMathOperator{\sgn}{sgn}

\usepackage{graphicx}
\usepackage{subfig}

\usepackage{subfiles}

\begin{document}

\title{Golden parachutes under the threat of accidents}

\author{Dylan {\sc Possama\"{i}} \footnote{ETH Z\"urich, Department of Mathematics, Switzerland, dylan.possamai@math.ethz.ch. This author gratefully acknowledges partial support by the SNF project MINT 205121-219818.}\and Chiara {\sc Rossato} \footnote{ETH Z\"urich, Department of Mathematics, Switzerland, chiara.rossato@math.ethz.ch.} }

\date{\today}

\maketitle

\begin{abstract}
This paper addresses a continuous-time contracting model that extends the problem introduced by \citeauthor*{sannikov2008continuous} \cite{sannikov2008continuous} and later rigorously analysed by \citeauthor*{possamai2020there} \cite{possamai2020there}. In our model, a principal hires a risk-averse agent to carry out a project. Specifically, the agent can perform two different tasks, namely to increase the instantaneous growth rate of the project's value, and to reduce the likelihood of accidents occurring. In order to compensate for these costly actions, the principal offers a continuous stream of payments throughout the entire duration of a contract, which concludes at a random time, potentially resulting in a lump-sum payment. We examine the consequences stemming from the introduction of accidents, modelled by a compound Poisson process that negatively impact the project's value. Furthermore, we investigate whether certain economic scenarii are still characterised by a golden parachute as in \citeauthor*{sannikov2008continuous}'s model. A golden parachute refers to a situation where the agent stops working and subsequently receives a compensation, which may be either a lump-sum payment leading to termination of the contract or a continuous stream of payments, thereby corresponding to a pension.

\bigskip
\noindent{\bf Key words:} principal--agent problem, random horizon, integro-differential Hamilton--Jacobi--Bellman equation  \bigskip
\end{abstract}
%\setlength{\parindent}{0pt}

%\tableofcontents

\section{Introduction}

Risk naturally emerges in various interactions across all sectors of the economy and is well-known to be closely linked to the concept of uncertainty. It involves the potential occurrence of unexpected shocks that could lead to harm or loss, thus exerting a significant---and usually negative---influence on the considered economic scenario. Here, one is naturally led to consider situations such as accidents during the production process in a company, injuries in an health insurance context, or the devaluation of financial assets in specific corporations, due to the bankruptcy of a subsidiary. Although various other instances could be mentioned, they all have in common the potential for certain adverse events that can be categorised as accidents. It is therefore natural to study risk prevention in contract theory, specifically within principal--agent problems. 

\medskip
The principal--agent paradigm addresses the optimal contracting problem in the presence of information asymmetry between two economic actors: a principal (`\emph{she}') and an agent (`\emph{he}'). In the typical scenario, the principal designs a contract to hire the agent, who can accept or reject it based on whether it provides sufficiently adequate incentives to ensure that his reservation utility is attained. The primary challenge is that the principal designs this contract to optimise her utility while having imperfect information about the efforts provided by the agent. This problem originated in the 1970s within a discrete-time framework and was subsequently reformulated in a continuous-time model by \citeauthor*{holmstrom1987aggregation} \cite{holmstrom1987aggregation}. This seminal work initiated the study of moral hazard problems in continuous-time, providing important insights into formulating an optimal contract between a principal and an agent whose effort influences the drift of the output process. Following this, numerous results have been established to explore more general models and derive tractable solutions, thereby expanding the scope of this initial work. A comprehensive analysis of the related literature can be found in the survey paper of \citeauthor*{sung2001lectures} \cite{sung2001lectures} or in the book of \citeauthor*{cvitanic2012contract} \cite{cvitanic2012contract}. 

\medskip
It has nowadays been universally acknowledged that \citeauthor*{sannikov2008continuous} \cite{sannikov2008continuous} achieved a substantial breakthrough in the field. In this work, the author studied an infinite horizon principal--agent model, in which the principal provides continuous payments to the agent, who exclusively controls the drift of the output process until a random retirement time. The significance of this work lies in its pioneering approach, which enables the transformation of the original bi-level optimisation problem between the principal and the agent into a standard stochastic control problem, whose solution can be described through an Hamilton–Jacobi–Bellman equation. This approach led to numerous follow-up works that rigorously formalised \citeauthor*{sannikov2008continuous}'s ideas. Particularly noteworthy are the PhD thesis of \citeauthor*{choi2014sannikov} \cite{choi2014sannikov} and, more recently, the contribution of \citeauthor*{possamai2020there} \cite{possamai2020there}, which extended the principal--agent model introduced by \citeauthor*{sannikov2008continuous} by considering potentially different discount rates for both parties. Their primary focus was on technical rigour,  ultimately providing a complete characterisation of the solution to the contracting problem. Another significant extension was introduced by \citeauthor*{cvitanic2018dynamic} \cite{cvitanic2018dynamic}, presenting a comprehensive mathematical approach to solve principal--agent problems that allow control over both the drift and diffusion of the output process. While this analysis focused on a finite deterministic horizon setting, \citeauthor*{lin2022random} \cite{lin2022random} subsequently addressed principal--agent problems over random horizons.

\medskip
It is nonetheless important to point out that the Brownian framework considered in all the previously mentioned works fails to capture sudden and unpredictable events, and consequently it becomes necessary to consider principal--agent problems involving a jump component in the output process. There exists a vast literature that adopts jump processes to replicate negative shocks, with an early influential contribution from \citeauthor*{sung1997corporate} \cite{sung1997corporate}. This work extended \citeauthor*{holmstrom1987aggregation}'s model, addressing a corporate insurance problem driven by a multi-dimensional point process, where each component represents a specific type of accident, less likely to occur when the agent exerts effort. More recently, \citeauthor*{zhang2009dynamic} \cite{zhang2009dynamic} explored efficient allocations in a Mirrleesean economy influenced by persistent shocks that transition according to a continuous-time Markov chain that is beyond the control of the agent. Similarly, \citeauthor*{sannikov2010role} \cite{sannikov2010role} investigated the dynamic interaction among economic actors in a scenario where the information process is characterised by a discontinuous uncontrolled Poisson component associated to bad news. In contrast, \citeauthor*{biais2007environmental} \cite{biais2007environmental,biais2010large} simulated accidents and large and unfrequent losses, respectively, using a Poisson process whose intensity can be influenced by the agent's actions. Additionally, \citeauthor*{pages2014mathematical} \cite{pages2014mathematical}, later extended by \citeauthor*{hernandez2020bank} \cite{hernandez2020bank} to include adverse selection, explored the optimal securitisation of a pool of long-term loans exposed to Poisson default risk, where the default intensity can be reduced through the monitoring activity of a bank. Furthermore, \citeauthor*{martin2023risk} \cite{martin2023risk} modelled a large insurable risk capable of halting production entirely using a single jump process. However, it is worth noting that in existing literature, a jump process is not exclusively associated with negative events. \citeauthor*{el2021optimal} \cite{el2021optimal} linked a Poisson process with market orders to investigate a contract offered by an exchange to a market-maker, aiming to reduce the bid--ask spread for a single risky asset. Subsequently, this model was extended to consider multiple market participants by \citeauthor*{baldacci2021optimal} \cite{baldacci2021optimal} and to enable trading on dark liquidity pools by \citeauthor*{baldacci2023market} \cite{baldacci2023market}.

\medskip
All the previously mentioned works adopt jump processes with a constant jump size to simulate sudden shocks. To the best of our knowledge, the first study to consider accidents of different random sizes is by \citeauthor*{capponi2015dynamic} \cite{capponi2015dynamic}. The authors examined a principal--agent model where the agent can exert both effort and accident prevention, aiming to decrease the likelihood of negative events that are described by an exogenous loss distribution. Later on, \citeauthor*{bensalem2023continuous} \cite{bensalem2023continuous} employed a similar jump--diffusion framework in their investigation of an insurance demand model, enabling the protection buyer to exert prevention effort to reduce his risk exposure. An extension  of all the aforementioned principal--agent models can be found in the work of \citeauthor*{hernandez2022principal} \cite{hernandez2022principal}. Here, the author studied a general contracting problem between the principal and a finite number of agents, where each agent has control over the drift of the output process and the compensator of its associated jump measure. Additionally, \citeauthor*{mastrolia2022agency} \cite{mastrolia2022agency} formulated an energy-optimal demand--response model where the principal hires an infinite number of agents, incorporating accidents by introducing a Lévy process.

\medskip
Principal--agent models with jumps that go beyond simple Poisson process lead to Hamilton--Jacobi--Bellman equations with an additional integral term. Consequently, the notion of viscosity solution introduced by \citeauthor*{crandall1983viscosity} \cite{crandall1983viscosity} and \citeauthor*{lions1983optimal} \cite{lions1983optimal} requires particular attention as the equations become non-local. The first works in this direction are to be attributed to \citeauthor*{soner1986optimal} \cite{soner1986optimal,soner1986optimal2} for bounded measures and \citeauthor*{awatif1991equations} \cite{awatif1991equations,awatif1991equations2} for unbounded measures. Subsequently, \citeauthor*{alvarez1996viscosity} \cite{alvarez1996viscosity} proved a general existence and uniqueness result for viscosity solutions to integro-differential equations characterised by a bounded integral operator, while \citeauthor*{pham1998optimal} \cite{pham1998optimal} achieved a similar result for a special case involving singular measures. In recent years, there has been a surge of interest in extending the theory of viscosity solutions to integro-differential equations. This growth in the literature is largely due to the dependence of the results on the integrability of the singular measure characterising the integral operator. Among these advancements, \citeauthor*{jakobsen2006maximum} \cite{jakobsen2006maximum} presented a first version of a non-local maximum principle, which allowed the deduction of a comparison principle for viscosity sub-quadratic solutions. Subsequently, \citeauthor*{barles2008second} \cite{barles2008second} and later the PhD thesis of \citeauthor*{hollender2016levy} \cite{hollender2016levy} generalised their uniqueness result by considering solutions with arbitrary growth at infinity.

\medskip
In this paper, we examine the weak formulation of the model introduced by \citeauthor{sannikov2008continuous} \cite{sannikov2008continuous} in the presence of accidents, thereby making the model suitable for investigating contracting problems between a principal and an agent in an economic environment that might experience unpredictable and negative shocks of unknown size. Specifically, we explore optimal contracts offered by a risk-neutral principal to a risk-averse agent aimed at executing a project. To maximise the project's value, which is described by an underlying jump--diffusion process, the agent can employ two distinct actions simultaneously, influencing both the drift of the process and the likelihood of accidents occurring. We model the impact of accidents using a compound Poisson process, in line with the approach outlined by \citeauthor*{capponi2015dynamic}. However, there seems to be a subtle issue in the derivation of optimal contracts in the aforementioned work, stemming from an incorrect application of the martingale representation property in \cite[Lemma 3.1]{capponi2015dynamic}. This issue inevitably imposes a restriction on the set of the admissible contracts. Consequently, our contribution can be seen as an extension to this model, not only due to this concern but also because we assume that the principal offers compensation to the agent not solely at the random termination time but consistently throughout the entire duration of the contract.

\medskip
The main objective of this paper is to verify whether the rigorous study by \citeauthor*{possamai2020there} \cite{possamai2020there} regarding \citeauthor{sannikov2008continuous}'s contracting problem remains valid within our framework that incorporates accidents. More specifically, we aim to test the robustness of the economic conclusions drawn by these authors by studying the economic implications resulting from these unpredictable events that negatively impact the principal--agent interaction. Our focus is particularly on examining the concept of a `golden parachute' in the context of accident risk. This term refers to a situation where the agent ceases to work at a random time---or refrains from working completely---and subsequently receives a positive compensation. This compensation may take the form of a lump-sum payment, leading to termination of the contract, or a continuous stream of payments, indicating retirement. Therefore, our investigation aims to determine whether economic conditions exist that could lead to the existence of these two distinct instances of golden parachutes, similarly to what was observed in the problem without accidents. To conduct this analysis, we need to consider different regimes based on the impatience levels of both parties:
\begin{enumerate}[label=$(\roman*)$]
\item\label{Analysis_deg} when the principal is significantly more impatient than the agent, a golden parachute does not exist. This scenario depends only on the impatience of both economic actors and the agent's level of risk aversion, completely mirroring the case without accident risk. Indeed, the problem degenerates because the principal can achieve her maximum reward regardless of the agent's participation level. However, optimal contracts do not exist; 
\item\label{Analysis_NOdeg_deltaSmaller1} when the principal is more impatient than the agent but not excessively, as described in the previous scenario, a golden parachute can correspond to either of the previously described situations. This is in stark contrast with the analysis not involving accidents, where a golden parachute is solely linked to a pension and not to a contract termination with subsequent positive lump-sum payment. This difference is due to the impact of accidents on the economic scenario, leading the principal to prefer terminating the contract over retiring the agent for a sufficiently small continuation utility of the agent. This decision is influenced by the expenses associated with potential future losses outweighing the immediate cost of compensating the agent;
\item\label{Analysis_NOdeg_delta1} when the principal is as impatient as the agent, the analysis closely aligns with that conducted without considering accidents. In other words, we know that a golden parachute is likely to exist, and it corresponds to a lump-sum compensation;
\item\label{Analysis_NOdeg_deltaBigger1} when the agent is strictly more impatient than the principal, a golden parachute can be associated with both distinct compensation schemes, which further depends on the average loss per accident. When the size of potential accidents is not significant, contract termination is never optimal, while if accidents entail high costs, the principal prefers either firing the agent or retiring him, based on the continuation utility of the agent. Hence, we observe that the nature of a golden parachute is deeply linked to the average accident size. In case of large accidents, it translates to either a lump-sum payment or a continuous stream of payments. Conversely, as the accident size decreases, it corresponds to the retirement scenario, mirroring the situation in the absence of accidents.
\end{enumerate}

The motivation for analysing the diverse levels of patience exhibited by both parties stems from the necessity for a comprehensive comparison with the model without jumps in \cite{possamai2020there}. This is because, in most real-life principal-agent interactions, the principal is less impatient than the agent. An example where this situation is reversed is evident in the start-up environment---although it does not perfectly reflect the model outlined in this paper since the project's value is not constrained---especially in the early stages of the start-up process, as explained in \citeauthor*{frank2007significance} \cite{frank2007significance}. In fact, the combination of financial investment and risk exposure often makes owners more impatient than employees. This is because owners typically have a significant financial stake in the start-up that can contribute to their impatience to see the business succeed and quickly turn investments into a profit. As outlined in the description of the various regimes, in addition to the impatient levels of both parties, the economic consequences vary according to the average accident size. Indeed, in the insurance context, for example, if we consider insurance premiums that include an incentive component to implement risk reduction measures, such as premium discounts or policy endorsements, the insured accidents can vary in nature, leading to different consequences. Large risks are associated with insurance policies taken out by companies analysing sensitive national data since cyber attacks, for instance, can result in significant financial losses, or by large manufacturing firms with production processes characterised by environmental risks such as air or water pollution. Catastrophic safety risks and the considerable societal costs linked to accidents are addressed in \citeauthor*{biais2007environmental} \cite{biais2007environmental}, where they investigate insurance policies sold to a firm whose manager is incentivised to implement necessary risk-prevention measures. An example of a model where insurance claims are also sold to individuals is studied by \citeauthor*{cao2022stackelberg} \cite{cao2022stackelberg}. The authors examine insurance terms and premiums in a spectrally negative Lévy framework, where both the buyer and the seller of the insurance policy face ambiguity regarding the intensity and severity of insurable losses.

\medskip
We emphasise that the economic features of the model are closely related to the specific regimes \ref{Analysis_deg}, \ref{Analysis_NOdeg_deltaSmaller1}, \ref{Analysis_NOdeg_delta1} and \ref{Analysis_NOdeg_deltaBigger1} under consideration and thus rely on the impatience levels of both parties, exactly as in the problem without accidents. However, despite the unchanged structure of the problem, the introduction of unpredictable and adverse events allows for the possibility of both golden parachute scenarii. In other words, when the agent is not as impatient as the principal, a golden parachute can be associated with both retirement and termination, while retirement is always the preferred choice in the absence of accidents. This different characterisation stems from the different nature of an appropriate face-lifted version of the agent’s inverse utility function, which is defined through an obstacle problem. While in the accident-free framework, this problem simplifies to the study of an ordinary differential equation, in our model, the analysis of this face-lifted utility function is significantly more involved due to its correlation with the average accident size. We also mention that even with relatively small potential accidents, the impact on the principal is significant, and can result in a 50\% reduction of the value of the principal for an accident size which represents only 2.5\% of the average value of the project.

\medskip
Another important difference with the reference model without accidents is obvious in the study of the first-best problem, which is no longer directly solvable, except in the scenario where the principal and the agent are equally impatient. Despite the difficulty in fully characterising the first-best solution for every regime, we can still replicate similar economic considerations. However, the most challenging aspect of investigating the model incorporating accidents lies in describing the second-best problem in order to address the economic question of whether a golden parachute is optimal. Unfortunately, we cannot provide a definitive answer to this question due to the complexity arising from the integral term appearing in the Hamilton--Jacobi--Bellman equation associated with the contracting problem. We will discuss this in more detail in \Cref{section:discussProblemHJB}. Nonetheless, we are able to characterise the solution to the problem as the unique viscosity solution of the aforementioned equation through a generalisation of Tietze’s extension theorem that allows to overcome the obstacles posed by the limited liability constraint.

\medskip
The rest of the paper is structured as follows. In \Cref{section:modelASS}, we formulate the principal--agent model and provide a comprehensive explanation of the criteria of the agent and the principal. We then discuss the two notions of a golden parachute along with the definition of the face-lifted utility function that combines these concepts. In \Cref{section:faceliftedAnalysis}, we completely characterise the face-lifted utility in all the various cases \ref{Analysis_deg}, \ref{Analysis_NOdeg_deltaSmaller1}, \ref{Analysis_NOdeg_delta1} and \ref{Analysis_NOdeg_deltaBigger1} outlined above. In \Cref{section:firstBest}, we study the first-best contracting problem, which offers insights into the subsequent analysis of the second-best contracting problem in \Cref{section:secondBest}. We defer technical proofs to the appendices.

\medskip
{\small\textbf{Notation.} Let $\N$ be the set of positive integers, and $\R_+$ the set of non-negative real-numbers. For any $(a,b) \in [-\infty,\infty]^2$, we write $a \wedge b \coloneqq \min\{a, b\}$ and $a \vee b \coloneqq \max\{a,b\}$. Consider a probability space $(\Omega, \cF, \P)$ carrying a filtration $\F \coloneqq (\cF_t)_{t \geq 0}$. We use the convention $\cF_{0-} \coloneqq \cF_{0}$ and $\cF_\infty \coloneqq \cF_{\infty -} \coloneqq \sigma(\cup_{t \geq 0} \cF_t)$. For $t \in [0, \infty]$, we write $\E^\P_t[\xi]$ for the conditional expectation of a random variable $\xi$ with respect to $\cF_t$. For any $\F$--stopping time $\tau$, we denote by $\cT_{\tau}(\F)$ the set of $\F$--stopping times $\theta$ such that $\theta \leq \tau$, $\P$--a.s., and by $\llbracket \theta,\tau \rrbracket$ (resp. $\rrbracket \theta,\tau \rrbracket$) the stochastic interval $\{(t,\omega) \in \R_+ \times \Omega: \; \theta(\omega) \leq (\text{resp.} < )\; t \leq \tau(\omega) \}$. The notation $\widetilde{\cP}(\F)$ (resp. $\widetilde{\cO}(\F)$) refers to the predictable (resp. optional) $\sigma$-algebra on $\Omega \times \R_+ \times \R$. For a random measure $\mu$ on $\R_+ \times \R$, we denote by $\mu^p$ (resp. $\tilde{\mu} \coloneqq \mu - \mu^p$) its $(\F,\P)$-compensator (resp. $(\F,\P)$-compensated random measure). We introduce the spaces
\begin{align*}
G_{\rm loc}(\mu,\F,\P) &\coloneqq\bigg\{\R \text{\rm-valued}, \; \widetilde{\cP}(\F) \text{\rm-measurable functions} \; U : \bigg(\int_0^{\cdot} \int_\R |U_s(\ell)|^2 \mu(\d s, \d \ell)\bigg)^\frac{1}{2}\; \text{\rm is} \; (\F,\P)\text{\rm--locally}\;  \P\text{\rm -{integrable}} \bigg\}, \; \\
\L^2_{\mathrm{loc}}(W,\F,\P) &\coloneqq \bigg\{\R \text{\rm-valued}, \; \F \text{\rm-predictable processes} \; Z :  \int_0^{\cdot} |Z_s|^2 \d s \; \text{\rm is} \; (\F,\P)\text{\rm--locally} \; \P\text{\rm-integrable} \bigg\},
\end{align*}
where $W$ is an $(\F,\P)$--Brownian motion. For an $(\F,\P)$--local martingale $X$ in the sense of \citeauthor*{jacod2003limit} \cite[Definition I.1.45]{jacod2003limit}, we denote its stochastic exponential by $\cE(X)$, that is, 
\begin{equation*}
\cE(X)_t = \mathrm{exp}\big(X_t - \frac{1}{2} [ X^c ]_t\big) \prod_{0 < s \leq t} (1+\Delta X_s) \mathrm{e}^{- \Delta X_s}, \; \text{for} \; t \geq 0,
\end{equation*}
where $[ X^c ]$ is the quadratic variation of the continuous martingale part of $X$ and $\Delta X$ denotes the jump process of $X$ (as in \cite[Theorem I.4.61]{jacod2003limit}).
}

\section{Model and assumptions}\label{section:modelASS}
We investigate a continuous-time contracting problem in which a risk-neutral principal hires a risk-averse agent to manage a project over a possibly infinite time horizon. During this period, the agent can exert effort to increase the instantaneous growth rate of the project's value, that is, the principal's profit, while simultaneously reducing the likelihood of accidents that negatively impact the total profit. It is important to emphasise that in the model considered here, the agent is unable to reduce the impact of accidents on the project's value as he has no control over their size, which is an exogenous factor.

\subsection{The setting}
Let $\left(\Omega,\cF,\P\right)$ be a probability space carrying the following jointly $\P$--independent objects: a one-dimensional Brownian motion $W=(W_t)_{t \geq 0}$, a Poisson process $N=(N_t)_{t \geq 0}$ with $\P$-intensity $\int_0^\cdot \d s$ (see \cite[Definition I.3.26]{jacod2003limit}), and a collection of bounded and positive $\P$--i.i.d. random variables $(L_i)_{i\in \N}$. Let $\sigma > 0$, and let $x_0 \in \R$ denote a given initial investment. The dynamics of the project's value, referred to as the output process hereafter, under no exertion of effort by the agent is given by 
\begin{equation*}
X_t \coloneqq x_0 + \sigma W_t - J_t, \; \text{for} \; t \geq 0,
\end{equation*}
where $J=(J_t)_{t \geq 0}$ denotes the compound Poisson process
\begin{equation*}
J_t \coloneqq \sum_{i = 1}^{N_t} L_i, \; \text{for} \; t \geq 0.
\end{equation*}
Here, $W$ represents a source of randomness, and $J$ models the losses that can occur throughout the lifetime of the project at random times. The Poisson process $N$ represents the total number of accidents, while $J$ the respective cumulative losses since, for $i\in\N$, each random variable $L_i$ quantifies the size of the $i$-th accident. The average loss per accident is denoted by
\begin{equation}\label{equation:distributionJumps}
m \coloneqq \E^\P[L_i] = \int_{\R} \ell \Phi(\d \ell),
\end{equation}
where $\Phi$ is the common distribution function associated to all the $L_i$, $i\in\N$. Note that $m$ is positive since $L_i$ is assumed to be positive. Finally, we let $\F= (\cF_t)_{t \geq 0}$ be the $\P$-augmentation of the filtration generated by the Brownian motion $W$ and the compound Poisson process $J$. Note that $\F$ satisfies the usual conditions under $\P$ by \citeauthor*{protter2005stochastic} \cite[Theorem I.31]{protter2005stochastic}. 

\medskip
In this setting, a natural question to ask is whether the predictable martingale representation property with respect to $W$ and $J$ holds. This can indeed be achieved by introducing the associated jump measure of $J$ as follows. Let
\begin{align*}
\mu^{J}(\d t, \d \ell) \coloneqq \sum_{s > 0} \mathbf{1}_{\{\Delta J_\smalltext{s} \neq 0\}} \varepsilon_{(s,\Delta J_\smalltext{s})} (\d t, \d \ell),
\end{align*}
where $\varepsilon$ denotes the Dirac measure, the $(\F,\P)$-compensator (see \cite[Theorem II.1.8]{jacod2003limit} for the definition) is then given by
\begin{align*}
\mu^{J,p}(\d t, \d \ell) \coloneqq \Phi(\d \ell) \d t.
\end{align*}
To ease the notation, we write $\tilde{\mu}^J(\d t, \d \ell)$ for the compensated random measure $\mu^{J}(\d t, \d \ell) - \mu^{J,p}(\d t, \d \ell)$. Furthermore, for any $Z \in \L^2_{\mathrm{loc}}(W,\F,\P)$ and  $U \in G_{\rm loc}(\mu^J,\F,\P)$, we denote by 
\begin{equation}
	\int_0^{t} Z_s \d W_s \; \text{and} \; \int_0^{t} \int_{\R} U_s(\ell) \tilde{\mu}^{J}(\d s, \d \ell), \; t \geq 0,
\end{equation}
the stochastic integral of $Z$ with respect to $W$ and the compensated stochastic integral of $U$ with respect to the random measure $\mu^J$, respectively, in the sense of \cite[Chapter 1 \& 2]{jacod2003limit}. Having collected the necessary ingredients, we can now state the following martingale representation theorem, which is derived from \citeauthor*{cohen2015stochastic} \cite[Theorem 14.5.7]{cohen2015stochastic} combined with \cite[Observation I.4.1]{jacod2003limit}.

\begin{lemma}\label{lemma:martingale Representation}
Let $M$ be an $(\F,\P)$--local martingale. Then, there exist unique $Z \in \L^2_{\mathrm{loc}}(W,\F,\P)$ and $U \in G_{\rm loc}(\mu^J,\F,\P)$ such that
\begin{equation*}
M_{t} = M_0 + \int_0^{t } Z_s \d W_s + \int_0^{t} \int_{\R} U_s(\ell) \tilde{\mu}^{J}(\d s, \d \ell), \;t \geq 0,\; \P \text{\rm --a.s.}
\end{equation*}
Moreover, this property is preserved under an equivalent change of measure, see for instance {\rm\citep[Theorem III.5.24]{jacod2003limit}}.
\end{lemma}

\subsection{Actions and contracts}
We assume that through specific costly efforts, the agent can influence the distribution of the output process $X$, that is, he can affect the probability measure under which the problem is described in its weak formulation. For this, we introduce a compact subset $A$ of $\R_+$ containing $0$ and bounded by $\bar{a}\in A$, for some $\bar{a}>0$, representing the possible effort values that the agent can exert on the drift of the output process. Additionally, we define another compact subset $B$ of $\R_+$ bounded by $m$. This set $B$ denotes all possible accident frequencies, which is why it needs to satisfy the technical condition of being bounded away from zero. Specifically, there exists some $\eps_m\in B $ such that $\eps_m \in (0,m \wedge \bar{a})$ and $\eps_m \leq b$ for any $b \in B$. It is important to note that the fact that $\eps_m$ is positive implies that accidents are less likely to happen if the agent exerts effort, but he cannot eliminate them completely. We denote for simplicity $U\coloneqq A \times B$\footnote{The decision to consider $A$ and $B$ as generic closed sets instead of the intervals $[0,\bar{a}]$ and $[\varepsilon_m,m]$ is made to encompass the binary action case mentioned in \cite{sannikov2008continuous}. In this scenario, the agent can only choose between working, indicated by $(a, b) = (\bar{a}, \varepsilon_m)$, or not working, represented by $(a, b) = (0, m)$.}.

\medskip
The collection of admissible controls $\cU$ for the agent consists of all $\F$-predictable processes $\nu \coloneqq (\alpha, \beta)$ with values in $U$ such that the following stochastic exponential 
\begin{equation*}
M^\nu_t \coloneqq \cE\left(\int_0^\cdot \frac{\alpha_s}{\sigma} \d W_s\right)_t \cE\left(\int_0^\cdot \left(\frac{\beta_s}{m}-1\right) \left(\d N_s - \d s\right)\right)_t = \cE\bigg(\int_0^\cdot \frac{\alpha_s}{\sigma} \d W_s\bigg)_t \cE\bigg(\int_0^\cdot \int_{\R} \bigg(\frac{\beta_s}{m}-1\bigg) \tilde{\mu}^{N}(\d s, \d \ell) \bigg)_t, \; \text{for} \; t \geq 0,
\end{equation*}
is a $\P$--uniformly integrable $(\F,\P)$-martingale. It is well-known that this implies the existence of an $\cF_{\infty}$-measurable random variable $M^\nu_{\infty}$ such that $(M^\nu_t)_{t \in [0,\infty]}$ is still a $\P$--uniformly integrable $(\F,\P)$-martingale.

\medskip
For any $\nu \in \cU$, we define $\P^{\nu}$ as the probability measure on $(\Omega,\cF)$ whose density with respect to $\P$ is given by 
\begin{align}\label{align:changeMeasureNU}
\frac{\d\P^{\nu}}{\d\P} \coloneqq M^\nu_{\infty}.
\end{align}

While the agent can control the growth rate of the output process and reduce the intensity at which accidents occur through costly effort, the objective of the principal is to design a contract that incentivises the agent to increase the overall value of the project. The execution of the contract starts at time 0 and terminates at the random time $\tau$. Throughout that period, the agent receives a remuneration for his work in the form of a continuous stream of payments $\pi$ and a lump-sum payment $\xi$ at termination. We assume that $\tau$ is an $\F$--stopping time, $\pi$ is an $\F$-predictable non-negative process, and $\xi$ is a non-negative $\cF_\tau$-measurable random variable. We denote by $\cC$ the collection of contracts $\bC \coloneqq (\tau, \pi, \xi)$.

\begin{remark}
\begin{enumerate}[label=$(\roman*)$, wide,  labelindent=0pt]
\item To be precise, we only need the process $(M^\nu_t)_{t \in [0,\tau]}$ to be a $\P$--uniformly integrable $(\F,\P)$-martingale, where $\tau$ represents the termination of the contract between the principal and the agent. However, it is straightforward to notice that we are not asking for more by requiring that $M^\nu$ is a $\P$--uniformly integrable $(\F,\P)$-martingale since we can simply redefine the process $\nu$ as $\nu_t = (\alpha_t,\beta_t) = (0,m) $ for all $t \in (\tau,\infty)$ if $\tau$ is finite.
\item By {\rm \cite[Theorem 15.2.6]{cohen2015stochastic}}, for any $\nu\in\Uc$, the process $W^{\nu} \coloneqq W - \int_0^\cdot (\alpha_s/\sigma) \d s$ is an $(\F,\P^{\nu})$--Brownian motion, and the compound Poisson process $J$ and the Poisson process $N$ have $\P^{\nu}$-intensity given by $\int_0^\cdot \beta_s  \d s$ and $\int_0^\cdot (\beta_s/m) \d s$, respectively. Moreover, the measure $\mu^{J}(\d t, \d \ell)$ has $(\F,\P^{{\nu}})$-compensator $\mu^{J^{\smalltext{\nu}},p}(\d t, \d \ell) \coloneqq ({\beta_t}/{m})  \Phi(\d \ell) \d t$. Notice that we highlight the dependence on the probability measure $\P^\nu$ with the superscript $J^\nu$. Accordingly, for any $\nu\in\Uc$, the output process $X$ can be expressed as 
\begin{equation*}
X_t = x_0 + \int_0^t \alpha_s \d s + \sigma W^{\nu}_t - J_t \; \text{\rm for any} \; t \geq 0,\; \P\text{\rm--a.s.}
\end{equation*}
\end{enumerate}
\end{remark}

\subsection{The problem of the agent}\label{section:agentProblem}
The preferences of the agent are determined by a utility function $u: [0, \infty) \longrightarrow [0, \infty)$, which is supposed to be increasing\footnote{Throughout the paper, the term `increasing' is understood to mean strictly increasing (and similarly for `decreasing').}, strictly concave and twice continuously differentiable. Moreover, we assume $u(0)= 0$ and the condition $\lim_{\pi \rightarrow \infty} u^{\prime}(\pi) = 0$. The function $u$ satisfies the following growth condition:
\begin{equation}\label{eq:growthUtility}
c_0 \big(-1 + \pi^{\frac{1}{\gamma}}\big) \leq u(\pi) \leq c_1 \big(1 + \pi^{\frac{1}{\gamma}}\big), \; \text{for} \; \pi \geq 0, \; \text{for some} \; (c_0,c_1) \in( 0,\infty)^2, \; \text{and some} \; \gamma > 1.
\end{equation}

We introduce the opposite of its inverse along with its concave conjugate that we will be using throughout the paper. Precisely,
\begin{equation}\label{eq:defF}
F(y)\coloneqq - u^{(-1)} (y), \; \text{for} \; y  \geq 0, \; \text{and} \; F^\star(p)\coloneqq \inf_{y \geq 0}\{ y p  - F(y)\}, \; \text{for} \; p \in \R.
\end{equation}

\begin{remark}\label{remark:FFStarProperties}
The required conditions on the agent's utility function immediately imply that $F$ is twice continuously differentiable, decreasing, strictly concave, and that $F(0) = 0$. We have
\begin{equation}\label{eq:growthF}
\tilde{c}_0 \big(-1 - y^{\gamma}\big) \leq F(y) \leq \tilde{c}_1 \big(1 -y^{{\gamma}}\big) \; \text{for any} \; y \geq 0, \; \text{for some} \; (\tilde{c}_0,\tilde{c}_1) \in( 0,\infty)^2.
\end{equation}
It is straightforward to verify that $F^\star$ is null on $[F^\prime(0),\infty)$, and it is negative on $(-\infty,F^\prime(0))$. Furthermore, $F^\star$ is twice continuously differentiable, increasing and strictly concave on $(-\infty,F^\prime(0))$. The growth condition \eqref{eq:growthUtility} imposed on the utility function of the agent results in
\begin{equation}
-c^\star_0 \big(1 + (-p)^{\frac{\gamma}{\gamma-1}}\big) \leq F^\star(p) \leq c^\star_1 \big(1 - (-p)^{\frac{\gamma}{\gamma-1}}\big), \; \text{\rm for} \; p < F^\prime(0), \; \text{\rm for some} \; (c^\star_0,c^\star_1) \in( 0,\infty)^2.
\end{equation}
\end{remark}

\medskip
Given a contract $\bC = (\tau, \pi, \xi) \in \cC$ offered by the principal, the agent's optimisation problem is
\begin{equation}\label{eq:AgentReward}
V^{\rm A}(\bC) \coloneqq \sup_{\nu \in \cU} J^{\rm A}(\bC, \nu),  \; \text{where} \; J^{\rm A}(\bC, \nu) \coloneqq \E^{\P^{\smalltext{\nu}}}\bigg[\mathrm{e}^{-r \tau} u(\xi) \mathbf{1}_{\{\tau < \infty\}}+ \int_0^\tau r \mathrm{e}^{-rs} (u(\pi_s) - h(\nu_s)) \d s\bigg].
\end{equation}
In order to simplify notations in the subsequent discussion, it is convenient to rewrite the criterion of the agent as follows
\begin{align*}
J^{\rm A}(\bC, \nu) &= \E^{\P^{\smalltext{\nu}}}\bigg[\mathrm{e}^{-r \tau} \zeta \mathbf{1}_{\{\tau < \infty\}}+ \int_0^\tau r \mathrm{e}^{-rs} (\eta_s - h(\nu_s)) \d s\bigg]\; \text{for all} \; (\bC,\nu) \in \cC \times \cU.
\end{align*}
Here, we have denoted $\zeta \coloneqq u(\xi) $ and $\eta \coloneqq u(\pi)$, respectively. Henceforth, we indifferently identify as a contract $\bC \in \cC$ the triplet $(\tau, \pi, \xi)$ or the triplet $(\tau, \eta, \zeta)$ due to the one-to-one correspondence between them.

\medskip
The agent chooses the effort $\nu\coloneqq (\alpha,\beta)$ that maximises his utility from remuneration subject to the corresponding cost $h(\nu) = h(\alpha,\beta)$. Here, we abuse notation and indifferently identify the one-argument function $h(\cdot)$ with the two-arguments function $h(\cdot,\cdot)$. We assume the cost function $h : [0,\bar{a}] \times [\eps_m,m] \longrightarrow [0, \infty)$ to be twice continuously differentiable, strictly convex, and satisfying $h(0,m) = 0$. This latter assumption captures the fact that there is no cost for exerting no effort. Moreover, we suppose that $h(\cdot, b) : [0,\bar{a}] \longrightarrow [0, \infty)$ is increasing for any $b \in [\eps_m,m]$, and that $h(a, \cdot) : [\eps_m,m]\longrightarrow [0, \infty)$ is decreasing when $a\in [0,\bar{a}]$. This description of the cost function mathematically reflects the idea that exerting effort to increase the instantaneous growth rate of the value of the project and to decrease the likelihood of negative events causes discomfort to the agent as his utility is reduced. Finally, in order to consider the time-value of money, the agent exponentially discounts future income at the constant discount rate $r > 0$.

\begin{remark}
The perspective taken in this work is that in the model we analyse here, accidents are driven exogenously and are not dependent on the agent's effort. However, if we adopt the perspective {\rm`}the greater the effort, the greater the risk{\rm '}, in which accidents are caused endogenously and depend on the consequences of the agent's actions, we can adapt the analysis in this work.\footnote{We thank an anonymous referee for suggesting this variant.} In this case, the control of the agent is solely given by $\alpha \in \mathcal{A}$, and the cost function is defined as $\tilde{h}(a) \coloneqq h(a,b(a))$, for $a \in A$ and for some $b: A \longrightarrow [\varepsilon_m, m]$, which we require to be twice continuously differentiable, increasing, strictly convex, and satisfying $b(0) = \varepsilon_m$ and $b(\bar{a}) = m$. The underlying mathematics and results remain unchanged. However, we would need to replace the constant $m$ with $\varepsilon_m$ in the respective statements.
\end{remark}

\medskip
A control $\hat{\nu} \in \cU$ is considered an optimal response to the contract $\bC\in\cC$ if $V^{\rm A}(\bC) = J^A(\bC,\hat{\nu})$. We denote by $\cU^{\star}(\bC)$ the set of all optimal responses of the agent to the contract $\bC$. In addition, as usual in contract theory, we suppose that the agent has a
reservation utility $u(R)$, for some $R \geq 0$, meaning that he is only willing to accept the contract $\bC$ if his participation constraint is satisfied, that is, $J^A(\bC,\hat\nu) \geq u(R)$. Consequently, the agent only accepts contracts in 
\begin{equation*}
\mathfrak{C}_R \coloneqq \big\{\bC \in \mathfrak{C}: V^{\rm A}(\bC) \geq u(R)\big\},
\end{equation*}
where the set $\mathfrak{C}$ denotes the collection of admissible contracts that will be introduced at the end of \Cref{section:GP}, as we need to impose some additional integrability requirements first.

\subsection{The problem of the principal}
Anticipating the optimal response of the agent, a risk-neutral principal seeks to design the contract which best serves her objective under his participation constraint. Specifically, her problem is
\begin{equation}\label{secondBest_Pproblem}
V^{\rm P} \coloneqq \sup_{\bC \in \mathfrak{C}_\smalltext{R}} \sup_{\nu \in \cU^{\smalltext{\star}}(\bC)} J^{\rm P}(\bC, \nu),
\end{equation}
where
\begin{equation}\label{eq:PrincipalReward}
J^{\rm P}(\bC, \nu) \coloneqq \E^{\P^{\smalltext{\nu}}}\bigg[-\mathrm{e}^{-\rho \tau} \xi  \mathbf{1}_{\{\tau < \infty\}} + \int_0^\tau \rho \mathrm{e}^{-\rho s} (\d X_s - \pi_s \d s)\bigg] = \E^{\P^{\smalltext{\nu}}}\bigg[-\mathrm{e}^{-\rho \tau} \xi \mathbf{1}_{\{\tau < \infty\}} + \int_0^\tau \rho \mathrm{e}^{-\rho s} (\alpha_s - \beta_s - \pi_s)\d s\bigg],
\end{equation}
as a consequence of Doob’s optional sampling theorem. With the notation introduced in the previous section and using the utility function $F$ given in \eqref{eq:defF}, we can rewrite the criterion of the principal to
\begin{equation*}
J^{\rm P}(\bC, \nu) = \E^{\P^{\smalltext{\nu}}}\bigg[\mathrm{e}^{-\rho \tau} F(\zeta)  \mathbf{1}_{\{\tau < \infty\}} + \int_0^\tau \rho \mathrm{e}^{-\rho s} (\alpha_s - \beta_s + F(\eta_s))\d s\bigg] \; \text{for any} \; (\bC,\nu) \in \cC \times \cU.
\end{equation*}
It is worth noting that the future rewards of the principal are discounted with a constant discount rate $\rho > 0$, which may differ from the one of the agent. Additionally, we adopt the convention $\sup \varnothing = -\infty$ that implies that the principal has an interest in only offering contracts $\bC$ which induce an optimal response from the agent, that is, contracts for which the set $\cU^{\star}(\bC)$ is not empty.

\subsection{Golden parachute and reformulation}\label{section:GP}

A golden parachute is a situation where the agent ceases to exert any effort but receives a compensation from the principal. Within the framework presented here, there are two ways in which a golden parachute can occur
\begin{itemize}
\item retirement: the agent is retired by the principal at some non-negative random time, and continues to receive a stream of positive payments;
\item termination: the contract is terminated at some non-negative random time because the agent is fired, and the agent receives a lump-sum compensation.
\end{itemize}

As pointed out in \cite{possamai2020there}, we will prove that although the agent is unconcerned by the two scenarii, the discrepancy between the two discount rates $r$ and $\rho$ can lead to a situation where the principal can improve her reward by retiring the agent rather than firing him. More precisely, for any given state $\omega \in \Omega$, the agent is indifferent between receiving $\xi(\omega)$ at some non-negative random time $\tau(\omega)$ or a continuous stream of payments $\pi(\omega)$ over the interval $[\tau(\omega), \tau(\omega) + T(\omega)]$ for some $T(\omega) \geq 0$, with subsequent termination of the contract postponed to time $\tau(\omega) + T(\omega)$ with a lump-sum compensation $\xi^{\prime}(\omega)$ verifying the following condition
\begin{align}\label{align:sameUtilityInformally}
u(\xi(\omega)) = \mathrm{e}^{-r T(\omega)} u(\xi^\prime(\omega)) \mathbf{1}_{\{T(\omega) < \infty\}} + \int_0^{T(\omega)} r \mathrm{e}^{-rs} u(\pi_s(\omega)) \d s.
\end{align}
On the contrary, the principal may prefer retiring the agent in this case since her criterion compared to the termination scenario can be improved by the quantity informally expressed as
\begin{equation*}
\sup_{\pi(\omega)} \sup_{T(\omega)} \bigg\{-\mathrm{e}^{-\rho T(\omega)} \xi^\prime(\omega) \mathbf{1}_{\{T(\omega) < \infty\}} + \int_0^{T(\omega)} \rho \mathrm{e}^{-\rho s} (- m - \pi_s(\omega)) \d s\bigg\} .
\end{equation*}

Therefore, it is evident that the problem of the principal exhibits the so-called face-lifting phenomenon, as her reward can be improved by the introduction of the face-lifted utility $\bar{F}$, which is defined by the following deterministic mixed control--stopping problem
\begin{equation}\label{eq:FbarEquation}
\bar{F}(y_0) \coloneqq \sup_{p \in \cB_{\smalltext{\R}_\tinytext{+}}} \sup_{T \in [0, T^{\smalltext{y}_\tinytext{0}\smalltext{,}\smalltext{p}}_\smalltext{0}]} \bigg\{\mathrm{e}^{-\rho T} F\big(y^{y_\smalltext{0},p}(T)\big) \mathbf{1}_{\{T < \infty\}} + \int_0^T \rho \mathrm{e}^{-\rho t} \big(- m + F(p(t))\big) \d t\bigg\}, \; \text{for} \; y_0 \geq 0,
\end{equation}
where the function $F$ is introduced in \eqref{eq:defF}, $\cB_{{\R}_\smalltext{+}}$ denotes the set of Borel-measurable maps from $\R_+$ to $\R_+$, and for all $(y_0, p) \in \R_+ \times \cB_{\R_\smalltext{+}}$, $T^{y_\smalltext{0},p}_0 \coloneqq \inf\{t\geq 0: y^{y_\smalltext{0},p}(t) \leq 0\}$, where the state process $y^{y_\smalltext{0},p}$ is defined by the controlled first-order ODE
\begin{equation}\label{eq:stateYdynamics}
\dot{y}^{y_\smalltext{0},p}(t) = r\big(y^{y_\smalltext{0},p}(t)-p(t)\big), \; \text{for} \; t >0,\;  y^{y_\smalltext{0},p}(0) = y_0.
\end{equation}
It follows that
\begin{equation*}
y_0 = \mathrm{e}^{-r t} y^{y_\smalltext{0},p}(t) + \int_0^t r \mathrm{e}^{-r s} p(s) \d s  \; \text{for any} \; t \in \big[0,T^{y_\smalltext{0},p}_0\big).
\end{equation*}
When comparing this equation with \eqref{align:sameUtilityInformally}, it becomes clear that $y^{y_\smalltext{0},p}(t)$ represents the continuation utility the agent gets from the lump-sum payment $u^{{(-1)}}(y^{y\smalltext{0},p}(t))$, for every time $t \in [0,T^{{y}_\smalltext{0},{p}}_0)$. It is important to emphasise that we need to restrict our attention to the subset $[0,T^{{y}_\smalltext{0},{p}}_0)$. This results from the fact that the agent cannot receive negative payments as he is protected by limited liability.

\medskip
After having introduced the face-lifted utility $\bar{F}$, it would be natural to define what is known as the relaxed criterion of the principal in \cite{possamai2020there}. However, before proceeding with this, we first need to define the set of admissible contracts. We denote by $\mathfrak{C}$ the collection of contracts $\bC = (\tau,\pi,\xi) \in \cC$ satisfying the integrability condition
\begin{equation}\label{eq:integrabilityCondition}
\sup_{\nu \in \cU} \E^{\P^{\smalltext{\nu}}}\bigg[\mathrm{e}^{-r^\smalltext{\prime} q\tau} \xi^{\frac{q}{\gamma}} \mathbf{1}_{\{\tau < \infty\}} + \int_0^\tau \mathrm{e}^{-r^\smalltext{\prime} q s} \pi_s^{\frac{q}{\gamma}} \d s\bigg]<\infty,
\end{equation}
for some $r^\prime \in (0,r)$ and $q>2 \vee \gamma $.

\begin{remark}
\begin{enumerate}[label=$(\roman*)$, wide, labelindent=0pt]
\item Note that the value function of the agent $V^{\rm A}(\bC)$ is finite for any $\bC \in \mathfrak{C}$, that is, for any contract that verifies the integrability condition \eqref{eq:integrabilityCondition}. Additionally, the value function of the principal $V^{\rm P}$ is locally bounded; it is trivially bounded from above by $\bar{a} - \varepsilon_m$ and locally bounded from below since there exists at least a contract $\bC \in \mathfrak{C}_R$ such that the set $\cU^{{\star}}(\bC)$ is not empty. 
\item The requirement $q>2$ in the integrability condition \eqref{eq:integrabilityCondition} is necessary in order to apply {\rm \citeauthor*{kruse2015bsdes} \cite[Proposition 2]{kruse2015bsdes} }in the reduction argument in {\rm\Cref{appendix:reductionSecond}}. Specifically, this condition is crucial for the application of the Burkholder--Davis--Gundy inequality for non-continuous local martingales.
\end{enumerate}
\end{remark}

Analogously to \cite{possamai2020there}, we can introduce the collection $\mathfrak{C}^0$ of admissible contracts $\bC^0 =  (\tau^0, \pi^0, \xi^0) \in \mathfrak{C}$ where
\begin{equation*}
\tau^0 \coloneqq \tau + T, \; \pi^0 \coloneqq \pi \; \mathbf{1}_{\llbracket 0,\tau \rrbracket} + p \; \mathbf{1}_{\rrbracket \tau,\tau^0 \rrbracket}, \; \text{and} \; u(\xi^0) \coloneqq y^{u(\xi),p}(T),
\end{equation*}
for some $\bC = (\tau,\pi,\xi) \in \mathfrak{C}$, $\cF_{\tau}$-measurable $p$ with values in $\cB_{\R_\smalltext{+}}$, and $T \in [0,T^{{u(\xi)},{p}}_0]$. Since it can be easily shown that the two sets $\mathfrak{C}$ and $\mathfrak{C}^0$ coincide, we have the following equivalent formulations of the problem of the principal
\begin{equation*}
V^{\rm P} \coloneqq \sup_{\bC \in \mathfrak{C}_\smalltext{R}} \sup_{\nu \in \cU^{\smalltext{\star}}(\bC)} J^{\rm P}(\bC, \nu) = \bar{V}^{\rm P} \coloneqq \sup_{\bC \in \mathfrak{C}_\smalltext{R}} \sup_{\nu \in \cU^{\smalltext{\star}}(\bC)} \bar{J}^{\rm P}(\bC, \nu),
\end{equation*}
where the relaxed reward of the principal is given by 
\begin{equation}\label{eq:newVprincipal}
\bar{J}^{\rm P}(\bC, \nu) \coloneqq  \E^{\P^{\smalltext{\nu}}}\bigg[\mathrm{e}^{-\rho \tau} \bar{F}(\zeta)  \mathbf{1}_{\{\tau < \infty\}} + \int_0^\tau \rho \mathrm{e}^{-\rho s} \big(\alpha_s - \beta_s + F(\eta_s)\big) \d s\bigg], \; \text{for} \; (\bC,\nu) \in \cC \times \cU.
\end{equation}

\medskip
Consequently, the introduction of the face-lifted utility $\bar{F}$ and its corresponding relaxed reward for the principal plays a crucial role as it combines both notions of a golden parachute in the termination scenario.

\begin{definition}\label{def:goldenparachute}
The relaxed contracting problem $\bar{V}^{\rm P}$ exhibits a golden parachute if $\xi^{\star} > 0$ on the event set $\{\tau^{\star} < \infty\}$, for any optimal contract $\bC^{\star} = (\tau^{\star}, \pi^{\star}, \xi^{\star}) \in \mathfrak{C}$.
\end{definition}
It is worth noticing that in the definition of a golden parachute, the optimal termination $\tau^{\star}$ of the contract $\bC^{\star}$ can be null, indicating that the agent receives some remuneration without starting to work. This scenario is a consequence of the formulation of the problem, as the principal cannot choose not to hire the agent, and thus may simply decide to hire him and terminate the contract simultaneously.

\section{On the face-lifted utility}\label{section:faceliftedAnalysis}
We first start by investigating the problem $\eqref{eq:FbarEquation}$ to provide an explicit characterisation of the face-lifted utility $\bar{F}$. The associated Hamilton--Jacobi equation is given by
\begin{align}\label{align:hjFbar}
\min\big\{\bar{F}(y)- F(y), F^\star(\delta \bar{F}^\prime(y)) - \delta y\bar{F}^\prime(y) + \bar{F}(y) + m\big\}=0,\; y\in(0,\infty), \; \bar{F}(0) = 0,
\end{align}
where $\delta \coloneqq r/\rho$, and $F^\star$ denotes the concave conjugate of $F$ introduced in \eqref{eq:defF}.

\medskip
The subsequent sections are dedicated to the analysis of the aforementioned Hamilton--Jacobi equation to compute the face-lifted utility $\bar{F}$ in closed-form. We prove that when the agent and the principal are equally impatient, namely when $\delta = 1$, the principal never benefits from delaying the termination of the contract and allowing the agent to exert no effort for a while. Only this particular scenario corresponds exactly to the analogous one without accidents studied in \cite{possamai2020there}. Indeed, when the two discount factors differ, the possibility of accidents may lead the principal to consider termination as a more favourable option. However, this decision is never optimal when $m =0$, unless the continuation utility reaches the value $0$.

\medskip
If the agent is more impatient than the principal, and thus $\delta >1$, we need to divide the analysis into two cases based on the size of possible accidents. When the average size per accident is small, precisely $m \leq - F^\star (\delta F^\prime(0))$, we show that nothing differs from the results in \cite{possamai2020there} because the face-lifted function $\bar{F}$ is always strictly above $F$, meaning that the retirement scenario is always preferred by the principal. This decision originates from the principal's relatively limited concern for accidents, given that the potential losses she might face are not significant. 

\medskip
However, when the size of the accidents becomes larger, specifically when $m$ is above the aforementioned threshold, the decision to retire the agent is no longer optimal in general because the two functions $\bar{F}$ and $F$ coincide for small values of the continuation utility of the agent. This difference can be attributed solely to the principal's concern regarding accidents since the occurrence of an accident now significantly impacts the value of the project. Consequently, the principal chooses to terminate the contract by firing the agent when his continuation utility decreases sufficiently. This is in stark contrast with the results of \cite{possamai2020there}, where termination is never optimal for $\delta>1$.

\medskip 
When the principal becomes much more impatient than the agent, namely $\rho \geq \gamma r$, the problem degenerates exactly as in \citep{possamai2020there} since $\bar{F}$ coincides with $F$ until the latter reaches the value $-m$, after which it stays identically equal to $-m$. Consequently, for small values of the continuation utility of the agent, the presence of accidents leads the principal to prefer dismissing the agent rather than risking the occurrence of an accident that is more costly for her. Conversely, when the continuation utility of the agent increases sufficiently, the principal chooses not to terminate the contract, and the cost of retiring the agent is given exactly by the average loss incurred in case of an accident.

\medskip
Finally, in case where the principal is still more impatient than the agent but $\rho < r \gamma$, there is one point at which the concavity of $\bar{F}$ breaks, a feature which never occurred in \citep{possamai2020there}, where $\bar F$ is always concave. In fact, the face-lifted utility $\bar{F}$ coincides with $F$ up to this point, and is then a non-trivial majorant of $F$. This implies that the principal can prefer dismissing the agent or retiring him based on the actual level of his continuation utility, owing to the combined effect of her impatience and her concerns regarding possible accidents. Surprisingly, unlike in the case where $\delta >1$, the introduction of potential losses ensures that termination is always preferred over retiring the agent for sufficiently small values of the agent's continuation utility. For every positive value of $m$, termination of the contract is a possibility, which never arises in the analysis without accidents, namely when $m=0$.

\subsection{The equally impatient case}
Here we examine the case where $r = \rho$, or equivalently, $\delta = 1$.
\begin{proposition}\label{theorem:FbarDelta1}
When $\delta=1$, we have $\bar{F}=F$ on $[0,\infty)$.
\end{proposition}
\begin{proof}
Let us fix some $y_0 \geq 0$. The definition of the control--stopping problem in \Cref{eq:FbarEquation} easily implies that $\bar{F}(y_0) \geq F(y_0)$. To show the reverse inequality, we fix some $p \in \cB_{\R_\smalltext{+}}$ and $T \in [0, T^{y_\smalltext{0},p}_0]$. If $T < \infty$, the fundamental theorem of calculus implies that 
\begin{equation*}
\mathrm{e}^{-\rho T} F(y^{y_\smalltext{0},p}(T)) = F(y_0) + \int_0^T \rho \mathrm{e}^{-\rho t} \Big(- F\big(y^{y_\smalltext{0},p}(t)\big) - F^\prime\big(y^{y_\smalltext{0},p}(t)\big) \big(p(t) - y^{y_\smalltext{0},p}(t)\big)\Big) \d t,
\end{equation*}
given that the state $y^{y_\smalltext{0},p}$ is determined by the ODE \eqref{eq:stateYdynamics} and the fact that $\rho = r$. Accordingly, it holds that 
\begin{align*}
&\mathrm{e}^{-\rho T} F(y^{y_\smalltext{0},p}(T)) + \int_0^T \rho \mathrm{e}^{-\rho t} \big(- m + F(p(t))\big) \d t \\
&= F(y_0) + \int_0^T \rho \mathrm{e}^{-\rho t} \Big(- F\big(y^{y_\smalltext{0},p}(t)\big) - F^\prime\big(y^{y_\smalltext{0},p}(t)\big) \big(p(t) - y^{y_\smalltext{0},p}(t) \big) + F(p(t)) - m \Big) \d t \leq  F(y_0) + m\big (\mathrm{e}^{-\rho T} - 1\big) \leq F(y_0),
\end{align*}
due to the concavity of the function $F$. On the other hand, if $T = \infty$, similar computations as before lead to
\begin{align*}
\lim_{T^\prime \rightarrow \infty} \bigg(\mathrm{e}^{-\rho T^\prime} F(y^{y_\smalltext{0},p}(T^\prime)) \mathbf{1}_{\{T^\prime < \infty\}}  + \int_0^{T^\prime} \rho \mathrm{e}^{-\rho t} \big(- m + F(p(t))\big) \d t\bigg) \leq F(y_0).
\end{align*}
Thus, we can conclude that $\bar{F}(y_0) \leq F(y_0)$ by arbitrariness of $p \in \cB_{\R_\smalltext{+}}$ and $T \in [0, T^{y_\smalltext{0},p}_0]$.
\end{proof}

\subsection{Different discount rates}
In general, the Hamilton--Jacobi equation \eqref{align:hjFbar} corresponding to the mixed control--stopping problem $\bar{F}$ cannot be simplified into an ODE as presented in \cite[Equation A.1]{possamai2020there}. Another challenge arises from the positivity of the constant $m$, implying that there may exist a non-empty set $\mathcal{I}$ on which $F$ is a super-solution of the ODE described in \eqref{align:hjFbar}, thereby possibly resulting in $\bar{F}$ coinciding with $F$ on $\mathcal{I}$. In cases where the discount factors $r$ and $\rho$ differ, there is a possibility that $\mathcal{I}$ could be disconnected. To prevent this occurrence, we introduce the following assumption, which, in economic terms, suggests that if the principal prefers postponing the termination of the contract for a certain continuation utility of the agent, she maintain the same preference for larger continuation utility, assuming all other factors remain constant.

\begin{assumption}\label{assumption:FdeltamDec}
If $\delta \neq 1$, the function $F_{\delta m}:[0,\infty)\longrightarrow \R$ defined by
\begin{align}\label{eq:fdecreasing}
F_{\delta m}(y) \coloneqq F^\star\big(\delta F^\prime(y)\big) - \delta y F^\prime(y) + F(y) + m, \; \text{for} \; y \geq 0,
\end{align}
is decreasing and such that $\lim_{y \rightarrow \infty} F_{\delta m}(y) = -\infty$.
\end{assumption}
The assumption naturally implies the existence of the following point
\begin{align}\label{eq:ybar}
\bar{y}\coloneqq \inf\big\{y \geq 0: F_{\delta m}(y) \leq 0\big\} \in [0, \infty),
\end{align}
and we will implicitly assume this throughout the paper. We show below that \Cref{assumption:FdeltamDec}, which is introduced for mathematical convenience, also holds in a simple and reasonable example. 
\begin{remark}
It is important to highlight that within the accident-free framework examined by \cite{possamai2020there}, \Cref{assumption:FdeltamDec} is unnecessary. Specifically, when $m=0$ $($and $\delta \neq 1)$, the function $F_{\delta m}$ is strictly negative on $(0, \infty)$ due to the strict concavity of $F$ over the same interval. Consequently, the latter is a sub-solution, not a solution, to the {\rm ODE} in \eqref{align:hjFbar}. This observation is also confirmed by {\rm\cite[Proposition 2.1]{possamai2020there}}, which establishes that, under these conditions, the face-lifted utility $\bar{F}$ is a strict majorant of $F$.
\end{remark}

\begin{example}\label{example:positivePower}
Suppose $\gamma >1$ and consider the utility function $u(\pi) = \pi^{\frac{1}{\gamma}}$, for $\pi \geq 0$. Then,
\begin{equation*}
F(y) = - y^\gamma, \; \text{\rm for} \; y \geq 0, \; \text{\rm and} \; F^{\star}(p) = - (\gamma -1) \bigg({\frac{-p}{\gamma}}\bigg)^{\frac{\gamma}{\gamma - 1}} \; \mathbf{ 1}_{(-\infty,0]}(p), \; \text{\rm for} \; p \in \R.
\end{equation*}
It is straightforward to verify that {\rm\Cref{assumption:FdeltamDec}} holds. Indeed,
\begin{equation*}
F_{\delta m}(y) = \big((1 - \gamma) \delta^{\frac{\gamma}{\gamma - 1}} + \delta \gamma - 1\big) y^\gamma + m ,\; y \geq 0.
\end{equation*}
The factor in front of the positive power function $y^\gamma$ is negative when $\delta \neq 1$, which implies that the map $F_{\delta m}$ is decreasing with $\lim_{y \rightarrow \infty} F_{\delta m}(y) = -\infty$. In particular, we find that
\begin{align*}
\bar{y}= \left(\frac{m}{(\gamma - 1) \delta^{\frac{\gamma}{\gamma - 1}} - \delta \gamma + 1}\right)^\frac{1}{\gamma}, \; \text{\rm and consequently} \; F(\bar{y}) = - \frac{m}{(\gamma - 1) \delta^{\frac{\gamma}{\gamma - 1}} - \delta \gamma + 1}.
\end{align*}
\end{example}

\medskip
In the scenario where $\delta > 1$, the definition of the point $\bar{y}$ in \eqref{eq:ybar} suffices to describe the face-lifted utility function $\bar{F}$, while the case $\delta <1$ requires the introduction of two additional fundamental points. The first one is defined as
\begin{align}\label{eq:yHat}
\hat{y} \coloneqq \inf\{y \geq 0: F(y) \leq -m\} \in (0,\bar{y}).
\end{align}
Given the assumption $\delta < 1$, we initially note that the interval $(0,\bar{y})$ is non-degenerate, as the fact that $F_{\delta m}(0) = F^{\star}(\delta F^{\prime}(0)) + m = m > 0$ implies that $\bar{y}$ is positive. The proof that $\hat{y}$ belongs to $(0,\bar{y})$, or in other words, that $\hat{y}$ is smaller than $\bar{y}$, is given in \Cref{lemma:FyBarlessMinusM}, which we will introduce later for clarity. Furthermore, if, along with $\delta < 1$, we require $\gamma\delta > 1$, as will be proven in \Cref{lemma:existenceOfyTilde}, we can define
\begin{align}\label{eq:yTilde}
\tilde{y} \coloneqq \{y \in (\hat{y},\bar{y}): F(y) = w_0(y) -m\},
\end{align}
where $w_0$ is the face-lifted reward introduced in \cite[Proposition 2.1]{possamai2020there}, which is given by
\begin{align}\label{align:w0nojumps}
w_0(y) \coloneqq (w^\star_0)^\star(y) = \inf_{p \leq 0} \{ y p - w^\star_0(p)\},\;\text{for}\;y\geq 0, \; \text{where} \; w^\star_0(p)\coloneqq \int_0^\infty \rho \mathrm{e}^{-\rho t} F^\star(\delta \mathrm{e}^{\rho (1-\delta)t}p) \d t, \; \text{for} \; p \leq 0.
\end{align}
Furthermore, \cite[Proposition 2.1]{possamai2020there} proves that $w_0$ is twice-continuously differentiable, decreasing with $w_0^\prime(0) = 0$, and satisfies \begin{equation}\label{eq:eqw0NoJumps}
F^\star(\delta w_0^\prime(y)) - \delta y w_0^\prime(y) + w_0(y) =0,\; y\in(0,\infty), \; w_0(0) = 0.
\end{equation}
Consequently, both the function $w_0$ and its first derivative $w_0^\prime$ are negative on the interval $(0,\infty)$.

\subsubsection{A more impatient agent}
Throughout this section, we characterise the face-lifted reward function $\bar{F}$ in the case where $\delta >1$.
\begin{proposition}\label{proposition:FBarGeneralDeltaBigger1}
Let us assume that $\delta >1$. We introduce the function
\begin{align*}
w^\star(p) \coloneqq \frac{(-p)^{\frac{1}{1-\delta}}}{\delta-1} \int_p^{p^\smalltext{\star}} \frac{F^\star(\delta x)}{(-x)^{1+\frac{1}{1-\delta}}} \d x + m \bigg(1- \bigg(\frac{p}{p^\star}\bigg)^{\frac{1}{1-\delta}}\bigg) + \bigg(\frac{p}{F^\prime(\bar{y})}\bigg)^{\frac{1}{1-\delta}} F^{\star}(F^\prime(\bar{y})) \mathbf{1}_{\{m > - F^\smalltext{\star}(\delta F^\smalltext{\prime}(0))\}}, \; \text{for} \; p \leq p^\star,
\end{align*}
where $p^\star \coloneqq ((F^\star)^{(-1)}(-m)/\delta) \mathbf{1}_{\{m \leq - F^\smalltext{\star}(\delta F^\smalltext{\prime}(0))\}} + F^\prime(\bar{y}) \mathbf{1}_{\{m > - F^\smalltext{\star}(\delta F^\smalltext{\prime}(0))\}}$. Then, the face-lifted utility function $\bar{F}$ is such that
\begin{enumerate}[label=$(\roman*)$]
\item \label{mSmall}if $m \leq - F^\star(\delta F^\prime(0))$, we have $\bar{F}(y) = (w^\star)^\star(y)$, for $y \geq 0;$
\item \label{mLarge}if $m > - F^\star(\delta F^\prime(0))$, we have $\bar{F}(y) = F(y)  \mathbf{1}_{[0,\bar{y}]}(y) + (w^\star)^\star(y) \mathbf{1}_{(\bar{y},\infty)}(y) $, for $y \geq 0$.
\end{enumerate}
In particular, $\bar{F}$ is continuously differentiable with $\bar{F}^\prime(0) = ((F^\star)^{(-1)}(-m)/\delta) \mathbf{1}_{\{m \leq - F^\smalltext{\star}(\delta F^\smalltext{\prime}(0))\}} + F^\prime(0) \mathbf{1}_{\{m > - F^\smalltext{\star}(\delta F^\smalltext{\prime}(0))\}} \leq 0$. Moreover, it is decreasing and strictly concave, and satisfies the following growth condition:
\begin{equation*}
\bar{c}_0 \big(-1 - y^{{\gamma}}\big) \leq \bar{F}(y) \leq \bar{c}_1 \big(1 - y^{{\gamma}}\big), \; \text{for} \; y \geq 0, \; \text{\rm for some} \; (\bar{c}_0,\bar{c}_1) \in( 0,\infty)^2.
\end{equation*}
\end{proposition}
The proof is postponed to \Cref{appendix:face-lifted utility}.

\begin{remark}\label{remark:onlyOneCaseDeltaBigger1}
{\rm \Cref{proposition:FBarGeneralDeltaBigger1}} distinguishes between two cases based on the size of the average loss per accident. When this value is small, specifically less than $- F^\star(\delta F^\prime(0))$, the point $\bar{y}$ introduced in \eqref{eq:ybar} becomes zero, and therefore the function $F$ cannot be a solution of the Hamilton--Jacobi equation \eqref{align:hjFbar}. Conversely, when $m$ is strictly above the aforementioned threshold, $F$ solves the equation on the non-degenerate interval $[0,\bar{y}]$.
\end{remark}

\begin{example}\label{example:positivePowerDeltaBigger1}
In the specific positive power utility setting of {\rm\Cref{example:positivePower}}, only the second scenario described in {\rm\Cref{proposition:FBarGeneralDeltaBigger1}} occurs, as $F^\star(\delta F^\prime(0)) = 0$. Consequently, the face-lifted utility $\bar{F}$ coincides with $F$ on $[0,\bar{y}]$, and then it is described by the concave conjugate of $w^{\star}$, which is given by
\begin{align*}
w^{\star}(p) =-\left(\frac{-p}{\gamma}\right)^{\frac{1}{1-\delta}} \left(\frac{m}{1 - \gamma\delta + (\gamma-1)\delta^{\frac{\gamma}{\gamma-1}}}\right)^{\frac{\gamma-1}{\gamma(\delta-1)}} \left(\frac{m \gamma (\delta - 1)}{\gamma\delta -1}\right) - (-p)^{\frac{\gamma}{\gamma-1}} \frac{(1-\gamma)^2}{\gamma\delta-1} \left(\frac{\delta}{\gamma}\right)^{\frac{\gamma}{\gamma-1}} + m, \;\text{for} \; p \leq F^{\prime}(\bar{y}).
\end{align*}
Unfortunately, providing an explicit expression for the face-lifted reward $\bar{F}$ on $(\bar{y},\infty)$ is hindered by the presence of the two different powers that characterise the above formula. Nevertheless, we can approximate $\bar{F}$ using numerical techniques, as it is shown in the graphs at the end of this section.
\end{example}

\subsubsection{When the principal becomes strictly more impatient}
In our current scenario where $\delta < 1$, our aim is to provide a comprehensive characterisation of the solution to the Hamilton--Jacobi equation \eqref{align:hjFbar}. As previously noted, the point $\bar{y}$ introduced in \eqref{eq:ybar} is positive, suggesting the possibility of the solution aligning with the function $F$ in a right-neighbourhood of zero. We will demonstrate that it indeed aligns with $F$, at least until $F$ reaches the value of $-m$, or equivalently, until the point $\hat{y}$ defined in \eqref{eq:yHat} is attained. We prove the relation that $\hat{y}$ is smaller than $\bar{y}$ in the following result.

\begin{lemma}\label{lemma:FyBarlessMinusM}
If $\delta < 1$, then it holds that $F(\bar{y}) < -m$. 
\end{lemma}
\begin{proof}
It is sufficient to prove that $F^{\star}(\delta F^{\prime}(\bar{y})) - \delta \bar{y} F^{\prime}(\bar{y}) > 0$ since $F^{\star}(\delta F^{\prime}(\bar{y})) - \delta \bar{y} F^{\prime}(\bar{y})  + F(\bar{y}) + m = F_{\delta m}(\bar{y}) = 0$ due to continuity. To achieve this, consider $y^{\star}$ as the maximiser in the definition of $F^{\star}(\delta F^{\prime}(\bar{y}))$, which is the unique value $y^{\star}>0$ such that $\delta F^{\prime}(\bar{y}) = F^{\prime}(y^{\star})$ when $\delta F^{\prime}(\bar{y}) < F^{\prime}(0)$, otherwise $y^{\star} = 0$. In the first case, the condition $\delta < 1$ implies that $\bar{y} > y^{\star}$, allowing us to conclude that
\begin{align*}
F^{\star}(\delta F^{\prime}(\bar{y})) - \delta \bar{y} F^{\prime}(\bar{y}) = \delta F^{\prime}(\bar{y}) (y^{\star}-\bar{y}) - F(y^{\star}) > 0,
\end{align*}
since the function $F$ is decreasing. In the second case, it follows that $F^{\star}(\delta F^{\prime}(\bar{y})) - \delta \bar{y} F^{\prime}(\bar{y}) = -\delta \bar{y} F^{\prime}(\bar{y}) > 0$ as $y^{\star} = 0$. This completes the proof.
\end{proof}

It is now fundamental to distinguish between two distinct cases, depending on $\gamma\delta \leq 1$ or $\gamma\delta > 1$. Specifically, when $\gamma\delta \leq 1$, the problem \eqref{eq:FbarEquation} degenerates, leading to the face-lifted utility $\bar{F}$ coinciding with the function $F$ up to the point $\hat{y}$, after which $\bar{F}$ remains constant at the value of $-m$. Conversely, if $\delta < 1$ and $\gamma\delta > 1$, we show that the two aforementioned functions coincide up to the point $\hat{y}$ within the interval $(\hat{y},\bar{y})$ introduced in \eqref{eq:yTilde}. Beyond this point, $\bar{F}$ adopts the form of an affine transformation of the face-lifted reward without jumps $w_{0}$ recalled in \eqref{align:w0nojumps}.

\begin{lemma}\label{lemma:existenceOfyTilde}
Let $\delta<1$ and $\gamma\delta >1$. There exists a unique $\tilde{y} \in (\hat{y},\bar{y})$ that satisfies $F(\tilde{y}) = w_0(\tilde{y}) - m$.
\end{lemma}
\begin{proof}
We show that the function $\chi(y) \coloneqq F(y) - w_0(y) + m$, for $y \geq 0$, has a unique zero on the interval $(\hat{y},\bar{y})$. To achieve this, we assert that $\chi$ is decreasing on $(0,\infty)$; however, the proof is postponed to the end for the sake of clarity. As a result, the statement follows directly because $\chi(\hat{y})> 0$ and $\chi(\bar{y}) < 0$. The former inequality arises from the definition of the point $\hat{y}$ in \eqref{eq:yHat}, given that $\chi(\hat{y}) = F(\hat{y}) - w_0(\hat{y}) + m = - w_0(\hat{y}) > 0$. To prove that $\chi(\bar{y}) < 0$, we first observe that the definition of the point $\bar{y}$ stated in \eqref{eq:ybar} and the \Cref{eq:eqw0NoJumps} satisfied by $w_0$ lead to the following:
\begin{align*}
\upsilon_{\bar{y}}(w_0^\prime(\bar{y})) - \upsilon_{\bar{y}}(F^\prime(\bar{y})) = F^\star\big(\delta w_0^\prime(\bar{y})\big) - \delta \bar{y} w_0(\bar{y}) - \big(F^\star\big(\delta F^\prime(\bar{y})\big) - \delta \bar{y} F^\prime(\bar{y})\big)= F(\bar{y}) - m - w_0(\bar{y}) = \chi(\bar{y}),
\end{align*}
where $\upsilon_{\bar{y}}(p) \coloneqq F^{\star}(\delta p) - \delta \bar{y} p$, for $p \in \R$. By noticing that $\upsilon_{\bar{y}}$ is decreasing on $(F^\prime(\bar{y})/\delta, \infty)$ and $F^\prime(\bar{y})/\delta < F^\prime(\bar{y}) < w_0^\prime(\bar{y})$, we can conclude that $\chi(\bar{y}) < 0$. The proof is complete once we verify the claim that $F^\prime < w_0^\prime$ on $(0,\infty)$. We proceed by contradiction by assuming there exists some $y_0 \in (0,\infty)$ such that $F^\prime(y_0) \geq w_0^\prime(y_0)$. Then, by differentiating the ODE satisfied by $w_0$ and the function $F_{\delta m}$, combined with \Cref{assumption:FdeltamDec}, we can observe that 
\begin{align*}
\delta w^{\prime\prime}_{0}(y) (y - (F^{\star})^{\prime}(\delta w^{\prime}_{0}(y))) = w^{\prime}_{0}(y) (1-\delta) \; \text{and} \; \delta F^{\prime\prime}(y) (y - (F^{\star})^{\prime}(\delta F^{\prime}(y))) > F^{\prime}(y) (1-\delta) \; \text{for any} \; y \in (0,\infty).
\end{align*}
Given that $\delta < 1$, we deduce that 
\begin{align*}
\delta w^{\prime\prime}_{0}(y_0) (y_0 - (F^{\star})^{\prime}(\delta w^{\prime}_{0}(y_0))) = w^{\prime}_{0}(y_0) (1-\delta) \leq F^{\prime}(y_0) (1-\delta) &< \delta F^{\prime\prime}(y_0) (y_0 - (F^{\star})^{\prime}(\delta F^{\prime}(y_0)))\\& \leq \delta F^{\prime\prime}(y_0) (y_0 - (F^{\star})^{\prime}(\delta w^{\prime}_{0}(y_0))) ,
\end{align*}
where we have used the fact that $(F^{\star})^{\prime}(\delta w^{\prime}_{0}(y_0)) \geq (F^{\star})^{\prime}(\delta F^{\prime}(y_0))$ as $w^{\prime}_0(y_0) \leq F^{\prime}(y_0)$. It follows that $w^{\prime\prime}_{0}(y_0) < F^{\prime\prime}(y_0)$. Now, let us define the set of sets
\begin{align*}
\cI^{y_\smalltext{0}} \coloneqq \{(y_0,y^\ast): w^{\prime}_{0}(y) \leq F^{\prime}(y) \; \text{for any} \; y \in (y_0,y^\ast)\}.
\end{align*}
It holds that $\cup_{\cI \in \cI^{\smalltext{y}_\tinytext{0}}} \cI = (y_0, b_0)$ for some $b_0 \in (y_0, \infty)$. We have $b_0 = \infty$; otherwise, the same argument as before would lead to the existence of some $\varepsilon > 0$ such that $\cup_{\cI \in \cI^{\smalltext{y}_\tinytext{0}}} \cI = (y_0, b_0) \subset (y_0, b_0 + \varepsilon) \subseteq \cup_{\cI \in \cI^{\smalltext{y}_\tinytext{0}}} \cI$. Consequently, $w^{\prime}_{0} \leq F^{\prime}$, and $w^{\prime\prime}_{0} < F^{\prime\prime}$ on $[y_0, \infty)$. As a result, the function $w_{0} - F$ is strictly concave and decreasing on $(y_0, \infty)$, implying that $\lim_{y \rightarrow  \infty} (w_{0} - F)(y) = -\infty$. However, this contradicts the fact that $w_{0}$ dominates $F$ on $(0, \infty)$ (see \cite[Lemma A.3]{possamai2020there}).
\end{proof}

\begin{proposition}\label{proposition:FBarGeneralDeltaSmaller1}
Let us assume that $\delta < 1$. Then, the face-lifted utility function $\bar{F}$ can be described as follows:
\begin{enumerate}[label=$(\roman*)$]
\item \label{degenerateFBar}if $\gamma\delta \leq 1$, we have $\bar{F}(y) = F(y) \mathbf{1}_{[0,\hat{y}]}(y) -m  \mathbf{1}_{(\hat{y},\infty)}(y) $, for $y \geq 0;$
\item \label{nonDegenerateFBar}if $\gamma\delta > 1$, we have $\bar{F}(y) = F(y)  \mathbf{1}_{[0,\tilde{y}]}(y) + (w_0(y)-m)  \mathbf{1}_{(\tilde{y},\infty)}(y) $, for $y \geq 0$,
\end{enumerate}
where $\hat{y}$ is defined in {\rm\Cref{eq:yHat}}, whereas $\tilde{y}$ is introduced in {\rm\Cref{lemma:existenceOfyTilde}}. It holds that $\bar{F}$ is twice-continuously differentiable, except at a single point. Furthermore, when $\gamma\delta > 1$, it is decreasing and exhibits the following growth:
\begin{equation*}
\bar{c}_0 \big(-1 - y^{{\gamma}}\big) \leq \bar{F}(y) \leq \bar{c}_1 \big(1 - y^{{\gamma}}\big), \; \text{for} \; y \geq 0, \; \text{for some} \; (\bar{c}_0,\bar{c}_1) \in( 0,\infty)^2.
\end{equation*}
\end{proposition}
This result is proved in \Cref{appendix:face-lifted utility}.

\begin{example}\label{example:positivePowerDeltaSmaller1}
Within the specific framework of the positive power utility as described in  {\rm\Cref{example:positivePower}}, we find that
\begin{align*}
\hat{y} = m^{\frac{1}{\gamma}}, \; \text{\rm and} \; \tilde{y} = m^{\frac{1}{\gamma}} \bigg(1- \frac{1}{\delta}\bigg(\frac{\gamma\delta - 1}{\delta(\gamma-1)}\bigg)^{\gamma-1}\bigg)^{-\frac{1}{\gamma}}.
\end{align*}
Consequently, if $\gamma\delta>1$, then 
\begin{align*}
\bar{F}(y) = -y^\gamma  \mathbf{1}_{[0,\tilde{y}]}(y) - \bigg(\bigg(\frac{\gamma\delta-1}{\gamma-1}\bigg)^{\gamma - 1} \bigg(\frac{y}{\delta}\bigg)^\gamma+m\bigg)  \mathbf{1}_{(\tilde{y},\infty)}(y), \;\text{\rm for} \; y \geq 0.
\end{align*}
\end{example}

\subsection{Numerical simulations}
In order to gain a deeper insight into the implications stemming from the accidents incorporated in \citeauthor*{sannikov2008continuous}'s contracting problem explored by \cite{possamai2020there}, we perform some numerical simulations.  Initially, we consider the specific positive power utility setting outlined in \Cref{example:positivePower}, which we elaborate on in \Cref{example:positivePowerDeltaBigger1} and \Cref{example:positivePowerDeltaSmaller1}, depending on the value of the ratio $\delta$ of the impatient levels of both parties. For this analysis, we set $\gamma = 2$ and $m=6$. Specifically, in the two graphs provided in \Cref{fig:faceliftedcompWithFw0}, we illustrate the face-lifted utility $\bar{F}$ and compare it with its barrier $F$ and the face-lifted utility $w_0$ that characterises the problem without accidents. As expected, the latter is always a strict upper bound, particularly since termination is never optimal when $m=0$ but becomes an option for small values of the continuation utility of the agent when $m>0$. However, as the continuation utility exceeds the value $\bar{y}$ introduced in \eqref{eq:ybar} for the case $\delta >1$, or respectively $\tilde{y}$ introduced in \eqref{eq:yTilde} for $\delta < 1$, the principal prefers retiring the agent since her reward $\bar{F}$ begins to deviate from the barrier $F$. This results from the requirement to ensure that the agent receives a sufficiently high utility, leading the principal to prefer the risk of potential future losses.

\begin{figure}[ht!]
  \centering
  \subfloat[$\delta = 2$]{\includegraphics[width=0.48\textwidth]{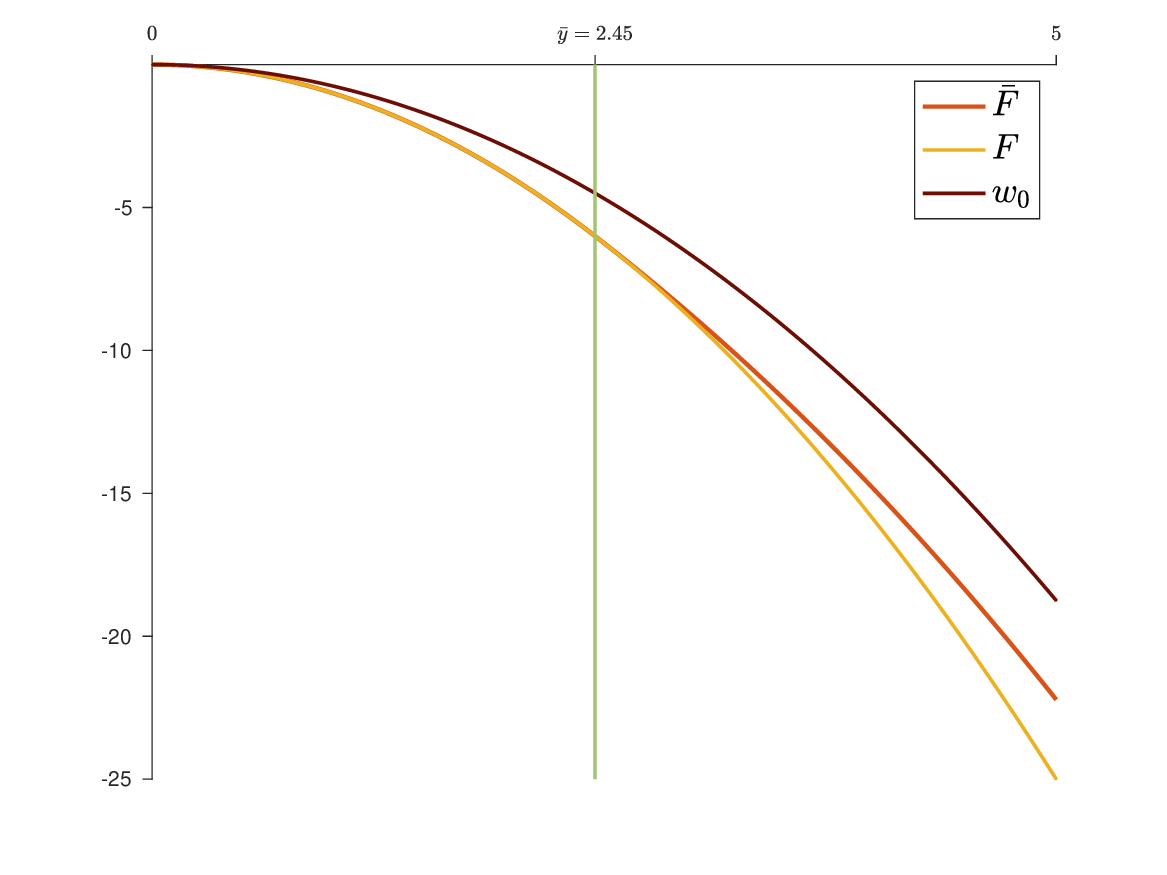}}
  \hfill
  \subfloat[$\delta = 3/4$]{\includegraphics[width=0.48\textwidth]{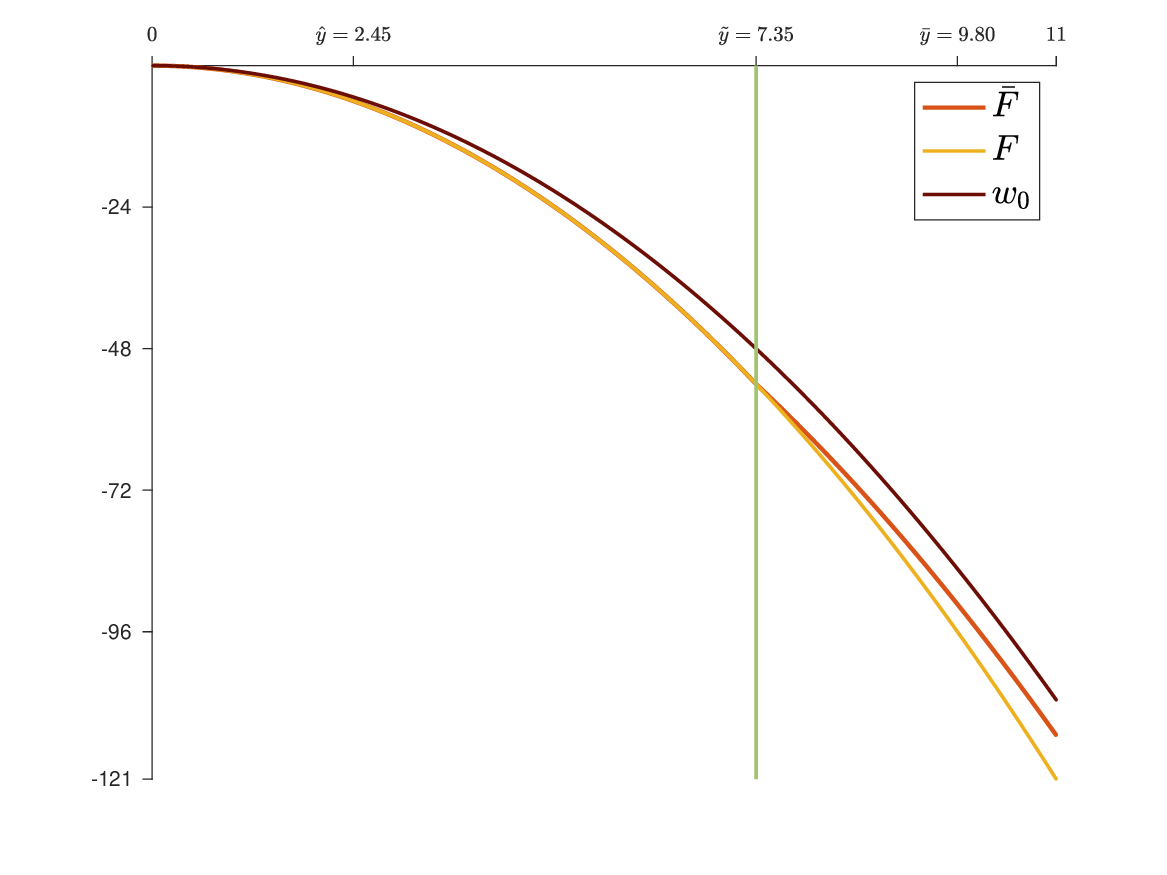}}
  \caption{the face-lifted utility $\bar{F}$}
  \label{fig:faceliftedcompWithFw0}
\end{figure}

\medskip
\begin{figure}[ht!]
  \centering
  \subfloat[$\delta = 2$]{\includegraphics[width=0.48\textwidth]{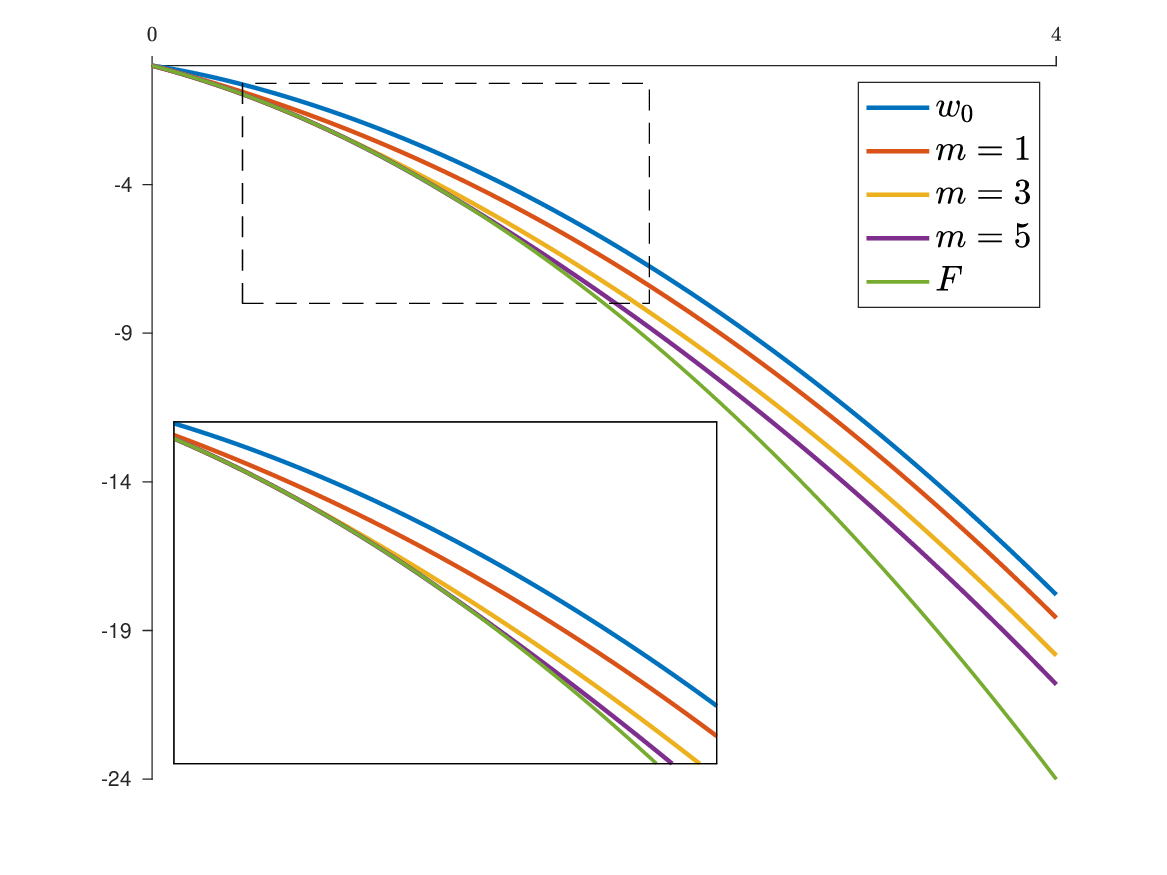}\label{fig:deltaBigger1_diffm}}
  \hfill
  \subfloat[$\delta = 3/4$]{\includegraphics[width=0.48\textwidth]{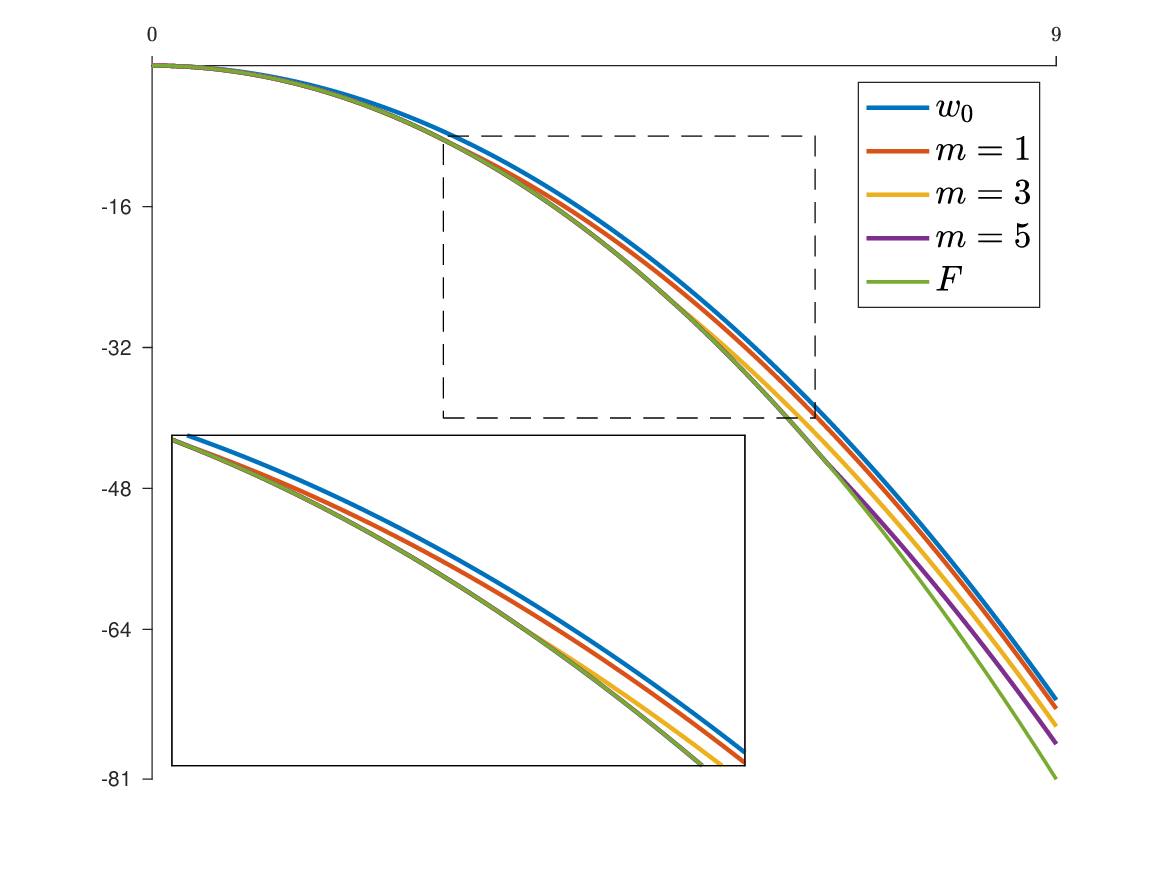}\label{fig:deltaSmaller1_diffm}}
  \caption{the face-lifted utility $\bar{F}$ for different values of $m$}
\end{figure}

As discussed in \Cref{example:positivePowerDeltaBigger1}, the positive power utility exhibits only a singular scenario when $\delta > 1$. Consequently, we slightly modify the utility function by considering $F(y) = - y^\gamma - \gamma y$, for $y \geq 0$. It is important to note that since this function is injective, it corresponds to a specific utility $u$, which cannot be expressed analytically. However, this is unnecessary for our analysis; what is significant is that $F^\star(\delta F^\prime(0)) = -(\gamma-1) (\delta-1)^{\frac{\gamma}{\gamma-1}} < 0$ since this condition ensures that both scenarios in \Cref{proposition:FBarGeneralDeltaBigger1} are feasible. We proceed by illustrating the face-lifted utility $\bar{F}$ in the case when $\gamma=2$ as before, but for different values of $m$. The plot in \Cref{fig:deltaBigger1_diffm} highlights a continuous dependence between the value of $m$ and the value of $\bar{y}$, and therefore a continuous dependence between the average size of accidents and the length of the interval where termination is the preferred choice. Specifically, it is evident that, as the size of the potential losses reduces, the value of $\bar{y}$ also decreases, while the function $\bar{F}$ increases until converging to the accident-free reward $w_0$. A similar analysis can be carried out for $\delta <1$; however, for numerical convenience, we consider the standard positive power utility setting $F(y) = - y^\gamma$, for $y \geq 0$, where we set $\gamma=2$. The results are plotted in \Cref{fig:deltaSmaller1_diffm}. Analogously to the case $\delta >1$, we can observe a correlation between the value $m$ and $\tilde{y}$. The only difference is that the latter is null only in the absence of accidents, that is, when $m=0$.

\section{The first-best problem}\label{section:firstBest}

In this section, we analyse the first-best contracting problem, which represents the benchmark situation without moral hazard. Indeed, in this case, the principal has all the bargaining power, in the sense that she actually chooses both the contract and the effort of the agent under the constraint that his utility is above the reservation utility:
\begin{equation}\label{eq:firstbestV}
V^{\rm P,FB} \coloneqq \sup_{\substack{ (\bC, \nu) \in \mathfrak{C} \times \cU
\\ J^{\smalltext{\rm A}}(\bC,\nu) \geq u(R)}} J^{\rm P}(\bC, \nu) = \sup_{\substack{ (\bC, \nu) \in \mathfrak{C} \times \cU
\\ J^{\smalltext{\rm A}}(\bC,\nu) \geq u(R)}} \bar{J}^{\rm P}(\bC, \nu).
\end{equation}
Note that this equality directly follows from the analysis in \Cref{section:GP}, as $\mathfrak{C} = \mathfrak{C}^0$. Introducing a Lagrange multiplier $\lambda \leq 0$ associated to the participation constraint of the agent and applying the classical Karush--Kuhn--Tucker method, we can always rewrite the first-best problem $V^{\rm P,FB}$ as
\begin{equation*}
\inf_{\lambda \leq 0}\bigg\{\lambda u(R) + \sup_{(\bC, \nu) \in \mathfrak{C} \times \cU} \E^{\P^{\smalltext{\nu}}}\bigg[\big(\mathrm{e}^{-\rho \tau} F(\zeta)- \lambda \mathrm{e}^{-r \tau} \zeta\big) \mathbf{1}_{\{\tau < \infty\}}  + \int_0^\tau \Big( \rho \mathrm{e}^{-\rho s} \big(\alpha_s - \beta_s + F(\eta_s) \big) - \lambda r \mathrm{e}^{-rs} \big(\eta_s - h(\alpha_s,\beta_s)\big) \Big)  \d s  \bigg]\bigg\}.
\end{equation*}
The above expression can be expressed in a more compact form as
\begin{equation}\label{eq:firstbestV_KT}
V^{\rm P,FB} = \inf_{\lambda \leq 0}\bigg\{\lambda u(R) + \sup_{T \geq 0} \bigg\{
-\mathrm{e}^{-\rho T} F^\star\big(\lambda \mathrm{e}^{\rho (1-\delta) T}\big) \mathbf{1}_{\{T < \infty\}} + \int_0^T \rho \mathrm{e}^{-\rho t} \big(G^\star - F^\star\big)\big(\delta\lambda\mathrm{e}^{\rho(1-\delta)t}\big) \d t \bigg\}\bigg\},
\end{equation}
where we have introduced $G^{\star}(p) \coloneqq \sup_{(a,b) \in U} \{a - b + p h(a,b)\}, \; \text{for} \; p \in \R.$ 

\begin{remark}\label{remark:propertiesGStar}
Notice that the function $G^\star$ is clearly continuously differentiable, non-decreasing and convex. In particular, we have $G^{\star}\geq -m$. It remains constantly equal to $-m$ if and only if its derivative is null. Moreover, it holds that $G^{\star}(p) \in [-m,\bar{a}-\varepsilon_m]$ for any $p < 0$, and $G^{\star}(p) = \bar{a}-\varepsilon_m + h(\bar{a},\varepsilon_m) p$ for any $p \geq 0$. We can say that $G^{\star}$ is twice continuously differentiable and increasing on $(\lambda_{inc}, \infty)$, where $\lambda_{inc} \coloneqq \inf \{\lambda \leq 0: (G^{\star})^\prime(\lambda) > 0\} \in (-\infty,0)$.
\end{remark}

\medskip
We introduce the notation $v^{\smalltext{\rm FB}}(y)$ to denote the maximal reward obtained by the principal while maintaining the utility of the agent at the level $y \in \R$. This function is referred to as the first-best value function, and it holds that $V^{\rm P,FB} = v^{\smalltext{\rm FB}}(u(R)).$ In the subsequent analysis, the study of the first-best value function and consequently the first-best problem \eqref{eq:firstbestV} is divided into two separate investigations based on the ratio $\delta$ of the impatience levels of the agent and the principal. Specifically, when $\delta = 1$, the problem can be easily tackled with the aforementioned Karush--Kuhn--Tucker method, while this classical method is not successful to provide a complete characterisation of the first-best value function when the two discount rates differ. Nonetheless, for $\delta \neq 1$ and under the additional condition $\gamma\delta > 1$, we can identify the first-best value function as the unique viscosity solution of the following obstacle problem:
\begin{equation}\label{align:hjbFB}
\min\big\{v^{\smalltext{\rm FB}}(y)-\bar{F}(y), F^{\star}(\delta (v^{\smalltext{\rm FB}})^\prime(y)) - \delta y (v^{\smalltext{\rm FB}})^\prime(y) + v^{\smalltext{\rm FB}}(y) -\cJ^{\smalltext{\rm FB}}\big((v^{\smalltext{\rm FB}})^\prime(y), (v^{\smalltext{\rm FB}})^{\prime\prime}(y)\big)\big\}=0, \; y >-h(\bar{a},\varepsilon_m) \text{\footnotemark},
\footnotetext{The definition of the face-lifted utility $\bar{F}$ in \eqref{eq:FbarEquation} implies that $\bar{F} = -\infty$ on $(-\infty, 0)$ due to the convention $\sup \varnothing = -\infty$.}
\end{equation}
with the initial value $v^{\smalltext{\rm FB}}(-h(\bar{a},\varepsilon_m)) = \bar{a}-\varepsilon_m$, and the operator $\cJ^{\smalltext{\rm FB}}$ defined as
\begin{equation}\label{eq:operatorfirstbest}
\cJ^{\smalltext{\rm FB}}(p,q) \coloneqq G^{\star}(\delta p)  \mathbf{1}_{(-\infty,0]}(q)  +\infty \mathbf{1}_{(0,\infty)}(q), \; \text{for} \; (p, q) \in \R^2.
\end{equation}

\begin{remark}
\begin{enumerate}[label=$(\roman*)$, wide, labelindent=0pt]\label{remark:v0FirstBest}
\item Notice that the obstacle problem for the first-best problem is formulated on a larger domain than the non-negative half-line because the limited liability constraint $J^{{\rm A}}(\bC,\nu) \geq u(R)$ does not ensure a non-negative continuation utility for the agent. Indeed, we will prove that the principal can allow the continuation utility of the agent to take values in the interval $(-h(\bar{a},\varepsilon_m),0)$ and still ensure that it reaches the non-negative value $u(R)$ upon termination of the contract.
\item\label{notattainablevalue} Since $\bar{F} = -\infty$ on $(-h(\bar{a},\varepsilon_m),0)$, it is evident that the first-best value function $v^{\smalltext{\rm FB}}$ does not coincide with its barrier on that interval, but instead solves the {\rm ODE} described in \eqref{align:hjbFB}. However, our focus lies on the values assumed by the first-best value function on $[0,\infty)$, as $V^{\rm P,FB} = v^{\smalltext{\rm FB}}(u(R))$, where $u(R) \geq 0$. Therefore, the initial value we are concerned with is
\begin{equation}\label{eq:initialValue}
v^{\smalltext{\rm FB}}(0) = \inf_{\lambda \leq 0} \sup_{T \geq 0} \bigg\{-\mathrm{e}^{-\rho T} F^{\star}\big(\lambda \mathrm{e}^{\rho(1-\delta) T}\big) \mathbf{1}_{\{T < \infty\}} + \int_0^T \rho \mathrm{e}^{-\rho s} (G^{\star}-F^{\star})\big(\lambda \delta \mathrm{e}^{\rho(1-\delta) s}\big) \d s\bigg\} \in [0, \bar{a} - \varepsilon_m).
\end{equation}
We would like to emphasise that the value $\bar{a} - \varepsilon_m$ is not attainable. Achieving this value would require choosing $\alpha^\star = \bar{a}$, $\beta^\star = \varepsilon_m$, $\tau^\star = \infty$ and $\pi^\star = 0$, but the contract $(\tau^\star, \pi^\star, \xi)$ does not satisfy the participation constraint of the agent for any choice of the lump-sum payment $\xi$ that is $\cF_{\infty}$-measurable. Moreover, observe that we cannot further simplify the expression in \eqref{eq:initialValue} since the function $G^\star$ is not always non-negative. As a result, the concavity of the conjugate $F^\star$ is not sufficient to guarantee that  the supremum is reached at $T=\infty$, as it is shown in the proof of {\rm\cite[Theorem 3.1]{possamai2020there}} for the problem without accidents.
\end{enumerate}
\end{remark}

Despite the apparent differences emphasised previously, we get analogous results to those in \cite[Theorem 3.1]{possamai2020there}. Specifically, we can prove that when $\delta =1$, the first-best value function touches the barrier $\bar{F}$ and, when this occurs, the two functions coincide forever. Conversely, when $\delta > 1$, it can be shown that the two aforementioned functions never intersect, given an additional condition on the sign of $G^\star$ for certain values of the derivative of $\bar{F}$. Unfortunately, in the case where $\delta <1$, we are unable to proceed with the analysis; it is possible only when $\gamma\delta \leq 1$, as we can demonstrate that the problem degenerates. In fact, analogously to \cite[Theorem 3.1]{possamai2020there}, it can be shown that the principal can achieve her maximal value $\bar{a}-\varepsilon_m$, motivating the agent to exert maximal effort continuously and promising him a substantial lump-sum payment at a large retirement time. 

\begin{theorem}\label{thm:FirstBestCompleteC} 
\begin{enumerate}[label=$(\roman*)$, wide, labelindent=0pt]
\item \label{degFB}Let $\gamma\delta \leq 1$. Then $v^{\smalltext{\rm FB}} = \bar{a} - \varepsilon_m$ on $[0,\infty)$ and an optimal contract does not exist.
\item \label{delta1_firstBest}Let $\delta = 1$, and introduce $\lambda_{pos} \coloneqq \inf \{\lambda \leq 0: G^{\star}(\lambda) \geq 0\} \in (-\infty,0)$.
\begin{enumerate}[label=$(\roman{enumi}$-$\arabic*)$, leftmargin=*]
\item If $\lambda_{pos} \geq F^\prime(0)$, then $v^{\smalltext{\rm FB}}= F$ on $[0,\infty);$
\item otherwise if $\lambda_{pos} < F^\prime(0)$, we have that $v^{\smalltext{\rm FB}}(y) = F(y)$ for any $y \geq y^{F^\smalltext{\star}} \coloneqq(F^\star)^\prime(\lambda_{pos})>0$. Additionally, the first-best  value function is continuously differentiable and, defining $y^{F^\smalltext{\star}, G^\smalltext{\star}} \coloneqq(F^\star)^\prime(\lambda_{pos}) - (G^\star)^\prime(\lambda_{pos})$\footnote{Notice that $y^{F^\smalltext{\star}, G^\smalltext{\star}} < y^{F^\smalltext{\star}}$ as $(G^\star)^\prime(\lambda_{pos}) > 0$ by definition of $\lambda_{pos}$.}, determined by 
\begin{align*}
v^{\smalltext{\rm FB}}(y) = \begin{cases}
	\inf_{\lambda \leq 0} \big\{y \lambda- F^{\star}(\lambda) + G^{\star}(\lambda) \big\}  \mathbf{1}_{[0,0 \wedge y^{\smalltext{F}^\tinytext{\star}\smalltext{,} \smalltext{G}^\tinytext{\star}})}(y) + \big(y \lambda_{pos}- F^{\star}(\lambda_{pos})\big)  \mathbf{1}_{[0 \wedge y^{\smalltext{F}^\tinytext{\star}\smalltext{,} \smalltext{G}^\tinytext{\star}},y^{\smalltext{F}^\tinytext{\star}})}(y), \; y \in [0,y^{F^\smalltext{\star}}), \\
		F(y),  \; y \in [y^{F^\smalltext{\star}},\infty).
	\end{cases}
\end{align*}
\item The optimal contract $(\bC^\star,\nu^\star)$ for the first-best problem \eqref{eq:firstbestV} exists and ensures that $J^{{\rm A}}(\bC^\star,\nu^\star) = u(R)$. Additionally, when if $\lambda_{pos} < F^\prime(0)$ and $u(R) \in [0,y^{F^\star})$, let $\lambda^\star$ denote the unique solution to the first-order condition $u(R) - (F^\star - G^\star)^\prime(\lambda^\star)=0$, then the optimal contract takes the form
\begin{equation*}
\tau^\star = \infty, \; \text{and} \; \pi^\star_t = -F\big((F^\prime)^{(-1)}\big)(\lambda^\star) \; \text{and} \; \nu^\star_t \in \argmax_{u \in U} G^\star(\lambda^\star), \; \text{for} \; t \geq 0.
\end{equation*}
If $\lambda_{pos} \geq F^\prime(0)$ or $u(R) \geq y^{F^\smalltext{\star}}$, then $\tau^\star = 0$ with a lump-sum payment $\xi^\star = R$.
\end{enumerate}
\item \label{viscosityCFB}Let $\delta \neq 1$ and $\gamma\delta > 1$. We have that the first-best  value function $v^{\smalltext{\rm FB}}$ is the unique continuous viscosity solution\footnote{We are referring to the definition of viscosity solutions for local equations as provided by \citeauthor*{crandall1992user} \cite[Definition 2.2]{crandall1992user}.} of {\rm\Cref{align:hjbFB}} in the class of functions $v$ such that $|v(y) - \bar{F}(y)| \leq c^{\smalltext{\rm FB}}$ for any $y \geq 0$, for some $c^{\smalltext{\rm FB}}>0$. 
\begin{enumerate}[label=$(\roman{enumi}$-$\arabic*)$, leftmargin=*]
\item\label{deltaBigger1} Let $\delta >1$. If either $m \leq -F^\star(\delta F^\prime(0))$ and $G^\star((F^\star)^{(-1)}(-m)/\delta) \geq 0$, or if $m > -F^\star(\delta F^\prime(0))$ and $G^\star(F^\prime(\bar{y})) > 0$, where $\bar{y}$ is introduced in \eqref{eq:ybar}, then it is true that $v^{\smalltext{\rm FB}} > \bar{F}$ on $[-h(\bar{a},\varepsilon_m),\infty)$. Furthermore, $v^{\smalltext{\rm FB}}$ is twice continuously differentiable, and such that
\begin{equation*}
v^{\smalltext{\rm FB}} = \big(v^{\smalltext{\rm FB},\star}\big)^\star, \; \text{\rm where} \; v^{\smalltext{\rm FB},\star}(p) = \frac{(-p)^{\frac{1}{1-\delta}}}{{\delta-1}} \int_p^0 \frac{F^\star(\delta x) - G^\star(\delta x)}{(-x)^{1+\frac{1}{1-\delta}}} \d x, \; \text{for}\; p \leq 0.
\end{equation*}
\item\label{deltaSmaller1}Let $\delta <1$, and assume that $\frac{\partial h}{\partial a}(0,b) > 0$ for any $b \in B$ and $\frac{\partial h}{\partial b}(a,m) < 0$ for any $a \in A$. We can define $y_m \coloneqq \inf\{y \geq 0: G^{\star}(\delta w_0^\prime(y)) = -m\} \in (0,\infty)$, recalling that $w_0$ represents the face-lifted reward in the problem without accidents as introduced in \eqref{align:w0nojumps}. If $v^{\smalltext{\rm FB}}(y^0) = \bar{F}(y^0)$ for some $y^0 \in [y_m \vee \tilde{y}, \infty)$, then $v^{\smalltext{\rm FB}} = \bar{F}$ on $[y^0, \infty)$.
\end{enumerate}
\end{enumerate}
\end{theorem}

The proof of the first result follows; we show the remaining items in the subsequent sections and in \Cref{appendix:FBImpatientAgent}.
\begin{proof}[Proof of \Cref{thm:FirstBestCompleteC}.\ref{degFB}]
The fact that the first-best value function $v^{\smalltext{\rm FB}}$ is constantly equal to $\bar{a} - \varepsilon_m$ on $[0,\infty)$ follows directly from \Cref{thm:SecondtBestCompleteC}.\ref{degSB}. Furthermore, this value can only be achieved by selecting the first-best contract $(\bC^\star,\nu^\star)$ with $\tau^\star = \infty$, and $(\pi^\star_t,\alpha^\star_t,\beta^\star_t)= (0,\bar{a},\varepsilon_m)$ for all $t \geq 0$. However, as highlighted in \Cref{remark:v0FirstBest}.\ref{notattainablevalue}, this contract fails to meet the participation constraint of the agent.
\end{proof}

\begin{remark}
\begin{enumerate}[label=$(\roman*)$, wide, labelindent=0pt]\label{remark:commentsFirstBest}
\item {\rm\Cref{def:goldenparachute}} can be readily translated into the first-best setting. Specifically, we say that the first-contracting problem \eqref{eq:firstbestV} exhibits a golden parachute if an optimal solution satisfies $\xi^{\star} > 0$ on the event set $\{\tau^{\star} < \infty\}$. Based on {\rm\Cref{thm:FirstBestCompleteC}.\ref{delta1_firstBest}}, we can conclude that when the agent is as impatient as the principal, a first-best golden parachute always exists, and it represents a termination scenario. In other words, if $\lambda_{pos} \geq F^\prime(0)$ and the agent has a positive reservation utility, or if $\lambda_{pos} < F^\prime(0)$ and his participation constraint exceeds $y^{F^\smalltext{\star}}$, the contract does not start, meaning that the agent never begins working, but instead immediately receives a positive lump-sum payment.
\item {\rm\Cref{thm:FirstBestCompleteC}.\ref{deltaBigger1}} states that a first-best golden parachute does not exist when $\delta >1$ since the first-best value function $v^{\smalltext{\rm FB}}$ never touches its barrier $\bar{F}$. This mirrors the result in \cite[Theorem 3.1]{possamai2020there} in the absence of accidents. However, an additional assumption regarding the sign of $G^\star$ for specific values of the derivative of $\bar{F}$ is introduced, and it is worth noting that this additional requirement is always satisfied in the analysis without jumps.
\item The case of $\delta <1$ and $\gamma\delta >1$ is the only one that we are unable to analyse in detail. Given the explicit form of the face-lifted utility $\bar{F}$ expressed in {\rm\Cref{proposition:FBarGeneralDeltaSmaller1}.\ref{nonDegenerateFBar}}, particularly the fact that it is only continuous, we expect that the first-best value function $v^{\smalltext{\rm FB}}$ might not be smooth on the entire interval of definition $[-h(\bar{a},\varepsilon_m),\infty)$. Furthermore, we suspect that, similar to the accident-free framework, there exists a point from which $v^{\smalltext{\rm FB}}$ and its barrier $\bar{F}$ coincide if $\frac{\partial h}{\partial a}(0,b) > 0$ for any $b \in B$ and $\frac{\partial h}{\partial b}(a,m) < 0$ for any $a \in A$. If this were to be confirmed, we could conclude that a first-best golden parachute exists whenever $v^{\smalltext{\rm FB}}(0) > 0$, and this might also be a termination scenario, unlike the case with $m=0$ where the golden parachute always represents a retirement scenario. Conversely, if $v^{\smalltext{\rm FB}}(0) = 0$, we could not automatically claim the existence of a golden parachute, we would need to pay attention to the interval of the non-negative half-line that we are considering. Note that the condition $\frac{\partial h}{\partial a}(0,b) > 0$ for any $b \in B$ and $\frac{\partial h}{\partial b}(a,m) < 0$ for any $a \in A$ is fundamental to ensure that the face-lifted utility $\bar{F}$ is a super-solution---and hence a solution, considering that $G^\star \geq -m$ on the whole real line---of {\rm\Cref{align:hjbFB}} from a certain point onwards. This is the same condition required in {\rm\cite[Theorem 3.1]{possamai2020there}} for $m=0$. In this case, it is proved that when the marginal cost of effort at zero is null, the first-best value function never touches the barrier $\bar{F}$, thereby ruling out the existence of a first-best golden parachute.
\end{enumerate}
\end{remark}

\subsection{The Karush–Kuhn–Tucker method}
\begin{proof}[Proof of \Cref{thm:FirstBestCompleteC}.\ref{delta1_firstBest}]
The first-best value function, as expressed in \eqref{eq:firstbestV_KT}, can be further simplified when $\delta = 1$:
\begin{equation*}
v^{\smalltext{\rm FB}}(y) = \inf_{\lambda \leq 0}\bigg\{\lambda y- F^\star(\lambda) +\sup_{T \geq 0} \big\{ G^\star(\lambda) \big(1-\mathrm{e}^{-\rho T}\big)\big\}\bigg\} \; \text{for any} \; y \geq 0.
\end{equation*}
The optimisation with respect to the variable $T$ yields the following problem
\begin{align}\label{eq:firstBestReduced}
v^{\smalltext{\rm FB}}(y)  = \inf_{\lambda \leq 0} \Big\{\lambda y - F^\star(\lambda) + G^\star(\lambda)  \mathbf{1}_{[\lambda_\smalltext{pos},0]}(\lambda)\Big\}.
\end{align}

If $\lambda_{pos} \geq F^\prime(0)$, the optimisation with respect to $\lambda$ implies that $v^{\smalltext{\rm FB}}(y) = \inf_{\lambda \leq 0} \big\{\lambda y - F^\star(\lambda)\big\} = F(y).$ This is because $y - (F^\star)^\prime(\lambda) + (G^\star)^\prime(\lambda)  = y + (G^\star)^\prime(\lambda) > 0$ for any $\lambda \in  [\lambda_{pos},0]$. We can deduce that $v^{\smalltext{\rm FB}} = F$ on $[0, \infty)$.

\medskip
On the other hand, we assume now that $\lambda_{pos} < F^\prime(0)$. Considering the reformulation of the first-best  problem as shown in \eqref{eq:firstBestReduced}, we describe the function $v^{\smalltext{\rm FB}}$ by distinguishing three distinct cases depending on which interval of the non-negative half-line we are considering. To begin, let us suppose that $y \in (0,0 \vee y^{F^\smalltext{\star}, G^\smalltext{\star}})$. It is worth noting that this interval can be empty because $y^{F^\smalltext{\star}, G^\smalltext{\star}}$ can be non-positive. If this is not the case, then we have that $y - (F^\star)^\prime(\lambda_{pos}) + (G^\star)^\prime(\lambda_{pos}) < 0$, and  
\begin{align*}
v^{\smalltext{\rm FB}}(y) = \inf_{\lambda \in (\lambda_\smalltext{pos},0)} \big\{\lambda y - F^\star(\lambda) + G^\star(\lambda)\big\} = \inf_{\lambda \leq 0} \big\{\lambda y - F^\star(\lambda) + G^\star(\lambda)\big\}.
\end{align*} 
From the initial equality, it is evident that $v^{\smalltext{\rm FB}} > F$ on $[0,0 \vee y^{F^\smalltext{\star}}]$, as $G^\star(\lambda) >0$ for any $\lambda> \lambda_{pos}$. We then consider the case where $y \in [0 \vee y^{F^\smalltext{\star}, G^\smalltext{\star}},y^{F^\smalltext{\star}})$. It holds that $y - (F^\star)^\prime(\lambda_{pos}) + (G^\star)^\prime(\lambda_{pos}) \geq 0$ and $y - (F^\star)^\prime(\lambda_{pos}) < 0$ because of the definition of the two points $y^{F^\smalltext{\star}, G^\smalltext{\star}}$ and $y^{F^\smalltext{\star}}$. Hence, $v^{\smalltext{\rm FB}}$ is a linear function of the form $v^{\smalltext{\rm FB}}(y) = y \lambda_{pos}- F^\star(\lambda_{pos}).$ Additionally, the assumption $\lambda_{{pos}} < F^\prime(0)$, under which we are operating, implies the existence of a unique minimiser of $F^\star(\lambda_{{pos}})$. However, since $y - (F^\star)^\prime(\lambda_{{pos}}) < 0$, it follows that
\begin{align*}
F^\star(\lambda_{pos}) = \inf_{y \geq 0} \{y \lambda_{pos} - F(y)\} =  (F^\star)^\prime(\lambda_{pos}) \lambda_{pos} - F((F^\star)^\prime(\lambda_{pos})) < y \lambda_{pos} - F(y).
\end{align*}
By rearranging the terms, we can deduce that $v^{\smalltext{\rm FB}} > F$ on $[0 \vee y^{F^\smalltext{\star}},y^{F^\smalltext{\star}})$. Lastly, when we have $y \in [y^{F^\smalltext{\star}},\infty)$, it follows that $y - (F^\star)^\prime(\lambda_{pos}) \geq 0$. As a consequence,
\begin{align*}
v^{\smalltext{\rm FB}}(y) = \inf_{\lambda \leq \lambda_{pos}} \big\{ \lambda y - F^\star(\lambda) \big\} = \inf_{\lambda \leq 0} \big\{ \lambda y - F^\star(\lambda) \big\} = F(y).
\end{align*}

Reviewing the computations done so far, we can deduce that the first-best value function $v^{\smalltext{\rm FB}}$ is not only continuous but also continuously differentiable since
\begin{align*}
(v^{\smalltext{\rm FB}})^\prime_{-}(y^{F^\smalltext{\star}, G^\smalltext{\star}}) \mathbf{1}_{\{y^{\smalltext{F}^\tinytext{\star}\smalltext{,} \smalltext{G}^\tinytext{\star}} > 0\}} = (v^{\smalltext{\rm FB}})^\prime_{+}(y^{F^\smalltext{\star}, G^\smalltext{\star}})  \mathbf{1}_{\{y^{\smalltext{F}^\tinytext{\star}\smalltext{,} \smalltext{G}^\tinytext{\star}} > 0\}} = \lambda_{pos} = (v^{\smalltext{\rm FB}})^\prime_{-}(y^{{F^\star}}) = (v^{\smalltext{\rm FB}})^\prime_{+}(y^{F^\smalltext{\star}}).
\end{align*}
\end{proof}

\begin{example}\label{example:FB}
In the specific framework of the positive power utility described in {\rm\Cref{example:positivePower}} with $\delta = 1$, we have $\bar{F}(y) = F(y) = -y^\gamma , \;\text{\rm for} \; y \geq 0.$ Additionally, let us set some $\bar{a}$, $\varepsilon_m$ and $m$ such that $0<\varepsilon_m < m \wedge \bar{a}$, and consider the following cost function:
\begin{align*}
h(a,b) \coloneqq h_a(a) + h_b(b) \coloneqq \frac{a^2}{2} + \bigg(\frac{1}{b} - \frac{1}{m}\bigg), \;\text{\rm where} \; (a,b) \in A \times B \coloneqq [0,\bar{a}] \times [\varepsilon_m,m].
\end{align*}
The decision to consider a separate cost in the two tasks of applying effort and reducing accident risk is solely for purely analytical reasons, as it significantly simplifies the calculations, allowing us to explicitly characterise the convex dual
\begin{align*}
G^{\star}(p) = G^{\star}_a(p) + G^{\star}_b(p) \coloneqq \sup_{a \in [0,\bar{a}]} \{a + p h_a(a)\} + \sup_{b \in [\varepsilon_\smalltext{m},m]} \{-b + p h_b(b)\} \; \text{\rm for any} \; p \in \R.
\end{align*}
Under the present specification, we have that
\begin{align*}
G^{\star}_a(p) & = - \frac{1}{2p} \mathbf{1}_{(-\infty,- 1 / \bar{a}]}(p) + \bar{a} \Big(1+  \frac{\bar{a}}{2} p \Big) \mathbf{1}_{(- 1/\bar{a},\infty)}(p), \\ 
G^{\star}_b(p) & =- m \mathbf{1}_{(-\infty, - m^\smalltext{2}]}(p)  - \big(2m - (-p)^{\frac{1}{2}}\big) \frac{(-p)^{\frac{1}{2}}}{m} \mathbf{1}_{(- m^\smalltext{2}, - \varepsilon_\smalltext{m}^\smalltext{2})}(p) - \bigg(\varepsilon_m + \bigg(\frac{1}{m}-\frac{1}{\varepsilon_m}\bigg)p\bigg) \mathbf{1}_{[- \varepsilon_m^\smalltext{2},\infty)}(p), \; \text{\rm for} \; p \in \R.
\end{align*}

\begin{figure}[ht!]
  \centering
 {\includegraphics[width=0.5\textwidth]{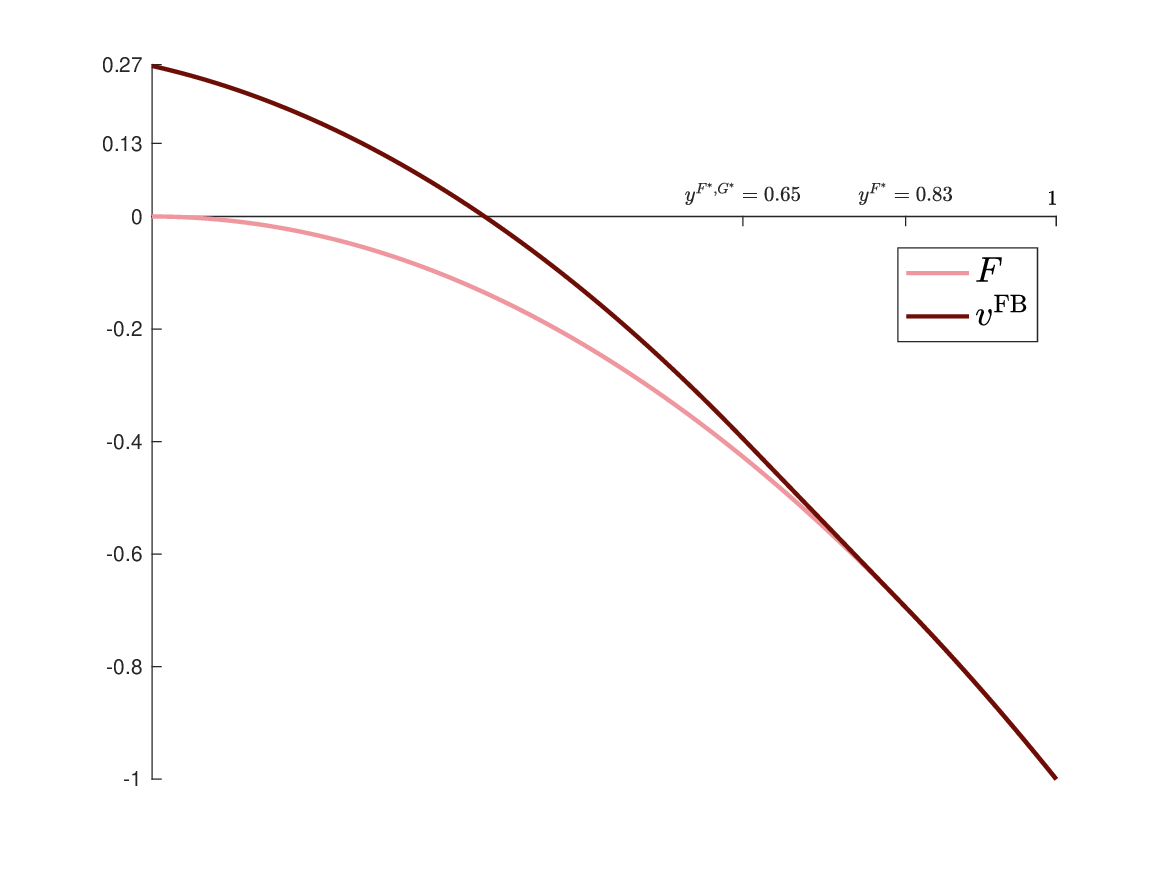}\label{fig:delta1_firstBest}}
  \caption{the first-best value function $v^{\smalltext{\rm FB}}$ and its barrier $F$}
  \label{fig:firstBestAndBarrier}
\end{figure}

Given this specific choice of the utility function $F$, particularly noting that $F^\prime(0) = 0$, we fall into the second scenario of {\rm\Cref{thm:FirstBestCompleteC}}.\ref{delta1_firstBest}. Consequently, the first-best value function is not identically equal to its barrier $F$, but we are unable to provide explicit formulas since $y^{{F}^\smalltext{\star}{,} {G}^\smalltext{\star}}$ and $y^{{F}^\smalltext{\star}}$ cannot be computed analytically. Setting $\gamma = 2$, $\bar{a} = 0.6$, $\varepsilon_m = 0.1$ and $m=0.3$, we plot our numerical results in {\rm\Cref{fig:firstBestAndBarrier}}. Specifically, we observe that $y^{{F}^\smalltext{\star}{,} {G}^\smalltext{\star}}$ is positive, and therefore, the function $v^{\smalltext{\rm FB}}$ is a strictly concave solution to the {\rm ODE} that characterises the obstacle problem \eqref{align:hjbFB} within the interval $[0,y^{{F}^\smalltext{\star}{,} {G}^\smalltext{\star}})$, is affine on $[y^{{F}^\smalltext{\star}{,} {G}^\smalltext{\star}},y^{{F}^\smalltext{\star}}]$, and then coincides with $F$ on $(y^{{F}^\smalltext{\star}}, \infty)$. It is worth noting that contrary to the results in {\rm\cite{possamai2020there}}, the function is not strictly concave on the whole domain of definition. Additionally, it is only continuously differentiable, as the gluing of the functions $\inf_{\lambda \leq 0} \{y \lambda- F^{\star}(\lambda) + G^{\star}(\lambda)\}$, for $y \in [0,y^{\smalltext{F}^\tinytext{\star}\smalltext{,} \smalltext{G}^\tinytext{\star}})$, $y \lambda_{pos}- F^{\star}(\lambda_{pos})$, for $y \in [y^{\smalltext{F}^\tinytext{\star}\smalltext{,} \smalltext{G}^\tinytext{\star}},y^{\smalltext{F}^\tinytext{\star}})$, and $F(y)$, for $y \in [y^{\smalltext{F}^\tinytext{\star}},\infty)$, is only continuously differentiable. These differences arise from the fact that $G^\star$ can be negative in our framework with accidents.

\end{example}

\subsection{HJB approach}

To address the first-best version \eqref{eq:firstbestV} of the contracting problem under the conditions $\delta \neq 1$ and $\gamma\delta >1$, we follow the general approach outlined in \cite{sannikov2008continuous}, developed in \cite{cvitanic2018dynamic,lin2022random}. Specifically, we show that there is no loss of generality when the principal solves her first-best problem in maximising over the set $\cV^{\smalltext{\rm FB}}(Y^{{\rm A}}_0)$, which we introduce later in \Cref{theorem:ReductionFirstBest}, rather than optimising over $\mathfrak{C}\times\cU$ such that the participation constraint is satisfied. This reduction simplifies the first-best problem to a standard mixed control--stopping stochastic problem, with the state variable being the continuation utility of the agent. 

\medskip
Let us fix an $\F$--stopping time $\tau$, an $\F$-predictable non-negative process $\pi$ and a control $\nu \in \cU$. For a constant $Y^{{\rm A}}_0 \in \R$, an $\R$-valued $\F$-predictable process $Z^{{\rm A}}$ and an $\R$-valued $\widetilde{\cP}(\F)$-measurable function $U^{{\rm A}}$, we introduce the process $Y^{\smalltext{\rm FB}, {Y}^{\smalltext{\rm A}}_\smalltext{0}{,}{Z}^{\smalltext{\rm A}} {,}{U}^{\smalltext{\rm A}},\pi,\nu}$ defined by the following SDE, for $t\geq 0$:
\begin{align}\label{align:defYAfb}
Y^{\smalltext{\rm FB}, {Y}^{\smalltext{\rm A}}_\smalltext{0}{,}{Z}^{\smalltext{\rm A}} {,}{U}^{\smalltext{\rm A}}{,}{\pi}{,}{\nu}}_t = Y^{{\rm A}}_0 + \int_0^t r \big(Y^{\smalltext{\rm FB}, {Y}^{\smalltext{\rm A}}_\smalltext{0}{,}{Z}^{\smalltext{\rm A}} {,}{U}^{\smalltext{\rm A}}{,}{\pi}{,}{\nu}}_s - u(\pi_s) +h(\nu_s)\big) \d s + r \sigma \int_0^t Z^{{\rm A}}_s \d W^{\nu}_s + r \int_0^t \int_{\R} U^{{\rm A}}_{s}(\ell) \tilde{\mu}^{{J}^{\smalltext{\nu}}}(\d s, \d \ell),
\end{align}
where $\tilde{\mu}^{{J}^{\smalltext{\nu}}}$ denotes the $(\F,\P^{\nu})$-compensated random measure $\mu^{J} - \mu^{J^{\smalltext{\nu}},p}$. Then, we introduce the set $\cV^{\smalltext{\rm FB}}_\tau$ as the collection of all $Z^{{\rm A}}$ and $U^{{\rm A}}$, defined as before, satisfying in addition the integrability conditions
\begin{align}\label{align:intCondYA}
\sup_{\nu \in \cU} \E^{\P^{\smalltext{\nu}}}\bigg[\sup_{t \geq 0} \Big|\mathrm{e}^{-\kappa (t\wedge\tau)} Y^{\smalltext{\rm FB}, {Y}^{\smalltext{\rm A}}_\smalltext{0}{,}{Z}^{\smalltext{\rm A}} {,}{U}^{\smalltext{\rm A}}{,}{\pi}{,}{\nu}}_{t\wedge\tau}\Big|^p\bigg] < \infty,
\end{align}
\begin{align}\label{align:intCondZAUA}
\sup_{\nu \in \cU} \E^{\P^{\smalltext{\nu}}}\Bigg[\bigg(\int_0^{\tau} \mathrm{e}^{-2 \kappa s} |Z^{{\rm A}}_s|^2 \d s\bigg)^{\frac{p}{2}}\Bigg] < \infty, \; \text{and} \;\sup_{\nu \in \cU} \E^{\P^{\smalltext{\nu}}}\Bigg[\bigg(\int_0^{\tau} \int_{\R}\mathrm{e}^{-2 \kappa s} |U^{{\rm A}}_s(\ell)|^2 \Phi(\d \ell) \d s\bigg)^{\frac{p}{2}}\Bigg] < \infty,
\end{align}
for some $\kappa \in (0, r)$, $p > 2 \vee \gamma $.

\begin{remark}\label{remark:existenceYsde} 
Note that based on {\rm\citeauthor*{oksendal2007applied} \cite[Theorem 1.19]{oksendal2007applied}}, the Lipschitz-continuity in the $y$-variable of the drift implies that the {\rm SDE} \eqref{align:defYAfb} has a pathwise unique solution up to the random time $\tau$. Furthermore, we always consider a c\`adl\`ag $\P$-modification of the process $Y^{\smalltext{\rm FB}, {Y}^{\smalltext{\rm A}}_\smalltext{0}{,}{Z}^{\smalltext{\rm A}} {,}{U}^{\smalltext{\rm A}}{,}{\pi}{,}{\nu}}$.
\end{remark}

For any $Y^{{\rm A}}_0 \in \R$, we define the set $\cV^{\smalltext{\rm FB}}(Y^{{\rm A}}_0)$ as the collection of all quintuples $(\tau, \pi, \nu, Z^{{\rm A}}, U^{{\rm A}})$, where $\tau$ is an $\F$--stopping time, $\pi$ is an $\F$-predictable non-negative process, and $(Z^{{\rm A}},U^{{\rm A}}) \in \cV^{\smalltext{\rm FB}}_\tau$, that verify that
\begin{align*}
\text{the contract} \; \bC\coloneqq \big(\tau, \pi, u^{(-1)}\big(Y^{\smalltext{\rm FB}, {Y}^{\smalltext{\rm A}}_\smalltext{0}{,}{Z}^{\smalltext{\rm A}} {,}{U}^{\smalltext{\rm A}}{,}{\pi}{,}{\nu}}_\tau\big)\big) \; \text{satisfies the integrability conditions \eqref{eq:integrabilityCondition}}.
\end{align*}
\begin{theorem}\label{theorem:ReductionFirstBest}
It holds that
\begin{align*}
V^{\rm P,FB} =  \sup_{(\tau, \pi, \nu, Z^{\smalltext{\rm A}}, U^{\smalltext{\rm A}}) \in \cV^\smalltext{\rm FB}(u(R))} \E^{\P^{\smalltext{\nu}}}\bigg[\mathrm{e}^{-\rho \tau} \bar{F}\big(Y^{\smalltext{\rm FB}, u(R){,}{Z}^{\smalltext{\rm A}} {,}{U}^{\smalltext{\rm A}}{,}{\pi}{,}{\nu}}_\tau\big) \mathbf{1}_{\{\tau < \infty\}} + \int_0^\tau \rho \mathrm{e}^{-\rho s} (\alpha_s - \beta_s - \pi_s) \d s\bigg].
\end{align*}
\end{theorem}

\begin{proof}
The proof follows by simply adapting the arguments in \Cref{theorem:reductionSecondBest} taking into account that, for any contract $\bC = (\tau, \pi, \xi) \in \mathfrak{C}$ and any control $\nu \in \cU$, the process defined as
\begin{align}\label{align:Mfb}
M^{\smalltext{\rm A}\smalltext{,}\smalltext{\rm FB}}_t(\bC,\nu) \coloneqq V^{\smalltext{\rm A}\smalltext{,}\smalltext{\rm FB}}_{t \wedge \tau}(\bC,\nu) \mathrm{e}^{-r (t \wedge \tau)} + \int_0^{t \wedge \tau} r \mathrm{e}^{-r s} (u(\pi_s) - h(\nu_s)) \d s, \; \text{for} \; t \geq 0,
\end{align}
is a $\P^{\nu}$--uniformly integrable $(\F,\P^{{\nu}})$-martingale (see for instance \cite[Theorem I.1.42]{jacod2003limit}). Here,
\begin{align*}
V^{\smalltext{\rm A}\smalltext{,}\smalltext{\rm FB}}_t(\bC,\nu) \coloneqq \E^{\P^{\smalltext{\nu}}}_t \bigg[\mathrm{e}^{-r (\tau-t)} u(\xi) \mathbf{1}_{\{\tau < \infty\}} + \int_t^\tau r \mathrm{e}^{-r(s-t)} (u(\pi_s) - h(\nu_s)) \d s\bigg], \; \text{for} \; t \in [0,\tau].
\end{align*}
\end{proof}

\begin{remark}\label{remark:FBContUtilityAgent}
\begin{enumerate}[label=$(\roman*)$, wide, labelindent=0pt]
\item The proof of {\rm\Cref{theorem:ReductionFirstBest}}, or more precisely the arguments presented to prove \Cref{theorem:reductionSecondBest}, show that the process $Y^{\smalltext{\rm FB},Y^{\smalltext{\rm A}}_\smalltext{0}{,}{Z}^{\smalltext{\rm A}} {,}{U}^{\smalltext{\rm A}}{,}{\pi^\smalltext{\star}}{,}{\nu^\smalltext{\star}}}$, with $Y^{{\rm A}}_{0} = u(R)$ represents the continuation utility of the agent given the first-best contract $(\bC^\star, \nu^\star)$, where $\xi = \xi^{Y^{\smalltext{\rm A}}_\smalltext{0}{,}Z^{\smalltext{\rm A}}{,}U^{\smalltext{\rm A}}{,}{\pi^\smalltext{\star}}{,}{\nu^\smalltext{\star}}} \coloneqq u^{(-1)}\big(Y^{\smalltext{\rm FB},Y^{\smalltext{\rm A}}_\smalltext{0}{,}{Z}^{\smalltext{\rm A}} {,}{U}^{\smalltext{\rm A}}{,}{\pi^\smalltext{\star}}{,}{\nu^\smalltext{\star}}}_\tau\big)$. In particular, the reward of the agent is exactly given by his participation constraint, meaning $V^{\rm A}(\bC^\star) = Y^{\smalltext{\rm FB},Y^{\smalltext{\rm A}}_\smalltext{0}{,}{Z}^{\smalltext{\rm A}} {,}{U}^{\smalltext{\rm A}}{,}{\pi^\smalltext{\star}}{,}{\nu^\smalltext{\star}}}_0 = Y^{{\rm A}}_{0} = u(R)$.
\item\label{remark:lowerboundYFB} Applying a comparison theorem for {\rm BSDEs} $($for instance, one can adapt the arguments of {\rm \citeauthor*{lin2020second} \cite[Theorem 3.7]{lin2020second}} and {\rm\citeauthor*{possamai2023reflections} \cite[Proposition 7.3]{possamai2023reflections}}$)$, we can derive the constraint $Y^{\smalltext{\rm FB}, {Y}^{\smalltext{\rm A}}_\smalltext{0}{,}{Z}^{\smalltext{\rm A}} {,}{U}^{\smalltext{\rm A}}{,}{\pi}{,}{\nu}}_{t \wedge \tau} \geq -h(\bar{a},\varepsilon_m)$ for any $t \geq 0$, $\P${\rm--a.s.}, given the non-negativity condition on the utility function $u$ and the upper boundedness of the cost function $h$ by $h(\bar{a},\varepsilon_m)$. 
\end{enumerate}
\end{remark}

\Cref{theorem:ReductionFirstBest} significantly simplifies the complexity of the first-best contracting problem. In fact, it reduces the Stackelberg game to a mixed control--stopping stochastic problem, to which we can associate a Hamilton--Jacobi--Bellman equation using classical control theory reasoning. To identify the domain over which the equation holds, we need to address the following stochastic target problem, for $y \geq 0$:
\begin{align*}
{d}^{\smalltext{\rm FB}}(y) \coloneqq \inf \mathcal{D}^{\smalltext{\rm FB}}_y \coloneqq \inf \Big\{Y^{{\rm A}}_0 \in \R: \; \big(Y^{\smalltext{\rm FB}, {Y}^{\smalltext{\rm A}}_\smalltext{0}{,}{Z}^{\smalltext{\rm A}} {,}{U}^{\smalltext{\rm A}}{,}{\pi}{,}{\nu}}_\tau - y\big) \mathbf{1}_{\{\tau < \infty\}}  \geq 0 \; \P\text{{\rm--a.s.}} \; \text{\rm for some} \; (\tau, \pi, \nu, Z^{\smalltext{\rm A}}, U^{\smalltext{\rm A}}) \in \cV^\smalltext{\rm FB}\big(Y^{{\rm A}}_0\big) \Big\}.
\end{align*}
This problem consists in determining the minimum $Y^{{\rm A}}_0$ for which a control of the reduced control--stopping stochastic problem allows the state process $Y^{\smalltext{\rm FB}, {Y}^{\smalltext{\rm A}}_\smalltext{0}{,}{Z}^{\smalltext{\rm A}} {,}{U}^{\smalltext{\rm A}}{,}{\pi}{,}{\nu}}$ to reach the target $y$, representing the participation constraint of the agent, at the terminal time $\tau$ when the latter is finite. To address this stochastic target problem, let us choose $y \geq 0$ and some $n \in \N$, and introduce the control $(\tau^{n,y}, \pi^{n,y}, \nu^{n,y}, Z^{\smalltext{\rm A},n,y}, U^{\smalltext{\rm A},n,y})$ such that 
\begin{align*}
\tau^{n,y} \coloneqq \frac{\log{(n(y+h(\bar{a},\varepsilon_m)))}}{r}, \; \text{and} \; (\pi^{n,y}_t, \nu^{n,y}_t, Z^{\smalltext{\rm A},n,y}_t, U^{\smalltext{\rm A},n,y}_t) \coloneqq (0, (\bar{a},\varepsilon_m),0,0), \; \text{for} \; t \in [0,\tau^{n,y} ]
\end{align*}
From \eqref{align:defYAfb}, it is evident that 
\begin{align*}
Y^{\smalltext{\rm FB}, -h(\bar{a},\varepsilon_m) +1/n{,}{Z}^{\smalltext{\rm A},n,y} {,}{U}^{\smalltext{\rm A},n,y}{,}{\pi^{n,y}}{,}{\nu^{n,y}}}_{\tau^{n,y}} =  y.
\end{align*}
Hence, we can deduce that $-h(\bar{a},\varepsilon_m) +1/n \in \mathcal{D}^{\smalltext{\rm FB}}_y$ for any $y \geq 0$ and $n \in \N$. We conclude that ${d}^{\smalltext{\rm FB}}= -h(\bar{a},\varepsilon_m)$ on $[0,\infty)$ since $\mathcal{D}^{\smalltext{\rm FB}}_y \subseteq [-h(\bar{a},\varepsilon_m), \infty)$ for all $y \geq 0$, as stated in \Cref{remark:FBContUtilityAgent}.\ref{remark:lowerboundYFB}. Consequently, the domain of the Hamilton-Jacobi-Bellman equation associated with the first-best value function $v^{\smalltext{\rm FB}}$ is $[-h(\bar{a},\varepsilon_m), \infty)$. Our next objective is to prove that $v^{\smalltext{\rm FB}}$ is the unique viscosity solution of the aforementioned equation in the class of functions that behave like $\bar{F}$ at infinity. However, before verifying this statement, it is necessary to prove that our value function exhibits the desired growth.

\begin{lemma}\label{lemma:growthfirstbest}
If $\delta \neq 1$ and $\gamma\delta >1$, then there exists some $c^{\smalltext{\rm FB}}>0$ such that $| v^{\smalltext{\rm FB}}(y) - \bar{F}(y) | \leq c^{\smalltext{\rm FB}}$ for any $ y \geq 0$.
\end{lemma}
\begin{proof}
The reformulation of the first-best problem in \Cref{theorem:ReductionFirstBest} straightforwardly implies that $v^{\smalltext{\rm FB}} - \bar{F} \geq -c^{\smalltext{\rm FB}}$ on $[0,\infty)$, where $c^{\smalltext{\rm FB}} > 0$. To prove the reverse bound, we first introduce the function 
\[\iota(y) \coloneqq \inf_{\lambda \leq 0}\bigg\{\lambda y - \int_0^{\infty} \rho \mathrm{e}^{-\rho s} F^{\star}\big(\lambda \delta \mathrm{e}^{\rho(1-\delta) s}\big) \d s \bigg\}, \; \text{for} \; y \geq 0. \]
For any $y \geq 0$, we note that
\begin{align*}
v^{\smalltext{\rm FB}}(y) &= \inf_{\lambda \leq 0}\bigg\{\lambda y + \sup_{T \geq 0} \bigg\{-\mathrm{e}^{-\rho T} F^{\star}\big(\lambda \mathrm{e}^{\rho(1-\delta) T}\big) \mathbf{1}_{\{T < \infty\}} + \int_0^T \rho \mathrm{e}^{-\rho s} \big(G^{\star}-F^{\star}\big)\big(\lambda \delta \mathrm{e}^{\rho(1-\delta) s}\big) \d s\bigg\} \bigg\} \\
&\leq \inf_{\lambda \leq 0}\bigg\{\lambda y + \sup_{T \geq 0} \bigg\{-\mathrm{e}^{-\rho T} F^{\star}\big(\lambda e^{\rho(1-\delta) T}\big) \mathbf{1}_{\{T< \infty\}} - \int_0^T \rho \mathrm{e}^{-\rho s} F^{\star}\big(\lambda \delta \mathrm{e}^{\rho(1-\delta) s}\big) \d s \bigg\} + \bar{a} - \varepsilon_m = \iota(y) + \bar{a} - \varepsilon_m,
\end{align*}
where the inequality is a consequence of the definition of $G^\star$, while the last equality follows from the concavity of $F^\star$. Consequently, it is sufficient to verify that there exists a constant $c^{\smalltext{\rm FB}}>0$ such that 
\begin{align}\label{align:boundAim}
\iota(y) - \bar{F}(y) \leq c^{\smalltext{\rm FB}} \; \text{for all} \; y \geq 0.
\end{align}
We can divide the proof into three distinct parts, depending on the specific case under consideration and the corresponding face-lifted utility $\bar{F}$. In the scenario where $\delta > 1$ and $m \leq - F^{\star}(\delta F^{\prime}(0))$, the function $\bar{F}$ is given in terms of its concave conjugate $w^{\star}$ introduced in \eqref{align:wstarexp}. Based on its description, we can deduce that
\begin{align*}
w^{\star}(p) \leq m + \frac{(-p)^{\frac{1}{1-\delta}}}{\delta-1} \int_p^0 \frac{F^{\star}(\delta x)}{(-x)^{1+\frac{1}{1-\delta}}} \d x = m + \int_0^{\infty} \rho \mathrm{e}^{-\rho t} F^{\star}\big(\delta p \mathrm{e}^{\rho(1-\delta)t}\big) \d t \; \text{for any} \; p \leq \frac{(F^\star)^{(-1)}(-m)}{\delta}.
\end{align*}
In the first inequality, we have used the fact that $F^{\star}(\delta x) + m \geq 0$ for any $x \in [f_{\delta m}, 0]$ followed by the change of variables $x \longmapsto p \mathrm{e}^{\rho(1-\delta)t}$. From the previous inequality, it becomes evident that $\iota(y) - \bar{F}(y) \leq m$ for all $y \geq 0$. We can conclude that the inequality in \eqref{align:boundAim} is satisfied. 

\medskip
When the jumps become larger, specifically in the case of $\delta > 1$ and $m > - F^{\star}(\delta F^{\prime}(0))$, we have that $\bar{F}=F$ on $[0,\bar{y}]$, see \Cref{proposition:FBarGeneralDeltaBigger1}.\ref{mLarge}. As the function $\iota$ is upper--semi-continuous since it is an infimum of an arbitrary family of continuous functions, we can deduce the existence of some constant $c^{\smalltext{\rm FB}}>0$ such that $(\iota-F)(y) \leq c^{\smalltext{\rm FB}}$ for any $y \in [0,\bar{y}]$. Nonetheless, we still need to provide a similar bound on $(F^{\prime}(\bar{y}),\infty)$, where $\bar{F}$ is described by the concave conjugate of $w^{\star}_{\bar{y}}$ as given in \eqref{align:wstarybar}. Its explicit formula is as follows, for any $p \in (-\infty,F^{\prime}(\bar{y}))$:
\begin{align*}
w^{\star}_{\bar{y}}(p) &= \frac{(-p)^{\frac{1}{1-\delta}}}{\delta-1} \int_p^{F^{\prime}(\bar{y})} \frac{F^{\star}(\delta x)+m}{(-x)^{1+\frac{1}{1-\delta}}} \d x + F^{\star}(F^{\prime}(\bar{y})) \left(\frac{p}{F^{\prime}(\bar{y})}\right)^{\frac{1}{1-\delta}} \\
&\leq \frac{(-p)^{\frac{1}{1-\delta}}}{\delta-1} \int_p^{F^{\prime}(\bar{y})} \frac{F^{\star}(\delta x)+m}{(-x)^{1+\frac{1}{1-\delta}}} \d x \\
&= \frac{(-p)^{\frac{1}{1-\delta}}}{\delta-1} \int_p^0 \frac{F^{\star}(\delta x)+m}{(-x)^{1+\frac{1}{1-\delta}}} \d x - \frac{(-p)^{\frac{1}{1-\delta}}}{\delta-1} \int_{F^{\prime}(\bar{y})}^0 \frac{F^{\star}(\delta x)+m}{(-x)^{1+\frac{1}{1-\delta}}} \d x \leq C + \frac{(-p)^{\frac{1}{1-\delta}}}{\delta-1} \int_p^0 \frac{F^{\star}(\delta x)+m}{(-x)^{1+\frac{1}{1-\delta}}} \d x,
\end{align*}
for some $C > 0$. This allows us to conclude by applying the same change of variables as we did previously.

\medskip
Lastly, we examine the scenario where $\delta < 1$ and $\gamma\delta > 1$. It is straightforward to prove that the inequality in \eqref{align:boundAim} is satisfied on the interval $[0,\tilde{y}]$. Additionally, for $y \in (\tilde{y},\infty)$, it holds that $\iota(y) = w_0(y) = \bar{F}(y)$, where the function $w_0$ is introduced in \eqref{align:w0nojumps}. This concludes the proof.
\end{proof}

\begin{proof}[Proof of \Cref{thm:FirstBestCompleteC}.\ref{viscosityCFB}]
First, it is important to note that the operator $\cJ^{\smalltext{\rm FB}}$ introduced in \eqref{eq:operatorfirstbest} is Lipschitz-continuous and non-decreasing with respect to its second variable. These properties are sufficient to suggest that adapting the same techniques outlined in \cite[Lemma B.1]{possamai2020there} to our context is sufficient to prove a comparison theorem on $[-(\bar{a}-\varepsilon),\infty)$ for the Hamilton--Jacobi--Bellman equation \eqref{align:hjbFB} in the class of functions $v$ verifying $ |v(y) - \bar{F}(y)| \leq c^{\smalltext{\rm FB}}$, $y \geq 0$, for some $c^{\smalltext{\rm FB}}>0$. Furthermore, standard arguments from viscosity solution theory show that $v^{\smalltext{\rm FB}}$ is a discontinuous viscosity solution of \Cref{align:hjbFB} (the approach is similar to the one in \Cref{theorem:viscositySolFbar}). Consequently, \Cref{lemma:growthfirstbest} allows to apply the aforementioned comparison theorem to the function $v^{\smalltext{\rm FB}}$, leading to the conclusion that $v^{\smalltext{\rm FB}}$ is the unique viscosity solution, in that class of functions, and it is continuous. 
\end{proof}

\begin{proof}[Proof of \Cref{thm:FirstBestCompleteC}.\ref{deltaSmaller1}]
By assumption, there exists some $y^0 \in [y_m \vee \tilde{y}, \infty)$ such that $v^{\smalltext{\rm FB}}(y^0) = \bar{F}(y^0)$. Since $y^0 \geq \tilde{y}$, it follows that $\bar{F}(y) = w_0(y) -m$ for any $y \in [y^0, \infty)$ due to \Cref{proposition:FBarGeneralDeltaSmaller1}.\ref{nonDegenerateFBar}. Additionally, on this interval, $\bar{F}$ solves \Cref{align:hjbFB} because the fact that $w_0$ is a decreasing function implies that $G^{\star}(\delta w_0^\prime(y)) = -m$ for any $y \in [y^0, \infty)$. Therefore, the comparison theorem  mentioned in the proof of \Cref{thm:FirstBestCompleteC}.\ref{viscosityCFB} allows us to conclude.
\end{proof}

\section{Reduction of the second-best contracting problem}\label{section:secondBest}

This section focuses on reducing the Stackelberg game \eqref{secondBest_Pproblem} of the principal to a standard mixed control--stopping stochastic problem. As in the analysis of the first-best contracting problem when $\delta \neq 1$ and $\gamma\delta >1$, we identify a specific class of contracts $\bC$ that reveal the optimal response $\nu^\star(\bC)$ of the agent, while ensuring no loss of generality. Before exploring further into this approach, it is essential to note that the Hamiltonian associated to the problem of the agent is defined as
\begin{align}\label{align:defhA}
H^{\rm A}(z^{\smalltext{\rm A}},u^{\smalltext{\rm A}}) \coloneqq \sup_{ (a,b) \in U} h^{\rm A}(z^{\smalltext{\rm A}},u^{\smalltext{\rm A}}, a,b) \coloneqq \sup_{ (a,b) \in U} \bigg\{a z^{\smalltext{\rm A}} -  \frac{m-b}{m} \int_{\R} u^{\smalltext{\rm A}}(\ell) \Phi(\d \ell) - h(a,b)\bigg\}, \; \text{for} \;(z^{\smalltext{\rm A}}, u^{\smalltext{\rm A}}) \in \R\times\cB_{\R}.
\end{align}

Let us introduce an $\R$-valued $\F$-predictable process $Z^{\rm A}$ and an $\R$-valued $\widetilde{\cP}(\F)$-measurable function $U^{\rm A}$. For any given $\F$--stopping time $\tau$, we define the set $\cU^\star_\tau(Z^{{\rm A}},U^{{\rm A}})$ as the collection of efforts $\nu = (\alpha,\beta) \in \cU$ such that
\begin{align}\label{align:HamiltonianAgent}
(\alpha_t,\beta_t) \in \argmax_{(a,b) \in U}  h^{\rm A}(Z^{\rm A}_t,U^{\rm A}_t,a,b), \; \d t \otimes \d \P\text{--a.e.} \; \text{on} \; \llbracket 0, \tau \rrbracket.
\end{align}
Similarly to the approach taken in \eqref{align:defYAfb} for the first-best problem, we also need to introduce a family of processes that will be used to reformulate the admissible lump-sum payments at terminal time received by the agent. Let us fix a constant $Y^{\rm A}_0 \in \R$ along with a non-negative $\F$-predictable process $\pi$. We define the process $Y^{Y^{\smalltext{\rm A}}_\smalltext{0},Z^{\smalltext{\rm A}},U^{\smalltext{\rm A}},\pi}$ as the solution of the following SDE, for $t\geq0$:
\begin{align}\label{align:defYA}
Y^{Y^{\smalltext{\rm A}}_\smalltext{0},Z^{\smalltext{\rm A}},U^{\smalltext{\rm A}},\pi}_t = Y^{\rm A}_0 + \int_0^t r \big(Y^{Y^{\smalltext{\rm A}}_\smalltext{0},Z^{\smalltext{\rm A}},U^{\smalltext{\rm A}},\pi}_s - u(\pi_s) - H^{\rm A}(Z^{\rm A}_s,U^{\rm A}_s)\big) \d s + r \sigma \int_0^t Z^{\rm A}_s \d W_s + r \int_0^t \int_{\R} U^{\rm A}_s(\ell) \tilde{\mu}^{{J}}(\d s, \d \ell).
\end{align} 
Then, we introduce the set $\cV^{\smalltext{\rm A}}_\tau$ as the collection of all $\R$-valued, $\F$-predictable processes $Z^{\rm A}$ and $\R$-valued $\widetilde{\cP}(\F)$-measurable functions $U^{\rm A}$ such that the supremum \eqref{align:defhA} in the definition of $H^{\rm A}(Z^{{\rm A}},U^{{\rm A}})$ is attained, or equivalently, the set $\cU^\star_\tau(Z^{{\rm A}},U^{{\rm A}})$ is not empty, and the integrability conditions \eqref{align:intCondYA} and \eqref{align:intCondZAUA} are met.

\begin{remark}\label{remark:sdeContUtilityAgent}
\begin{enumerate}[label=$(\roman*)$, wide, labelindent=0pt]
\item\label{y0} As pointed out in {\rm\Cref{remark:existenceYsde}}, the previous {\rm SDE} admits a unique strong solution $Y^{Y^{\smalltext{\rm A}}_\smalltext{0},Z^{\smalltext{\rm A}},U^{\smalltext{\rm A}},\pi}$, which represents the continuation utility of the agent given the contract $\bC=(\tau,\pi,\xi)$, where $\xi = \xi^{Y^{\smalltext{\rm A}}_\smalltext{0},Z^{\smalltext{\rm A}},U^{\smalltext{\rm A}},\pi} \coloneqq u^{(-1)}\big(Y^{Y^{\smalltext{\rm A}}_\smalltext{0},Z^{\smalltext{\rm A}},U^{\smalltext{\rm A}},\pi}_\tau\big)$. 
\item We highlight that a comparison theorem for {\rm BSDEs} $($one can adapt the arguments of {\rm \cite[Theorem 3.7]{lin2020second}} and {\rm \cite[Proposition 7.3]{possamai2023reflections}} to the present framework$)$ imposes a non-negativity constraint of the continuation utility of the agent, namely $Y^{Y^{\smalltext{\rm A}}_\smalltext{0},Z^{\smalltext{\rm A}},U^{\smalltext{\rm A}},\pi}_{t \wedge \tau} \geq 0$ for any $t \geq 0$, $\P${\rm--a.s.}, given the non-negativity condition satisfied by both the utility function $u$ and the Hamiltonian $H^{\rm A}$. Additionally, the limited liability constraint implies that 
\begin{align*}
Y^{Y^{\smalltext{\rm A}}_\smalltext{0},Z^{\smalltext{\rm A}},U^{\smalltext{\rm A}},\pi}_{t-} + r \int_{\R} U^{\rm A}_t(\ell) \Phi(\d \ell) \geq  0, \; \d t \otimes \d \P\text{\rm --a.e.} \; \text{\rm on} \; \llbracket 0, \tau \rrbracket.
\end{align*}
\end{enumerate}
\end{remark}

\medskip
Considering the previous discussion, we define the set $\cV^{\smalltext{\rm S}}(Y^{\rm A}_0)$ as the collection of all quadruples $(\tau, \pi, Z^{\rm A}, U^{\rm A})$ verifying that $\tau$ is an $\F$--stopping time, $\pi$ is an $\F$-predictable non-negative process, $(Z^{\rm A},U^{\rm A}) \in \cV_\tau^{\smalltext{\rm A}}$, and the contract $(\tau, \pi, u^{(-1)}(Y^{Y^{\smalltext{\rm A}}_\smalltext{0},Z^{\smalltext{\rm A}},U^{\smalltext{\rm A}},\pi}_\tau))$ satisfies the integrability condition \eqref{eq:integrabilityCondition}. With this setup, we can now proceed to formulate the reduction argument to rewrite the problem of the principal as a mixed control--stopping problem.

\begin{theorem}\label{theorem:reductionSecondBest}
It holds that\footnote{The proof of \Cref{theorem:reductionSecondBest} shows that $Y^{Y^{\smalltext{\rm A}}_\smalltext{0},Z^{\smalltext{\rm A}},U^{\smalltext{\rm A}},\pi}_0 = Y^{{\rm A}}_0 = V^{\rm A}(\bC)$, hence, the condition $Y^{{\rm A}}_0 \geq u(R)$ ensures that the participation constraint is satisfied.}
\begin{align*}
{V}^{\rm P} = \bar{V}^{\rm P} = &\sup_{Y^{\smalltext{\rm A}}_\smalltext{0} \geq u(R)} \bar{V}^{\rm P}(Y^{\rm A}_0),
\end{align*}
with
\begin{align}\label{align: reducedProblem}
 \bar{V}^{\rm P}(Y^{\rm A}_0) \coloneqq \sup_{(\tau, \pi, Z^{\smalltext{\rm A}}, U^{\smalltext{\rm A}}) \in \cV^{\smalltext{\rm S}}(Y^{\smalltext{\rm A}}_\smalltext{0})} \sup_{\nu\in \cU_\tau^{\smalltext\star}(Z^{\smalltext{\rm A}},U^{\smalltext{\rm A}})} \E^{\P^{\smalltext{\nu}}}\bigg[\mathrm{e}^{-\rho \tau} \bar{F}(Y^{Y^{\smalltext{\rm A}}_\smalltext{0},Z^{\smalltext{\rm A}},U^{\smalltext{\rm A}},\pi}_\tau) \mathbf{1}_{\{\tau < \infty\}} + \int_0^\tau \rho \mathrm{e}^{-\rho s} (\alpha_s - \beta_s - \pi_s) \d s\bigg].
\end{align}
\end{theorem}
The result is proved in \Cref{appendix:reductionSecond}.

\subsection{The associated HJB equation}\label{section: HJB}

In accordance with the results from the reference paper \cite{possamai2020there} in the accident-free framework, we show that the second-best contracting problem \eqref{secondBest_Pproblem} degenerates when $\gamma\delta \leq 1$. In the alternative scenario where $\gamma \delta >1$, this degeneracy is no longer observed, and in order to characterise the solution of the Stackelberg game, we can consider the rewriting of game as a simpler mixed control--stopping problem, presented in \Cref{align: reducedProblem}. Therefore, by using standard techniques in stochastic control (see for instance \cite{oksendal2007applied}), we can introduce the following second-order integro-differential equation:
\begin{align}\label{align:hjbSB}
\min\big\{v^{\smalltext{\rm SB}}(y)-\bar{F}(y), F^{\star}(\delta (v^{\smalltext{\rm SB}})^\prime(y)) - \delta y (v^{\smalltext{\rm SB}})^\prime(y) + v^{\smalltext{\rm SB}}(y) -\cJ^{\smalltext{\rm SB}}\big(y,(v^{\smalltext{\rm SB}})^\prime(y), (v^{\smalltext{\rm SB}})^{\prime\prime}(y),v^{\smalltext{\rm SB}}(\cdot))\big\}=0, \; y \in (0,\infty),
\end{align}
for the initial condition $v^{\smalltext{\rm SB}}(0) = 0$. Here, $v^{\smalltext{\rm SB}}(y)$ represents the maximal reward obtained by the principal while maintaining the utility of the agent at the level $y \geq 0$, following a notation similar to that used for the first-best problem. For any $(y,p,q, v(\cdot)) \in \R_+ \times \R \times \R \times \cB_{\R}$, the non-local operator is of the form
\begin{align}\label{eq:operatorsecondbest}
\begin{split}
\cJ^{\smalltext{\rm SB}}(y,p,q, v(\cdot)) \coloneqq \sup_{(z^{\smalltext{\rm A}}, u^{\smalltext{\rm A}}) \in \mathfrak{V}_y}\sup_{(a,b) \in U^\smalltext{\star}(z^{\smalltext{\rm A}}, u^{\smalltext{\rm A}})} \bigg\{ & a-b+ \delta h(a,b) p + \frac{r\delta \sigma^2}{2} (z^{\rm A})^2 q \\
&+ \frac{b}{m \rho} \int_{\R} \big(v(y + r u^{\rm A}(\ell)) - v(y) - r p u^{\rm A}(\ell) \big)  \Phi(\d \ell) \bigg\}  \mathbf{1}_{(-\infty,0]} (q) + \infty  \mathbf{1}_{(0,\infty)}(q),
\end{split}
\end{align}
where the set $\mathfrak{V}_y$ is the collection of all $z^{\rm A} \in \R$ and $u^{\rm A} \in \cB_\R$ such that
\begin{equation}\label{equation:minPlusIntCond}
\min_{\ell \in \R_\smalltext{+}} \big\{r u^{{\rm A}}(\ell) \big\}\geq -y, \; \text{and} \; \int_{\R} |u^{{\rm A}}(\ell)|^\mu \Phi(\d \ell) <  \infty,
\end{equation}
for some $\mu \geq 1$. Additionally, recalling that the function $h^{\rm A}$ is introduced in \Cref{align:defhA}, we denote 
\begin{align*}
U^\star(z^{\rm A},u^{\rm A}) \coloneqq \bigg\{u^\star = (a^\star,b^\star) \in U :  u^\star \in \argmax_{(a,b) \in U} \big\{ h^{\rm A}(z^{\rm A},u^{\rm A},a,b)\big\}\bigg\}.
\end{align*}

\begin{remark}\label{remark:wellDefinedOperator}
It is important to note that the integral operator is of order zero due to the fact that $\Phi$ is a bounded measure by assumption since it represents the accident size distribution. Moreover, we have supposed that its support is contained in $[c_m,C_m]$, for some $0<c_m \leq C_m$. Consequently, there is no singularity even at infinity and the integrability condition stated in {\rm\Cref{equation:minPlusIntCond}} is sufficient to guarantee that the integral operator is well-defined for any function $v \in \cG_\mu$, where 
\begin{equation*}
\cG_\mu \coloneqq \bigg\{v: [0,\infty) \longrightarrow \R:  \sup_{y \in [0,\infty)} \frac{|v(y)|}{1 + |y|^\mu} <  \infty\bigg\}.
\end{equation*}
\end{remark}

\medskip
\Cref{align:hjbSB} should be interpreted in a weaker sense because we do not know \emph{a priori} whether the value function $v^{\smalltext{\rm SB}}$ is smooth, given that it is well-known it does not hold true in various applications. Therefore, we need to refer to an appropriate concept of viscosity solution.

\subsubsection{Viscosity solutions for the integro-differential HJB equation}

In this section, we introduce the notion of viscosity solution that we refer to throughout the paper since it strictly depends on the integral term appearing in \Cref{align:hjbSB}, especially on the regularity of the measure $\Phi$. For the purpose of our application, it is convenient to adopt the definition of viscosity solution from \cite[Definition 2.1]{hollender2016levy} (analogous to the one used by \cite[Definition 2]{alvarez1996viscosity}), specifically the version derived from \cite[Lemma 2.6]{hollender2016levy}, as we seek a solution $v^{\smalltext{\rm SB}}$ to \eqref{align:hjbSB} that is defined on the interval $[0,\infty)$ rather than the entire real line.

\begin{definition}\label{def:viscositydef}
Let us fix $\mu \geq 1$. An upper--semi-continuous $($resp. lower--semi-continuous$)$ function $u:[0,\infty) \longrightarrow \R$ is a viscosity sub-solution $($resp. viscosity super-solution$)$ of \eqref{align:hjbSB} on $(0,\infty)$ if $u \in \cG_\mu$, and for any $(y, \phi) \in (0,  \infty) \times C^2((0,\infty))$ such that $\phi \in \cG_\mu $, and $u-\phi$ attains a global maximum $($resp. global minimum$)$ at $y$, it holds that 
\begin{align*}
\min\left\{u(y)-\bar{F}(y), F^{\star}(\delta \phi^\prime(y)) - \delta y \phi^\prime(y)+u(y) -\cJ^{\smalltext{\rm SB}}(y,\phi^\prime(y), \phi^{\prime\prime}(y),\phi(\cdot))\right\}\leq 0 \; (\text{resp.} \; \geq 0).
\end{align*}
Additionally, we say that a locally bounded function $u:[0,\infty) \longrightarrow \R$ is a viscosity solution of \eqref{align:hjbSB} on $(0,\infty)$ if its upper--semi-continuous envelope and its lower--semi-continuous envelope is a viscosity sub-solution and a viscosity super-solution of \eqref{align:hjbSB} on $(0,\infty)$, respectively.
\end{definition}

We can take advantage of the regularity of the integral term that characterises our Hamilton--Jacobi--Bellman equation \eqref{align:hjbSB} to provide two equivalent formulations of being a viscosity sub-solution---and consequently a viscosity super-solution. These alternative definitions will be more convenient in proving the comparison theorem for \eqref{align:hjbSB}, as we will employ a standard maximum principle that invokes the notion of semi-jets instead of being formulated in terms of test functions. 

\begin{lemma}\label{lemma:definitionsViscSub}
An upper--semi-continuous function $u:[0,\infty) \longrightarrow \R$ is a viscosity sub-solution of \eqref{align:hjbSB} on $(0,\infty)$ if and only if either of the two following conditions hold:
\begin{enumerate}[label=$(\roman*)$, wide, labelindent=0pt]
\item\label{SubWithJets} for any $y\in (0,  \infty)$, it is true that\footnote{We refer to \citeauthor*{touzi2013optimal} \cite[Section 6.4.2]{touzi2013optimal} for the definitions of the semi-jets $J^{2,+}u(y)$ and $J^{2,-}u(y)$ along with their respective closures $\bar{J}^{2,+}u(y)$ and $\bar{J}^{2,-}u(y)$. Additional properties can be found in \cite[Remarks 2.7]{crandall1992user}.}
\begin{align}\label{align:subSemijets}
\min\big\{u(y)-\bar{F}(y), F^{\star}(\delta p) - \delta y p + u(y) -\cJ^{\smalltext{\rm SB}}(y,p, q,u(\cdot))\big\}\leq 0, \; \text{\rm for} \; (p,q) \in J^{2,+}u(y);
\end{align}
\item\label{noJets}for any $(y, \phi) \in (0,  \infty) \times C^2((0,\infty))$ such that $u-\phi$ attains a global maximum at $y$, it is true that
\begin{align*}
\min\big\{u(y)-\bar{F}(y), F^{\star}(\delta \phi^\prime(y)) - \delta y \phi^\prime(y)+u(y) -\cJ^{\smalltext{\rm SB}}(y,\phi^\prime(y), \phi^{\prime\prime}(y),u(\cdot))\big\}\leq 0.
\end{align*}
\end{enumerate}
\end{lemma}
\begin{remark}
As noted in {\rm\cite[Remark 2]{alvarez1996viscosity}}, in {\rm\Cref{lemma:definitionsViscSub}}.\ref{SubWithJets}, it is equivalent to require that a sub-solution $u$ satisfies \eqref{align:subSemijets} for any $(p,q) \in \bar{J}^{2,+}u(y)$, where $y \in (0,\infty)$.
\end{remark}

\begin{proof}
It is straightforward to verify that \Cref{lemma:definitionsViscSub}.\ref{SubWithJets} implies that $u$ is a viscosity sub-solution of \eqref{align:hjbSB} on $(0,\infty)$. Indeed, consider $(y, \phi) \in (0,  \infty) \times C^2((0,\infty))$ such that $\phi \in \cG_\mu $ and $u-\phi$ attains a global maximum at $y$. Consequently, by definition, it holds that $(\phi^\prime(y), \phi^{\prime\prime}(y)) \in J^{2,+}u(y)$. Furthermore, we require that $(u-\phi)(y) = 0$, and thus $u \leq \phi$ on $[0,\infty)$. Hence, the monotonicity of the integral operator allows us to conclude since
\begin{align*}
&\min\left\{u(y)-\bar{F}(y), F^{\star}(\delta \phi^\prime(y)) - \delta y \phi^\prime(y)+u(y) -\cJ^{\smalltext{\rm SB}}(y,\phi^\prime(y), \phi^{\prime\prime}(y),\phi(\cdot))\right\} \\
&\leq 
\min\left\{u(y)-\bar{F}(y), F^{\star}(\delta \phi^\prime(y)) - \delta y \phi^\prime(y)+u(y) -\cJ^{\smalltext{\rm SB}}(y,\phi^\prime(y), \phi^{\prime\prime}(y),u(\cdot))\right\} \leq 0.
\end{align*}

\medskip
Analogously, we can show that \Cref{lemma:definitionsViscSub}.\ref{SubWithJets} implies \Cref{lemma:definitionsViscSub}.\ref{noJets}. Conversely, let us assume \Cref{lemma:definitionsViscSub}.\ref{noJets} and fix some $y >0$ and $(p,q) \in J^{2,+}u(y)$. By referencing \citeauthor*{fleming2006controlled} \cite[Lemma V.4.1]{fleming2006controlled} (or equivalently \citep[Theorem 1.23]{hollender2016levy}), it follows that there exists a $C^2((0,\infty))$-function $\phi$ such that $(p,q) = (\phi^\prime(y), \phi^{\prime\prime}(y))$, and $u-\phi$ achieves its maximum at $y$. Moreover, we can also assume that $(u-\phi)(y) = 0$. This implies \Cref{lemma:definitionsViscSub}.\ref{SubWithJets}, and thus proves the desired implication.

\medskip
Finally, in order to show that \Cref{def:viscositydef} implies \Cref{lemma:definitionsViscSub}.\ref{SubWithJets}, we consider some $y >0$ and $(\tilde{p},\tilde{q}) \in J^{2,+}u(y)$. As before, we can derive the existence of a $C^2((0,\infty))$-function $\phi_0$ such that $(\tilde{p},\tilde{q}) = (\phi_0^\prime(y), \phi_0^{\prime\prime}(y))$, and $u-\phi_0$ achieves its maximum at $y$, where $(u-\phi)(y) = 0$. The result in \cite[Theorem 1.20]{hollender2016levy} allows us to select a sequence of $C^2((0,\infty))$-functions $(\phi_n)_{n \in \N}$ such that, for each $n \in \N$, $\phi_n \in \cG_\mu$ for some $\mu \geq 1$ and
\begin{align*}
    u(x) \leq \phi_{n + 1} (x) \leq \phi_n(x) \leq \phi_0(x) \; \text{for any} \; x >0.
\end{align*}
Moreover, it holds that $\lim_{n \rightarrow  \infty} \phi_n (x) = u (x)$ for any $x >0$. We can deduce that $u - \phi_n$ and $\phi_n - \phi_0$ also have global maxima at $y$. The latter maximum condition implies that $\phi_n^\prime(y) = \phi_0^\prime(y)= \tilde{p}$ and $\phi_n^{\prime\prime}(y) \leq  \phi_0^{\prime\prime}(y) = \tilde{q}$. Therefore, 
\begin{align*}
F^{\star}(\delta p(y)) - \delta y \tilde{p}(y) + u(y) -\cJ^{\smalltext{\rm SB}}(y,\tilde{p},\tilde{q},\phi_n(\cdot)) \leq F^{\star}(\delta \phi_n^\prime(y)) - \delta y \phi_n^\prime(y) + u(y) -\cJ^{\smalltext{\rm SB}}(y,\phi_n^\prime(y), \phi_n^{\prime\prime}(y),\phi_n(\cdot)) \leq 0.
\end{align*}
The following equality holds:
\begin{align*}
\int_{\R} \big(\phi_n(y + r u^{\rm A}(\ell)) - \phi_n(y) - r \tilde{p} u^{\rm A}(\ell) \big)  \Phi(\d \ell)  = \int_{\R} \big(\phi_n(y + r u^{\rm A}(\ell)) - u(y) - r \tilde{p} u^{\rm A}(\ell) \big)  \Phi(\d \ell),
\end{align*}
and thus the proof immediately follows from the the monotone convergence theorem, as it implies that
\begin{align*}
\int_{\R} \big(\phi_n(y + r u^{\rm A}(\ell)) - u(y) - r \tilde{p} u^{\rm A}(\ell) \big)  \Phi(\d \ell) \underset{n\to\infty}{\longrightarrow} \int_{\R} \big(u(y + r u^{\rm A}(\ell)) - u(y) - r \tilde{p} u^{\rm A}(\ell) \big)  \Phi(\d \ell).
\end{align*}
\end{proof}

\subsection{Characterisation of the solution to the contracting problem}\label{section:discussProblemHJB}

As already mentioned, the second-best contracting problem degenerates when $\gamma\delta \leq 1$. Specifically, we can show that the reward of the principal reaches its maximum $\bar{a} - \varepsilon_m$ by means of a sequence of admissible contracts that offer the agent small intermediate payments and promise him a large lump-sum payment at a large retirement time. These contracts are constructed in such a way that the agent exerts maximum effort throughout their extremely long duration. However, the large discrepancy between the discount rates of the agent and the principal ensures that even the utility of the agent reaches its maximum.

\medskip
The alternative scenario where $\gamma\delta >1$ is more interesting to analyse since the problem no longer degenerates. Nevertheless, it is quite challenging since the non-local nature of the Hamilton--Jacobi--Bellman equation \eqref{align:hjbSB} prevents us from replicating the regularity results presented in \citep{possamai2020there}. What we can still do is characterising the second-best value function $v^{\smalltext{\rm SB}}$ as the unique viscosity solution---within a specific class of functions---to the aforementioned equation by applying dynamic programming principle and some other standard arguments in control theory.

\begin{theorem}\label{thm:SecondtBestCompleteC}
\begin{enumerate}[label=$(\roman*)$, wide, labelindent=0pt]
\item \label{degSB}If $\gamma\delta \leq 1$, then $v^{\smalltext{\rm SB}} = \bar{a} - \varepsilon_m$ on $[0,\infty)$. Moreover, there is no admissible contract achieving this value.
\item \label{nodegSB}Let $\gamma\delta > 1$. The second-best value function $v^{\smalltext{\rm SB}}$ is the unique continuous viscosity solution of {\rm\Cref{align:hjbSB}} in the class of functions $v$ such that $|v(y) - \bar{F}(y)| \leq c^{\smalltext{\rm SB}}$ for any $y \geq 0$, for some $c^{\smalltext{\rm SB}}>0$. Additionally, $(v^{\smalltext{\rm SB}})^\prime(0) \leq 0$ if $\delta >1$.
\end{enumerate}
\end{theorem}
The proof of \Cref{thm:SecondtBestCompleteC}.\ref{degSB} is given in \Cref{degeneracySeconBest}.
\begin{proof}[Proof of \Cref{thm:SecondtBestCompleteC}.\ref{nodegSB}]
The result is obtained in two main steps. First, the fact that the second-best value function $v^{\smalltext{\rm SB}}$ is a viscosity solution of \Cref{align:hjbSB} is an immediate consequence of some standard arguments in control theory (see for instance \cite{oksendal2007applied}). Second, \Cref{thm:FirstBestCompleteC} implies that $|v^{\smalltext{\rm SB}}(y)- \bar{F}(y)| \leq c^{\smalltext{\rm SB}} \; \text{for any} \; y \in [0,\infty), \; \text{for some} \; c^{\smalltext{\rm SB}}>0,$ given the fact that the first-best value function $v^{\smalltext{\rm FB}}$ is an upper bound for $v^{\smalltext{\rm SB}}$ on their domain of definition $[0,\infty)$. In \Cref{compResultapp}, we prove a comparison result for the integro-differential variational inequality \eqref{align:hjbSB} that immediately allows us to deduce that the second-best value function $v^{\smalltext{\rm SB}}$ is the unique viscosity solution of \Cref{align:hjbSB}, and that it is continuous. Lastly, if $\delta>1$, we can prove that $(v^{\smalltext{\rm SB}})^\prime(0) \leq 0$, where $(v^{\smalltext{\rm SB}})^\prime(0)$ exists as the right derivative of a continuous function. This is an immediate consequence of \cite[Theorem 3.4]{possamai2020there} as the second-best value function in the accident-free setting is a viscosity super-solution of our Hamilton--Jacobi--Bellman equation \eqref{align:hjbSB}.
\end{proof}

Unfortunately, we are unable to replicate the regularity results proved in the reference paper without accidents. Specifically, the fact that $\cJ^{\smalltext{\rm SB}}$ maximises over the set of controls $\mathfrak{V}_y$ that depend on the considered point $y \geq 0$ prevents us from showing that whenever the second-best value function $v^{\smalltext{\rm SB}}$ intersects its barrier $\bar{F}$, then it coincides with it forever (this is not possible even under the assumption of \cite[Lemma 8.1]{possamai2020there} that $F^\prime$ is concave). Consequently, we cannot characterise the stopping region $\{ v^{\smalltext{\rm SB}} = \bar{F} \}$ of the contracting problem. Moreover, the introduction of accidents poses another challenge stemming from the definition of the operator $\cJ^{\smalltext{\rm SB}}$. The latter is described as the sum of two components: one related to the effort $a$ on the drift of the output process and the other related to the effort $b$ on the intensity of accidents. For the sake of simplicity, let us assume that the cost function is separable, specifically $h(a,b) = h_a(a) + h_b(b)$, for $(a,b) \in A \times B$. As a result,
\begin{align}\label{eq:operatorsecondbestSEP}
\cJ^{\smalltext{\rm SB}}(y,p,q, v(\cdot)) = \big(\cJ^{\smalltext{\rm SB} \smalltext{:} \smalltext{a}}(p,q)  + \cJ^{\smalltext{\rm SB} \smalltext{:} \smalltext{b}}(y,p, v(\cdot))\big)  \mathbf{1}_{(-\infty,0]}(q) + \infty  \mathbf{1}_{(0,\infty)}(q) \; \text{for any} \; (y,p,q, v(\cdot)) \in \R_+ \times \R \times \R \times \cB_{\R},
\end{align}
where 
\begin{align*}
\cJ^{\smalltext{\rm SB} \smalltext{:} \smalltext{a}}(p,q) &\coloneqq  \sup_{z^{\smalltext{\rm A}}\in \R}\sup_{a \in A^\smalltext{\star}(z^{\smalltext{\rm A}})} \bigg\{ a + \delta h_a(a) p + \frac{r\delta \sigma^2}{2} (z^{\rm A})^2 q \bigg\} \; \text{and} \\
\cJ^{\smalltext{\rm SB} \smalltext{:} \smalltext{b}}(y,p, v(\cdot)) &\coloneqq \sup_{\substack{ u^{\smalltext{\rm A}}\in  \cB_\R \; \text{s.t.} \\ \text{conditions \eqref{equation:minPlusIntCond} are satisfied}}} \sup_{b \in B^\smalltext{\star}(u^{\smalltext{\rm A}})} \bigg\{-b + h_b(b) p + \frac{b}{m \rho} \int_{\R} \big(v(y + r u^{\rm A}(\ell)) - v(y) - r p u^{\rm A}(\ell) \big)  \Phi(\d \ell) \bigg\}.
\end{align*}
Hence, it is evident that $\cJ^{\smalltext{\rm SB} \smalltext{:} \smalltext{a}} \geq 0$ and $\cJ^{\smalltext{\rm SB} \smalltext{:} \smalltext{b}} \geq -m$. To analyse the differences with the framework without accidents, we consider an open interval $\cI \subset (0,\infty)$ where the second-best value function $v^{\smalltext{\rm SB}}$ does not coincide with its barrier $\bar{F}$, and thus solves the ODE that characterises the variational inequality \eqref{align:hjbSB}. Two different scenarii can occur:

\begin{enumerate}[label=$(\roman*)$, wide, labelindent=0pt]
\item if the operator $\cJ^{\smalltext{\rm SB} \smalltext{:} \smalltext{a}}$ is positive on $\cI$, then it is possible to demonstrate that $v^{\smalltext{\rm SB}}$ is twice continuously differentiable on $\cI$. Indeed, similarly to Step 2 in the proof of \cite[Lemma 8.1]{possamai2020there}, we can show that the aforementioned ODE is uniformly elliptic and consequently admits a classical solution $w$ (unique if initial conditions are provided). This is because \Cref{thm:SecondtBestCompleteC}.\ref{nodegSB} implies that $v^{\smalltext{\rm SB}}$ is continuous, and therefore the ODE can be rewritten as \begin{align}\label{align:newEqSB}
F^{\star}(\delta w^\prime(y)) - \delta y w^\prime(y) + w(y) -\cJ\big(y,w^\prime(y), w^{\prime\prime}(y))=0, \; y \in \cI,
\end{align}
where
\begin{align}\label{eq:operatorsecondbestNEW}
\cJ(y,p,q) \coloneqq \big(\cJ^{\smalltext{\rm SB} \smalltext{:} \smalltext{a}}(p,q) + \cJ^{\smalltext{\rm SB} \smalltext{:} \smalltext{b}}(y,p, v^{\smalltext{\rm SB}}(\cdot)) \big) \mathbf{1}_{(-\infty,0]}(q) + \infty  \mathbf{1}_{(0,\infty)}(q), \; (p,q) \in \R^2.
\end{align}
We argue that the solution $w$ coincides with $v^{\smalltext{\rm SB}}$ on $\cI$ since the comparison theorem \cite[Lemma B.1]{possamai2020there} can still be applied to \Cref{align:newEqSB}. However, we are unable to find sufficient conditions that guarantee that the operator $\cJ^{\smalltext{\rm SB} \smalltext{:} \smalltext{a}}$ is positive. For example, in the accident-free framework with $\delta = 1$, the fact that the value function is necessarily strictly concave implies that $\cJ^{\smalltext{\rm SB} \smalltext{:} \smalltext{a}}$ is positive in a right-neighbourhood of zero. In the case with accidents, this condition only implies that the operator $\cJ$ is positive, and this is compatible with having $\cJ^{\smalltext{\rm SB} \smalltext{:} \smalltext{a}}(p,q) = 0$ and $\cJ^{\smalltext{\rm SB} \smalltext{:} \smalltext{b}} (y,p,v(\cdot))> 0$ if the first derivative $p$ is positive for any $(y,q,v) \in \R_+ \times \R \times \cB_{\R}$;
\item if the operator $\cJ^{\smalltext{\rm SB} \smalltext{:} \smalltext{a}}$ vanishes on $\cI$, this does not necessarily imply that $\cJ = -m$ on $\cI$, and thus the second-best value function $v^{\smalltext{\rm SB}}$ might solve a different equation from the one that characterises its barrier $\bar{F}$. This is a pure non-linear first-order integro-differential equation. This possibility can never occur in the problem without accidents since whenever $\cJ^{\smalltext{\rm SB} \smalltext{:} \smalltext{a}} = 0$, then the value function touches $\bar{F}$ or it solves the same equation but with a different initial condition. 
\end{enumerate}

These considerations lead us to believe that it is unlikely for our second-best value function $v^{\smalltext{\rm SB}}$ to be more regular than continuously differentiable, especially when considering that \Cref{thm:FirstBestCompleteC} proves that its upper bound $v^{\smalltext{\rm PB}}$ is only continuously differentiable when $\delta = 1$.

\begin{remark}
\begin{enumerate}[label=$(\roman*)$, wide, labelindent=0pt]
\item It is well-known that the existence of an optimal control for the reduced mixed control--stopping problem \eqref{align: reducedProblem}, and thus of an optimal contract for the principal--agent problem, is linked to the existence of a classical smooth solution to the corresponding Hamilton--Jacobi--Bellman equation \eqref{align:hjbSB}. This connection is evident in {\rm\cite[Theorem 3.4]{possamai2020there}}, where the smoothness of the solution to the Hamilton–Jacobi–Bellman equation is essential for deducing the existence of an optimal contract. Since we have been unable to prove that the second-best value function $v^{\smalltext{\rm SB}}$ is smooth enough due to the complexity of the second-order integro-differential equation characterising the problem, we are unable to deduce the existence of an optimal contract $\bC^\star$. Nonetheless, if an optimal contract $\bC^\star$ exists, it has to be of the form
\begin{align*}
\bC^\star \coloneqq \big(\tau^\star, \pi^\star, u^{(-1)}(Y^{Y^{\smalltext{\rm A},\star}_\smalltext{0}{,}Z^{\smalltext{\rm A},\star}{,}U^{\smalltext{\rm A},\star},\pi^\star}_\tau)\big),
\end{align*}
for some $(\tau^\star, \pi^\star, Z^{\rm A,\star}, U^{\rm A,\star}) \in \cV^{\smalltext{\rm S}}(Y^{\rm A,\star}_0)$, where $Y^{\rm A,\star}_0 \geq U(R)$ is such that $\sup_{Y^{\smalltext{\rm A}}_\smalltext{0} \geq u(R)} \bar{V}^{\rm P}(Y^{\rm A}_0) =  \bar{V}^{\rm P}(Y^{\rm A,\star}_0)$ in {\rm\Cref{theorem:reductionSecondBest}}. Moreover, the agent's optimal effort is determined by the maximiser of his Hamiltonian, i.e. $\nu^\star\in \cU_{\tau^\star}^{\star}(Z^{{\rm A,\star}},U^{{\rm A,\star}})$.
\item A possible approach to the existence of optimal controls found in the literature is to allow the agent to use measure-valued controls, although this would necessitate studying a different and more general principal-agent model, see {\rm\citeauthor*{krvsek2023randomisation} \cite{krvsek2023randomisation}}. An alternative method could involve constructing approximate optimal contracts starting from an approximate problem as in {\rm\citeauthor*{jakobsen2008error} \cite{jakobsen2008error}}.
\end{enumerate}
\end{remark}

\subsection{Numerical result}

Given the impossibility of explicitly characterising the second-best value function $v^{\smalltext{\rm SB}}$, it is necessary to perform some numerical simulations to illustrate its behaviour and provide insights into its properties. For computational tractability, we simplify the model by assuming that accidents have a constant size, specifically equal to $m$. Consequently, the distribution function $\Phi$ of each accident becomes a Dirac measure. This simplification facilitates the direct calculation of the maximisers of the Hamiltonian in \eqref{align:defhA} and, as a result, enables the use of a finite difference approximation to solve the Hamilton--Jacobi--Bellman equation \eqref{align:hjbSB}. To draw a comparison with the case analysed by \cite[Figure 1]{sannikov2008continuous}, later revisited by \cite[Example 3.10]{possamai2020there}, we consider the scenario $\delta=1$ within the specific positive power utility setting with separable cost
\begin{align*}
h(a,b) \coloneqq h_a(a) + h_b(b) \coloneqq \bigg(\frac{a^2}{2} + \frac{2 a}{5}\bigg) + \bigg(\frac{1}{b} - \frac{1}{m}\bigg), \;\text{\rm where} \; (a,b) \in A \times B \coloneqq [0,\bar{a}] \times [\varepsilon_m,m].
\end{align*}
We set the parameters as follows: $\gamma = 2$, $\bar{a} = 4.6$, $\varepsilon_m = 0.1$ and $m=0.2$, along with $r= 0.1 = \rho$ and $\sigma = 1$. Note that despite considering the set $A$ as bounded, fixing this bound to $\bar{a} = 4.6$ enables us to illustrate \citeauthor*{sannikov2008continuous}’s contracting problem within the context of accidents, even though \citeauthor*{sannikov2008continuous}’s problem is characterised by an unbounded $A$. This is because this parameter choice ensures that the optimal effort $a^\star$ remains bounded by $\bar{a}=4.6$. 

\begin{figure}[h!tbp]
  \centering
  \includegraphics[width=0.5\textwidth]{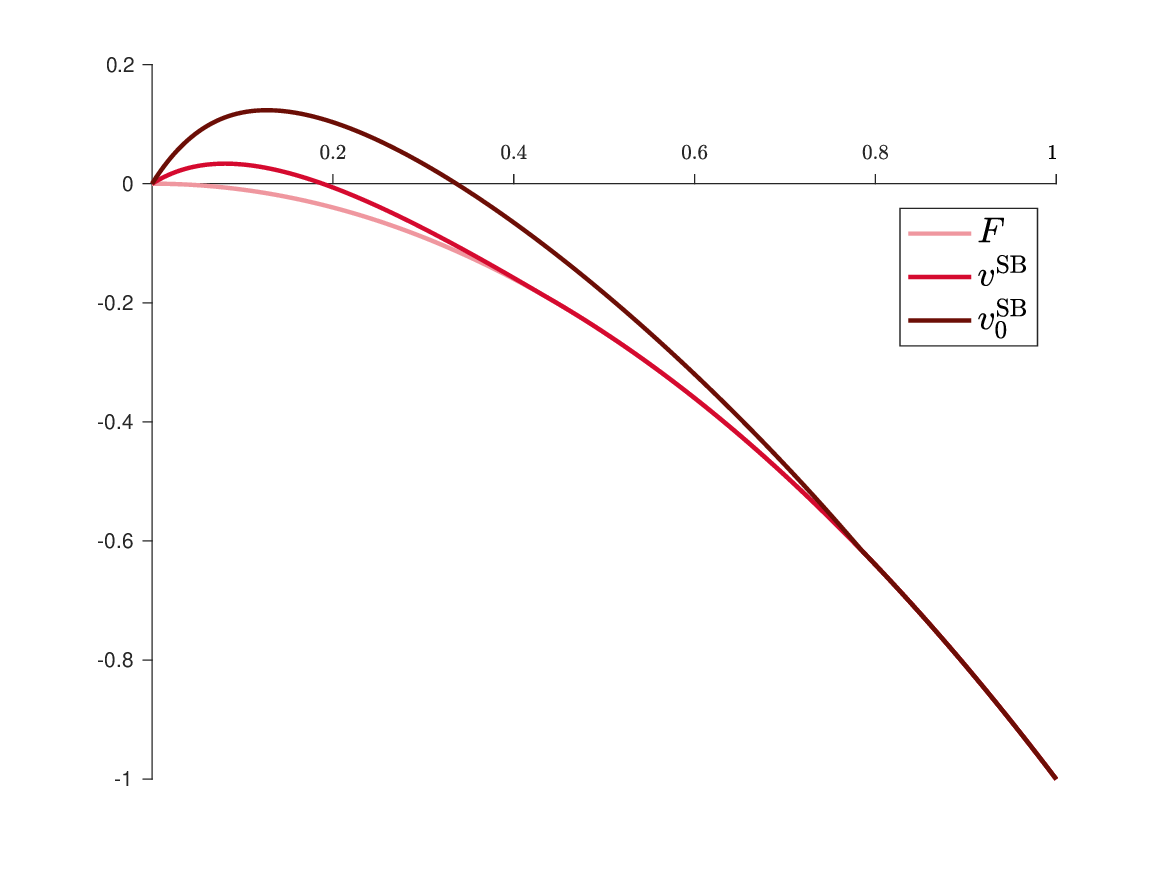}
  \caption{the second-best value function $v^{\smalltext{\rm SB}}$, its barrier $F$ and the accident-free second-best value function $v^{\smalltext{\rm SB}}_0$}
  \label{fig:deltaE1_secondBest}
\end{figure}

\Cref{fig:deltaE1_secondBest} illustrates the second-best value function $v^{\smalltext{\rm SB}}$ and its barrier $F$, along with the second-best value function within the accident-free framework $v^{\smalltext{\rm SB}}_0$. The latter obviously represents an upper bound for $v^{\smalltext{\rm SB}}$ as the criterion of the principal is higher in the absence of accidents. While these two functions do not coincide over the entire half-line, it is evident that the features defining $v^{\smalltext{\rm SB}}_0$ are transferred into the description of $v^{\smalltext{\rm SB}}$. In other words, this parameter choice leads us to infer that the agent with a small reservation utility benefits from an informational rent because the principal optimally offers him a contract that gives utility strictly higher than his participation value. Furthermore, we can notice that the second-best value function $v^{\smalltext{\rm SB}}$ does not coincide with its barrier ${F}$ in a neighbourhood of zero, but then becomes equal to it, indicating the existence of a golden parachute. This suggests that if the agent's reservation utility is higher than the point where these functions intersect, the contract terminates immediately. In such a scenario, the agent exerts no effort but receives a positive lump-sum compensation since we are considering the case where both parties are equally impatient, resulting in $\bar{F} = F$.

\medskip
To better understand how the second-best value function $v^{\smalltext{\rm SB}}$ is affected by the accidents' size, we compute this function for different values of $m$ while keeping the other parameters unchanged. \Cref{fig:secondBestDifferentValuesM} clearly shows the convergence of $v^{\smalltext{\rm SB}}$ to the second-best value function within the accident-free framework $v^{\smalltext{\rm SB}}_0$ as $m$ decreases to 0. Furthermore, it is evident that as $m$ increases, the value function $v^{\smalltext{\rm SB}}$ reaches the barrier $F$ sooner. This happens because, as $m$ becomes large, the size of potential accidents grows, leading the principal to decide earlier to terminate the contract due to her increased fear. Consequently, we can deduce that accidents significantly influence the second-best value function $v^{\smalltext{\rm SB}}$ when the agent's reservation utility is not too high. This occurs because in this case, the principal prefers to terminate the contract earlier than in the framework with fewer or no accidents. This preference arises from the increased cost associated with potential accidents, which exceeds the cost of offering the agent an immediate lump-sum compensation. This phenomenon is clearly visible in the zoomed-in box.

\begin{figure}[h!tbp]
  \centering
  \includegraphics[width=0.5\textwidth]{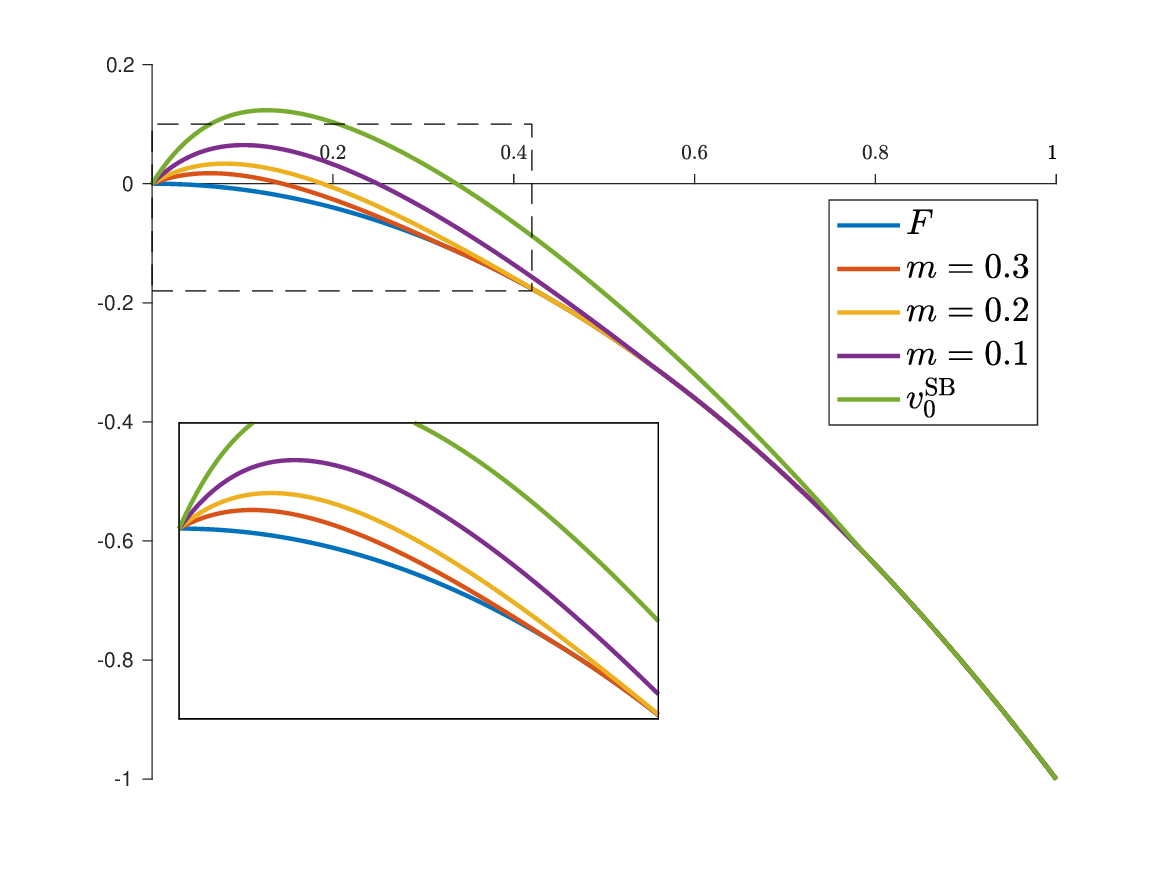}
  \caption{the second-best value function $v^{\smalltext{\rm SB}}$ for different values of $m$}
  \label{fig:secondBestDifferentValuesM}
\end{figure}

To conclude the analysis, we present in \Cref{table:maxAndLoss} the maximum values assumed by the functions $v^{\smalltext{\rm SB}}_0$ and $v^{\smalltext{\rm SB}}$ for different values of $m$. This helps us to better understand the loss in utility for the principal due to the presence of accidents. It is evident that as $m$ increases, the relative loss increases significantly. This phenomenon occurs because the larger the potential loss the principal could face, the more inclined she is to avoid risks and terminate the contract by compensating the agent with a lump-sum payment. It is worth noting that seemingly small values of $m$ can result in significant losses, approaching nearly 50\%, as seen in the case of $m=0.1$. 

\begin{table}[ht] 
\centering
\begin{tabular}[t]{lcccc}
\hline
&$v^{\smalltext{\rm SB}}_0$ & $m = 0.1$ & $m = 0.2$ & $m = 0.3$\\[0.5ex] 
\hline
maximum value & 0.1234 & 0.0648 & 0.0336 & 0.0176\\
relative loss & --- & 47.48\% & 72.75\% & 85.75\% \\
\hline
\end{tabular}
\caption{maximum value and relative loss for different values of $m$}\label{table:maxAndLoss}
\end{table}

To demonstrate that $m=0.1$ indeed represents a small value concerning potential accidents, we compare it against the average value assumed by the project's value $X$ at the termination of the contract. Specifically, we simulate several paths of the project's value $X$, whose dynamics under the optimal effort $\nu^\star = (\alpha^\star, \beta^\star)$ is given by
\begin{equation*}
X_{t} = x_0 + \int_0^{t} \alpha^\star_s \d s + \sigma W^{\nu^\star}_{t} - mN_{t}, \; \text{\rm for} \; t \geq 0,
\end{equation*}
where the intensity of the Poisson process $N$ is $\int_0^\cdot (\beta^\star_s/m) \d s$. We consider $x_0 = 0$. Additionally, the optimal control $\nu^\star$ is associated with the continuation value of the agent $Y^{Y^{\smalltext{\rm A}}_\smalltext{0},\hat{Z}^{\smalltext{\rm A}},\hat{U}^{\smalltext{\rm A}},\hat{\pi}}$, whose dynamics is given in \Cref{align:defYA}. More precisely, $\hat{Z}^{{\rm A}}_t$ and $\hat{U}^{{\rm A}}_t$ are the maximisers of the operator $\cJ^{\smalltext{\rm SB}}\big(Y^{{Y}^{\smalltext{\rm A}}_\smalltext{0},{\hat{Z}}^{\smalltext{\rm A}},{\hat{U}}^{\smalltext{\rm A}},{\hat{\pi}}}_t,(v^{\smalltext{\rm SB}})^\prime(Y^{{Y}^{\smalltext{\rm A}}_\smalltext{0},{\hat{Z}}^{\smalltext{\rm A}},{\hat{U}}^{\smalltext{\rm A}},{\hat{\pi}}}_t), (v^{\smalltext{\rm SB}})^{\prime\prime}(Y^{{Y}^{\smalltext{\rm A}}_\smalltext{0},{\hat{Z}}^{\smalltext{\rm A}},{\hat{U}}^{\smalltext{\rm A}},{\hat{\pi}}}_t),v^{\smalltext{\rm SB}}(\cdot))$, with associated optimal effort $\nu^\star_t = (\alpha^\star_t, \beta^\star_t) \in U^{\star}(\hat{Z}^{{\rm A}}_t, \hat{U}^{{\rm A}}_t)$, and $\hat{\pi}_t \coloneqq (F^\star)^\prime((v^{\smalltext{\rm SB}})^\prime(Y^{Y^{\smalltext{\rm A}}_\smalltext{0},\hat{Z}^{\smalltext{\rm A}},\hat{U}^{\smalltext{\rm A}},\hat{\pi}}_t))$. The termination of the contract then occurs at 
\begin{equation*}
\tau^\star \coloneqq \inf\Big\{ t \geq 0: \; v^{\smalltext{\rm SB}}\big(Y^{Y^{\smalltext{\rm A}}_\smalltext{0},\hat{Z}^{\smalltext{\rm A}},\hat{U}^{\smalltext{\rm A}},\hat{\pi}}_t\big) = F\big(Y^{Y^{\smalltext{\rm A}}_\smalltext{0},\hat{Z}^{\smalltext{\rm A}},\hat{U}^{\smalltext{\rm A}},\hat{\pi}}_t\big) \Big\}.
\end{equation*}
Since $Y^{\rm A}_0 = V^{\rm A}$ (as explained in \Cref{remark:sdeContUtilityAgent}.\ref{y0}), and \Cref{fig:secondBestDifferentValuesM} indicates that the maximiser is $0.1011$ for $m =0.1$, we set $Y^{\rm A}_0 = 0.11$. \Cref{fig:trajectoriesX} illustrates several trajectories of the stopped project's value $X^{\tau^\smalltext{\star}}$ under the optimal effort $\nu^\star$ on the time horizon $[0,20]$, along with its average value. We select this time interval because less than 5\% of the trajectories fail to reach the stopping region by that time, making the graph statistically significant. We can infer that the average value of $X$ at the time $\tau^\star$ is slightly below 4, indicating that $0.1$ is indeed small in comparison.

\begin{figure}[h!tbp]
  \centering
  \includegraphics[width=0.5\textwidth]{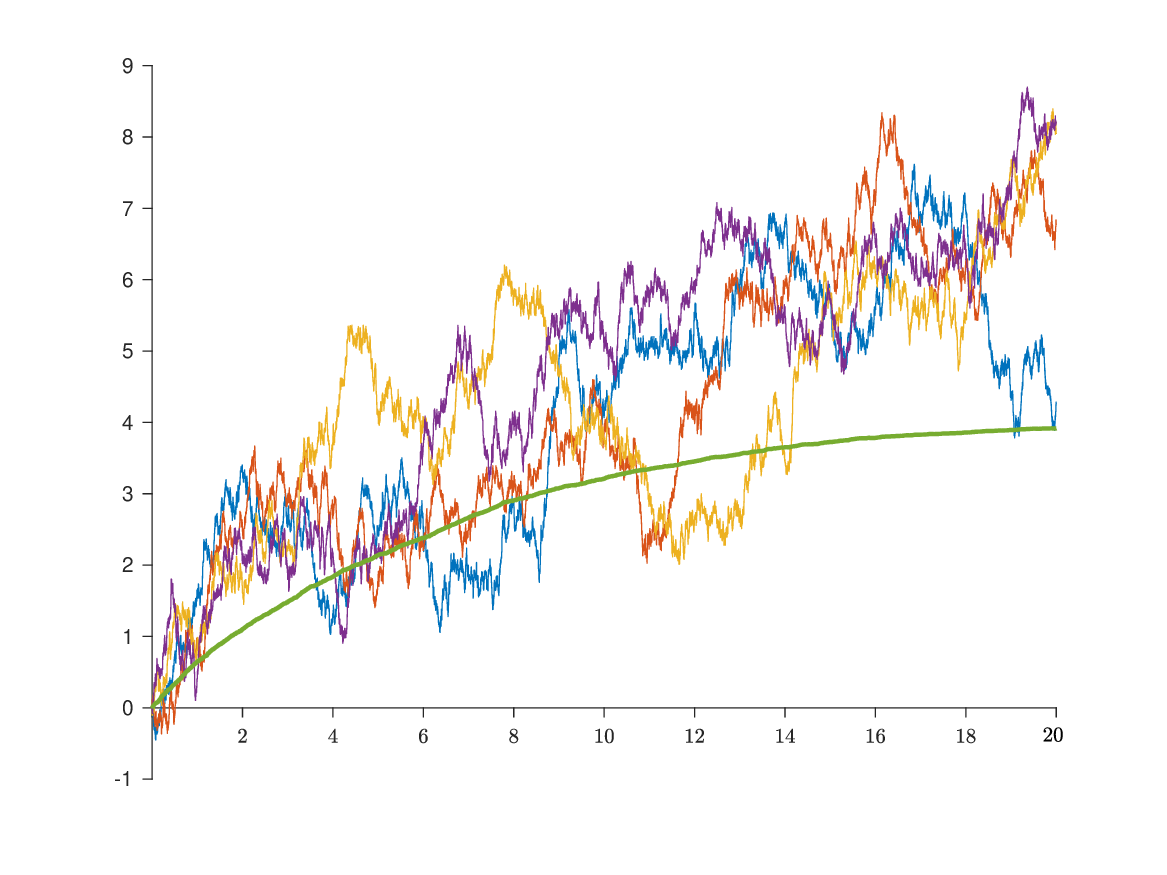}
  \caption{several trajectories depicting the project's value $X$, under the optimal effort $\nu^\star$, and its average value}
  \label{fig:trajectoriesX}
\end{figure}

{\small
\bibliography{BibliographyDylan}}

\begin{thebibliography}{60}
\providecommand{\natexlab}[1]{#1}
\providecommand{\url}[1]{\texttt{#1}}
\expandafter\ifx\csname urlstyle\endcsname\relax
  \providecommand{\doi}[1]{doi: #1}\else
  \providecommand{\doi}{doi: \begingroup \urlstyle{rm}\Url}\fi

\bibitem[Ahmad and Ambrosetti(2015)]{ahmad2015textbook}
S.~Ahmad and A.~Ambrosetti.
\newblock \emph{A textbook on ordinary differential equations}, volume~88 of
  \emph{UNITEXT}.
\newblock Springer Cham, 2nd edition, 2015.

\bibitem[Alvarez and Tourin(1996)]{alvarez1996viscosity}
O.~Alvarez and A.~Tourin.
\newblock Viscosity solutions of nonlinear integro-differential equations.
\newblock \emph{Annales de l'institut Henri Poincar{\'e}, Analyse non
  Lin{\'e}aire $(${\rm C}$)$}, 13\penalty0 (3):\penalty0 293--317, 1996.

\bibitem[Alvarez et~al.(1997)Alvarez, Lasry, and Lions]{alvarez1997convex}
O.~Alvarez, J.-M. Lasry, and P.-L. Lions.
\newblock Convex viscosity solutions and state constraints.
\newblock \emph{Journal de Math{\'e}matiques Pures et Appliqu{\'e}es},
  76\penalty0 (3):\penalty0 265--288, 1997.

\bibitem[Baldacci et~al.(2021)Baldacci, Possama{\"\i}, and
  Rosenbaum]{baldacci2021optimal}
B.~Baldacci, D.~Possama{\"\i}, and M.~Rosenbaum.
\newblock Optimal make--take fees in a multi market maker environment.
\newblock \emph{SIAM Journal on Financial Mathematics}, 12\penalty0
  (1):\penalty0 446--486, 2021.

\bibitem[Baldacci et~al.(2023)Baldacci, Manziuk, Mastrolia, and
  Rosenbaum]{baldacci2023market}
B.~Baldacci, I.~Manziuk, T.~Mastrolia, and M.~Rosenbaum.
\newblock Market making and incentives design in the presence of a dark pool: a
  {S}tackelberg actor--critic approach.
\newblock \emph{Management Science}, 71\penalty0 (2):\penalty0 727--749, 2023.

\bibitem[Bardi and Capuzzo-Dolcetta(1997)]{bardi1997optimal}
M.~Bardi and I.~Capuzzo-Dolcetta.
\newblock \emph{Optimal control and viscosity solutions of
  {H}amilton--{J}acobi--{B}ellman equations}.
\newblock Modern Birkh{\"a}user classics. Birkh{\"a}user Boston, MA, 1997.

\bibitem[Barles and Imbert(2008)]{barles2008second}
G.~Barles and C.~Imbert.
\newblock Second-order elliptic integro-differential equations: viscosity
  solutions' theory revisited.
\newblock \emph{Annales de l'institut Henri Poincar{\'e}, Analyse non
  Lin{\'e}aire $(${\rm C}$)$}, 25\penalty0 (3):\penalty0 567--585, 2008.

\bibitem[Bensalem et~al.(2023)Bensalem, Hern{\'a}ndez~Santib{\'a}{\~n}ez, and
  Kazi-Tani]{bensalem2023continuous}
S.~Bensalem, N.~Hern{\'a}ndez~Santib{\'a}{\~n}ez, and N.~Kazi-Tani.
\newblock A continuous-time model of self-protection.
\newblock \emph{Finance and Stochastics}, 27:\penalty0 503--537, 2023.

\bibitem[Biais et~al.(2007)Biais, Mariotti, Rochet, and
  Villeneuve]{biais2007environmental}
B.~Biais, T.~Mariotti, J.-C. Rochet, and S.~Villeneuve.
\newblock Environmental risk insurance under dynamic moral hazard.
\newblock Technical report, Universit{\'e} de Toulouse, 2007.

\bibitem[Biais et~al.(2010)Biais, Mariotti, Rochet, and
  Villeneuve]{biais2010large}
B.~Biais, T.~Mariotti, J.-C. Rochet, and S.~Villeneuve.
\newblock Large risks, limited liability, and dynamic moral hazard.
\newblock \emph{Econometrica}, 78\penalty0 (1):\penalty0 73--118, 2010.

\bibitem[Bouchard et~al.(2018)Bouchard, Possama{\"\i}, Tan, and
  Zhou]{bouchard2018unified}
B.~Bouchard, D.~Possama{\"\i}, X.~Tan, and C.~Zhou.
\newblock A unified approach to {\it a priori} estimates for supersolutions of
  {BSDE}s in general filtrations.
\newblock \emph{Annales de l'institut Henri Poincar{\'e}, Probabilit{\'e}s et
  Statistiques $(${\rm B}$)$}, 54\penalty0 (1):\penalty0 154--172, 2018.

\bibitem[Cao et~al.(2022)Cao, Li, Young, and Zou]{cao2022stackelberg}
J.~Cao, D.~Li, V.~R. Young, and B.~Zou.
\newblock Stackelberg differential game for insurance under model ambiguity.
\newblock \emph{Insurance: Mathematics and Economics}, 106:\penalty0 128--145,
  2022.

\bibitem[Capponi and Frei(2015)]{capponi2015dynamic}
A.~Capponi and C.~Frei.
\newblock Dynamic contracting: accidents lead to nonlinear contracts.
\newblock \emph{SIAM Journal on Financial Mathematics}, 6\penalty0
  (1):\penalty0 959--983, 2015.

\bibitem[Choi(2014)]{choi2014sannikov}
S.M. Choi.
\newblock \emph{On {S}annikov's continuous-time principal--agent problem}.
\newblock PhD thesis, University of California Berkeley, 2014.

\bibitem[Cohen and Elliott(2015)]{cohen2015stochastic}
S.N. Cohen and R.J. Elliott.
\newblock \emph{Stochastic calculus and applications}.
\newblock Probability and its applications. Springer New York, 2015.

\bibitem[Crandall and Ishii(1990)]{crandall1990maximum}
M.G. Crandall and H.~Ishii.
\newblock The maximum principle for semicontinuous functions.
\newblock \emph{Differential and Integral Equations}, 3\penalty0 (6):\penalty0
  1001--1014, 1990.

\bibitem[Crandall and Lions(1983)]{crandall1983viscosity}
M.G. Crandall and P.-L. Lions.
\newblock Viscosity solutions of {H}amilton--{J}acobi equations.
\newblock \emph{Transactions of the American Mathematical Society},
  277\penalty0 (1):\penalty0 1--42, 1983.

\bibitem[Crandall et~al.(1992)Crandall, Ishii, and Lions]{crandall1992user}
M.G. Crandall, H.~Ishii, and P.-L. Lions.
\newblock User's guide to viscosity solutions of second order partial
  differential equations.
\newblock \emph{Bulletin of the American Mathematical Society}, 27\penalty0
  (1):\penalty0 1--67, 1992.

\bibitem[Cvitani{\'c} and Zhang(2012)]{cvitanic2012contract}
J.~Cvitani{\'c} and J.~Zhang.
\newblock \emph{Contract theory in continuous-time models}.
\newblock Springer, 2012.

\bibitem[Cvitani{\'c} et~al.(2018)Cvitani{\'c}, Possama{\"\i}, and
  Touzi]{cvitanic2018dynamic}
J.~Cvitani{\'c}, D.~Possama{\"\i}, and N.~Touzi.
\newblock Dynamic programming approach to principal--agent problems.
\newblock \emph{Finance and Stochastics}, 22\penalty0 (1):\penalty0 1--37,
  2018.

\bibitem[Dellacherie and Lenglart(1981)]{dellacherie1981sur}
C.~Dellacherie and \'E. Lenglart.
\newblock Sur des probl{\`e}mes de r{\'e}gularisation, de recollement et
  d'interpolation en th{\'e}orie des martingales.
\newblock \emph{S\'eminaire de probabilit\'es de Strasbourg}, XV:\penalty0
  328--346, 1981.

\bibitem[El~Euch et~al.(2021)El~Euch, Mastrolia, Rosenbaum, and
  Touzi]{el2021optimal}
O.~El~Euch, T.~Mastrolia, M.~Rosenbaum, and N.~Touzi.
\newblock Optimal make--take fees for market making regulation.
\newblock \emph{Mathematical Finance}, 31\penalty0 (1):\penalty0 109--148,
  2021.

\bibitem[Fleming and Soner(2006)]{fleming2006controlled}
W.H. Fleming and H.M. Soner.
\newblock \emph{Controlled {M}arkov processes and viscosity solutions},
  volume~25 of \emph{Stochastic modelling and applied probability}.
\newblock Springer-Verlag New York, 2nd edition, 2006.

\bibitem[Frank et~al.(2007)Frank, Lueger, and Korunka]{frank2007significance}
H.~Frank, M.~Lueger, and C.~Korunka.
\newblock The significance of personality in business start-up intentions,
  start-up realization and business success.
\newblock \emph{Entrepreneurship \& Regional Development}, 19\penalty0
  (3):\penalty0 227--251, 2007.

\bibitem[Hern\'{a}ndez~Santib{\'a}{\~n}ez(2022)]{hernandez2022principal}
N.~Hern\'{a}ndez~Santib{\'a}{\~n}ez.
\newblock Principal--multiagents problem under equivalent changes of measure:
  general study and an existence result.
\newblock \emph{ArXiv preprint arXiv:2208.11575}, 2022.

\bibitem[Hern{\'a}ndez~Santib{\'a}{\~n}ez
  et~al.(2020)Hern{\'a}ndez~Santib{\'a}{\~n}ez, Possama{\"\i}, and
  Zhou]{hernandez2020bank}
N.~Hern{\'a}ndez~Santib{\'a}{\~n}ez, D.~Possama{\"\i}, and C.~Zhou.
\newblock Bank monitoring incentives under moral hazard and adverse selection.
\newblock \emph{Journal of Optimization Theory and Applications}, 184\penalty0
  (3):\penalty0 988--1035, 2020.

\bibitem[Hollender(2016)]{hollender2016levy}
J.~Hollender.
\newblock \emph{L\'evy-type processes under uncertainty and related nonlocal
  equations}.
\newblock PhD thesis, TU Dresden, 2016.

\bibitem[Holmstr{\"o}m and Milgrom(1987)]{holmstrom1987aggregation}
B.~Holmstr{\"o}m and P.~Milgrom.
\newblock Aggregation and linearity in the provision of intertemporal
  incentives.
\newblock \emph{Econometrica}, 55\penalty0 (2):\penalty0 303--328, 1987.

\bibitem[Jacod and Shiryaev(2003)]{jacod2003limit}
J.~Jacod and A.N. Shiryaev.
\newblock \emph{Limit theorems for stochastic processes}, volume 288 of
  \emph{Grundlehren der mathematischen Wissenschaften}.
\newblock Springer-Verlag Berlin Heidelberg, 2003.

\bibitem[Jakobsen and Karlsen(2006)]{jakobsen2006maximum}
E.R. Jakobsen and K.H. Karlsen.
\newblock A ``maximum principle for semicontinuous functions'' applicable to
  integro-partial differential equations.
\newblock \emph{Nonlinear Differential Equations and Applications NoDEA},
  13:\penalty0 137--165, 2006.

\bibitem[Jakobsen et~al.(2008)Jakobsen, Karlsen, and
  La~Chioma]{jakobsen2008error}
E.R. Jakobsen, K.H. Karlsen, and C.~La~Chioma.
\newblock Error estimates for approximate solutions to {B}ellman equations
  associated with controlled jump--diffusions.
\newblock \emph{Numerische Mathematik}, 110\penalty0 (2):\penalty0 221--255,
  2008.

\bibitem[Kr{\v{s}}ek and Possama{\"\i}(2023)]{krvsek2023randomisation}
D.~Kr{\v{s}}ek and D.~Possama{\"\i}.
\newblock Randomisation with moral hazard: a path to existence of optimal
  contracts.
\newblock \emph{ArXiv preprint arXiv:2311.13278}, 2023.

\bibitem[Kruse and Popier(2015)]{kruse2015bsdes}
T.~Kruse and A.~Popier.
\newblock {BSDE}s with monotone generator driven by {B}rownian and {P}oisson
  noises in a general filtration.
\newblock \emph{Stochastics: An International Journal of Probability and
  Stochastic Processes}, 88\penalty0 (4):\penalty0 491--539, 2015.

\bibitem[Kyprianou(2014)]{kyprianou2014fluctuations}
A.E. Kyprianou.
\newblock \emph{Fluctuations of {L}{\'e}vy processes with applications}.
\newblock Universitext. Springer Berlin Heidelberg, 2nd edition, 2014.

\bibitem[Kyprianou et~al.(2013)Kyprianou, Kuznetsov, and
  Rivero]{kyprianou2013theory}
A.E. Kyprianou, A.~Kuznetsov, and V.~Rivero.
\newblock The theory of scale functions for spectrally negative {L}{\'e}vy
  processes.
\newblock In S.~Cohen, A.~Kuznetsov, A.E. Kyprianou, and V.~Rivero, editors,
  \emph{L{\'e}vy matters II}, volume 2061 of \emph{Lecture notes in
  mathematics}, pages 97--186. Springer Berlin Heidelberg, 2013.

\bibitem[Lin et~al.(2020)Lin, Ren, Touzi, and Yang]{lin2020second}
Y.~Lin, Z.~Ren, N.~Touzi, and J.~Yang.
\newblock Second order backward {SDE} with random terminal time.
\newblock \emph{Electronic Journal of Probability}, 25\penalty0 (99):\penalty0
  1--43, 2020.

\bibitem[Lin et~al.(2022)Lin, Ren, Touzi, and Yang]{lin2022random}
Y.~Lin, Z.~Ren, N.~Touzi, and J.~Yang.
\newblock Random horizon principal--agent problem.
\newblock \emph{SIAM Journal on Control and Optimization}, 60\penalty0
  (1):\penalty0 355--384, 2022.

\bibitem[Lions(1982)]{lions1982generalized}
P.-L. Lions.
\newblock \emph{Generalized solutions of {H}amilton--{J}acobi equations},
  volume~69 of \emph{Research notes in mathematics}.
\newblock Pitman Advanced Publishing Program, Boston, London, Melbourne, 1982.

\bibitem[Lions(1983)]{lions1983optimal}
P.-L. Lions.
\newblock Optimal control of diffustion processes and
  {H}amilton--{J}acobi--{B}ellman equations. {III}---regularity of the optimal
  cost function.
\newblock In H.~Brezis and J.-L. Lions, editors, \emph{Nonlinear partial
  differential equations and their applications. Coll\`ege de France seminar
  volume V}, volume~93 of \emph{Research notes in mathematics}, pages 95--205.
  Taylor \& Francis, 1983.

\bibitem[Martin and Villeneuve(2023)]{martin2023risk}
J.~Martin and S.~Villeneuve.
\newblock Risk-sharing and optimal contracts with large exogenous risks.
\newblock \emph{Decisions in Economics and Finance}, 46:\penalty0 1--43, 2023.

\bibitem[Mastrolia and Zhang(2022)]{mastrolia2022agency}
T.~Mastrolia and J.~Zhang.
\newblock Agency problem and mean field system of agents with moral hazard,
  synergistic effects and accidents.
\newblock \emph{arXiv preprint arXiv:2207.11087}, 2022.

\bibitem[{\O}ksendal and Sulem(2007)]{oksendal2007applied}
B.~{\O}ksendal and A.~Sulem.
\newblock \emph{Applied stochastic control of jump diffusions}.
\newblock Universitext. Springer-Verlag Berlin Heidelberg, 2nd edition, 2007.

\bibitem[Pag{\`e}s and Possama{\"\i}(2014)]{pages2014mathematical}
H.~Pag{\`e}s and D.~Possama{\"\i}.
\newblock A mathematical treatment of bank monitoring incentives.
\newblock \emph{Finance and Stochastics}, 18\penalty0 (1):\penalty0 39--73,
  2014.

\bibitem[Pham(1998)]{pham1998optimal}
H~Pham.
\newblock Optimal stopping of controlled jump diffusion processes: a viscosity
  solution approach.
\newblock \emph{Journal of Mathematical Systems Estimation and Control},
  8\penalty0 (1):\penalty0 127--130, 1998.

\bibitem[Possama{\"\i} and Rodrigues(2024)]{possamai2023reflections}
D.~Possama{\"\i} and M.~Rodrigues.
\newblock Reflections on {BSDE}s.
\newblock \emph{Electronic Journal of Probability}, 29\penalty0 (66):\penalty0
  1--82, 2024.

\bibitem[Possama{\"\i} and Tangpi(2021)]{possamai2021non}
D.~Possama{\"\i} and L.~Tangpi.
\newblock Non-asymptotic convergence rates for mean-field games: weak
  formulation and {M}c{K}ean--{V}lasov {BSDE}s.
\newblock \emph{ArXiv preprint arXiv:2105.00484}, 2021.

\bibitem[Possama{\"\i} and Touzi(2020)]{possamai2020there}
D.~Possama{\"\i} and N.~Touzi.
\newblock Is there a golden parachute in {S}annikov's principal--agent problem?
\newblock \emph{Mathematics of Operations Research}, to appear, 2020.

\bibitem[Protter(2005)]{protter2005stochastic}
P.E. Protter.
\newblock \emph{Stochastic integration and differential equations}, volume~21
  of \emph{Stochastic modelling and applied probability}.
\newblock Springer-Verlag Berlin Heidelberg, 2nd edition, 2005.

\bibitem[Rockafellar(1970)]{rockafellar1970convex}
R.T. Rockafellar.
\newblock \emph{Convex analysis}.
\newblock Princeton University Press, 1970.

\bibitem[Sannikov(2008)]{sannikov2008continuous}
Y.~Sannikov.
\newblock A continuous-time version of the principal--agent problem.
\newblock \emph{The Review of Economic Studies}, 75\penalty0 (3):\penalty0
  957--984, 2008.

\bibitem[Sannikov and Skrzypacz(2010)]{sannikov2010role}
Y.~Sannikov and A.~Skrzypacz.
\newblock The role of information in repeated games with frequent actions.
\newblock \emph{Econometrica}, 78\penalty0 (3):\penalty0 847--882, 2010.

\bibitem[Sayah(1991{\natexlab{a}})]{awatif1991equations}
A.~Sayah.
\newblock {\'E}quations d'{H}amilton--{J}acobi du premier ordre avec termes
  int{\'e}gro-diff{\'e}rentiels, partie {I} : unicit{\'e} des solutions de
  viscosit{\'e}.
\newblock \emph{Communications in Partial Differential Equations}, 16\penalty0
  (6--7):\penalty0 1057--1074, 1991{\natexlab{a}}.

\bibitem[Sayah(1991{\natexlab{b}})]{awatif1991equations2}
A.~Sayah.
\newblock {\'E}quations d'{H}amilton--{J}acobi du premier ordre avec termes
  int{\'e}gro-diff{\'e}rentiels, partie {II} : existence de solutions de
  viscosit{\'e}.
\newblock \emph{Communications in Partial Differential Equations}, 16\penalty0
  (6--7):\penalty0 1075--1093, 1991{\natexlab{b}}.

\bibitem[Sch{\"a}l(1974)]{schal1974selection}
M.~Sch{\"a}l.
\newblock A selection theorem for optimization problems.
\newblock \emph{Archiv der Mathematik}, 25\penalty0 (1):\penalty0 219--224,
  1974.

\bibitem[Soner(1986{\natexlab{a}})]{soner1986optimal}
H.M. Soner.
\newblock Optimal control with state--space constraint {I}.
\newblock \emph{SIAM Journal on Control and Optimization}, 24\penalty0
  (3):\penalty0 552--561, 1986{\natexlab{a}}.

\bibitem[Soner(1986{\natexlab{b}})]{soner1986optimal2}
H.M. Soner.
\newblock Optimal control with state--space constraint {II}.
\newblock \emph{SIAM Journal on Control and Optimization}, 24\penalty0
  (6):\penalty0 1110--1122, 1986{\natexlab{b}}.

\bibitem[Sung(1997)]{sung1997corporate}
J.~Sung.
\newblock Corporate insurance and managerial incentives.
\newblock \emph{Journal of Economic Theory}, 74\penalty0 (2):\penalty0
  297--332, 1997.

\bibitem[Sung(2001)]{sung2001lectures}
J.~Sung.
\newblock Lectures on the theory of contracts in corporate finance: from
  discrete-time to continuous-time models.
\newblock \emph{Com2Mac Lecture Note Series}, 4, 2001.

\bibitem[Touzi(2013)]{touzi2013optimal}
N.~Touzi.
\newblock \emph{Optimal stochastic control, stochastic target problems, and
  backward {SDE}}, volume~29 of \emph{Fields institute monographs}.
\newblock Springer-Verlag New York, 2013.

\bibitem[Zhang(2009)]{zhang2009dynamic}
Y.~Zhang.
\newblock Dynamic contracting with persistent shocks.
\newblock \emph{Journal of Economic Theory}, 144\penalty0 (2):\penalty0
  635--675, 2009.

\end{thebibliography}

\begin{appendix}

\section{Representation of the face-lifted utility} \label{appendix:face-lifted utility}
In this section, our objective is twofold: first, to deduce the form of the face-lifted utility $\bar{F}$ through an analysis of the Hamilton--Jacobi equation \eqref{align:hjFbar}, and second, to establish that the conjectured function is indeed given by the mixed control--stopping problem. This is accomplished by means of a verification argument when the candidate solution exhibits sufficient regularity, or by invoking the dynamic programming principle in case where this regularity condition is not met.

\subsection{An impatient agent and small average accident size}
We assume that $\delta > 1$ and $m \leq - F^\star(\delta F^\prime(0))$. The latter assumption implies that point $\bar{y}$ introduced in \eqref{eq:ybar} is zero, and thus leads to the conclusion that the function $F$ cannot be a solution of the Hamilton--Jacobi equation \eqref{align:hjFbar}. Hence, we focus our attention to the study of the following ODE:
\begin{align}\label{eq:wode}
F^\star\big(\delta w^\prime(y)\big) - \delta yw^\prime(y) + w(y) + m = 0,\; y\in (0,\infty), \; w(0) = 0.
\end{align}
In what follows, we adopt the same approach as in \cite[Appendix A]{possamai2020there} to show that the utility $\bar{F}$ is a strictly concave solution to the previously introduced ODE, and it never touches the barrier $F$.

\begin{lemma}\label{lemma:derivative}
If $w$ is a strictly concave continuously differentiable solution of {\rm \Cref{eq:wode}}, then $w^\prime(0) = f_{\delta m}$, where  $f_{\delta m} \coloneqq (F^\star)^{(-1)}(-m)/\delta$.
\end{lemma}
\begin{proof}
If the right-derivative $w^{\prime}(0)$ is finite, the result follows by simply computing the limit for $y$ going to $0$ in \eqref{eq:wode}. If we rather assume that $w^{\prime}(0)$ is not finite, specifically, $w^{\prime}(0) = \infty$ given the strict concavity of the function $w$, then there exists some $\varepsilon > 0$ such that $w >0$ and $F^\star(\delta w^\prime) = 0$ on $(0,\varepsilon)$. Consequently, the following equation is satisfied:
\begin{align*}
- \delta yw^{\prime}(y) + w(y) + m = 0, \; y \in (0,\varepsilon).
\end{align*}
For any arbitrary $y_0 \in (0,\varepsilon)$, the previous equation implies that 
\begin{align*}
w(y) = \frac{w(y_0) + m}{y_0^{\frac{1}{\delta}}} y^{\frac{1}{\delta}} - m \; \text{for any} \; y \in (y_0,\varepsilon).
\end{align*}
Given that $w(0) =0$, we deduce that $\lim_{y_0 \rightarrow 0^\smalltext{+}} (w(y_0) + m)y_0^{-\frac{1}{\delta}} = \infty$. Then, $w = \infty$ on the interval $[0,\eps)$, which contradicts the initial condition.
\end{proof}

Let us introduce the notation $w$ for a strictly concave and continuously differentiable solution to the ODE given in {\rm \Cref{eq:wode}}. Our goal is to derive the ODE that its concave conjugate $w^\star$ satisfies in order to deduce some properties that characterise $w^\star$, and consequently the function $w$ by standard convex duality. According to \citeauthor*{alvarez1997convex} \cite[Proposition 5]{alvarez1997convex}, the ODE that $w^\star$ satisfies is
\begin{align}\label{eq:wStarode}
F^\star(\delta p) + (1-\delta) p (w^\star)^\prime(p) - w^\star(p) + m = 0, \; p\in (-\infty, f_{\delta m}), \; w^\star(f_{\delta m}) = 0.
\end{align}
We can express the previous ODE in a linear form. As a result, based on, for instance, \citeauthor*{ahmad2015textbook} \cite[Corollary 2.2.11]{ahmad2015textbook}, we can uniquely determine the solution $w^\star$ by the following expression:  
\begin{align}\label{align:wstarexp}
w^\star(p) = \frac{(-p)^{\frac{1}{1-\delta}}}{\delta-1} \int_p^{f_{\smalltext{\delta}\smalltext{m}}} \frac{F^\star(\delta x)}{(-x)^{1+\frac{1}{1-\delta}}} \d x + m \bigg(1- \bigg(\frac{p}{f_{\delta m}}\bigg)^{\frac{1}{1-\delta}}\bigg), \; \text{for} \; p \leq f_{\delta m}.
\end{align}

\begin{proposition}\label{proposition:propWStar}
It holds that $F^\star(p) \geq w^\star(p)$ for any $p \leq f_{\delta m}$. Moreover, $w^\star$ is strictly concave and increasing on $(-\infty, f_{\delta m})$.
\end{proposition}

\begin{proof}
The condition imposed on the parameter $m$, specifically $m \leq - F^{\star}(\delta F^{\prime}(0))$, implies that $\bar{y}$, as introduced in \eqref{eq:ybar}, is equal to zero. Consequently,
\begin{equation*}
-F^{\star}(p) + (1-\delta) p (F^{\star})^\prime(p) + F^{\star}(\delta p) + m \leq 0 \; \text{for any} \; p \leq F^\prime(0) \leq f_{\delta m},
\end{equation*} 
by standard convex duality (see \cite[Proposition 5]{alvarez1997convex}). Additionally, if $F^\prime(0) < f_{\delta m}$, then $F^{\star} = 0$ and $(F^{\star})^\prime = 0$ on the interval $(F^\prime(0), f_{\delta m}]$. This observation, together with the fact that $f_{\delta m} = (F^\star)^{(-1)}(-m)/\delta$, leads to the following inequality on the entire domain of definition of $w^\star$:
\begin{equation}\label{eq:inequalityOnTheWholeDomain}
-F^{\star}(p) + (1-\delta) p (F^{\star})^\prime(p) + F^{\star}(\delta p) + m \leq 0, \; p\in (-\infty,f_{\delta m}].
\end{equation} 
By defining $\phi(p) \coloneqq (-p)^{-\frac{1}{1-\delta}} (F^{\star}(p) - w^{\star}(p))$, for $p \leq f_{\delta m}$, and using the fact that $w^{\star}$ satisfies \eqref{eq:wStarode}, we derive that
\begin{align*}
\phi^{\prime}(p) &= \frac{1}{1-\delta} (-p)^{-\frac{1}{1-\delta}-1} \big(F^{\star}(p) - w^{\star}(p) -p (1-\delta)((F^{\star})^{\prime}(p) - (w^{\star})^{\prime}(p)\big) \\
&= \frac{1}{1-\delta} (-p)^{-\frac{1}{1-\delta}-1} \left(F^{\star}(p) - (1-\delta) p (F^{\star})^{\prime}(p) - F^{\star}(\delta p) - m\right).
\end{align*}
It follows that the function $\phi$ is non-increasing because of \eqref{eq:inequalityOnTheWholeDomain}. As a consequence, $F^{\star}$ is not below $w^{\star}$ on $(-\infty, f_{\delta m}]$ since
\begin{align*}
\phi(p) \geq \lim_{p \rightarrow f^\smalltext{-}_{\smalltext{\delta}\smalltext{ m}}} \phi(p) = (-f_{\delta m})^{-\frac{1}{1-\delta}} \big(F^{\star}(f_{\delta m}) - w^{\star}(f_{\delta m})\big) = 0 \; \text{for any} \; p \leq f_{\delta m}.
\end{align*}

We have showed that $F^\star(p) \geq w^\star(p)$ for any $p \leq f_{\delta m}$. However, it is still necessary to verify that the function $w^\star$ is strictly concave and to do this, we prove that its second derivative is negative by differentiating the ODE \eqref{eq:wStarode}:
\begin{align*}
(1-\delta)^2 p^2 (w^{\star})^{\prime\prime}(p) &= \delta \big((1-\delta) p (w^{\star})^\prime(p) - (1-\delta) p (F^{\star})^\prime(\delta p)\big) \\
	&= \delta (w^{\star}(p) - F^{\star}(\delta p) - m - (1-\delta) p (F^{\star})^\prime(\delta p))\leq \delta (w^{\star}(p) - F^{\star}(p) - m) < 0 \; \text{for any} \; p < f_{\delta m}.
\end{align*}
The fact that $w^{\star}$ is strictly concave on $(-\infty, f_{\delta m})$ implies that $w^{\star}$ is an increasing function since $w^\star(f_{\delta m}) = 0$, and it remains below $F^{\star}$, which is increasing on $(-\infty, F^\prime(0))$ and constant at $0$ on $[F^\prime(0),f_{\delta m}]$.
\end{proof}

Let us introduce the concave conjugate of $w^{\star}$, that is, $w^{\star\star}(y) \coloneqq \inf_{p \leq f_{\smalltext{\delta}\smalltext{ m}}}\{ yp - w^{\star}(p)\}$, for $y \geq 0$.

\begin{proposition}\label{theorem:propW}
Assuming that $\delta > 1$ and $m \leq - F^{\star}(\delta F^\prime(0))$, and considering the unique solution $w^{\star}$ to {\rm\Cref{eq:wStarode}}, it holds that $w = w^{\star\star}$ is a solution to \eqref{align:hjFbar} and satisfies the growth condition $\bar{c}_0 (-1-y^\gamma) \leq w(y) \leq \bar{c}_1 (1-y^\gamma)$, $y \geq 0$, for some $(\bar{c}_0,\bar{c}_1)\in(0,\infty)^2$. Moreover, $w^{\prime}(0) = f_{\delta m} = (F^\star)^{(-1)}(-m)/\delta$.
\end{proposition}

\begin{proof}
The result \citeauthor*{rockafellar1970convex} \cite[Theorem 12.2]{rockafellar1970convex} shows that $w = w^{\star\star}$ due to the strict concavity of $w^{\star}$. This strict concavity property further implies the continuously differentiability of $w$, as demonstrated by \cite[Theorem 26.6]{rockafellar1970convex}. Finally, concave duality ensures the growth of $w$ from the growth condition of $w^{\star}$, and  implies that $w$ is a solution of \eqref{eq:wode} but also of \eqref{align:hjFbar}, given that $w \geq F^{\star\star} = F$ on the interval $[0,\infty)$.
\end{proof}

We can conclude this section by employing a standard verification argument to prove that the previously mentioned function $w$ indeed represents the face-lifted utility defined by the deterministic mixed control--stopping problem in \eqref{eq:FbarEquation}.
\begin{proof}[Proof of \Cref{proposition:FBarGeneralDeltaBigger1}.\ref{mSmall}]
We first show that $w$ is an upper bound for $\bar{F}$ on $[0,\infty)$. To this aim, we select a starting point $y_0 \geq 0$, a control $p \in \mathcal{B}_{\R_\smalltext{+}}$ and a stopping time $T \in [0,T^{y_\smalltext{0},p}_0]$. Given that the function $w$ is not below $F$ on $[0,\infty)$, as proved in \Cref{theorem:propW}, we have
\begin{align*}
w(y_0) &= \mathrm{e}^{-\rho T} w(y^{y_\smalltext{0},p}(T)) \mathbf{1}_{\{T<\infty\}} + \int_0^T \rho \mathrm{e}^{-\rho t} \big(w(y^{y_\smalltext{0},p}(t)) - \delta w^\prime(y^{y_\smalltext{0},p}(t)) (y^{y_0,p}(t) - p(t))\big) \d t \\
& \geq \mathrm{e}^{-\rho T} F(y^{y_\smalltext{0},p}(T)) \mathbf{1}_{\{T<\infty\}} + \int_0^T \rho \mathrm{e}^{-\rho t} \big(F(p(t)) - m\big) \d t,
\end{align*}
due to the fact that $w$ solves {\rm\Cref{eq:wStarode}}. The arbitrariness of $p \in \mathcal{B}_{\R_\smalltext{+}}$ and $T \in [0,T^{y_\smalltext{0},p}_0]$ implies the desired inequality, that is, $w \geq F$ on the interval $[0,\infty)$. 

\medskip
To establish the converse inequality, let us fix some $y_0 \geq 0$. We consider the maximiser in the definition of $F^{\star}(\delta w^\prime(y^{y_\smalltext{0}, p^{\star}}))$ as a feedback control, described by
\begin{align*}
\dot{y}^{y_\smalltext{0},p^{\smalltext\star}}(t) = r\big(y^{y_\smalltext{0},p^{\smalltext\star}}(t) - p^{\star}(t)\big), \; \text{where} \; p^{\star}(t) \coloneqq (F^{\star})^\prime\big(\delta w^\prime(y^{y_\smalltext{0},p^{\smalltext{\star}}}(t))\big), \; \text{for} \; t\geq 0.
\end{align*}
It holds that 
\begin{align*}
\dot{y}^{y_\smalltext{0},p^{\smalltext\star}}(t) = r \bigg(\frac{1-\delta}{\delta}\bigg) \frac{w^\prime(y^{y_\smalltext{0},p^{\smalltext\star}}(t))}{w^{\prime\prime}(y^{y_\smalltext{0},p^{\smalltext\star}}(t))} \; \text{for any} \; t>0.
\end{align*}
Consequently, the trajectory $y^{y_\smalltext{0},p^{\smalltext\star}}$ is decreasing and thus remains well-defined until it reaches zero at its first hitting time $T^\star \coloneqq T^{y_\smalltext{0},p^{\star}}_0 < \infty$. The same calculations done in the first part of the proof lead to
\begin{align*}
w(y_0) = \mathrm{e}^{-\rho T^{\star}} w\big(y^{y_\smalltext{0},p^{\smalltext\star}}(T^{\star})\big) + \int_0^{T^\smalltext{\star}} \rho \mathrm{e}^{-\rho t} \big(F(p^{\star}(t)) - m\big) \d t &= \mathrm{e}^{-\rho T^{\smalltext\star}} w(0) + \int_0^{T^{\smalltext\star}} \rho \mathrm{e}^{-\rho t} \big(F(p^{\star}(t)) - m\big) \d t \\
&= \mathrm{e}^{-\rho T^{\smalltext\star}} F(0) + \int_0^{T^{\smalltext\star}} \rho \mathrm{e}^{-\rho t} \big(F(p^{\star}(t)) - m\big) \d t \\
&= \mathrm{e}^{-\rho T^\smalltext{\star}} F(y^{y_\smalltext{0},p^{\smalltext\star}}(T^{\star})) + \int_0^{T^{\smalltext\star}} \rho \mathrm{e}^{-\rho t} \big(F(p^{\star}(t)) - m\big) \d t,
\end{align*}
since $w(0) = 0 = F(0)$ by the boundary condition. We deduce that $(T^{\star},p^{\star})$ is an optimal control for the problem $\bar{F}$, thereby leading to the conclusion that the face-lifted utility $\bar{F}$ coincides with its upper bound $w$.
\end{proof}

\subsection{When the agent is more impatient and the accidents' size matters}\label{subsection:deltaGreater1_FBar}
Throughout this section, we investigate the mixed control--stopping problem \eqref{eq:FbarEquation} under the assumptions that $\delta > 1$ and $m > - F^{\star}(\delta F^\prime(0))$. It is important to note that the latter assumption implies that $\bar{y}$ introduced in \eqref{eq:ybar} is positive. In contrast to the previous section, it follows that $F$ can be a solution to the Hamilton--Jacobi equation \eqref{align:hjFbar} on a right-neighbourhood of zero. Indeed, this is precisely the case, as demonstrated by the following lemma for strictly concave solutions.

\begin{lemma}\label{lemma:wEqualF}
For any strictly concave solution $w$ of {\rm\Cref{align:hjFbar}}, there exists $\varepsilon > 0$ such that $w$ coincides with $F$ on $[0, \varepsilon)$.
\end{lemma}
\begin{proof}
Let us fix $\varepsilon>0$. We assume to the contrary that $w$ is a strictly concave solution of the Hamilton--Jacobi equation \eqref{align:hjFbar} such that it is above $F$ on $(0, \varepsilon)$. Consequently, it becomes a solution of \Cref{eq:wode} on the interval $(0, \varepsilon)$, leading to $w^{\prime}(0) = f_{\delta m}$, as inferred from \Cref{lemma:derivative}. The given hypothesis regarding the parameter $m$, that is, $F^{\star}(\delta F^{\prime}(0)) > -m = F^{\star}(\delta f_{\delta m})$, implies that $w^{\prime}(0) < F^{\prime}(0)$. However, this contradicts the fact that $w > F$ on $(0, \varepsilon)$. 
\end{proof}

Similar to the previous section, we now examine the equation solved by the concave conjugate $w^{\star}$ which is given by \cite[Proposition 5]{alvarez1997convex} to be 
\begin{align}\label{align:hjFbarStar}
-\min\Big\{(w^{\star})^\prime(p) p - w^{\star}(p) - F\big((w^{\star})^\prime(p)\big), (1-\delta) p (w^{\star})^\prime(p)- w^{\star}(p)  + F^{\star}(\delta p) + m\Big\}=0, \; p<F^\prime(0), \;w^{\star}(F^\prime(0)) = 0.
\end{align}
As we seek a strictly concave solution to the previous equation that coincides with $F^\star$ on the right-neighbourhood of $F^\prime(0)$ (as indicated by \Cref{lemma:wEqualF}), let us fix some $p_{\bar{y}} < F^\prime(0)$ and consider the function 
\begin{align}\label{align:barFyHat}
w^{\star}(p) \coloneqq \begin{cases}
	w^{\star}_{\bar{y}}(p),\; p \in (-\infty,p_{\bar{y}})\\
	F^{\star}(p), \; p \in [p_{\bar{y}},F^{\prime}(0)]
	\end{cases},
\end{align}
where $w^{\star}_{\bar{y}}$ is the unique solution of 
\begin{align}\label{align:wHatyBar}
- w^{\star}_{\bar{y}}(p) + (1-\delta) p (w^{\star}_{\bar{y}})^{\prime}(p) + F^{\star}(\delta p) + m = 0, \;p\in (-\infty, p_{\bar{y}}), \; w^{\star}_{\bar{y}}(p_{\bar{y}}) = F^{\star}(p_{\bar{y}}).
\end{align}
The solution to the equation above is explicitly given by
\begin{align}\label{align:wstarybar}
w^{\star}_{\bar{y}}(p) = \frac{(-p)^{\frac{1}{1-\delta}}}{\delta-1} \int_p^{p_{\smalltext{\bar{y}}}} \frac{F^{\star}(\delta x)}{(-x)^{1+\frac{1}{1-\delta}}} \d x + \bigg(\frac{p}{p_{\bar{y}}}\bigg)^{\frac{1}{1-\delta}} (F^{\star}(p_{\bar{y}}) - m) + m, \; \text{for} \; p \leq p_{\bar{y}}.
\end{align}

\begin{lemma}
The function described in \eqref{align:barFyHat} is continuously differentiable on $(-\infty,F^\prime(0))$ if and only if $p_{\bar{y}} = F^\prime(\bar{y})$.
\end{lemma}
\begin{proof}
The function $w^\star$ is continuously differentiable if and only if 
\begin{align*}
\frac{F^{\star}(\delta p_{\bar{y}}) - F^{\star}(p_{\bar{y}}) + m}{p_{\bar{y}} (\delta-1)} = (w^{\star}_{\bar{y}})^{\prime}(p_{\bar{y}}) = (F^{\star})^{\prime}(p_{\bar{y}}).
\end{align*}
The previous equality holds true if and only if 
\begin{align*}
F^{\star}(p_{\bar{y}}) - F^{\star}(\delta p_{\bar{y}}) - m = w^{\star}_{\bar{y}}(p_{\bar{y}}) - F^{\star}(\delta p_{\bar{y}}) - m = (1-\delta) p_{\bar{y}} (w^{\star}_{\bar{y}})^{\prime}(p_{\bar{y}}) = (1-\delta) p_{\bar{y}} (F^{\star})^{\prime}(p_{\bar{y}}).
\end{align*}
By standard convex duality and in accordance with Assumption \ref{assumption:FdeltamDec}, there exists a unique $p_{\bar{y}} \leq F^{\prime}(0)$ for which the aforementioned equality is satisfied. Specifically, this is determined as  $p_{\bar{y}} = F^{\prime}(\bar{y})$.
\end{proof}

Henceforth, we slightly abuse notations by setting $p_{\bar{y}} = F^{\prime}(\bar{y})$ in the definition of the function $w^{\star}$ as given in \Cref{align:barFyHat}. Then, we proceed to outline several properties of this function $w^\star$ and prove this latter solves the conjugate obstacle problem. This, in turn, allows us to conclude that the concave conjugate of $w^\star$ coincides with the face-lifted function $\bar{F}$.

\begin{remark}
It holds that $w^{\star} \leq F^{\star}$ on $(-\infty, F^{\prime}(0)]$, and $w^\star$ is strictly concave and increasing on $(-\infty, F^{\prime}(0))$. These properties are evident on $[F^{\prime}(\bar{y}),F^{\prime}(0)]$ because $w^{\star}$ coincides with $F^{\star}$ on this interval, while can be proved with arguments similar to those presented in {\rm\Cref{proposition:propWStar}} on $(-\infty, F^{\prime}(\bar{y}))$.
\end{remark}

\begin{proposition}
The function $w^{\star}$ introduced in {\rm \eqref{align:barFyHat}} satisfies {\rm\Cref{align:hjFbarStar}}.
\end{proposition}
\begin{proof}
The function $w^{\star}$ is given by $F^{\star}$ on $[F^{\prime}(\bar{y}), F^{\prime}(0)]$. Within this interval, the following equations hold:
\begin{align*}
(F^{\star})^{\prime}(p) p - F^{\star}(p) - F((F^{\star})^{\prime}(p)) = 0, \; \text{and} \; - F^{\star}(p) + (1-\delta) p (F^{\star})^{\prime}(p) + F^{\star}(\delta p) + m \geq 0,
\end{align*}
where the inequality is a consequence of the definition of $\bar{y}$ in \eqref{eq:ybar}. Then, we only need to show that the solution $w^{\star}_{\bar{y}}$ of the ODE \eqref{align:wHatyBar} satisfies $(w^{\star}_{\bar{y}})^{\prime}(p) p - w^{\star}_{\bar{y}}(p) - F((w^{\star}_{\bar{y}})^{\prime}(p)) \geq 0$ on $(-\infty, F^{\prime}(\bar{y}))$ in order to conclude the proof. However, this is verified because
\begin{align*}
(w^{\star}_{\bar{y}})^{\prime}(p) p - w^{\star}_{\bar{y}}(p) - F((w^{\star}_{\bar{y}})^{\prime}(p)) &\geq (w^{\star}_{\bar{y}})^{\prime}(p) p - F^{\star}(p) - F((w^{\star}_{\bar{y}})^{\prime}(p)) \\
&\geq (w^{\star}_{\bar{y}})^{\prime}(p) p - F^{\star}(p) - \big(F((F^{\star})^{\prime}(p)) + F^{\prime}((F^{\star})^{\prime}(p))((w^{\star}_{\bar{y}})^{\prime}(p)-(F^{\star})^{\prime}(p))\big) \\
&= (w^{\star}_{\bar{y}})^{\prime}(p) (p - F^{\prime}((F^{\star})^{\prime}(p))) = 0.
\end{align*}
\end{proof}

Up to this point, we have investigated the properties of the function $w^{\star}$. Our current objective is to translate them into characteristics of its concave conjugate $w^{\star\star}$. By standard convex duality, the latter can be expressed as
\begin{align}\label{align:wYBar}
w^{\star\star}(y) = \begin{cases}
	F(y),\; y \in [0,\bar{y}] \\
	w_{\bar{y}}(y), \;  y \in (\bar{y}, \infty)
	\end{cases},
\end{align}
where $w_{\bar{y}}$ represents the concave conjugate of $w^{\star}_{\bar{y}}$ and thus satisfies
\begin{align}\label{align:wybar}
F^{\star}\big(\delta w^\prime_{\bar{y}}(y)\big) - \delta y w^\prime_{\bar{y}}(y) + w_{\bar{y}}(y) + m = 0, \; y \in (\bar{y},\infty), \; w_{\bar{y}}(\bar{y}) = F(\bar{y}), \; w^\prime_{\bar{y}}(\bar{y}) = F^\prime(\bar{y}).
\end{align}

\begin{proposition}\label{theorem:propBarF}
Let $\delta > 1$ and $m > - F^{\star}(\delta F^{\prime}(0))$. The function $w^{\star\star}$ introduced in \eqref{align:wYBar} is a solution to the Hamilton--Jacobi equation \eqref{align:hjFbar} while satisfying the growth condition $\bar{c}_0 (-1-y^\gamma) \leq w^{\star\star}(y) \leq \bar{c}_1 (1-y^\gamma)$ for any $y \geq 0$, for some $(\bar{c}_0,\bar{c}_1) \in( 0,\infty)^2$.
\end{proposition}
\begin{proof}
As taking conjugates evidently reverses functional inequalities, it follows that $w^{\star\star}$ is always above $F$. This observation, together with the definition of $\bar{y}$, leads us to the conclusion that it is indeed a solution to \Cref{align:hjFbar}. Furthermore, concave duality ensures the growth of $w^{\star\star}$ in accordance with the growth condition established for $w^{\star}$.
\end{proof}

We can now present the main result of this section, which demonstrates that the aforementioned function $w^{\star\star}$ indeed represents the face-lifted utility $\bar{F}$.
\begin{proof}[Proof of \Cref{proposition:FBarGeneralDeltaBigger1}.\ref{mLarge}]
Similar to the proof of the statement in \Cref{proposition:FBarGeneralDeltaBigger1}.\ref{mSmall}, we can show that $w^{\star\star} \geq \bar{F}$ on $[0,\infty)$. To prove the reverse inequality, we need to distinguish two distinct cases. Firstly, when $y_0 \in [0, \bar{y}]$, the trivial control $(T^{\star},p^{\star}) : = (0,0)$ is optimal as it attains the upper bound $w^{\star\star}(y_0)$. This is straightforward to prove, as follows:
\begin{align*}
w^{\star\star}(y_0) = \mathrm{e}^{-\rho T^{\star}} w^{\star\star} \big(y^{y_\smalltext{0},p^{\star}}(T^{\star})\big) + \int_0^{T^{\smalltext\star}} \rho \mathrm{e}^{-\rho t} \big(F(p^{\star}(t)) - m\big) \d t = w^{\star\star}(y^{y_\smalltext{0},p^{\smalltext\star}}(0)) = F\big((y^{y_\smalltext{0},p^{\smalltext{\star}}}(0))\big) = F(y_0).
\end{align*}
Next, we turn our attention to the case where $y_0 > \bar{y}$. In this scenario, we define $p^{\star}(t)$ as the maximiser in the definition of $F^{\star}(\delta (\bar{F}^{\star\star})^\prime(y^{y_\smalltext{0},p^{\smalltext\star}}(t)))$, for $t \in [0,T^{\star}]$, where $T^{\star} \coloneqq\inf\{t\geq 0: y^{y_\smalltext{0},p^{\smalltext{\star}}}(t) \leq \bar{y}\}\in (0,T^{y_\smalltext{0},p^{\smalltext\star}}_0)$. Consequently, the state dynamics is given by
\begin{align*}
\dot{y}^{y_\smalltext{0},p^{\smalltext\star}}(t) = r \bigg(\frac{1-\delta}{\delta}\bigg) \frac{(w^{\star\star})^\prime(y^{y_\smalltext{0},p^{\smalltext\star}}(t))}{(w^{\star\star})^{\prime\prime}(y^{y_\smalltext{0},p^{\smalltext\star}}(t))} \; \text{for any} \; t\in (0,T^{\star}).
\end{align*}
It follows that $y^{y_\smalltext{0},p^{\smalltext\star}}$ is decreasing and, therefore, well-defined on $(0,T^{\star})$. Under the control $(T^{\star},p^{\star})$, it holds that
\begin{align*}
w^{\star\star}(y_0) = \mathrm{e}^{-\rho T^{\smalltext\star}} w^{\star\star}(y^{y_\smalltext{0},p^{\smalltext\star}}(T^{\star})) + \int_0^{T^{\smalltext\star}} \rho \mathrm{e}^{-\rho t} \big(F(p^{\star}(t)) - m\big) \d t 
	&= \mathrm{e}^{-\rho T^{\smalltext\star}} F(\bar{y}) + \int_0^{T^{\smalltext\star}} \rho \mathrm{e}^{-\rho t} \big(F(p^{\star}(t)) - m\big) \d t \\
	&= \mathrm{e}^{-\rho T^{\smalltext\star}} F(y^{y_0,p^{\star}}(T^{\star})) + \int_0^{T^{\smalltext\star}} \rho \mathrm{e}^{-\rho t} \big(F(p^{\star}(t)) - m\big) \d t.
\end{align*}
Here, we have used the fact that $y^{y_\smalltext{0},p^{\smalltext\star}}(T^{\star})=\bar{y}$ and that $v_{\bar{y}}(\bar{y}) = F(\bar{y})$. We deduce that $(T^{\star},p^{\star})$ is an optimal control for the problem $\bar{F}$ when $y_0 > \bar{y}$. This concludes the proof.
\end{proof}

\subsection{When the principal becomes impatient, but not too much}

We assume that $\delta < 1$ and $\gamma\delta > 1$. Reasoning analogous to that developed in \Cref{lemma:wEqualF} suggests that any solution $w$ to the Hamilton-Jacobi equation \eqref{align:hjFbar} coincides with the function $F$ on $[0,\varepsilon)$ for some $\varepsilon > 0$. Furthermore, $w$ is continuously differentiable only when it equals $F$ on the non-degenerate interval $[0,\bar{y}]$ (it is important to highlight that $\bar{y}$ is positive under the assumptions of this section). However, this choice is inadmissible since it leads to a solution $w$ that falls below the barrier $F$ on a right-neighbourhood of $\bar{y}$. Therefore, we propose the \emph{ansatz} that $w$ takes the form outlined in \Cref{proposition:FBarGeneralDeltaSmaller1}. Given that the suggested function $w$ fails to be differentiable everywhere, our next results necessitate the use of the theory of viscosity solutions. For a more comprehensive treatment in the case of Hamilton--Jacobi equations, we refer the reader to \citeauthor*{lions1982generalized} \cite{lions1982generalized}.

\begin{proposition}\label{theorem:propFbarYhat}
Assuming $\delta < 1$ and $\gamma\delta > 1$, the function
\begin{align*}
v_{\tilde{y}}(y) \coloneqq \begin{cases}
	F(y), \; y \in [0,\tilde{y}] \\
	w_{0}(y)-m, \; y \in (\tilde{y}, \infty)
	\end{cases},
\end{align*}
becomes a viscosity solution to the Hamilton--Jacobi equation \eqref{align:hjFbar}. Here, $w_{0}$ denotes the face-lifted utility that is introduced in {\rm\cite[Proposition 2.1]{possamai2020there}} and readdressed in {\rm\Cref{align:w0nojumps}}, while $\tilde{y}$ is defined in {\rm\Cref{lemma:existenceOfyTilde}}. Additionally, we have that $\bar{c}_0 (-1-y^\gamma) \leq v_{\tilde{y}}(y) \leq \bar{c}_1 (1-y^\gamma)$, for $y \geq 0$, for some $(\bar{c}_0,\bar{c}_1)\in(0,\infty)^2$.
\end{proposition}

\begin{proof}
Firstly, it is important to note that the growth condition directly follows from \cite[Proposition 2.1]{possamai2020there}. Subsequently, as highlighted in \Cref{lemma:existenceOfyTilde}, the function $v_{\tilde{y}}$ is continuous on $(0,\infty)$ and continuously differentiable on the same interval, except at the point $\tilde{y}>0$, by definition. We can then prove that $v_{\tilde{y}}$ satisfies the Hamilton--Jacobi equation \eqref{align:hjFbar} in three distinct steps based on the interval of the non-negative half-line under consideration. Initially, it is trivial that $v_{\tilde{y}}$ solves \Cref{align:hjFbar} on $(0,\tilde{y})$ since $\tilde{y} < \bar{y}$. The proof is also straightforward on $(\tilde{y},\infty)$, given that $w_{\tilde{y}}$ is a solution to the ODE \eqref{eq:eqw0NoJumps}. It is only sufficient to check whether $v_{\tilde{y}}$ is above $F$ on that interval but this is simply verified. In fact, the proof of \Cref{lemma:existenceOfyTilde} implies that the function $v_{\tilde{y}} - F = w_{0} -m -F$ is increasing on $(\tilde{y},\infty)$ and $(v_{\tilde{y}} - F)(\tilde{y}) =0$.

\medskip
Lastly, the proof of the viscosity solution property of $v_{\tilde{y}}$ at $\tilde{y}$ requires our focus on examining the super-differential $D^+v_{\tilde{y}}(\tilde{y})$ and the sub-differential $D^-v_{\tilde{y}}(\tilde{y})$ of the function at that particular point, see for instance \citeauthor*{bardi1997optimal} \cite[Lemma I.1.7]{bardi1997optimal}. As a direct consequence of this result, the viscosity super-solution property of $v_{\tilde{y}}$ at $\tilde{y}$ turns out to be equivalent to the following condition:
\begin{align}\label{eq:conditionSupersolyionYHat}
\min\left\{v_{\tilde{y}}(\tilde{y}) - F(\tilde{y}), F^{\star}(\delta p) - \delta \tilde{y} p + v_{\tilde{y}}(\tilde{y}) + m\right\} = \min\left\{0, F^{\star}(\delta p) - \delta \tilde{y} p + w_0(\tilde{y})\right\} \geq 0 \; \text{for any}\; p \in D^-v_{\tilde{y}}(\tilde{y}).
\end{align}
Therefore, \Cref{eq:conditionSupersolyionYHat} is verified if and only if $F^{\star}(\delta p) - \delta \tilde{y} p + w_0(\tilde{y}) \geq 0$ for any $p \in D^-v_{\tilde{y}}(\tilde{y})$. This condition is derived straightforwardly from the fact that the function $[F^\prime(\tilde{y}),w_0^\prime(\tilde{y})] \ni p \longmapsto F^{\star}(\delta p) - \delta \tilde{y} p + w_0(\tilde{y})$ is decreasing and is null at $w_0^\prime(\tilde{y})$ due to the ODE satisfied by $w_0$, while $D^-v_{\tilde{y}}(\tilde{y}) = \{p \in \R: F^\prime(\tilde{y}) \leq p \leq w^\prime_0(\tilde{y})\}$. However, the proof of this last assertion is deferred to the end for the sake of clarity. Consequently, we can conclude that $v_{\tilde{y}}$ constitutes a viscosity solution of the Hamilton--Jacobi equation \eqref{align:hjFbar} at $\tilde{y}$ because \cite[Lemma II.1.8]{bardi1997optimal} implies that the super-differential $D^+v_{\tilde{y}}(\tilde{y})$ is empty due to the function not being continuously differentiable at that specific point.

\medskip
To complete the proof, it remains to prove the claim $D^-v_{\tilde{y}}(\tilde{y}) = \{p \in \R: F^\prime(\tilde{y}) \leq p \leq w^\prime_0(\tilde{y})\}$. For any $p >w^\prime_0(\tilde{y})$, it is evident that
\begin{align*}
\liminf_{y \rightarrow \tilde{y}}\bigg\{ \frac{v_{\tilde{y}}(y) - v_{\tilde{y}}(\tilde{y})-p(y-\tilde{y})}{|\tilde{y}-y|}\bigg\} \leq \liminf_{y \rightarrow \tilde{y}^{\smalltext{+}}} \bigg\{\frac{w_{0}(y) - w_{0}(\tilde{y})-p(y-\tilde{y})}{|\tilde{y}-y|}\bigg\} \leq  w^\prime_{0}(\tilde{y})- p < 0.
\end{align*}
Here, we have used the fact that the function $w_0$ is strictly concave on its domain of definition. Therefore, it follows that $p \notin D^-v_{\tilde{y}}(\tilde{y})$, and the same conclusion holds for any $p < F^{\prime}(\tilde{y})$:
\begin{align*}
\liminf_{y \rightarrow \tilde{y}}\bigg\{ \frac{v_{\tilde{y}}(y) - v_{\tilde{y}}(\tilde{y})-p(y-\tilde{y})}{|\tilde{y}-y|}\bigg\} \leq \liminf_{y \rightarrow \tilde{y}^{\smalltext{-}}} \bigg\{\frac{F(y) - F(\tilde{y})-p(y-\tilde{y})}{|\tilde{y}-y|}\bigg\} \leq p - F^{\prime}(\tilde{y}) < 0,
\end{align*}
taking into account the concavity of $F$. For any $p \in [F^{\prime}(\tilde{y}),w^\prime(\tilde{y})]$, let us define the continuously differentiable function 
\begin{equation*}
\phi_{p} (y) \coloneqq F(\tilde{y}) + p (y-\tilde{y}) -  (y-\tilde{y})^{\frac{1}{\gamma}}, \; \text{for} \; y \geq 0.
\end{equation*}
We can justify that $\tilde{y}$ is a local minimum of the difference $v_{\tilde{y}} - \phi_p$ because $p$ falls within the interval $[(v_{\tilde{y}})_{-}^\prime(\tilde{y}),(v_{\tilde{y}})_{+}^\prime(\tilde{y})] = [F^\prime(\tilde{y}),0]$. According to \cite[Lemma II.1.7]{bardi1997optimal}, it follows that $p \in D^-v_{\tilde{y}}(\tilde{y})$.
\end{proof}

We have just proved that $v_{\tilde{y}}$ is a viscosity solution of the Hamilton--Jacobi equation \eqref{align:hjFbar}. Next, we proceed to show that the function $\bar{F}$, defined as the solution to the mixed control-stopping problem in \eqref{eq:FbarEquation}, also satisfies the same equation in the viscosity sense. Subsequently, by invoking a comparison theorem, we can effectively characterise $\bar{F}$ as the unique viscosity solution, with appropriate growth at infinity. This guarantees the equivalence of the two aforementioned functions. To this aim, we introduce the following result, which is fundamental for characterising $\bar{F}$ as a viscosity solution of \Cref{eq:FbarEquation}, as proven in \Cref{theorem:viscositySolFbar} below.

\begin{lemma}\label{lemma:dppINFSUP}
Let us fix some $y_0 \geq 0$. For any time $\theta \geq 0$, it holds that 
\begin{align*}
\bar{F}(y_0) &\leq \sup_{p \in \cB_{\smalltext{\R}_\tinytext{+}}} \sup_{T \in [0, T^{\smalltext{y}_\tinytext{0}\smalltext{,}\smalltext{p}}_\smalltext{0}]} \bigg\{ \int_0^{T \wedge \theta} \rho \mathrm{e}^{-\rho t} (- m + F(p(t))) \d t + \mathrm{e}^{-\rho T} F(y^{y_\smalltext{0},p}(T))  \mathbf{ 1}_{\{T < \theta\}} + \mathrm{e}^{-\rho \theta} \bar{F}^s (y^{y_\smalltext{0},p}(\theta))  \mathbf{ 1}_{\{T \geq \theta\}} \bigg\} \\
\bar{F}(y_0) &\geq \sup_{p \in \cB_{\smalltext{\R}_\tinytext{+}}} \sup_{T \in [0, T^{\smalltext{y}_\tinytext{0}\smalltext{,}\smalltext{p}}_\smalltext{0}]} \bigg\{ \int_0^{T \wedge \theta} \rho \mathrm{e}^{-\rho t} \big(- m + F(p(t))\big) \d t + \mathrm{e}^{-\rho T} F(y^{y_\smalltext{0},p}(T))\; \mathbf{ 1}_{\{T < \theta\}} + \mathrm{e}^{-\rho \theta} \bar{F}_i (y^{y_\smalltext{0},p}(\theta))\; \mathbf{ 1}_{\{T \geq \theta\}} \bigg\}.
\end{align*}
Here, the functions $\bar{F}^s$ and $\bar{F}_i$ denote the upper--semi-continuous and the lower--semi-continuous envelope of $\bar{F}$, respectively.
\end{lemma}

\begin{proof}
Let us consider some $p \in \cB_{\R_+}$ and $T \in [0, T^{y_\smalltext{0},p}_0]$, and denote
\begin{align*}
f^{y_{\smalltext{0}}}(p,T) \coloneqq &\int_0^T \rho \mathrm{e}^{-\rho t} \big(- m + F(p(t))\big) \d t + \mathrm{e}^{-\rho T} F\big(y^{y_\smalltext{0},p}(T)\big) \mathbf{ 1}_{\{T < \infty\}}.
\end{align*}
It follows that
\begin{align*}
f^{y_{\smalltext{0}}}(p,T) &= f^{y_{\smalltext{0}}}(p,T) \mathbf{1}_{\{T < \theta\}} \\
	&\quad+\bigg(\int_0^{\theta} \rho \mathrm{e}^{-\rho t} \big(- m + F(p(t))\big) \d t + \int_{\theta}^T \rho \mathrm{e}^{-\rho t} \big(- m + F(p(t))\big) \d t + \mathrm{e}^{-\rho T} F\big(y^{y_\smalltext{0},p}(T)\big)\mathbf{ 1}_{\{T < \infty\}}\bigg)\mathbf{ 1}_{\{T \geq \theta\}} \\
	&=f^{y_{\smalltext{0}}}(p,T) \mathbf{1}_{\{T < \theta\}} +\bigg(\int_0^{\theta} \rho \mathrm{e}^{-\rho t} \big(- m + F(p(t))\big) \d t \bigg)\mathbf{ 1}_{\{T \geq \theta\}}\\
	&\quad + \mathrm{e}^{-\rho \theta} \bigg(\int_0^{T - \theta} \rho \mathrm{e}^{-\rho t} \big(- m + F(\tilde{p}(t))\big) \d t + \mathrm{e}^{-\rho (T-\theta)} F\big(y^{{y^{\smalltext{y}_\tinytext{0}\smalltext{,}\smalltext{p}}(\theta)},\tilde{p}}(T-\theta)\big) \mathbf{ 1}_{\{T < \infty\}}\bigg) \mathbf{ 1}_{\{T \geq \theta\}}, 
\end{align*}
where $\tilde{p}(t) \coloneqq p(t-\theta)$, for $t \geq 0$. Since $\tilde{p} \in \cB_{\R_\smalltext{+}}$, we have
\begin{align*}
f^{y_{\smalltext{0}}}(p,T) &\leq f^{y_{\smalltext{0}}}(p,T) \mathbf{ 1}_{\{T < \theta\}} + \bigg(\int_0^{\theta} \rho \mathrm{e}^{-\rho t} \big(- m + F(p(t))\big) \d t + \mathrm{e}^{-\rho \theta} \bar{F}^s (y^{y_\smalltext{0},p}(\theta)) \bigg) \mathbf{ 1}_{\{T \geq \theta\}}.
\end{align*}
We deduce that the first inequality simply follows by taking the supremum over $p \in \cB_{\R_\smalltext{+}}$ and $T \in \left[0, T^{y_\smalltext{0},p}_0\right]$. In order to prove the second inequality, let us fix $\varepsilon > 0$. For any $y\geq 0$, there exist some $p^{\varepsilon, y} \in \cB_{\R_\smalltext{+}}$ and $T^{\varepsilon, y} \in [0, T^{y,p^{\smalltext{\varepsilon}\smalltext{,}\smalltext{y}}}_0]$ such that 
\begin{equation}\label{eq:DPPUImpoinequality}
\int_0^{T^{\smalltext{\varepsilon}\smalltext{,}\smalltext{y}}} \rho \mathrm{e}^{-\rho t} \big(- m + F(p^{\varepsilon, y}(t))\big) \d t + \mathrm{e}^{-\rho T^{\smalltext{\varepsilon}\smalltext{,}\smalltext{y}}} F\big(y^{y,p^{\smalltext{\varepsilon}\smalltext{,}\smalltext{y}}}(T^{\varepsilon, y})\big) \mathbf{ 1}_{\{T^{\smalltext{\varepsilon}\smalltext{,}\smalltext{y}} < \infty\}} \geq \bar{F}_i(y) - \varepsilon. 
\end{equation}
Let us fix $p \in \cB_{\R_\smalltext{+}}$ and $T \in [0, T^{y_\smalltext{0},p}_0]$. It holds that 
\begin{align*}
\bar{F}(y_0) &\geq f^{y_{\smalltext{0}}}(p,T) \mathbf{ 1}_{\{T < \theta\}} +\bigg(\int_0^{\theta + T^{\smalltext{\varepsilon}\smalltext{,}\smalltext{y}^{\tinytext{y}_\stinytext{0}\tinytext{,}\tinytext{p}}(\theta)}} \rho \mathrm{e}^{-\rho t} \big(- m + F(\tilde{p}^{\varepsilon, y^{y_\smalltext{0},p}(\theta)}(t))\big) \d t \bigg) \mathbf{1}_{\{T \geq \theta\}}\\
	&\quad + \mathrm{e}^{-\rho (\theta + T^{\smalltext{\varepsilon}\smalltext{,}\smalltext{y}^{\tinytext{y}_\stinytext{0}\tinytext{,}\tinytext{p}}(\theta)})} F\Big(y^{y_\smalltext{0},\tilde{p}^{\smalltext{\varepsilon}\smalltext{,}\smalltext{y}^{\tinytext{y}_\stinytext{0}\tinytext{,}\tinytext{p}}(\theta)}}\big(\theta + T^{\smalltext{\varepsilon}\smalltext{,}\smalltext{y}^{\tinytext{y}_\stinytext{0}\tinytext{,}\tinytext{p}}(\theta)}\big)\Big)  \mathbf{ 1}_{\{T^{\smalltext{\varepsilon}\smalltext{,}\smalltext{y}^{\tinytext{y}_\stinytext{0}\tinytext{,}\tinytext{p}}(\theta)} < \infty \}}  \mathbf{1}_{\{T \geq \theta\}},
\end{align*}
where 
\begin{align*}
\tilde{p}^{\varepsilon, y^{\smalltext{y}_\tinytext{0}\smalltext{,}\smalltext{p}}(\theta)}(t) \coloneqq 
\begin{cases} 
p(t), \; t \in [0, \theta)\\
p^{\varepsilon, y^{\smalltext{y}_\tinytext{0}\smalltext{,}\smalltext{p}}(\theta)}(t-\theta), \; t \in [\theta, \infty)
\end{cases}.
\end{align*}
As $\tilde{p}^\varepsilon \in \cB_{\R_\smalltext{+}}$ and $\theta + T^{\varepsilon, y^{\smalltext{y}_\tinytext{0}\smalltext{,}\smalltext{p}}_\theta} \in [0, T^{y_\smalltext{0},\tilde{p}^\varepsilon}_0]$, we can apply the previous inequality \eqref{eq:DPPUImpoinequality} to get
\begin{align*}
\bar{F}(y_0) &\geq \int_0^{T \wedge \theta} \rho \mathrm{e}^{-\rho t} \big(- m + F(p(t))\big) \d t + \mathrm{e}^{-\rho T} F(y^{y_\smalltext{0},p}(T)) \mathbf{ 1}_{\{T < \theta\}} \\
&\quad+\bigg(\mathrm{e}^{-\rho \theta} \int_{\theta}^{\theta + T^{\smalltext{\varepsilon}\smalltext{,}\smalltext{y}^{\tinytext{y}_\stinytext{0}\tinytext{,}\tinytext{p}}(\theta)}} \rho \mathrm{e}^{-\rho (t-\theta)} \big(- m + F(p^{\varepsilon, y^{y_\smalltext{0},p}(\theta)}(t-\theta))\big) \d t\bigg)  \mathbf{1}_{\{T \geq \theta\}} \\
&\quad +\mathrm{e}^{-\rho \theta} \mathrm{e}^{-\rho T^{\smalltext{\varepsilon}\smalltext{,}\smalltext{y}^{\tinytext{y}_\stinytext{0}\tinytext{,}\tinytext{p}}(\theta)}} F\Big(y^{y^{\smalltext{y}_\tinytext{0}\smalltext{,}\smalltext{p}}(\theta),p^{\smalltext{\varepsilon}\smalltext{,}\smalltext{y}^{\tinytext{y}_\stinytext{0}\tinytext{,}\tinytext{p}}(\theta)}}\big(T^{\varepsilon, y^{\smalltext{y}_\tinytext{0}\smalltext{,}\smalltext{p}}(\theta)}\big)\Big) \mathbf{ 1}_{\{T^{\smalltext{\varepsilon}\smalltext{,}\smalltext{y}^{\tinytext{y}_\stinytext{0}\tinytext{,}\tinytext{p}}(\theta)} < \infty \}}  \mathbf{ 1}_{\{T \geq \theta\}} \\
&\geq \int_0^{T \wedge \theta} \rho \mathrm{e}^{-\rho t} \big(- m + F(p(t))\big) \d t + \mathrm{e}^{-\rho T} F\big(y^{y_\smalltext{0},p}(T)\big)\; \mathbf{ 1}_{\{T < \theta\}} + \big(\bar{F}_i(y^{y_\smalltext{0},p}(\theta)) - \varepsilon\big) \mathbf{ 1}_{\{T \geq \theta\}}.
\end{align*}
The conclusion follows by taking the supremum over $p \in \cB_{\R_\smalltext{+}}$ and $T \in [0, T^{y_\smalltext{0},p}_0]$, and then by sending $\varepsilon$ to zero.
\end{proof}

\begin{proposition}\label{theorem:viscositySolFbar}
The function $\bar{F}$ is a viscosity solution of the Hamilton--Jacobi equation \eqref{align:hjFbar}.
\end{proposition}
\begin{proof}
Starting from the evident inequality $\bar{F} - F \geq 0$ on $[0,\infty)$, it follows immediately that $\bar{F}_i - F \geq 0$ on the same interval. Thus, it suffices to prove that $\bar{F}$ is a viscosity super-solution of the ODE presented in \eqref{align:hjFbar} to show the super-solution property. Let us fix a smooth test function $\phi$ and $y^0 \geq 0$ such that $y^0$ is a local minimum point for the difference $\bar{F}_i - \phi$ on $[0,\infty)$, with $(\bar{F}_i - \phi)(y^0) = 0$. Then, by definition of $\bar{F}_i$, there exists a non-negative sequence $(y_n)_{n\in\N}$ such that 
\begin{align*}
y_n \underset{n\to\infty}{\longrightarrow} y^0, \; \bar{F}(y_n) \underset{n\to\infty}{\longrightarrow} \bar{F}_i(y^0), \; \bar{F}(y_n) - \phi(y_n) \underset{n\to\infty}{\longrightarrow} 0.
\end{align*}
Consequently, we can define a positive sequence $(h_n)_{n\in\N}$ by
\begin{equation}\label{eq:sequencehN}
h_n \coloneqq |\bar{F}(y_n) - \phi(y_n)|^{\frac{1}{2}} \; \mathbf{ 1}_{\{|\bar{F}(y_\smalltext{n}) - \phi(y_\smalltext{n})| >0\}} + \frac{1}{n} \; \mathbf{ 1}_{\{|\bar{F}(y_\smalltext{n}) - \phi(y_\smalltext{n})| =0\}}, \; \text{for} \; n \in \N.
\end{equation}
It is true that 
\begin{align*}
h_n \underset{n\to\infty}{\longrightarrow} 0, \; \frac{\bar{F}(y_n)-\phi(y_n)}{h_n} \underset{n\to\infty}{\longrightarrow} 0.
\end{align*}
The second inequality of \Cref{lemma:dppINFSUP} implies that, for any $p \in \cB_{\R_\smalltext{+}}$ and sufficiently small values of $h_n$,
\begin{align*}
\bar{F}(y_n) &\geq \int_0^{h_{n}} \rho \mathrm{e}^{-\rho t} \big(- m + F(p(t))\big) \d t + \mathrm{e}^{-\rho h_\smalltext{n}} \bar{F}_i (y^{y_\smalltext{n},p}(h_n)) \\
&\geq \int_0^{h_{n}} \rho \mathrm{e}^{-\rho t} \big(- m + F(p(t))\big) \d t + \mathrm{e}^{-\rho h_\smalltext{n}} \phi(y^{y_\smalltext{n},p}(h_n)) \\
&= \int_0^{h_{n}} \rho \mathrm{e}^{-\rho t} \big(- m + F(p(t))\big) \d t + \bigg(\phi(y_n) + \int_0^{h_{n}} \rho \mathrm{e}^{-\rho t}\big (-\phi(y^{y_\smalltext{n},p}(t)) + \delta \phi^{\prime}(y^{y_\smalltext{n},p}(t))  (y^{y_\smalltext{n},p}(t) - p(t))\big) \d t\bigg).
\end{align*}
Here, the second inequality follows directly from the local minimiser property of $y^0$, while the last equality is a consequence of the fundamental theorem of calculus. It follows that 
\begin{align*}
\frac{\bar{F}(y_n) - \phi(y_n)}{h_n} + \frac{1}{h_n} \int_0^{h_{n}} \rho \mathrm{e}^{-\rho t} \big(m - F(p(t)) + \phi(y^{y_\smalltext{n},p}(t)) - \delta \phi^{\prime}(y^{y_\smalltext{n},p}(t)) (y^{y_\smalltext{n},p}(t) - p(t))\big) \d t \geq 0
\end{align*}
and, by taking the limit for $n$ going to $\infty$, we get that 
\begin{align*}
m - F(p(0)) + \phi(y^0) - \delta \phi^{\prime}(y^0) (y^0 - p(0)) \geq 0 \; \text{for any} \; p \in \cB_{\R_\smalltext{+}}.
\end{align*}
The super-solution property is therefore proved due to the arbitrariness of $p \in \cB_{\R_\smalltext{+}}$. Subsequently, we show that $\bar{F}$ is a sub-solution by contradiction. As before, we select a smooth test function $\phi$ and $y^0 \geq 0$ such that $y^0 $ is a point of local maximum for $\bar{F}^s - \phi$, with $(\bar{F}^s - \phi)(y^0 )=0$. By contradiction, we assume that 
\begin{align*}
\phi(y^0) - F(y^0) > 0, \; \text{and} \; F^{\star}(\delta \phi^{\prime}(y^0)) - \delta \phi^{\prime}(y^0) y^0  + \phi(y^0) + m > 0.
\end{align*}
Let us fix $\varepsilon > 0$. The definition of $\bar{F}^s$ implies the existence of a sequence $(y_n)_{n\in\N}$ taking values in $(y^0-\varepsilon, y^0+\varepsilon)\cap[0,\infty)$ such that
\begin{align*}
y_n \underset{n\to\infty}{\longrightarrow} y^0, \; \bar{F}(y_n) \underset{n\to\infty}{\longrightarrow} \bar{F}^s(y^0), \; \bar{F}(y_n) - \phi(y_n) \underset{n\to\infty}{\longrightarrow} 0.
\end{align*}
As in \eqref{eq:sequencehN}, we can construct a positive sequence $(h_n)_{n\in\N}$ satisfying 
\begin{align*}
h_n \underset{n\to\infty}{\longrightarrow} 0, \; \frac{\bar{F}(y_n)-\phi(y_n)}{h_n} \underset{n\to\infty}{\longrightarrow} 0.
\end{align*}
Then, the first inequality in \Cref{lemma:dppINFSUP} implies that, for sufficiently small values of $h_n$, there exists $p^\varepsilon \in \cB_{\R_+}$ such that 
\begin{align*}
\bar{F}(y_n) - \varepsilon h_n &\leq \int_0^{h_n} \rho \mathrm{e}^{-\rho t} \big(- m + F(p^\varepsilon(t))\big) \d t + \mathrm{e}^{-\rho h_n} \bar{F}^s \big(y^{y_\smalltext{n},p^\smalltext{\varepsilon}}(h_n)\big) \\
&\leq \int_0^{h_n} \rho \mathrm{e}^{-\rho t} \big(- m + F(p^\varepsilon(t))\big) \d t + \mathrm{e}^{-\rho h_n} \phi\big(y^{y_\smalltext{n},p^\smalltext{\varepsilon}}(h_n)) \\
&= \int_0^{h_n} \rho \mathrm{e}^{-\rho t} \big(- m + F(p^\varepsilon(t))\big) \d t + \phi(y_n) + \int_0^{h_n} \rho \mathrm{e}^{-\rho t} \big(-\phi(y^{y_\smalltext{n},p^\smalltext{\varepsilon}}(t)) + \delta \phi^{\prime}(y^{y_\smalltext{n},p^\smalltext{\varepsilon}}(t)) (y^{y_\smalltext{n},p^\smalltext{\varepsilon}}(t) - p^\varepsilon(t))\big) \d t.
\end{align*}
It follows that 
\begin{align*}
\frac{\bar{F}(y_n) - \phi(y_n)}{h_n} - \varepsilon + \frac{1}{h_n} \int_0^{h_n} \rho \mathrm{e}^{-\rho t} \big(m - F(p^\varepsilon(t)) + \phi(y^{y_\smalltext{n},p^\smalltext{\varepsilon}}(t)) - \delta \phi^{\prime}(y^{y_\smalltext{n},p^\smalltext{\varepsilon}}(t)) (y^{y_\smalltext{n},p^\smalltext{\varepsilon}}(t) - p^\varepsilon(t))\big) \d t \leq 0.
\end{align*}
Additionally, the mean-value theorem implies that 
\begin{align*}
- \varepsilon + \rho \big(m - F(p^\varepsilon(0)) + \phi(y^0) - \delta \phi^{\prime}(y^0) (y^0 - p^\varepsilon(0))\big) \leq 0.
\end{align*}
We get the desired contradiction by taking the limit for $\varepsilon$ going to zero.
\end{proof}

After showing that $\bar{F}$ is a viscosity solution of the Hamilton--Jacobi equation \eqref{align:hjFbar}, we can infer its equivalence with $v_{\tilde{y}}$ as defined in \Cref{theorem:propFbarYhat}. This conclusion is attained through the following result that relies on the application of the comparison theorem for the aforementioned equation in \Cref{proposition: comparisonFBar}.
\begin{proof}[Proof of \Cref{proposition:FBarGeneralDeltaSmaller1}.\ref{degenerateFBar}]
\Cref{theorem:viscositySolFbar} proves that $\bar{F}^s$ is an upper--semi-continuous viscosity sub-solution of (\ref{align:hjFbar}), whereas $\bar{F}_i$ is a lower--semi-continuous viscosity super-solution. The comparison result stated in \Cref{proposition: comparisonFBar} further indicates that $\bar{F}^s \leq \bar{F}_i$, leading to the conclusion that the function $\bar{F}$ is continuous on its domain of definition. It follows that $\bar{F}$ is the unique viscosity solution of the equation (\ref{align:hjFbar}) within an appropriate class of functions with a given growth condition. Moreover, \Cref{theorem:propFbarYhat} implies that $\bar{F} = v_{\tilde{y}}$ on $[0,\infty)$.
\end{proof}

\begin{proposition}\label{proposition: comparisonFBar}
Let $u$ be an upper--semi-continuous viscosity sub-solution and $v$ be a lower--semi-continuous viscosity super-solution to {\rm\Cref{align:hjFbar}} so that there exists $(\bar c_0,\bar c_1)\in(0,\infty)^2$ satisfying 
\[
-\bar{c}_0(1+y^\gamma) \leq \phi(y) \leq \bar{c}_1(1-y^\gamma),\; y\geq 0, \; \text{for} \; \phi \in \{u,v\}.
\]
If $u(0) = 0 = v(0)$, then $u \leq v$ on $[0,\infty)$.
\end{proposition}

\begin{proof}
By contradiction, we assume that there exists some $y^{\smalltext{\Delta}} \in (0,\infty)$ such that $(u-v)(y^{\smalltext{\Delta}}) > 0$. Let us fix $\varepsilon > 0$. We introduce a positive and non-decreasing sequence $(\alpha_n)_{n\in\N}$ such that $\alpha_n$ goes to $\infty$ as $n$ goes to $\infty$, and define 
\[
M_{\alpha_\smalltext{n}}^{\varepsilon} \coloneqq \sup_{(x,y) \in \R^\smalltext{2}_\smalltext{+}} \big\{u(x) - v(y) - \phi_{\alpha_n}^{\varepsilon}(x,y)\big\}, \; \text{where} \; \phi_{\alpha_\smalltext{n}}^{\varepsilon}(x,y) \coloneqq \frac{1}{2} \big(\alpha_n |x-y|^p + \varepsilon |x|^p + \varepsilon |y|^p\big), \; \text{for} \; (x,y) \in \R^2_+,
\]
for a generic $p >\gamma$. The growth condition on $u$ and $v$ implies that there exists a maximiser $(x_n,y_n)\coloneqq (x_{\alpha_\smalltext{n}}, y_{\alpha_\smalltext{n}})$ such that 
\begin{align*}
M_{\alpha_\smalltext{n}}^\varepsilon = u(x_n) - v(y_n) - \phi_{\alpha_\smalltext{n}}^{\varepsilon}(x_n,y_n).
\end{align*}
It is worth noting that we can find a compact set where the sequence $(x_n,y_n)_{n\in\N}$ takes values. Consequently, by considering a sub-sequence if necessary, we have that $(x_n,y_n)_{n\in\N}$ converges to $(x^0,y^0)$ as $n$ goes to $\infty$, for some non-negative $x^0$ and $y^0$. Let us first prove that the sequence $(\alpha_n)_{n\in\N}$ is such that 
\begin{align*}
(x_n,y_n) \underset{n\to\infty}{\longrightarrow} (y^0,y^0), \; \alpha_n |x_n - y_n|^p \underset{n\to\infty}{\longrightarrow} 0, \; M^\varepsilon_{\alpha_\smalltext{n}} \underset{n\to\infty}{\longrightarrow} M^\varepsilon_{\infty} \coloneqq \sup_{y \in \R_\smalltext{+}} \big\{(u-v)(y) - \varepsilon |y|^p\big\}.
\end{align*}
To this aim, note that
\begin{align*}
N^\varepsilon \coloneqq \sup_{y \in \R_\smalltext{+}} \big\{(u-v)(y)-\varepsilon |y|^p\big\} \leq M^\varepsilon_{\alpha_\smalltext{n}} = u(x_n) - v(y_n) - \phi^\varepsilon_{\alpha_\smalltext{n}}(x_n,y_n) < \infty.
\end{align*}
Then, taking into account the upper--semi-continuity of the function and the fact that $\lim_{y \rightarrow \infty} \{(u-v)(y)-\varepsilon |y|^p\} = -\infty$, we can deduce the existence of a maximiser $y^{\star,\varepsilon}$ satisfying $N^\varepsilon = (u-v)(y^{\star,\varepsilon})-\varepsilon |y^{\star,\varepsilon}|^p$. Therefore,
\begin{align*}
d^{\star,\varepsilon} \coloneqq \frac{1}{2} \limsup_{n \rightarrow \infty}\big\{ \alpha_n |x_n - y_n|^p\big\} &\leq \limsup_{n \rightarrow \infty} \bigg\{u(x_n) - v(y_n) - \frac{\varepsilon}{2} |x_n|^p - \frac{\varepsilon}{2} |y_n|^p\bigg\}- (u-v)(y^{\star}) + \varepsilon |y^{\star}|^p \\
&\leq u(x^0) - v(y^0) - \frac{\varepsilon}{2} |x^0|^p - \frac{\varepsilon}{2} |y^0|^p - (u-v)(y^{\star,\varepsilon}) + \varepsilon |y^{\star,\varepsilon}|^p,
\end{align*}
given the upper--semi-continuity of the map $(x, y) \longmapsto u(x) - v(y) - \frac{\varepsilon}{2} |x|^p - \frac{\varepsilon}{2} |y|^p$. We deduce that $d^{\star,\varepsilon} < \infty$, which in turn implies that $x^0 = y^0$ since $\alpha_n$ explodes as $n$ goes to $\infty$. Hence, we can conclude that $M^\varepsilon_{\alpha_\smalltext{n}}$ converges to $M^\varepsilon_\infty$ since, by definition of $y^{\star,\varepsilon}$, we have that
\begin{align*}
0 \leq d^{\star,\varepsilon} \leq u(y^0) - v(y^0) - \frac{\varepsilon}{2} |y^0|^p - \frac{\varepsilon}{2} |y^0|^p - (u-v)(y^{\star,\varepsilon}) + \varepsilon |y^{\star,\varepsilon}|^p \leq 0.
\end{align*}
 
We show that $y^0$ is positive. For $\varepsilon$ small enough, it holds that 
\begin{align*}
0 < (u-v)(y^{\smalltext{\Delta}}) - \varepsilon |y^{\smalltext{\Delta}}|^p \leq u(x_n) - v(y_n) - \phi_{\alpha_\smalltext{n}}^{\varepsilon}(x_n,y_n) = M^\varepsilon_{\alpha_\smalltext{n}} &\leq \limsup_{n \rightarrow \infty} M^\varepsilon_{\alpha_\smalltext{n}} \\
&\leq \limsup_{n \rightarrow \infty} (u(x_n) - v(y_n)) \\
&\leq \limsup_{n \rightarrow \infty} u(x_n) - \liminf_{n \rightarrow \infty} v(y_n) \leq u(y^0) - v(y^0).
\end{align*}
Therefore, we can deduce that $y^0 > 0$ since $u(0) = 0 = v(0)$ by assumption.

\medskip
We have now all the necessary elements to achieve the desired contradiction. In fact, we have that $x_n$ is a local maximiser of the function $x \longmapsto u(x) - \phi^{1,\varepsilon}_{\alpha_n}(x) \coloneqq u(x) - \frac{1}{2}(\alpha_n |x - y_n|^p + \varepsilon |x|^p)$. We can deduce from the fact that $u$ is a viscosity sub-solution that
\begin{align}\label{align:uSub}
\min\bigg\{u(x_n) - F(x_n), F^{\star}\big(\delta \phi^{{1,\varepsilon}^{\smalltext\prime}}_{\alpha_n}(x_n,y_n)\big) - \frac{\delta p}{2} x_n (\alpha_n \mathrm{sgn}(x_n-y_n) |x_n-y_n|^{p-1} + \varepsilon |x_n|^{p-1}) + u(x_n) + m\bigg\} \leq 0.
\end{align}
Similarly, the point $y_n$ is a local minimiser of $y \longmapsto v(y) - \phi^{2,\varepsilon}_{\alpha_n}(y) \coloneqq v(y) - \left(-\frac{1}{2}(\alpha_n |x_n - y|^p - \varepsilon |y|^p)\right)$, and $v$ is a viscosity super-solution. Therefore, we get that
\begin{align}\label{align:vSuper}
\min\bigg\{v(y_n) - F(y_n), F^{\star}\big(\delta \phi^{{2,\varepsilon}^{\smalltext\prime}}_{\alpha_n}(x_n,y_n)\big) - \frac{\delta p}{2} y_n (\alpha_n \mathrm{sgn}(x_n-y_n) |x_n-y_n|^{p-1} - \varepsilon |y_n|^{p-1}) + v(y_n) + m\bigg\} \geq 0.
\end{align}
We divide our analysis into two distinct cases. The first case occurs if $u(x_n) - F(x_n) \leq 0$ along some subsequence. In this scenario, \Cref{align:vSuper} implies that $u(x_n) \leq v(y_n)$, which contradicts the fact that $(u-v)(y^{\smalltext{\Delta}}) > 0$. Therefore, we can only consider the second case, where along some subsequence, it holds that
\begin{align}\label{align:uSubPrime}
F^{\star}\big(\delta \phi^{{1,\varepsilon}^{\smalltext\prime}}_{\alpha_n}(x_n,y_n)\big) - \frac{\delta p}{2} x_n (\alpha_n \mathrm{sgn}(x_n-y_n) |x_n-y_n|^{p-1} + \varepsilon |x_n|^{p-1}) + u(x_n) + m \leq 0.
\end{align}
We combine \eqref{align:uSubPrime} with \eqref{align:vSuper}, yielding
\begin{align*}
u(x_n) - v(y_n) &\leq F^{\star}\big(\delta \phi^{{2,\varepsilon}^{\smalltext\prime}}_{\alpha_n}(x_n,y_n)\big) - F^{\star}\big(\delta \phi^{{1,\varepsilon}^{\smalltext\prime}}_{\alpha_n}(x_n,y_n)\big) + \frac{\delta p}{2} \alpha_n \mathrm{sgn}(x_n -y_n) |x_n-y_n|^p + \frac{\delta p}{2} \varepsilon |x_n|^p + \frac{\delta p}{2} \varepsilon |y_n|^p  \\
&= F^{\star}\bigg(\frac{\delta p}{2} (\alpha_n \mathrm{sgn}(x_n-y_n) |x_n-y_n|^{p-1} - \varepsilon |y_n|^{p-1})\bigg) - F^{\star}\bigg(\frac{\delta p}{2} (\alpha_n \mathrm{sgn}(x_n-y_n) |x_n-y_n|^{p-1} + \varepsilon |x_n|^{p-1})\bigg) \\
&\quad+ \frac{\delta p}{2} \alpha_n \mathrm{sgn}(x_n -y_n) |x_n-y_n|^p + \frac{\delta p}{2} \varepsilon |x_n|^p + \frac{\delta p}{2} \varepsilon |y_n|^p \\
&\leq \frac{\delta p}{2} \alpha_n \mathrm{sgn}(x_n -y_n) |x_n-y_n|^p + \frac{\delta p}{2} \varepsilon |x_n|^p + \frac{\delta p}{2} \varepsilon |y_n|^p.
\end{align*}
We arrive at a contradiction with the hypothesis that $(u-v)(y^{\smalltext{\Delta}}) > 0$ by first letting $\varepsilon$ go to zero, and then $n$ go to infinity. This completes the proof.
\end{proof}

\subsection{A very impatient principal}
We consider here the remaining case where $\delta < 1$ and $\gamma\delta \leq 1$. In analogy with the reference paper \citep{possamai2020there}, which demonstrates that the face-lifted utility is null, we posit that the mixed control--stopping problem degenerates. In what follows, we prove that $\bar{F}$ coincides with the function already introduced in \Cref{proposition:FBarGeneralDeltaSmaller1}:
\begin{align}\label{align:vm}
v_{m}(y) \coloneqq F(y)  \mathbf{1}_{[0,\hat{y}]}(y) -m \mathbf{1}_{(\hat{y}, \infty)}(y),\; y\geq 0.
\end{align}
To achieve this, we once again need to employ the theory of viscosity solutions, as $v_{m}$ is not continuously differentiable.

\begin{lemma}\label{lemma:FBarlsc}
The function $\bar{F}$ introduced in {\rm\Cref{eq:FbarEquation}} is lower--semi-continuous.
\end{lemma}

\begin{proof}
The lemma is proved by showing that the epigraph $\mathrm{epi}(\bar{F}) \coloneqq \{ (y,t) \in [0,\infty) \times \R:  t \geq \bar{F}(y)\}$ is closed. We first observe that, for any $y_0 \geq 0$, the function $\bar{F}$ can be expressed as 
\begin{align*}
\bar{F}(y_0) = \sup_{p \in \cB_{\R_\smalltext{+}}} \sup_{T \geq 0} \bigg\{\bigg(\mathrm{e}^{-\rho T} F\big(y^{y_\smalltext{0},p}(T)\big) \mathbf{1}_{\{T < \infty\}} + \int_0^T \rho \mathrm{e}^{-\rho t} \big(- m + F(p(t))\big) \d t\bigg) \mathbf{1}_{\{T \in [0, T^{\smalltext{y}_\tinytext{0}\smalltext{,}\smalltext{p}}_\smalltext{0}]\}} - \infty  \mathbf{1}_{\{T \in (T^{\smalltext{y}_\tinytext{0}\smalltext{,}\smalltext{p}}_\smalltext{0},\infty)\}}\bigg\}.
\end{align*}
Let us consider a sequence $(y_n,t_n)_{n\in\N}$ such that 
\begin{align*}
(y_n,t_n) \in \mathrm{epi}(\bar{F}) \; \text{for any} \; n \in \N, \; \text{and} \; (y_n,t_n) \underset{n\to\infty}{\longrightarrow} (y_0,t), \; \text{for some} \; y_0 \geq 0,\; t \in \R.
\end{align*} 
Furthermore, for any $p \in \cB_{\R_\smalltext{+}}$ and $T \geq 0$, we have
\begin{align*}
t_n \geq \bar{F}(y_n) \geq \bigg(\mathrm{e}^{-\rho T} F\big(y^{y_\smalltext{n},p}(T)\big) \mathbf{1}_{\{T < \infty\}}+ \int_0^T \rho \mathrm{e}^{-\rho t} \big(- m + F(p(t))\big) \d t\bigg) \mathbf{ 1}_{\{T \in [0, T^{\smalltext{y}_\tinytext{n}\smalltext{,}\smalltext{p}}_\smalltext{0}]\}} - \infty \mathbf{ 1}_{\{T \in (T^{\smalltext{y}_\tinytext{n}\smalltext{,}\smalltext{p}}_\smalltext{0},\infty)\}}.
\end{align*}
By extracting a monotone sub-sequence $(y_{n_\smalltext{k}})_{k\in\N}$, we observe that
\begin{align*}
t \geq \limsup_{k \rightarrow \infty} \bigg\{ \bigg(\mathrm{e}^{-\rho T} F\big(y^{y_{\smalltext{n}_\tinytext{k}},p}(T)\big) \mathbf{1}_{\{T < \infty\}} + \int_0^T \rho \mathrm{e}^{-\rho t} \big(- m + F(p(t))\big) \d t\bigg) \mathbf{1}_{\{T \in [0, T^{\smalltext{y}_{\tinytext{n}_\stinytext{k}}\smalltext{,}\smalltext{p}}_\smalltext{0}]\}} - \infty \mathbf{ 1}_{\{T \in (T^{\smalltext{y}_{\tinytext{n}_\stinytext{k}}\smalltext{,}\smalltext{p}}_\smalltext{0},\infty)\}}\bigg\}.
\end{align*}
The fact that $(y_{n_\smalltext{k}})_{k\in\N}$ is a convergent monotone sequence implies that so is $(T^{y_{n_\smalltext{k}},p}_0)_{k\in\N}$ for any $p \in \cB_{\R_\smalltext{+}}$. We can therefore conclude that 
\begin{align*}
t \geq \sup_{p \in \cB_{\smalltext{\R}_\tinytext{+}}} \sup_{T \geq 0} \bigg\{\bigg(\mathrm{e}^{-\rho T} F\big(y^{y_{\smalltext{n}_\tinytext{k}},p}(T)\big) \mathbf{1}_{\{T < \infty\}} + \int_0^T \rho \mathrm{e}^{-\rho t} \big(- m + F(p(t))\big) \d t\bigg) \mathbf{ 1}_{\{T \in [0, T^{\smalltext{y}_\tinytext{0}\smalltext{,}\smalltext{p}}_\smalltext{0})\}} - \infty \mathbf{ 1}_{\{T \in (T^{\smalltext{y}_\tinytext{0}\smalltext{,}\smalltext{p}}_\smalltext{0},\infty)\}}\bigg\} = \bar{F}(y_0).
\end{align*}
This completes the proof.
\end{proof}

Now, we can show that the solution $\bar{F}$ of the mixed control--stopping problem coincides with the function $v_{m}$ defined in {\rm\Cref{align:vm}}.
\begin{proof}[Proof of \Cref{proposition:FBarGeneralDeltaSmaller1}.\ref{nonDegenerateFBar}]
First, we prove that $v_m \geq \bar{F}$ on $[0,\infty)$. We fix some $\varepsilon > 0$ and easily construct a strictly convex, continuously differentiable function $g_{\varepsilon}:[\hat{y}-\varepsilon,\hat{y}+\varepsilon] \longrightarrow \R$ such that
\begin{align*}
v_{\varepsilon,m}(y) \coloneqq \begin{cases}
	v_{m}(y) = F(y),\;  y \in [0,\hat{y}-\varepsilon) \\
	g_{\varepsilon}(y),\; y \in [\hat{y}-\varepsilon,\hat{y}+\varepsilon] \\
	v_{m}(y) = -m,\; y \in (\hat{y}+\varepsilon,\infty)
	\end{cases}
\end{align*}
is continuously differentiable. Our goal is to prove that $v_{\varepsilon,m}$ is a viscosity super-solution of the Hamilton--Jacobi equation \eqref{align:hjFbar}. To achieve this, it is sufficient to focus our analysis on $[\hat{y}-\varepsilon,\hat{y}+\varepsilon]$, as this property is clearly satisfied outside this interval. We begin by noting that $g_{\varepsilon} \geq F$. This is a consequence of the fact that the derivative $g^{\prime}_{\varepsilon}$ is decreasing, as evidenced by
\begin{align*}
g_{\varepsilon}(y) = g_{\varepsilon}(\hat{y}-\varepsilon) + \int_{\hat{y}-\varepsilon}^y g^{\prime}_{\varepsilon}(s) \d s \geq F(\hat{y}-\varepsilon) + \int_{\hat{y}-\varepsilon}^y F^{\prime}(s) \d s = F(y) \; \text{for any} \; y \in [\hat{y}-\varepsilon,\hat{y}+\varepsilon].
\end{align*}
Therefore, we need to determine the sign of $F^{\star}(\delta g^{\prime}_{\varepsilon}(y)) - \delta y g^{\prime}_{\varepsilon}(y) + g_{\varepsilon}(y) + m$ for any $y \in [\hat{y}-\varepsilon,\hat{y}+\varepsilon]$. Let us fix some $y$ within this interval and consider the concave function $[F^{\prime}(y), 0] \ni p \longmapsto F^{\star}(\delta p) - \delta y p$.  Since this function is decreasing and null at $0$, we have that
\begin{align*}
F^{\star}(\delta g^{\prime}_{\varepsilon}(y)) - \delta y g^{\prime}_{\varepsilon}(y) \geq 0 \; \text{for any} \; y \in [\hat{y}-\varepsilon,\hat{y}+\varepsilon].
\end{align*}
Furthermore, it is evident that $g_{\varepsilon} + m \geq F + m \geq 0$ holds on $[\hat{y}-\varepsilon,\hat{y}+\varepsilon]$. Consequently, we can conclude that the function $v_{\varepsilon,m}$ is a viscosity super-solution of the Hamilton--Jacobi equation \eqref{align:hjFbar}. Therefore, by mimicking the proof of \citep[Lemma A.2]{possamai2020there}, it can be easily proved that $v_{\varepsilon,m} \geq \bar{F}$ on $[0,\infty)$. 

\medskip
We have showed the required inequality for the function $v_{\varepsilon,m}$, but we still need to demonstrate it for $v_m$, which is what we aim to accomplish. It is evident that $v_m \geq \bar{F}$ on $[0,\hat{y}) \cup (\hat{y},\infty)$ since $v_{\varepsilon,m}(y)$ converges to $v_m(y)$ as $\epsilon$ goes to zero, for any $y$ in that interval. However, this is also true at $\hat{y}$ because \Cref{lemma:FBarlsc} proves that $\bar{F}$ is lower--semi-continuous (and $v_m$ is continuous by definition). In fact, by considering a non-negative non-constant sequence $(y_n)_{n\in\N}$ converging to $\hat{y}$, we can see that $v_{\varepsilon,m}(y_n) \geq \bar{F}(y_n)$. Hence, $v_m(y_n) \geq \bar{F}(y_n)$ for any $n \in \N$, and thus
\begin{align*}
v_m(\hat{y}) \geq \liminf_{n \rightarrow \infty} \bar{F}(y_n) \geq \bar{F}(\hat{y}).
\end{align*}

\medskip
Finally, we need to show the reverse inequality, that is, $v_m \leq \bar{F}$ on $[0,\infty)$. For any $y \in [0, \hat{y}]$, it is easy to see that the trivial control $(T^{\star},p^{\star}) = (0,0)$ is optimal since it attains the upper bound $v_m(y)=F(y)$. On the other hand, if $y > \hat{y}$, we can mimic the proof in \cite[Section A.1]{possamai2020there} by constructing a sequence of controls $(T^{\star},p^{\star})$ which induces $\bar{F}(y)$ to reach its upper bound $v_m(y)=-m$. This concludes the proof.
\end{proof}

\bigskip

\section{Analysis of the first-best problem when the agent is more impatient} \label{appendix:FBImpatientAgent}

\Cref{thm:FirstBestCompleteC}.\ref{viscosityCFB} identifies the first-best  value function as the unique viscosity solution of the Hamilton--Jacobi--Bellman equation \eqref{align:hjbFB} when $\delta\neq 1$ and $\gamma\delta>1$. As a consequence, $v^{\smalltext{\rm FB}}$ is continuous, implying that it admits both a left- and right-derivative $v^{\smalltext{\rm FB}}_-(y)$ and $v^{\smalltext{\rm FB}}_+(y)$ at any $y \in (-h(\bar{a},\varepsilon_m),\infty)$. This also proves the existence of $(v^{\smalltext{\rm FB}})^{\prime}(-(\bar{a}-\varepsilon_m))$, where it is meant as the right derivative, and we show it is null. 

\begin{lemma}\label{lemma:vPrime0Negative}
If we assume that $\delta\neq 1$ and $\gamma\delta>1$, then $(v^{\smalltext{\rm FB}})^{\prime}(-h(\bar{a},\varepsilon_m)) = 0$.
\end{lemma}
\begin{proof}
Applying the Karush--Kuhn--Tucker approach to the first-best value function  $v^{\smalltext{\rm FB}}$ results in
\begin{align*}
v^{\smalltext{\rm FB}}(-h(\bar{a},\varepsilon_m)) =& \inf_{\lambda \leq 0}\bigg\{-\lambda h(\bar{a},\varepsilon_m) + \sup_{T \geq 0} \bigg\{
-\mathrm{e}^{-\rho T} F^\star\big(\lambda \mathrm{e}^{\rho (1-\delta) T}\big) \mathbf{1}_{\{T < \infty\}} + \int_0^T \rho \mathrm{e}^{-\rho t} \big(G^\star - F^\star\big)\big(\delta\lambda\mathrm{e}^{\rho(1-\delta)t}\big) \d t \bigg\}\bigg\} \\
\geq& \inf_{\lambda \leq 0}\bigg\{-\lambda h(\bar{a},\varepsilon_m) + \int_0^\infty \rho \mathrm{e}^{-\rho t} \big(G^\star - F^\star\big)\big(\delta\lambda\mathrm{e}^{\rho(1-\delta)t}\big) \d t\bigg\} \\
\geq& \bar{a}-\varepsilon_m +\inf_{\lambda \leq 0}\bigg\{- \int_0^\infty \rho \mathrm{e}^{-\rho t} F^\star\big(\delta\lambda\mathrm{e}^{\rho(1-\delta)t}\big) \d t\bigg\} = \bar{a}-\varepsilon_m,
\end{align*}
where the second inequality follows from $G^\star(p) \geq \bar{a}-\varepsilon_m + p h(\bar{a},\varepsilon_m)$ for all $p \in \mathbb{R}$ by definition, while the last equality stems from $F^\star$ being a non-positive function, as explained in \Cref{remark:FFStarProperties}. Given that the initial condition stated in \eqref{align:hjbFB} is $v^{\smalltext{\rm FB}}(-h(\bar{a},\varepsilon_m)) = \bar{a}-\varepsilon_m$, we infer that $(v^{\smalltext{\rm FB}})^{\prime}(-h(\bar{a},\varepsilon_m)) = 0$ since $(v^{\smalltext{\rm FB}})^{\prime}(-h(\bar{a},\varepsilon_m)) = \lambda^\star_{{\bar{a}},\varepsilon_{\smalltext{m}}}$, where $\lambda^\star_{{\bar{a}},\varepsilon_{\smalltext{m}}} = 0$ minimises the expression of $v^{\smalltext{\rm FB}}$.
\end{proof}

In analogy with the case without accidents studied in \cite[Theorem 3.1]{possamai2020there}, we conjecture that the first-best value function $v^{\smalltext{\rm FB}}$ does not intersect the barrier $\bar{F}$ when $\delta > 1$. This means that it effectively solves the ODE in the Hamilton--Jacobi--Bellman equation \eqref{align:hjbFB} over the entire interval $(-h(\bar{a},\varepsilon_m),\infty)$, and thus is concave due to the definition of the operator $\cJ^{\smalltext{\rm FB}}$ introduced in \eqref{eq:operatorfirstbest}. Indeed, our aim is to show that $v^{\smalltext{\rm FB}}$ is strictly concave and therefore decreasing, as we have just proved that $(v^{\smalltext{\rm FB}})^{\prime}(-h(\bar{a},\varepsilon_m))$ is zero. Consequently, it seems natural to introduce the concave dual function $v^{\smalltext{\rm FB},\star}$ and the equation that it should satisfy. This approach allows us to subsequently to deduce the properties that define $v^{\smalltext{\rm FB}}$ from those of $v^{\smalltext{\rm FB},\star}$. We consider the equation
\begin{align}\label{align:hjbODEVfbStar}
- v^{\smalltext{\rm FB},\star}(p) + (1-\delta) p (v^{\smalltext{\rm FB},\star})^\prime(p) + F^{\star}(\delta p) - G^\star(\delta p)=0, \; p\in (-\infty,0), \; v^{\smalltext{\rm FB},\star}(0) = -(\bar{a}-\varepsilon_m).
\end{align}
For $\delta >1$, this linear ODE has a unique solution given by
\begin{align}\label{align:firstbestStar_expl}
v^{\smalltext{\rm FB},\star}(p) = \frac{(-p)^{\frac{1}{1-\delta}}}{{\delta-1}} \int_p^{0} \frac{F^\star(\delta x) - G^\star(\delta x)}{(-x)^{1+\frac{1}{1-\delta}}} \d x \; \text{for any} \; p \leq 0.
\end{align}

\begin{lemma}\label{lemma:vStarFBconcIneq}
Suppose $\delta > 1$. The solution \eqref{align:firstbestStar_expl} to {\rm\Cref{align:hjbODEVfbStar}} is strictly concave over $(-\infty,0)$. Additionally, we have that $v^{\smalltext{\rm FB},\star}(p) \leq (F^\star - G^\star)(p)$ for all $p \leq 0$.
\end{lemma}
\begin{proof}
We prove the statement by adopting a similar approach as the proof of \cite[Lemma A.4]{possamai2020there}. First, it is important to note that the concavity of $F^\star-G^\star$ implies that for any $ p < 0$, it holds that
\begin{align*}
(F^\star- v^{\smalltext{\rm FB},\star} - G^\star)(p) = F^\star(p) - F^\star(\delta p) - G^\star(p) + G^\star(\delta p) - (1-\delta) p (v^{\smalltext{\rm FB},\star})^\prime(p) \geq (1-\delta) p (F^\star - v^{\smalltext{\rm FB},\star} - G^\star)^\prime(p).
\end{align*}
Now, consider $\phi(p) \coloneqq (-p)^{-\frac{1}{1-\delta}} (F^{\star}(p) - v^{\smalltext{\rm FB},\star}(p) - G^\star(p))$, for $p \leq 0$. This function is non-increasing, and therefore,
\begin{align*}
\phi(p) \geq \lim_{p \rightarrow 0^-} \phi(p) = 0 \; \text{for any} \; p \leq 0.
\end{align*}
By differentiating \Cref{align:hjbODEVfbStar} and substituting the expression for the derivative $(v^{\smalltext{\rm FB},\star})^\prime$, we get that $v^{\smalltext{\rm FB},\star}$ is concave, as indicated by the following inequality:
\begin{align*}
(1-\delta)^2 p^2 (v^{\smalltext{\rm FB},\star})^{\prime\prime}(p) \leq \delta (v^{\smalltext{\rm FB},\star}(p) - F^\star(p) + G^\star(p)) = -\delta (-p)^{\frac{1}{1-\delta}} \phi(p)  \leq 0 \; \text{for any} \; p < 0.
\end{align*}
Actually, the function $v^{\smalltext{\rm FB},\star}$ is strictly concave because there does not exists a non-empty interval $\cI \subset (-\infty, 0)$ where this function exhibit a linear growth, as evident from its representation in \eqref{align:firstbestStar_expl}.
\end{proof}

We have showed that the unique solution to \Cref{align:hjbODEVfbStar} is strictly concave and twice continuously differentiable on the interval $(-\infty,0)$. Consequently, we can deduce that it satisfies the equation
\begin{align*}
- v^{\smalltext{\rm FB},\star}(p) + (1-\delta) p (v^{\smalltext{\rm FB},\star})^\prime(p) + F^{\star}(\delta p) - \cJ^{\smalltext{\rm FB}}\bigg(p,\frac{1}{(v^{\smalltext{\rm FB},\star})^{\prime\prime}(p)}\bigg) = 0, \; p\in (-\infty,0), \; v^{\smalltext{\rm FB},\star}(0) = -(\bar{a}-\varepsilon_m).
\end{align*}
We now introduce the concave dual function $v^{\smalltext{\rm FB},\star \star}(y) \coloneqq \inf_{p \leq 0} \{y p - v^{\smalltext{\rm FB},\star}(p)\}$, for $y \geq -(\bar{a}-\varepsilon_m)$.

\begin{proof}[Proof of \Cref{thm:FirstBestCompleteC}.\ref{deltaBigger1}]
According to \cite[Proposition 5 and Lemma 5]{alvarez1997convex}), $v^{\smalltext{\rm FB},\star \star}$ is a twice continuously differentiable solution of the following equation
\begin{align}\label{align:eqFBstarstar}
F^{\star}(\delta (v^{\smalltext{\rm FB},\star \star})^\prime(y)) - \delta y (v^{\smalltext{\rm FB},\star \star})^\prime(y) + v^{\smalltext{\rm FB},\star \star}(y) -\cJ^{\smalltext{\rm FB}}((v^{\smalltext{\rm FB},\star \star})^\prime(y), (v^{\smalltext{\rm FB},\star \star})^{\prime\prime}(y))=0, \; y \in (-h(\bar{a},\varepsilon_m),\infty),
\end{align}
with initial condition $v^{\smalltext{\rm FB},\star \star}(-h(\bar{a},\varepsilon_m) )= - v^{\smalltext{\rm FB},\star}(0) = \bar{a} - \varepsilon_m$. Hence, showing that $v^{\smalltext{\rm FB},\star \star} \geq \bar{F}$ on $[-h(\bar{a},\varepsilon_m),\infty)$ leads to the conclusion that the function $v^{\smalltext{\rm FB},\star \star}$ satisfies the Hamilton--Jacobi--Bellman equation \eqref{align:hjbFB}. Consequently, $v^{\smalltext{\rm FB},\star \star}$ coincides with the first-best value function $v^{\smalltext{\rm FB}}$, as a consequence of the comparison theorem mentioned in the proof of \Cref{thm:FirstBestCompleteC}.\ref{viscosityCFB}. This stems from its required growth at infinity, which is inherited from the growth at infinity of its dual function $v^{\smalltext{\rm FB},\star}$.

\medskip
First, let us discuss the case $m \leq - F^{\star}(\delta F^{\prime}(0))$. Here, the face-lifted utility $\bar{F}$ is the concave conjugate of the function $w^\star$, introduced in \Cref{align:wstarexp}. We then compare $v^{\smalltext{\rm FB},\star}$ and $w^\star$ over the interval $(-\infty, f_{\delta m} \wedge 0) = (-\infty, f_{\delta m})$, where $f_{\delta m} = (F^\star)^{(-1)}(-m)/\delta < 0$, as defined in \Cref{lemma:derivative}. Since we are working under the assumption  $G^\star(f_{\delta m}) \geq 0$, \Cref{lemma:vStarFBconcIneq} implies that
\begin{align}\label{align:inequality0FBdeltaBigmSmall}
v^{\smalltext{\rm FB},\star}(f_{\delta m}) - w^\star(f_{\delta m}) =  v^{\smalltext{\rm FB},\star}(f_{\delta m}) \leq F^\star(f_{\delta m}) - G^\star(f_{\delta m})= -G^\star(f_{\delta m}) \leq 0,
\end{align}
where the second equality follows from the definition of the function $F^\star$, considering that $f_{\delta m} \geq F^\prime(0)$. Given that both concave conjugates $v^{\smalltext{\rm FB},\star}$ and $w^\star$ are solutions of \eqref{align:hjbODEVfbStar} and \eqref{eq:wStarode}, respectively, their difference $d^\star(p)\coloneqq v^{\smalltext{\rm FB},\star}(p) - w^\star(p)$, for $p \leq f_{\delta m}$, satisfies
\begin{align}\label{align:differencemsmallFB}
(1-\delta) p (d^\star)^\prime(p) - d^\star(p) - m - G^\star(\delta p) = 0, \; p\in (-\infty, f_{\delta m}), \; d^\star(f_{\delta m}) \leq 0.
\end{align}
The unique solution of \eqref{align:differencemsmallFB}, for any $p < f_{\delta m}$, is given by 
\begin{align}\label{align:inequalitydiffFBDeltaBigger1}
d^\star(p) = (-p)^{\frac{1}{1-\delta}} \Bigg(\frac{d^\star(f_{\delta m} )}{(-f_{\delta m})^{\frac{1}{1-\delta}}} + \frac{1}{1-\delta} \int_p^{f_{\smalltext{\delta}\smalltext{m}}} \frac{G^\star(\delta x) + m}{(-x)^{1+\frac{1}{1-\delta}}}\bigg) \leq  \frac{(-p)^{\frac{1}{1-\delta}}}{1-\delta} \int_p^{f_{\smalltext{\delta}\smalltext{m}}} \frac{G^\star(\delta x) + m}{(-x)^{1+\frac{1}{1-\delta}}} < 0,
\end{align}
since the sum $G^\star + m $ is non-negative and there exists some $\varepsilon>0$ such that $G^\star + m >0$ on $(f_{\delta m}-\varepsilon,f_{\delta m}]$, and $\delta > 1$ by assumption. Then, as taking conjugates clearly reverse functional inequalities, we conclude that $v^{\smalltext{\rm FB},\star \star}$ is always not below the face-lifted utility $\bar{F}$. Furthermore, if $G^\star(f_{\delta m}) > 0$, \Cref{align:inequality0FBdeltaBigmSmall} implies that $d^\star < 0$ on $(-\infty, f_{\delta m}]$, and thus $v^{\smalltext{\rm FB},\star \star} > \bar{F}$ on $[-h(\bar{a},\varepsilon_m),\infty)$. On the other hand, if $G^\star(f_{\delta m}) = 0$, the computations done previously, along with concave duality, imply that $v^{\smalltext{\rm FB},\star \star} > \bar{F}$ on $(0,\infty)$, and it is evident that $v^{\smalltext{\rm FB},\star \star} > \bar{F}$ on $[-h(\bar{a},\varepsilon_m),0)$ by definition of $\bar{F}$. If it were true that $v^{\smalltext{\rm FB},\star \star}(0) = \bar{F}(0) = 0$, then $(v^{\smalltext{\rm FB}})^\prime(0) > \bar{F}^\prime(0) = f_{\delta m}$. From the ODE \eqref{align:eqFBstarstar} satisfied by $v^{\smalltext{\rm FB},\star \star}$, we would have
\begin{equation*}
0 = F^{\star}(\delta (v^{\smalltext{\rm FB}})^\prime(0)) + v^{\smalltext{\rm FB}}(0) -G^\star\big(\delta (v^{\smalltext{\rm FB}})^\prime(0)) > F^{\star}(\delta f_{\delta m})  -G^\star\big(\delta (v^{\smalltext{\rm FB}})^\prime(0) = -m  -G^\star\big(\delta (v^{\smalltext{\rm FB}})^\prime(0)),
\end{equation*}
which is absurd since $G^\star \geq -m$ on $\R$. The strict inequality follows from the fact that $(v^{\smalltext{\rm FB}})^\prime(0) > \bar{F}^\prime(0) = f_{\delta m}$, which, by definition, satisfies $F^{\star}(\delta f_{\delta m}) = -m <0$.

\medskip
Now, we show that the statement remains true even in the case when $m > - F^{\star}(\delta F^{\prime}(0))$ employing the same arguments as previously. Here, the face-lifted utility $\bar{F}$ is the concave conjugate of $w^\star$ introduced in \Cref{align:barFyHat}, so we need to study the sign of the difference $d^\star(p)\coloneqq v^{\smalltext{\rm FB},\star}(p) - w^\star(p)$, for $p \leq F^\prime(0) \wedge 0 = F^\prime(0)$. To begin, let us observe that the function $d^\star$ is such that
\begin{align*}
d^\star(p) = v^{\smalltext{\rm FB},\star}(p) - w^\star(p) = v^{\smalltext{\rm FB},\star}(p) - F^\star(p) \leq -G^\star(p) < 0 \; \text{for any} \; p\in [F^\prime(\bar{y}), F^\prime(0)],
\end{align*}
where the first inequality is a consequence of \Cref{lemma:vStarFBconcIneq}, while the second one follows by the assumption $G^\star(F^\prime(\bar{y})) > 0$, which  implies that $G^\star(p) > 0$ for any $p >F^\prime(\bar{y})$. Furthermore, we have that
\begin{align*}
(1-\delta) p (d^\star)^\prime(p) - d^\star(p) - m - G^\star(\delta p) = 0, \; p\in (-\infty, F^\prime(\bar{y})), \; d^\star(F^\prime(\bar{y})) < 0
\end{align*}
from which we conclude, following the same reasoning as in \eqref{align:inequalitydiffFBDeltaBigger1}.
\end{proof}

\section{Proof of the reduction argument}\label{appendix:reductionSecond}

\subsection{The dynamic programming principle for the problem of the agent}
Let us fix an arbitrary contract $\bC = (\tau, \pi, \xi) \in \mathfrak{C}_R$ such that the set $\cU^\star(\bC)$ is not empty. Below, we will derive several properties of the dynamic version of the response of the agent to the given contract $\bC$, which is defined as the following family of random variables:
\begin{align}\label{align:VArandomVar}
V^{\rm A}_\theta(\bC) \coloneqq \esssup_{\nu \in \cU} \tilde{V}^{\rm A}_{\theta}(\bC,\nu) \coloneqq \esssup_{\nu \in \cU} \E^{\P^{\smalltext{\nu}}}_{\theta} \bigg[\mathrm{e}^{-r (\tau-\theta)} u(\xi)  \mathbf{1}_{\{\tau < \infty\}} + \int_\theta^\tau r \mathrm{e}^{-r(s-\theta)} (u(\pi_s) - h(\nu_s)) \d s\bigg],
\end{align}
where $\theta$ is an $\F$--stopping time such that $\theta \leq \tau$, $\P$--a.s. 

\begin{lemma}\label{lemma:dppVA}
For any $\F$--stopping times $\theta$ and $\tilde{\theta}$ such that $\theta \leq \tilde{\theta}\leq \tau$, $\P$--{\rm a.s.}, it holds that 
\begin{align*}
V^{\rm A}_\theta(\bC)  = \esssup_{\nu \in \cU} \E^{\P^{\smalltext{\nu}}}_{\theta} \bigg[\mathrm{e}^{-r (\tilde{\theta}-\theta)} V^{\rm A}_{\tilde{\theta}}(\bC)  \mathbf{1}_{{\{\tilde{\theta}} < \infty\}} + \int_\theta^{\tilde{\theta}} r \mathrm{e}^{-r(s-\theta)} (u(\pi_s) - h(\nu_s)) \d s\bigg], \; \P\text{\rm --a.s.}
\end{align*}
\end{lemma}

\begin{proof}
Proving that the right-hand side of the previous equality is greater than the left-hand side is straightforward. In fact, the tower property implies that 
\begin{align*}
V^{\rm A}_\theta(\bC) = \esssup_{\nu \in \cU} \E^{\P^{\smalltext{\nu}}}_{\theta} \bigg[\mathrm{e}^{-r (\tilde{\theta}-\theta)} \tilde{V}^{\rm A}_{\tilde{\theta}}(\bC,\nu) \mathbf{1}_{{\{\tilde{\theta}} < \infty\}} + \int_\theta^{\tilde{\theta}} r \mathrm{e}^{-r(s-\theta)} (u(\pi_s) - h(\nu_s)) \d s \bigg] , \; \P\text{--a.s.}
\end{align*}
Since $\tilde{V}^{\rm A}_{\tilde{\theta}}(\bC,\nu) \leq V^{\rm A}_{\tilde{\theta}}(\bC)$ on the event set $\{\tilde{\theta} < \infty\}$ by definition, then we can conclude that 
\begin{align*}
V^{\rm A}_\theta(\bC) \leq \esssup_{\nu \in \cU} \E^{\P^{\smalltext{\nu}}}_{\theta} \bigg[\mathrm{e}^{-r (\tilde{\theta}-\theta)} V^{\rm A}_{\tilde{\theta}}(\bC) \mathbf{1}_{{\{\tilde{\theta}} < \infty\}} + \int_\theta^{\tilde{\theta}} r \mathrm{e}^{-r(s-\theta)} (u(\pi_s) - h(\nu_s)) \d s\bigg], \; \P\text{--a.s.}
\end{align*}
We follow the arguments of \cite[Lemma A.4]{el2021optimal} to prove the reverse inequality. First, let us introduce two arbitrary controls $(\nu^1, \nu^2) \in \cU^2$, and define the process $\tilde{\nu}$ as follows:
\begin{align*}
\tilde{\nu}_s \coloneqq \nu^1_s \; \mathbf{1}_{[0,\tilde{\theta}]}(s) + \nu^2_s \; \mathbf{1}_{(\tilde{\theta},\tau]}(s), \; \text{for} \; s \geq 0.
\end{align*}
As it is evident that $\tilde{\nu} \in \cU$, we see that
\begin{align*}
V^{\rm A}_\theta(\bC) \geq \tilde{V}^{\rm A}_{\theta}(\bC,\tilde{\nu}) &= \E^{\P^{\smalltext{\tilde{\nu}}}}_{\theta}\bigg[\mathrm{e}^{-r (\tilde{\theta}-\theta)} \tilde{V}^{\rm A}_{\tilde{\theta}}(\bC,\tilde{\nu}) \mathbf{1}_{{\{\tilde{\theta}} < \infty\}} + \int_\theta^{\tilde{\theta}} r \mathrm{e}^{-r(s-\theta)} (u(\pi_s) - h(\tilde{\nu}_s)) \d s \bigg] \\
&= \E^{\P^{\smalltext{\tilde{\nu}}}}_{\theta}\bigg[\mathrm{e}^{-r (\tilde{\theta}-\theta)} \tilde{V}^{\rm A}_{\tilde{\theta}}(\bC,\nu^2) \mathbf{1}_{{\{\tilde{\theta}} < \infty\}} + \int_\theta^{\tilde{\theta}} r \mathrm{e}^{-r(s-\theta)} (u(\pi_s) - h(\nu^1_s)) \d s \bigg], \; \P\text{--a.s.}
\end{align*}
because $\tilde{V}^{\rm A}_{\tilde{\theta}}(\bC,\tilde{\nu})$ does not depend on the values that $\tilde{\nu}$ takes before the stopping time $\tilde{\theta}$. A change of measure implies that 
\begin{gather}\begin{aligned}\label{align:VATilde}
\tilde{V}^{\rm A}_{\theta}(\bC,\tilde{\nu}) &= \E^\P_{\theta}\Bigg[\frac{M^{\smalltext{\tilde{\nu}}}_{\tilde{\theta}}}{M^{\smalltext{\tilde{\nu}}}_{\theta}} \bigg(
\mathrm{e}^{-r (\tilde{\theta}-\theta)} \tilde{V}^{\rm A}_{\tilde{\theta}}(\bC,\nu^2) \mathbf{1}_{{\{\tilde{\theta}} < \infty\}} + \int_\theta^{\tilde{\theta}} r \mathrm{e}^{-r(s-\theta)} (u(\pi_s) - h(\nu^1_s)) \d s \bigg)\Bigg]\\
&= \E^\P_{\theta}\Bigg[\frac{M^{\smalltext{{\nu^1}}}_{\tilde{\theta}}}{M^{\smalltext{{\nu^1}}}_{\theta}} \bigg(
\mathrm{e}^{-r (\tilde{\theta}-\theta)} \tilde{V}^{\rm A}_{\tilde{\theta}}(\bC,\nu^2) \mathbf{1}_{{\{\tilde{\theta}} < \infty\}} + \int_\theta^{\tilde{\theta}} r \mathrm{e}^{-r(s-\theta)} (u(\pi_s) - h(\nu^1_s)) \d s \bigg)\Bigg] \\
&=\E^{\P^{\smalltext{\nu}^\tinytext{1}}}_{\theta}\bigg[\mathrm{e}^{-r (\tilde{\theta}-\theta)} \tilde{V}^{\rm A}_{\tilde{\theta}}(\bC,\nu^2) \mathbf{1}_{{\{\tilde{\theta}} < \infty\}} + \int_\theta^{\tilde{\theta}} r \mathrm{e}^{-r(s-\theta)} (u(\pi_s) - h(\nu^1_s)) \d s \bigg], \; \P\text{--a.s.}
\end{aligned}\end{gather}
It is evident that $\{\tilde{V}^{\rm A}_{\tilde{\theta}}(\bC,\nu^2)\}_{\nu^2 \in \cU}$ is an upward directed family of random variables. Consequently, there exists a sequence of processes $(\nu^2_n)_{n\in\N}$, where each $\nu_n \in \cU$, such that
\begin{align*}
V^{\rm A}_{\tilde{\theta}}(\bC) = \esssup_{\nu^\smalltext{2} \in \cU} \tilde{V}^{\rm A}_{\tilde{\theta}}(\bC,\nu^2) =  \lim_{n \rightarrow \infty}{\uparrow} \tilde{V}^{\rm A}_{\tilde{\theta}}(\bC,\nu^2_n), \; \P \text{--a.s.}
\end{align*}
Hence, the monotone convergence theorem, together with \eqref{align:VATilde}, implies that
\begin{align*}
V^{\rm A}_\theta(\bC) &\geq \lim_{n \rightarrow \infty} \E^{\P^{\smalltext{{\nu^1}}}}_{\theta}\bigg[\mathrm{e}^{-r (\tilde{\theta}-\theta)} \tilde{V}^{\rm A}_{\tilde{\theta}}(\bC,\nu^2_n) \mathbf{1}_{{\{\tilde{\theta}} < \infty\}} + \int_\theta^{\tilde{\theta}} r \mathrm{e}^{-r(s-\theta)} (u(\pi_s) - h(\nu^1_s)) \d s \bigg] \\
&= \E^{\P^{\smalltext{\nu}^\tinytext{1}}}_{\theta}\bigg[\mathrm{e}^{-r (\tilde{\theta}-\theta)} \lim_{n \rightarrow \infty} \tilde{V}^{\rm A}_{\tilde{\theta}}(\bC,\nu^2_n)  \mathbf{1}_{{\{\tilde{\theta}} < \infty\}} + \int_\theta^{\tilde{\theta}} r \mathrm{e}^{-r(s-\theta)} (u(\pi_s) - h(\nu^1_s)) \d s \bigg]\\
&= \E^{\P^{\smalltext{\nu}^\tinytext{1}}}_{\theta}\bigg[\mathrm{e}^{-r (\tilde{\theta}-\theta)} V^{\rm A}_{\tilde{\theta}}(\bC) \mathbf{1}_{{\{\tilde{\theta}} < \infty\}} + \int_\theta^{\tilde{\theta}} r \mathrm{e}^{-r(s-\theta)} (u(\pi_s) - h(\nu^1_s)) \d s \bigg], \; \P\text{--a.s.}
\end{align*}
We conclude thanks to the arbitrariness of the control $\nu^1 \in \cU$, that is,
\begin{align*}
V^{\rm A}_\theta(\bC)  \geq \esssup_{\nu^\smalltext{1}\in \cU} E^{\P^{\smalltext{{\nu^1}}}}_{\theta} \bigg[\mathrm{e}^{-r (\tilde{\theta}-\theta)} V^{\rm A}_{\tilde{\theta}}(\bC) \mathbf{1}_{{\{\tilde{\theta}} < \infty\}} + \int_\theta^{\tilde{\theta}} r \mathrm{e}^{-r(s-\theta)} (u(\pi_s) - h(\nu^1_s)) \d s\bigg], \; \P\text{--a.s.}
\end{align*}
\end{proof}

The previous dynamic programming principle allows us to prove the following result, which, together with the subsequent one, will play a crucial role in the reformulation of the problem of the principal as a standard mixed control--stopping stochastic problem.
\begin{proposition}\label{proposition:Mmartingale}
For an arbitrary control $\nu \in \cU$, we define the process
\begin{align}\label{align:Mprocess}
M^{\rm A}_t(\bC,\nu) \coloneqq V^{\rm A}_{t \wedge \tau}(\bC) \mathrm{e}^{-r (t \wedge \tau)} + \int_0^{t \wedge \tau} r \mathrm{e}^{-r s} (u(\pi_s) - h(\nu_s)) \d s, \;t \geq 0,
\end{align}
where $V^{\rm A}(\bC)$ is introduced in \eqref{align:VArandomVar}, that is,
\begin{align*}
V^{\rm A}_t(\bC) = \esssup_{\nu \in \cU} \E^{\P^{\smalltext{\nu}}}_{t\wedge\tau} \bigg[\mathrm{e}^{-r (\tau-(t\wedge \tau))} u(\xi) \mathbf{1}_{\{\tau < \infty\}} + \int_{t\wedge\tau}^\tau r \mathrm{e}^{-r(s-t)} (u(\pi_s) - h(\nu_s)) \d s\bigg], \; \text{for} \; t \geq 0.
\end{align*}
Then, $M^{\rm A}(\bC,\nu)$ is an $(\F,\P^{{\nu}})$-super-martingale. Moreover, it is an $(\F,\P^{\nu^\smalltext{\star}})$-martingale for any $\nu^\star\in \cU^\star(\bC)$.
\end{proposition}

\begin{proof}
First, we prove that $M^{\rm A}(\bC,\nu)$ is an $\F$-optional process by following the same reasoning of \citeauthor*{possamai2021non} \cite[Footnote 6]{possamai2021non}. Let $\cT_{\tau}(\F)$ be the set of $\F$--stopping times $\theta$ such that $\theta \leq \tau$, $\P$--a.s. We introduce a family of random variables $(M^{\rm A}_{\theta}(\bC,\nu))_{\theta \in \cT_{\tau}(\F)}$ such that
\begin{align*}
M^{\rm A}_{\theta}(\bC,\nu) \coloneqq V^{\rm A}_{\theta}(\bC) \mathrm{e}^{-r \theta} \mathbf{1}_{\{\theta  < \infty\}}+ \int_0^{\theta } r \mathrm{e}^{-r s} (u(\pi_s) - h(\nu_s)) \d s.
\end{align*}
This family is an $(\F,\P^{{\nu}})$-super-martingale system (see \citeauthor*{dellacherie1981sur} \cite[Definition 10]{dellacherie1981sur}). In fact, the super-martingale property is a trivial consequence of \Cref{lemma:dppVA} and the $\P^{{\nu}}$-integrability of $M^{\rm A}_{\theta}(\bC,\nu)$ follows from the integrability conditions \eqref{eq:integrabilityCondition}. Let us fix $\theta \in \cT_{\tau}(\F)$, we have that there exists a constant $C>0$ such that 
\begin{align*}
\Big|V^{\rm A}_{\theta}(\bC) \mathrm{e}^{-r \theta} \mathbf{1}_{\{\theta < \infty\}}\Big| & \leq \esssup_{\tilde{\nu} \in \cU} \E^{\P^{\smalltext{\tilde{\nu}}}}_{\theta} \bigg[\big|\mathrm{e}^{-r \tau}u(\xi) \mathbf{1}_{\{\tau < \infty\}} \big|+ \int_{\theta}^{\tau} r \mathrm{e}^{-rs} |u(\pi_s) - h(\tilde{\nu}_s) | \d s\bigg]  \\
& \leq C \bigg(1+ \esssup_{\tilde{\nu} \in \cU} \E^{\P^{\smalltext{\tilde{\nu}}}}_{\theta} \bigg[\left|\mathrm{e}^{-r \tau} u(\xi) \mathbf{1}_{\{\tau < \infty\}} \right| + \int_{\theta}^{\tau} \mathrm{e}^{-rs} u(\pi_s) \d s\bigg]\bigg) \\
& = C \bigg(1+ \lim_{n \rightarrow \infty}\uparrow \E^{\P^{\smalltext{\tilde{\nu}_n}}}_{\theta} \bigg[\mathrm{e}^{-r \tau}u(\xi) \mathbf{1}_{\{\tau < \infty\}} + \int_{\theta}^{\tau} \mathrm{e}^{-rs} u(\pi_s)  \d s\bigg]\bigg), \; \P\text{--a.s.},
\end{align*}
where the last equality is a consequence of the fact that $\big\{\E^{\P^{\smalltext{\tilde{\nu}}}}_\theta \big[\mathrm{e}^{-r \tau} u(\xi) \mathbf{1}_{\{\tau < \infty\}} + \int_\theta^\tau \mathrm{e}^{-r s} u(\pi_s) \d s\big]\big\}_{\tilde{\nu} \in \cU}$ is an upward directed family of random variables. Then, the monotone convergence theorem implies that 
\begin{align*}
\E^{\P^{\smalltext{\nu}}}\Big[\big|V^{\rm A}_\theta(\bC)\mathrm{e}^{-r \theta} \mathbf{1}_{\{\theta < \infty\}}\big|\Big] &\leq C \bigg(1 +  
\E^{\P^{\smalltext{\nu}}}\bigg[\lim_{n \rightarrow \infty}\uparrow \E^{\P^{\smalltext{\tilde{\nu}}_\tinytext{n}}}_\theta \bigg[\mathrm{e}^{-r \tau} u(\xi) \mathbf{1}_{\{ \tau < \infty\}} + \int_\theta^\tau \mathrm{e}^{-r s} u(\pi_s) \d s\bigg] \bigg]\bigg) \\
&= C \bigg(1 + \lim_{n \rightarrow \infty}
\E^{\P^{\smalltext{\nu}}}\bigg[\E^{\P^{\smalltext{\tilde{\nu}}_\tinytext{n}}}_\theta \bigg[\mathrm{e}^{-r \tau} u(\xi) \mathbf{1}_{\{\tau < \infty\}} + \int_\theta^\tau \mathrm{e}^{-r s} u(\pi_s)\d s\bigg] \bigg]\bigg) \\
&= C \bigg(1 + \lim_{n \rightarrow \infty}
\E^{\P^{\smalltext{\nu}}}\bigg[\E^{\P}_\theta \bigg[ \frac{M^{\smalltext{\tilde{\nu}}_\tinytext{n}}_{\tau}}{M^{\smalltext{\tilde{\nu}}_\tinytext{n}}_{\theta}} \bigg(\mathrm{e}^{-r \tau} u(\xi) \mathbf{1}_{\{\tau < \infty\}} + \int_\theta^\tau \mathrm{e}^{-r s} u(\pi_s) \d s\bigg)\bigg] \bigg]\bigg) \\
&= C \bigg(1 + \lim_{n \rightarrow \infty}
\E^{\P^{\smalltext{\nu}}}\bigg[\E^{\P^{\smalltext{\bar{\nu}}_\tinytext{n}}}_\theta \bigg[\mathrm{e}^{-r \tau} u(\xi) \mathbf{1}_{\{\tau < \infty\}} + \int_\theta^\tau \mathrm{e}^{-r s} u(\pi_s) \d s\bigg] \bigg]\bigg),\; \P\text{--a.s.},
\end{align*}
where $(\bar{\nu}_n)_s \coloneqq \nu_s \; \mathbf{1}_{[0,\theta]}(s) + (\nu_n)_s \; \mathbf{1}_{(\theta,\tau]}(s)$, for $s \geq 0$. By definition, it holds that $\P^{{\nu}}$ and $\P^{\bar{\nu}_\smalltext{n}}$ coincide on $\cF_\theta$, and thus
\begin{align*}
\E^{\P^{\smalltext{\nu}}}\Big[\big|V^{\rm A}_\theta(\bC)\mathrm{e}^{-r \theta} \mathbf{1}_{\{\theta < \infty\}} \big|\Big] 
\leq C \bigg(1 + \sup_{\nu \in \cU} \E^{\P^{\smalltext{\nu}}} \bigg[\mathrm{e}^{-r \tau} u(\xi) \mathbf{1}_{\{\tau < \infty\}} + \int_\theta^\tau \mathrm{e}^{-r s} u(\pi_s) \d s\bigg]\bigg) <\infty.
\end{align*}
We can conclude that 
\begin{align*}
\E^{\P^{\smalltext{\nu}}}\Big[\big|M^{\rm A}_{\theta}(\bC,\nu)\big|\Big] &\leq C \bigg(1+ \E^{\P^{\smalltext{\nu}}}\bigg[\big|V^{\rm A}_{\theta }(\bC) \mathrm{e}^{-r \theta} \mathbf{1}_{\{\theta < \infty\}} \big| + \int_0^{\theta} r \mathrm{e}^{-r s} u(\pi_s) \d s\bigg]  \bigg)< \infty.
\end{align*}
Consequently, \cite[Theorem 15]{dellacherie1981sur} implies that $(M^{\rm A}_{\theta}(\bC,\nu))_{\theta \in \cT_{\tau}(\F)}$ can be aggregated into an $\F$-optional process given by \eqref{align:Mprocess}. Additionally, we can immediately say that the process $M^{\rm A}(\bC,\nu)$ is an $(\F,\P^{{\nu}})$-super-martingale. Moreover, let us fix some $\nu^\star\in \cU^*(\bC)$ and an $\F$--stopping time $\theta$. It holds that
\begin{align*}
V^{\rm A}(\bC)  = V^{\rm A}_0(\bC) &= M^{\rm A}_0(\bC,\nu^\star) \\
&\geq \E^{\P^{\nu^\star}}\Big[M^A_{\theta}(\bC,\nu^\star)\Big] \\
&\geq \E^{\P^{\nu^\star}}\Big[M^A_{\tau}(\bC,\nu^\star)\Big] = \E^{\P^{\nu^\star}}\bigg[V^{\rm A}_{\tau}(\bC) \mathrm{e}^{-r \tau} \mathbf{1}_{\{\tau < \infty\}} + \int_0^{\tau} r \mathrm{e}^{-r s} (u(\pi_s) - h(\nu^\star_s)) \d s\bigg] \geq J^A(\bC,\nu^\star) = V^{\rm A}(\bC).
\end{align*}
We can conclude that $M^{\rm A}(\bC,\nu^\star)$ is an $(\F,\P^{\nu^\star})$-martingale (see for instance \cite[Lemma I.1.44]{jacod2003limit}).
\end{proof}

\begin{lemma}
Let us fix $\kappa \in (0, r)$. The family $\left(\mathrm{e}^{-\kappa \theta} V^{\rm A}_{\theta}(\bC) \mathbf{1}_{\{\theta < \infty\}} \right)_{\theta \in \cT_{\tau}(\F)}$ is $\P^{{\nu}}$-uniformly integrable for any $\nu \in \cU$. 
\end{lemma}
\begin{proof}
Let us fix $q > 2 \vee \gamma$. For any $\nu \in \cU$ and $\theta \in \cT_{\tau}(\F)$, there exists some $C>0$ such that the following holds:
\begin{align*}
\bigg|\int_\theta^\tau r \mathrm{e}^{-\kappa s} \mathrm{e}^{(s-\theta) (\kappa-r)} |u(\pi_s) - h({\nu}_s)| \d s\bigg|^q &\leq C \bigg(\int_\theta^\tau \mathrm{e}^{(s-\theta) (\kappa-r) \frac{q}{q-1}} \d s\bigg)^{q-1} \int_\theta^\tau \mathrm{e}^{-\kappa q s}  |u(\pi_s) - h({\nu}_s)|^q \d s \\
&= C \bigg(\frac{q-1}{q(r-\kappa)} \Big(1-\mathrm{e}^{(\kappa-r)(\tau-\theta)\frac{q}{q-1}}\Big) \bigg)^{q-1} \int_\theta^\tau \mathrm{e}^{-\kappa q s}  |u(\pi_s) - h({\nu}_s)|^q \d s\\
&\leq C \int_\theta^\tau \mathrm{e}^{-\kappa q s}  |u(\pi_s) - h({\nu}_s)|^q \d s ,\; \P\text{--a.s.}
\end{align*}
Here, the first inequality is a direct consequence of H\"older's inequality, while the last one follows from the fact that $\kappa< r$. Then, we have that 
\begin{align*}
\big|\mathrm{e}^{-\kappa \theta} V^{\rm A}_{\theta}(\bC) \mathbf{1}_{\{\theta < \infty\}} \big|^q &= \bigg|\esssup_{{\nu} \in \cU} \E^{\P^{\smalltext{{\nu}}}}_\theta \bigg[\mathrm{e}^{-\kappa\tau} \mathrm{e}^{(\tau-\theta) (\kappa-r)} u(\xi) \mathbf{1}_{\{\tau < \infty\}} + \int_\theta^\tau r \mathrm{e}^{-\kappa s} \mathrm{e}^{(s-\theta) (\kappa-r)} (u(\pi_s) - h({\nu}_s)) \d s \bigg]\bigg|^q \\
&\leq  C \esssup_{{\nu} \in \cU} \E^{\P^{\smalltext{\tilde{\nu}}}}_\theta \bigg[\mathrm{e}^{-\kappa q\tau} u(\xi)^q \mathbf{1}_{\{\tau < \infty\}} + \int_\theta^{\tau} r \mathrm{e}^{-\kappa q s} |u(\pi_s) - h({\nu}_s)|^q \d s \bigg] ,
\end{align*}
where we have once more used the fact that $\kappa < r$. Hence, we can follow the same computations done in the proof of \Cref{proposition:Mmartingale} to show that
\begin{align*}
\E^{\P^{\smalltext{\nu}}}\left[\left|\mathrm{e}^{-\kappa \theta} V^{\rm A}_{\theta}(\bC) \mathbf{1}_{\{\theta < \infty\}} \right|^q\right] \leq C \left( 1+\E^{\P^{\smalltext{\nu}}}\left[\esssup_{{\nu} \in \cU} \E^{\P^{\smalltext{\tilde{\nu}}}}_\theta \left[\mathrm{e}^{-\kappa q \tau} u(\xi)^q \mathbf{1}_{\{\tau < \infty\}} + \int_\theta^{\tau}  \mathrm{e}^{-\kappa q s} u(\pi_s)^q  \d s \right]\right] \right) < \infty
\end{align*}
because of the integrability condition \eqref{eq:integrabilityCondition}. The proof follows by applying de la Vall\'ee Poussin's criterion. 
\end{proof}

\subsection{Restricted class of contracts}
In this section, we prove the result stated in \Cref{theorem:reductionSecondBest} which shows that any contract $\bC$ can be represented as $\big(\tau, \pi, u^{(-1)}\big(Y^{Y^{\smalltext{\rm A}}_\smalltext{0},Z^{\smalltext{\rm A}},U^{\smalltext{\rm A}},\pi}_\tau\big)\big)$, where the process $Y^{Y^{\smalltext{\rm A}}_\smalltext{0},Z^{\smalltext{\rm A}},U^{\smalltext{\rm A}},\pi}$ is introduced in \eqref{align:defYA}. Additionally, under contracts of this form, the value function of the agent coincides with the process $Y^{Y^{\smalltext{\rm A}}_\smalltext{0},Z^{\smalltext{\rm A}},U^{\smalltext{\rm A}},\pi}$ and the corresponding optimal efforts are identified as the maximisers of the Hamiltonian $H^{\rm A}$.

\begin{proof}[Proof of \Cref{theorem:reductionSecondBest}]
We prove the statement by adapting the same argument as in \cite[Theorem 4.2]{lin2022random}. In order to prove the inequality $\bar{V}^{\rm P} \geq \sup_{Y^{\smalltext{\rm A}}_\smalltext{0} \geq u(R)} \bar{V}^{\rm P}(Y^{\rm A}_0)$, we first fix $Y^{\rm A}_0 \geq u(R)$ and some processes $(\tau, \pi, Z^{\rm A}, U^{\rm A}) \in \cV^{\smalltext{\rm S}}(Y^{\rm A}_0)$. Note that we have assumed implicitly that the set $\cV^{\smalltext{\rm S}}(Y^{\rm A}_0)$ is non-empty. However, if this assumption does not hold, the aforementioned inequality is trivially satisfied due to our convention that $\sup \varnothing = -\infty$. Next, we introduce the following contract:
\begin{align*}
\bC = \Big(\tau, \pi, \xi^{Y^{\smalltext{\rm A}}_\smalltext{0},Z^{\smalltext{\rm A}},U^{\smalltext{\rm A}},\tau,\pi}\Big) \coloneqq \Big(\tau, \pi, u^{(-1)}\Big(Y^{Y^{\smalltext{\rm A}}_\smalltext{0},Z^{\smalltext{\rm A}},U^{\smalltext{\rm A}},\pi}_\tau\Big)\Big).
\end{align*}
By its definition, it is clear that $\xi^{Y^{\smalltext{\rm A}}_\smalltext{0},Z^{\smalltext{\rm A}},U^{\smalltext{\rm A}},\tau,\pi}$ is a non-negative, $\F_{\tau}$-measurable random variable satisfying \eqref{eq:integrabilityCondition}. Furthermore, when we introduce $\tau_n \coloneqq \tau \wedge n$, $n \in \N$, for an arbitrary control $\nu \in \cU$, it holds that 
\begin{align*}
J^{\rm A}(\bC,\nu) &=  \E^{\P^{\smalltext{\nu}}} \bigg[\mathrm{e}^{-r \tau} Y^{Y^{\smalltext{\rm A}}_\smalltext{0},Z^{\smalltext{\rm A}},U^{\smalltext{\rm A}},\pi}_{\tau} \mathbf{1}_{\{\tau < \infty\}} + \int_0^{\tau} r \mathrm{e}^{-r s} (u(\pi_s) - h(\nu_s)) \d s\bigg] \\
&=  \E^{\P^{\smalltext{\nu}}} \bigg[\lim_{n \rightarrow \infty} \bigg(\mathrm{e}^{-r \tau_\smalltext{n}} Y^{Y^{\smalltext{\rm A}}_\smalltext{0},Z^{\smalltext{\rm A}},U^{\smalltext{\rm A}},\pi}_{\tau_\smalltext{n}}  + \int_0^{\tau_\smalltext{n}} r \mathrm{e}^{-r s} (u(\pi_s) - h(\nu_s)) \d s\bigg)\bigg] \\
&= \lim_{n \rightarrow \infty}  \E^{\P^{\smalltext{\nu}}} \bigg[\mathrm{e}^{-r \tau_\smalltext{n}} Y^{Y^{\smalltext{\rm A}}_\smalltext{0},Z^{\smalltext{\rm A}},U^{\smalltext{\rm A}},\pi}_{\tau_n}  + \int_0^{\tau_\smalltext{n}} r \mathrm{e}^{-r s} (u(\pi_s) - h(\nu_s)) \d s \bigg] \eqqcolon \lim_{n \rightarrow \infty} J_n^{\rm A}(\bC,\nu)
\end{align*}
by dominated convergence theorem as $(\tau, \pi, Z^{\rm A}, U^{\rm A}) \in \cV^{\smalltext{\rm S}}(Y^{\rm A}_0)$. A direct application of It\^o's formula implies that
\begin{align*}
&J_n^{\rm A}(\bC,\nu) \\
&= \E^{\P^{\smalltext{\nu}}} \bigg[Y^{\rm A}_0 - \int_0^{\tau_\smalltext{n}} r \mathrm{e}^{-r s} Y^{Y^{\smalltext{\rm A}}_\smalltext{0},Z^{\smalltext{\rm A}},U^{\smalltext{\rm A}},\pi}_s \d s + \int_0^{\tau_\smalltext{n}} \mathrm{e}^{-r s} \d Y^{Y^{\smalltext{\rm A}}_\smalltext{0},Z^{\smalltext{\rm A}},U^{\smalltext{\rm A}},\pi}_s + \int_0^{\tau_\smalltext{n}} r \mathrm{e}^{-r s} (u(\pi_s) - h(\nu_s)) \d s \bigg] \\
&= \E^{\P^{\smalltext{\nu}}} \bigg[Y^{\rm A}_0 - \int_0^{\tau_\smalltext{n}} r \mathrm{e}^{-r s} \big(h(\nu_s) + H^{\rm A}(Z^{\rm A}_s,U^{\rm A}_s)\big) \d s + r \sigma \int_0^{\tau_n} \mathrm{e}^{-r s} Z^{\rm A}_s \d W_s + r \int_0^{\tau_\smalltext{n}} \int_{\R} \mathrm{e}^{-r s} U^{\rm A}_s(\ell) \tilde{\mu}^{{J}}(\d s, \d \ell) \bigg] \\
&= \E^{\P^{\smalltext{\nu}}} \bigg[Y^{\rm A}_0 + \int_0^{\tau_\smalltext{n}} r \mathrm{e}^{-r s} \big(h^{\rm A}(Z^{\rm A}_s,U^{\rm A}_s,\alpha_s,\beta_s) - H^{\rm A}(Z^{\rm A}_s,U^{\rm A}_s)\big) \d s + r \sigma \int_0^{\tau_\smalltext{n}} \mathrm{e}^{-r s} Z^{\rm A}_s \d W^{\nu}_s + r \int_0^{\tau_\smalltext{n}} \int_{\R} \mathrm{e}^{-r s} U^{\rm A}_s(\ell) \tilde{\mu}^{{J}^{\smalltext{\nu}}}(\d s, \d \ell) \bigg].
\end{align*}

This expression, along with \cite[Lemma I.1.44]{jacod2003limit}, the integrability conditions in \eqref{eq:integrabilityCondition} and the dominated convergence theorem, implies the following:
\begin{align}\label{align:aAgentNewCriterion}
J^{\rm A}(\bC,\nu) = \lim_{n \rightarrow \infty} J^{\rm A}_n(\bC,\nu) =  Y^{\rm A}_0 + \E^{\P^{\smalltext{\nu}}} \bigg[\int_0^{\tau} r \mathrm{e}^{-r s} \left(h^{\rm A}(Z^{\rm A}_s,U^{\rm A}_s) - H^{\rm A}(Z^{\rm A}_s,U^{\rm A}_s)\right) \d s \bigg].
\end{align}
Hence, by definition of $H^{\rm A}$ we can conclude that $J^{\rm A}(\bC,\nu) \leq Y^{\rm A}_0$ for any $\nu \in \cU$, and the equality $J^{\rm A}(\bC,\nu^\star) = Y^{\rm A}_0$ in \eqref{align:aAgentNewCriterion} holds if and only if $\nu^\star$ is the maximiser of the Hamiltonian $H^{\rm A}$. In other words:
\begin{align*}
\nu^\star_t \in \argmax_{(a,b) \in U} \bigg\{a Z^{\rm A}_t - \frac{m-b}{m} \int_{\R} U^{\rm A}_t(\ell) \Phi(\d \ell) - h(a,b)\bigg\}, \; \d t \otimes \d \P\text{--a.e.}\; \text{on} \; \llbracket 0, \tau \rrbracket.
\end{align*}

We notice that the set $\argmax_{(a,b)\in U} h^{\rm A}(Z^{\rm A}_t,U^{\rm A}_t,a,b)$ is not empty due to the continuity of the function $h$ and the compactness of the set $U$. Moreover, by applying a classical measurable selection argument (see for instance \citeauthor*{schal1974selection} \cite[Theorem 3]{schal1974selection}), we can deduce that there exists a Borel-measurable map $v$ such that $v(Z^{\rm A}_{t\wedge\tau},U^{\rm A}_{t\wedge\tau}) = \nu^\star_{t\wedge\tau}$, $\d t \otimes \d \P$--a.e., for any $t \geq 0$. We also have that
\begin{align*}
J^{\rm A}\big(\bC,v\big(Z^{\rm A},U^{\rm A}\big)\big) = J^{\rm A}(\bC,\nu^\star) = Y^{\rm A}_0 = V^{\rm A}(\bC).
\end{align*}
Here, we slightly abuse notations by considering $v\big(Z^{\rm A},U^{\rm A}\big)$ as the process $(v(Z^{\rm A}_{t\wedge\tau},U^{\rm A}_{t\wedge\tau}))_{t \geq 0}$. Then, since the constructed contract $\bC \in \mathfrak{C}_R$, it immediately follows that $\bar{V}^{\rm P} \geq \sup_{Y^{\smalltext{\rm A}}_\smalltext{0} \geq u(R)} \bar{V}^{\rm P}(Y^{\rm A}_0)$.

\medskip
We turn to the reverse inequality. To begin, we fix a contract $\bC = (\tau,\pi,\xi) \in \mathfrak{C}_R$ and notice that $\bar{J}^{\rm P}(\bC, \nu) = -\infty$ for any $\nu \notin \cU^{{\star}}(\bC)$. Thus, without loss of generality, we can suppose that the contract $\bC$ is such that $\cU^{\star}(\bC) \neq \emptyset$ and consider a control ${\nu}^\star \in \cU^\star(\bC)$. Consequently, \Cref{proposition:Mmartingale} implies that the process $M^{\rm A}(\bC,{\nu}^\star)$ defined in \eqref{align:Mprocess} is an $(\F,\P^{\nu^\star})$-martingale. This, in turn, leads to the application of the martingale representation property, implying the existence of $(\mathrm{e}^{-r (t \wedge \tau)} Z^{\rm A}_{t \wedge \tau})_{t \geq 0} \in \L^2_{\rm loc}(W^{\nu^\star},\F,\P)$ and $(\mathrm{e}^{-r (t \wedge \tau)} U^{\rm A}_{t \wedge \tau}(\ell))_{t,\ell \geq 0} \in G_{\rm loc}(\mu,\F,\P)$ such that
\begin{align}\label{align:martRepAgent}
M^{\rm A}_t(\bC,\nu^\star) = M^{\rm A}_0(\bC,\nu^\star) + \int_0^{t \wedge \tau} r \sigma \mathrm{e}^{-r s} {Z}^{\rm A}_s \d W^{\nu^\star}_s + \int_0^{t \wedge \tau} \int_{\R} r \mathrm{e}^{-r s} {U}^{\rm A}_s(\ell) \tilde{\mu}^{{J}^{\smalltext{\nu^\star}}}(\d s, \d \ell), \; \P\text{--a.s.}, \; \text{for} \; t \geq 0.
\end{align}
Here, $\tilde{\mu}^{{J}^{\smalltext{\nu}^\tinytext{\star}}}(\d s, \d \ell)$ is the $(\F,\P^{\nu^\star})$-compensated random measure $\mu^{J}(\d s, \d \ell) - \mu^{J^{\smalltext{\nu}^\tinytext{\star}},p}(\d s, \d \ell)$. Applying Itô's formula directly, we observe that 
\begin{align*}
V^{\rm A}_{t \wedge \tau} (\bC) =  V^{\rm A}_0(\bC) + r \int_0^{t \wedge \tau} (V^{\rm A}_s(\bC) - u(\pi_s) + h(\nu^\star_s)) \d s + \int_0^{t \wedge \tau} r \sigma Z^{\rm A}_s \d W^{\nu^\star}_s + \int_0^{t \wedge \tau} \int_{\R} r U^A_s(\ell) \tilde{\mu}^{{J}^{\smalltext{\nu}^\tinytext{\star}}}(\d s, \d \ell), \; \P\text{--a.s.}
\end{align*}
Then, by defining the $\R$-valued process $Y^{\rm A} \coloneqq V^{\rm A}(\bC)$, we obtain that $(Y^{\rm A}, Z^{\rm A}, U^{\rm A})$ is a solution to the following BSDE:
\begin{align}\label{align:bsdeYA}
Y^{\rm A}_{t \wedge \tau} = u(\xi) - r \int_{t \wedge \tau}^{\tau} (Y^{\rm A}_s - u(\pi_s) + h(\nu^\star_s)) \d s - \int_{t \wedge \tau}^{\tau} r \sigma Z^{\rm A}_s \d W^{\nu^\smalltext{\star}}_s - \int_{t \wedge \tau}^{\tau} \int_{\R} r U^{\rm A}_s(\ell) \tilde{\mu}^{{J}^{\smalltext{\nu}^\tinytext{\star}}}(\d s, \d \ell), \; \P\text{--a.s.}
\end{align}

Now, let us consider an arbitrary control $\nu = (\alpha,\beta) \in \cU$ instead of restricting it to $\cU^\star(\bC)$. By definition, the $(\F,\P^{\nu})$--super-martingale $M^{\rm A}(\bC,\nu)$ defined in \eqref{align:Mprocess} can be rewritten as 
\begin{align*}
M^{\rm A}_t(\bC,\nu) &= M^{\rm A}_t(\bC,\nu^\star) + \int_0^{t \wedge \tau} r \mathrm{e}^{-r s} (h(\nu^\star_s) - h(\nu_s)) \d s \\
&= M^{\rm A}_0(\bC,\nu) + \int_0^{t \wedge \tau} r \mathrm{e}^{-r s} \bigg(h(\nu^\star_s) - h(\nu_s) + Z^{\rm A}_s (\alpha_s - \alpha^\star_s) + (\beta_s - \beta^\star_s) \int_{\R} \frac{U^{\rm A}_s(\ell)}{m} \Phi(\d \ell)\bigg) \d s \\
&\quad+ \int_0^{t \wedge \tau} r \sigma \mathrm{e}^{-r s} Z^{\rm A}_s \d W^{\nu}_s + \int_0^{t \wedge \tau} \int_{\R} r \mathrm{e}^{-r s} U^{\rm A}_s(\ell) \tilde{\mu}^{{J}^{\smalltext{\nu}}}(\d s, \d \ell)
\; \text{for any} \; t \geq 0, \; \P\text{\rm --a.s.},
\end{align*}
which follows from \eqref{align:martRepAgent}. Since $M^{\rm A}(\bC,\nu)$ is an $(\F,\P^{{\nu}})$-super-martingale, we deduce that
\begin{align*}
\nu^\star_t \in \argmax_{(a,b) \in U} \bigg\{a Z^{\rm A}_t -  \frac{m-b}{m} \int_{\R} U^{\rm A}_t(\ell) \Phi(\d \ell) - h(a,b)\bigg\}, \; \d t \otimes \d \P\text{--a.e.} \; \text{on} \; \llbracket 0, \tau \rrbracket,
\end{align*}
and, as in the first part of the proof, we can introduce an optimal feedback control $v(Z^{\rm A}_{t\wedge\tau},U^{\rm A}_{t\wedge\tau}) =\nu^\star_{t\wedge\tau}$, $\d t \otimes \d \P$--a.e., for any $t \geq 0$, for some Borel-measurable map $v$. We conclude that $v^\star \in \cU^\star(Z^{\rm A},U^{\rm A})$. To complete the proof, it suffices to verify that the quadruple $(\tau, \pi, Z^{\rm A}, U^{\rm A}) \in \cV^{\smalltext{\rm S}}(Y^{\rm A}_0)$. To do this, let us fix $r^\prime \in (0,r)$ and $q > 2 \vee \gamma$ associated to the integrability conditions \eqref{eq:integrabilityCondition} of the contract $\bC$, and introduce $\kappa \in (0, r^\prime]$. For any $t$, $t^\prime \geq 0$ such that $t \leq t^\prime$, a direct application of It\^o--Tanaka--Meyer formula to $\mathrm{e}^{-\kappa \cdot} |Y_\cdot^{\rm A}|$ on $[t \wedge \tau, t^\prime \wedge \tau]$ shows that 
\begin{align*}
\mathrm{e}^{-\kappa (t^\prime \wedge \tau)} |Y_{t^\prime \wedge \tau}^{\rm A}| - \mathrm{e}^{-\kappa (t \wedge \tau)} |Y_{t \wedge \tau}^{\rm A}| &= - \int_{t \wedge \tau}^{t^\prime \wedge \tau} \kappa \mathrm{e}^{-\nu s} |Y^{\rm A}_s| \d s + \int_{t \wedge \tau}^{t^\prime \wedge \tau} r \mathrm{e}^{-\kappa s} \sgn(Y^{\rm A}_s) (Y^{\rm A}_s - u(\pi_s) + h(\nu^\star_s)) \d s \\
&\quad+ \int_{t \wedge \tau}^{t^\prime \wedge \tau} r \sigma \mathrm{e}^{-\kappa s} \sgn(Y^{\rm A}_{s-}) Z^{\rm A}_s \d W^{\nu^\smalltext{\star}}_s + \int_{t \wedge \tau}^{t^\prime \wedge \tau} \int_{\R} r \mathrm{e}^{-\kappa s} \sgn(Y^{\rm A}_{s-}) U^{\rm A}_s(\ell) \tilde{\mu}^{{J}^{\smalltext{\nu}^\tinytext{\star}}}(\d s, \d \ell) \\
&\quad+ \int_{t \wedge \tau}^{t^\prime \wedge \tau} \int_{\R} \mathrm{e}^{-\kappa s} (|Y^{\rm A}_{s-} + r U^{\rm A}_s(\ell)| - |Y^{\rm A}_{s-}| - r \sgn(Y^{\rm A}_{s-}) U^{\rm A}_s(\ell)) \mu^{J}(\d s, \d \ell), \; \P\text{--a.s.}
\end{align*}
Since the last term in the above expression is non-negative, and given that $\kappa < r$, we can deduce that
\begin{align*}
\mathrm{e}^{-\kappa (t^\prime \wedge \tau)} |Y_{t^\prime \wedge \tau}^{\rm A}| - \mathrm{e}^{-\kappa (t \wedge \tau)} |Y_{t \wedge \tau}^{\rm A}| \geq&  -\int_{t \wedge \tau}^{t^\prime \wedge \tau} r \mathrm{e}^{-\kappa s} \sgn(Y^{\rm A}_s) (u(\pi_s) - h(\nu^\star_s)) \d s \\
&+ \int_{t \wedge \tau}^{t^\prime \wedge \tau} r \sigma \mathrm{e}^{-\kappa s} \sgn(Y^{\rm A}_{s-}) Z^{\rm A}_s \d W^{\nu^\smalltext{\star}}_s + \int_{t \wedge \tau}^{t^\prime \wedge \tau} \int_{\R} r \mathrm{e}^{-\kappa s} \sgn(Y^{\rm A}_{s-}) U^{\rm A}_s(\ell) \tilde{\mu}^{{J}^{\smalltext{\nu}^\tinytext{\star}}}(\d s, \d \ell), \; \P\text{--a.s.}
\end{align*}
We now substitute $t^\prime$ with $\tau_n \coloneqq n \wedge \tau$, where $n \in \N$ such that $n \geq t$. By applying the monotone convergence theorem, the dominated convergence theorem, and considering the continuity of the cost function $h$ defined on the compact set $U$, we can say that there exists some $C>0$ such that the following inequality holds:
\begin{align*}
\mathrm{e}^{-\kappa (t \wedge \tau)} |Y_{t \wedge \tau}^{\rm A}| &\leq C \bigg( 1 + \lim_{n \rightarrow \infty} \E_{t \wedge \tau}^{\P^{\smalltext{\nu}^\tinytext{\star}}}\bigg[\mathrm{e}^{-\kappa \tau_\smalltext{n}} |Y_{\tau_\smalltext{n}}^{\rm A}| + \int_0^{\tau_\smalltext{n}} \mathrm{e}^{-\kappa s} u(\pi_s) \d s \bigg] \bigg)\\
&= C  \bigg( 1 + \E_{t\wedge\tau}^{\P^{\smalltext{\nu}^\tinytext{\star}}}\bigg[\mathrm{e}^{-\kappa \tau} u(\xi) \mathbf{1}_{\{\tau < \infty\}}  + \int_0^{\tau} \mathrm{e}^{-\kappa s} u(\pi_s) \d s \bigg]\bigg), \; \P\text{--a.s.}
\end{align*}
Then, for any $p \in (0, q)$, we deduce that 
\begin{align}\label{align:intsupYproof}
\begin{split}
\sup_{\nu \in \cU} \E^{\P^{\smalltext{\nu}}}\bigg[\sup_{t\geq 0} \big|\mathrm{e}^{-\kappa (t\wedge\tau)} Y^{\rm A}_{t\wedge\tau}\big|^p\bigg] &\leq  C \bigg( 1+ \sup_{\nu \in \cU} \E^{\P^{\smalltext{\nu}}}\bigg[\sup_{t \geq 0} \bigg(\E_{t \wedge \tau}^{\P^{\smalltext{\nu}^\tinytext{\star}}}\bigg[\mathrm{e}^{-\kappa \tau} u(\xi) \mathbf{1}_{\{\tau < \infty\}}  + \int_0^{\tau} \mathrm{e}^{-\kappa s} u(\pi_s) \d s \bigg]\bigg)^p\bigg]\bigg) \\
&\leq C \bigg(1+ \frac{q}{q-p} \bigg(\sup_{\nu \in \cU} \E^{\P^{\smalltext{\nu}}}\bigg[\bigg(\E_{\tau}^{\P^{\smalltext{\nu}^\tinytext{\star}}}\bigg[\mathrm{e}^{-\kappa \tau} u(\xi) \mathbf{1}_{\{\tau < \infty\}}  + \int_0^{\tau} \mathrm{e}^{-\kappa s} u(\pi_s) \d s \bigg]\bigg)^q \bigg]\bigg)^{\frac{p}{q}}\bigg),
\end{split}
\end{align}
where the last inequality follows from \citeauthor*{lin2020second} \cite[Lemma 4.1]{lin2020second}. However, for any $\nu \in \cU$, Jensen's inequality implies that
\begin{align*}
\E^{\P^{\smalltext{\nu}}}\bigg[\bigg(\E_{\tau}^{\P^{\smalltext{\nu}^\tinytext{\star}}}\bigg[\mathrm{e}^{-\kappa \tau} u(\xi) \mathbf{1}_{\{\tau < \infty\}}  + \int_0^{\tau} \mathrm{e}^{-\kappa s} u(\pi_s) \d s \bigg]\bigg)^q \bigg] &\leq \E^{\P^{\smalltext{\nu}}}\bigg[\E_{\tau}^{\P^{\smalltext{\nu}^\tinytext{\star}}}\bigg[\bigg|\mathrm{e}^{-\kappa \tau} u(\xi) \mathbf{1}_{\{\tau < \infty\}}  + \int_0^{\tau} \mathrm{e}^{-\kappa s} u(\pi_s) \d s \bigg|^q\bigg] \bigg] \\
&= \E^{\P^{\smalltext{\nu}}}\bigg[\bigg|\mathrm{e}^{-\kappa \tau} u(\xi) \mathbf{1}_{\{\tau < \infty\}}  + \int_0^{\tau} \mathrm{e}^{-\kappa s} u(\pi_s) \d s \bigg|^q\bigg].
\end{align*}
Consequently, the integrability condition in \eqref{align:intCondYA} immediately follows from \Cref{align:intsupYproof} along with the last observation. Then, to derive the integrability conditions in \eqref{align:intCondZAUA}, our goal is to apply the estimates in \cite[Proposition 2]{kruse2015bsdes} to the BSDE in \eqref{align:bsdeYA} on the finite time horizon $[0, \tau_n]$, where $\tau_n \coloneqq n \wedge \tau$, for $n \in \N$. However, to do so, we need to rewrite it under the measure $\P^\nu$, for some $\nu = (\alpha,\beta) \in \cU$. This leads us to consider the following expression:
\begin{align}\label{align:BSDEYproofSecond}
\begin{split}
Y^{\rm A}_{t \wedge \tau_{\smalltext{n}}} &= Y^{\rm A}_{ \tau_{\smalltext{n}}} - r \int_0^{t \wedge \tau_{\smalltext{n}}} \bigg(Y^{\rm A}_s- u(\pi_s) + h(\nu_s) + h^{\rm A}(Z^{\rm A}_s,U^{\rm A}_s,\alpha_s,\beta_s) - H^{\rm A}(Z^{\rm A}_s,U^{\rm A}_s)\bigg) \d s \\
&\quad- \int_0^{t \wedge \tau_{\smalltext{n}}} r \sigma Z^{\rm A}_s \d W^{\nu}_s - \int_0^{t \wedge \tau_{\smalltext{n}}} \int_{\R} r U^{\rm A}_s(\ell) \tilde{\mu}^{{J}^{\smalltext{\nu}}}(\d s, \d \ell), \; \P\text{--a.s.}, \; \text{for} \; t \geq 0.
\end{split}
\end{align}
Based on \cite[Proposition 2]{kruse2015bsdes}, we can introduce $\tilde{p} \in (2 \vee \gamma, q)$ and conclude that there exists some $C>0$ such that 
\begin{align*}
&\E^{\P^{\smalltext{\nu}}}\Bigg[\bigg(\int_0^{\tau_{\smalltext{n}}}  \mathrm{e}^{-2 \kappa s} |Z^{\rm A}_s|^2 \d s\bigg)^{\frac{\tilde{p}}{2}}\Bigg] + \E^{\P^{\nu}}\Bigg[\bigg(\int_0^{\tau_{\smalltext{n}}} \int_{\R} \mathrm{e}^{-2 \kappa s} |U^{\rm A}_s(\ell)|^2 \Phi(\d \ell)\d s \bigg)^{\frac{\tilde{p}}{2}}\Bigg] \\
&\leq  C \Bigg(1+ \E^{\P^{\nu}}\Bigg[\sup_{t \geq 0} \big|\mathrm{e}^{-\kappa (t\wedge\tau)} Y^{\rm A}_{t\wedge\tau}\big|^{\tilde{p}}\Bigg] + \E^{\P^{\nu}}\Bigg[\int_0^{\tau} \mathrm{e}^{-\kappa \tilde{p} s} u(\pi_s)^{\tilde{p}} \d s \Bigg] \Bigg) <  \infty,
\end{align*}
given the integrability condition \eqref{align:intCondYA}, which we have previously proved, and the one in \eqref{eq:integrabilityCondition} that the contract $\bC$ satisfies by assumption. Consequently, the monotone convergence theorem allows us to deduce \eqref{align:intCondZAUA}. In conclusion, for any contract $\bC = (\tau,\pi,\xi) \in \mathfrak{C}_R$, we can find $(Z^{\rm A},U^{\rm A}) \in \cV^{\rm A}$ such that $\xi = u^{(-1)}(Y^{\rm A}_{\tau})$, $\P$--a.s., where $Y^{\rm A}$ solves the BSDE \eqref{align:bsdeYA}, or equivalently \eqref{align:BSDEYproofSecond}, and $Y^{\rm A}_0 \geq u(R)$ by construction. Moreover, it holds that the quadruple $(\tau, \pi, Z^{\rm A}, U^{\rm A}) \in \cV^{\smalltext{\rm S}}(Y^{\rm A}_0)$. If we denote the contract $\bC^{Y^{\smalltext{\rm A}}_\smalltext{0},Z^{\smalltext{\rm A}},U^{\smalltext{\rm A}},\pi} \coloneqq \big(\tau,\pi,u^{(-1)}\big(Y^{\rm A}_{\tau}\big)\big)$, then the first part of the proof also implies that the agent's control problem can be explicitly solved given $\bC^{Y^{\smalltext{\rm A}}_\smalltext{0},Z^{\smalltext{\rm A}},U^{\smalltext{\rm A}},\pi}$, and we obtain that $V^{\rm A}(\bC^{Y^{\smalltext{\rm A}}_\smalltext{0},Z^{\smalltext{\rm A}},U^{\smalltext{\rm A}},\pi}) = Y^{\rm A}_0$. The arbitrariness of the contract $\bC$ allows us to deduce that $\bar{V}^{\rm P} \leq \sup_{Y^{\smalltext{\rm A}}_\smalltext{0} \geq u(R)} \bar{V}^{\rm P}(Y^{\rm A}_0)$. This concludes the proof. 
\end{proof}

The reduction presented in \Cref{theorem:reductionSecondBest}, which has just been demonstrated, motivates the introduction of the Hamilton--Jacobi--Bellman equation associated with the problem of the principal. This equation is presented in \eqref{align:hjbSB}, and we can note that is characterised by a non-local operator defined in \eqref{eq:operatorsecondbest}. To ensure the latter is well-defined, it is essential the first integrability condition in \eqref{equation:minPlusIntCond} is satisfied. Consequently, we find it necessary to prove the following result, which we use to show that the principal's value function is the unique viscosity solution to the corresponding Hamilton--Jacobi--Bellman equation.

\begin{lemma}\label{lemma:finiteIntegral}
Let us consider a quadruple $(\tau, \pi, Z^{\rm A}, U^{\rm A}) \in \cV^{\smalltext{\rm S}}(Y^{\rm A}_0)$ such that the set $\cU^{\star}(Z^{\rm A},U^{\rm A})$ is not empty. For some given $q > 2 \vee \gamma$, it holds that 
\begin{align*}
\int_{\R} \Big||Y^{\rm A}_{t-} + r U^{\rm A}_t(\ell)|^q - |Y^{\rm A}_{t-}|^q - q r U^{\rm A}_t(\ell)|Y^{\rm A}_{t-}|^{q-1} \Big| \Phi(\d \ell) < \infty, \; \d t \otimes \d \P\text{\rm--a.e.} \; \text{\rm on} \; \llbracket 0, \tau \rrbracket,
\end{align*}
where $(Y^{\rm A},Z^{\rm A},U^{\rm A})$ solves the {\rm BSDE} \eqref{align:bsdeYA}, or equivalently \eqref{align:BSDEYproofSecond}.
\end{lemma}

\begin{proof}
Let us fix $q \geq 2\vee \gamma$ and $\kappa \in (0, r )$. For any $n \in \N$, we define the stopping time $\tau_n$ as follows
\begin{align*}
\tau_n \coloneqq \inf \bigg\{t \geq 0:  \int_0^t\bigg( \big|\mathrm{e}^{-q \kappa s} |Y^{\rm A}_{s}|^{q-1} Z^{\rm A}_s\big|^2 + \int_{\R} \big|\mathrm{e}^{-q \kappa s} |Y^{\rm A}_{s}|^{q-1} U^{\rm A}_s(\ell)\big|^2 \Phi(\d \ell)+ \mathrm{e}^{- q\kappa s} |Y^{\rm A}_s|^{q-1} u(\pi_s)\bigg) \d s\geq n\bigg\} \wedge n \wedge \tau.
\end{align*}

Applying It\^o formula to $\mathrm{e}^{-q \kappa \cdot} |Y_\cdot^{\rm A}|^q$ on $[0, \tau_n]$, we get that
\begin{align*}
\mathrm{e}^{-q \kappa \tau_\smalltext{n}} |Y_{\tau_\smalltext{n}}^{\rm A}|^q &\geq \mathrm{e}^{-q \kappa \tau_\smalltext{n}} |Y_{\tau_\smalltext{n}}^{\rm A}|^q - |Y_0^{\rm A}|^q \\
&= - \int_0^{\tau_\smalltext{n}} q \kappa \mathrm{e}^{-q \kappa s} |Y^{\rm A}_s|^q \d s + \int_0^{\tau_\smalltext{n}} q r \mathrm{e}^{-q \kappa s} |Y^{\rm A}_s|^{q-1} (Y^{\rm A}_s - u(\pi_s) + h(\nu^\star_s)) \d s \\
&\quad+ \int_0^{\tau_\smalltext{n}} q r \sigma \mathrm{e}^{-q \kappa s} |Y^{\rm A}_{s-}|^{q-1} Z^{\rm A}_s \d W^{\nu^\smalltext{\star}}_s + \int_0^{\tau_\smalltext{n}} \int_{\R} q r \mathrm{e}^{-q \kappa s} |Y^{\rm A}_{s-}|^{q-1} U^{\rm A}_s(\ell) \tilde{\mu}^{{J}^{\smalltext{\nu}^\tinytext{\star}}}(\d s, \d u) \\
&\quad+ \frac{q(q-1)}{2}\int_0^{\tau_\smalltext{n}} r^2 \sigma^2 \mathrm{e}^{-q \kappa s} |Y^{\rm A}_{s}|^{q-2} |Z^{\rm A}_s|^2 \d s \\
&\quad+ \int_0^{\tau_\smalltext{n}} \int_{\R} \mathrm{e}^{-q \kappa s} (|Y^{\rm A}_{s-} + r U^{\rm A}_s(\ell)|^q - |Y^{\rm A}_{s-}|^q - q r U^{\rm A}_s(\ell)|Y^{\rm A}_{s-}|^{q-1}) \mu^{J}(\d s, \d \ell)\\
&\geq - \int_0^{\tau_\smalltext{n}} q r \mathrm{e}^{-q \kappa s} |Y^{\rm A}_s|^{q-1} (u(\pi_s) - h(\nu^\star_s)) \d s \\
&\quad+ \int_0^{\tau_\smalltext{n}} q r \sigma \mathrm{e}^{-q \kappa s} |Y^{\rm A}_{s-}|^{q-1} Z^{\rm A}_s \d W^{\nu^\smalltext{\star}}_s + \int_0^{\tau_\smalltext{n}} \int_{\R} q r \mathrm{e}^{-q \kappa s} |Y^{\rm A}_{s-}|^{q-1} U^{\rm A}_s(\ell) \tilde{\mu}^{{J}^{\smalltext{\nu}^\tinytext{\star}}}(\d s, \d u) \\
&\quad+ \int_0^{\tau_\smalltext{n}} \int_{\R} \mathrm{e}^{-q \kappa s} (|Y^{\rm A}_{s-} + r U^{\rm A}_s(\ell)|^q - |Y^{\rm A}_{s-}|^q - q r U^{\rm A}_s(\ell)|Y^{\rm A}_{s-}|^{q-1}) \mu^{J}(\d s, \d \ell),\; \P\text{--a.s.},
\end{align*}
given the fact that $\kappa < r$. Additionally, considering the continuity of the cost function $h$ over the compact set $U$, we deduce that
\begin{align*}
\mathrm{e}^{-q \kappa \tau_\smalltext{n}} |Y_{\tau_\smalltext{n}}^{\rm A}|^q &\geq - \int_0^{\tau_\smalltext{n}} q r |\mathrm{e}^{- \kappa s} Y^{\rm A}_s|^{q-1} \mathrm{e}^{- \kappa s} u(\pi_s) \d s \\
&\quad+ \int_0^{\tau_\smalltext{n}} q r \sigma \mathrm{e}^{-q \kappa s} |Y^{\rm A}_{s-}|^{q-1} Z^{\rm A}_s \d W^{\nu^\smalltext{\star}}_s + \int_0^{\tau_\smalltext{n}} \int_{\R} q r \mathrm{e}^{-q \kappa s} |Y^{\rm A}_{s-}|^{q-1} U^{\rm A}_s(\ell) \tilde{\mu}^{{J}^{\smalltext{\nu}^\tinytext{\star}}}(\d s, \d u) \\
&\quad+ \int_0^{\tau_\smalltext{n}} \int_{\R} \mathrm{e}^{-q \kappa s} (|Y^{\rm A}_{s-} + r U^{\rm A}_s(\ell)|^q - |Y^{\rm A}_{s-}|^q - q r U^{\rm A}_s(\ell)|Y^{\rm A}_{s-}|^{q-1}) \mu^{J}(\d s, \d \ell), \; \P\text{--a.s.}
\end{align*}
Now, we can rearrange terms to get
\begin{align*}
&\mathrm{e}^{-q \kappa \tau_\smalltext{n}} |Y_{\tau_\smalltext{n}}^{\rm A}|^q + \int_0^{\tau_\smalltext{n}} q r |\mathrm{e}^{- \kappa s} Y^{\rm A}_s|^{q-1} \mathrm{e}^{- \kappa s} u(\pi_s) \d s \\
&\geq \int_0^{\tau_\smalltext{n}} q r \sigma \mathrm{e}^{-q \kappa s} |Y^{\rm A}_{s-}|^{q-1} Z^{\rm A}_s \d W^{\nu^\smalltext{\star}}_s + \int_0^{\tau_\smalltext{n}} \int_{\R} q r \mathrm{e}^{-q \kappa s} |Y^{\rm A}_{s-}|^{q-1} U^{\rm A}_s(\ell) \tilde{\mu}^{{J}^{\smalltext{\nu}^\tinytext{\star}}}(\d s, \d u) \\
&\quad+ \int_0^{\tau_\smalltext{n}} \int_{\R} \mathrm{e}^{-q \kappa s} (|Y^{\rm A}_{s-} + r U^{\rm A}_s(\ell)|^q - |Y^{\rm A}_{s-}|^q - q r U^{\rm A}_s(\ell)|Y^{\rm A}_{s-}|^{q-1}) \mu^{J}(\d s, \d \ell),\; \P\text{--a.s.}
\end{align*}
We deduce that there exists some $C>0$ such that 
\begin{align*}
&\E^{\P^{\smalltext{\nu}^\tinytext{\star}}} \bigg[\big|\mathrm{e}^{-\kappa (t\wedge\tau)} Y^{\rm A}_{t\wedge\tau}\big|^q\bigg] + \E^{\P^{\smalltext{\nu}^\tinytext{\star}}} \bigg[\sup_{t\geq 0}\big|\mathrm{e}^{-\kappa (t\wedge\tau)} Y^{\rm A}_{t\wedge\tau}\big|^{q-1} \int_0^{\tau_\smalltext{n}} \mathrm{e}^{- \kappa s} u(\pi_s) \d s \bigg] \\
&\geq C \E^{\P^{\smalltext{\nu}^\tinytext{\star}}} \bigg[\int_0^{\tau_\smalltext{n}} \int_{\R} \mathrm{e}^{-q \kappa s} \big(|Y^{\rm A}_{s-} + r U^{\rm A}_s(\ell)|^q - |Y^{\rm A}_{s-}|^q - q r U^{\rm A}_s(\ell)|Y^{\rm A}_{s-}|^{q-1}\big) \mu^{J}(\d s, \d \ell) \bigg]. 
\end{align*}
We can apply H\"older's inequality with H\"older conjugates $q/(q-1)$ and $q$, and then Jensen's inequality to observe that
\begin{align*}
\infty &> \E^{\P^{\smalltext{\nu}^\tinytext{\star}}} \bigg[\big|\mathrm{e}^{-\kappa (t\wedge\tau)} Y^{\rm A}_{t\wedge\tau}\big|^q\bigg] + \E^{\P^{\smalltext{\nu}^\tinytext{\star}}} \bigg[\int_0^{\tau} \mathrm{e}^{- \kappa q s} u(\pi_s)^q \d s \bigg] \\
&\geq C\E^{\P^{\smalltext{\nu}^\tinytext{\star}}} \bigg[\int_0^{\tau_\smalltext{n}} \int_{\R} \mathrm{e}^{-\kappa q s} \big(|Y^{\rm A}_{s-} + r U^{\rm A}_s(\ell)|^q - |Y^{\rm A}_{s-}|^q - q r U^{\rm A}_s(\ell)|Y^{\rm A}_{s-}|^{q-1}\big) \mu^{J}(\d s, \d \ell) \bigg]. 
\end{align*}
Since the integrand is non-negative (due to the convexity of $|\cdot|^q$), the monotone convergence theorem implies that
\begin{align*}
\infty > \E^{\P^{\smalltext{\nu}^\tinytext{\star}}} \bigg[\int_0^{\tau} \int_{\R} \mathrm{e}^{-\kappa q s} \big(|Y^{\rm A}_{s-} + r U^{\rm A}_s(\ell)|^q - |Y^{\rm A}_{s-}|^q - q r U^{\rm A}_s(\ell)|Y^{\rm A}_{s-}|^{q-1}\big) \mu^{J}(\d s, \d \ell) \bigg].
\end{align*}
Then, the conclusion of the proof can be derived from \cite[Theorem II.1.8]{jacod2003limit}, along with the observation that all the measures $(\P^\nu)_{\nu \in \cU}$ are equivalent to $\P$.
\end{proof}

\section{On the reduced Stackelberg game}

\subsection{When the principal is very impatient}\label{degeneracySeconBest}
\begin{proof}[Proof of \Cref{thm:SecondtBestCompleteC}.\ref{degSB}]
We can prove the result by employing a similar argument as in \citep[Theorem 3.4]{possamai2020there}. To do so, let us fix some $y_0 >0$. Furthermore,
\begin{equation*}
z^{\smalltext{\rm A}} \in \R \; \text{such that} \; z^{\smalltext{\rm A}} \geq \max_{b \in B} \frac{\partial h(\bar{a},b)}{\partial a}, \; \text{and} \; u^{\smalltext{\rm A}} \in \cB_\R \; \text{such that} \; u^{\smalltext{\rm A}}(\ell) = u^{\smalltext{\rm A}} \leq \min_{a \in A} \frac{\partial h(a,\varepsilon_m)}{\partial b} \; \text{for any} \; \ell \in \R.
\end{equation*}
This choice evidently leads to $U^\star(z^{\smalltext{\rm A}},u^{\smalltext{\rm A}}) = \{(\bar{a},\varepsilon_m)\}$. Henceforth, we denote $\bar{u} \coloneqq (\bar{a},\varepsilon_m)$ to simplify the notation. In order to apply the reduction from \Cref{theorem:reductionSecondBest}, we introduce the parameter $\varepsilon \in (0,r \wedge 1)$. Moreover, we consider the continuous payment $\pi^\varepsilon \coloneqq u^{(-1)}(\varepsilon Y^{\frac{\smalltext{y}_\tinytext{0}}{\smalltext{\sqrt{\varepsilon}}},z,u,\pi^{\smalltext{\varepsilon}}})$, for $t \geq 0$, in the definition of the continuation utility of the agent. Hence, the process $Y^\varepsilon \coloneqq Y^{\frac{\smalltext{y}_\tinytext{0}}{\smalltext{\sqrt{\varepsilon}}},z,u,\pi^{\smalltext{\varepsilon}}}$ is described by
\begin{align*}
Y^{\varepsilon}_t = \frac{y_0}{\sqrt{\varepsilon}} + \int_0^t \Big((r-\varepsilon) Y^{\varepsilon}_s + r h(\bar{a},\varepsilon) \Big) \d s + r \sigma z^{\smalltext{\rm A}} W^{\bar{u}}_t + r u^{\smalltext{\rm A}} \tilde{N}_t^{\bar{u}} \; \text{for any}\; t \geq 0,
\end{align*} 
or equivalently by,
\begin{align*}
Y^{\varepsilon}_t = \mathrm{e}^{(r-\varepsilon)t} \frac{y_0}{\sqrt{\varepsilon}} + \frac{r}{r-\varepsilon} h(\bar{a},\varepsilon) \big(\mathrm{e}^{(r-\varepsilon)t} -1 \big) + r \sigma z^{\smalltext{\rm A}} \int_0^t \mathrm{e}^{(r-\varepsilon)(t-s)} \d W^{\bar{u}}_s + r u^{\smalltext{\rm A}} \int_0^t \mathrm{e}^{(r-\varepsilon)(t-s)} \d \tilde{N}^{\bar{u}}_s, \; \P\text{--a.s.},\; \text{for any}\; t \geq 0.
\end{align*} 
Let us introduce $T^\varepsilon_0 \coloneqq \inf\big\{ t > 0: Y^{\varepsilon}_t < 0 \big\}$, and consequently $\tau^\varepsilon \coloneqq -\log(\varepsilon)/\varepsilon \wedge T^\varepsilon_0$. Then, we can say that the quadruple $(\tau^\varepsilon, \pi^\varepsilon, z^{\rm A}, u^{\rm A}) \in \cV^{\smalltext{\rm S}}({y_0}/{\varepsilon})$. This is because the fact that $z^{\rm A}$ and $u^{\rm A}$ are deterministic, and $\tau^\varepsilon$ is bounded implies that the integrability conditions \eqref{align:intCondYA} and \eqref{align:intCondZAUA} are satisfied, given also the explicit formula for $Y^\varepsilon$.

\medskip
The proof of \Cref{theorem:reductionSecondBest} shows that $(\tau^\varepsilon, \pi^\varepsilon, z^{\rm A}, u^{\rm A})$ corresponds to a contract $\bC^\varepsilon$ that guarantees the agent a utility of $y_0/\varepsilon$, an offer which the agent accepts for sufficiently small values of $\varepsilon$. When it comes to the principal, her reward is given by the following:
\begin{equation}\label{eq:criterionCVar}
\bar{J}^{\rm P}(\bC^\varepsilon, \bar{u}) = \E^{\P^{\smalltext{\bar{u}}}} \Big[\mathrm{e}^{-\rho \tau^\smalltext{\varepsilon}} \bar{F}(Y^{\varepsilon}_{\tau^\varepsilon}) \Big] + \E^{\P^{\smalltext{\bar{u}}}}\bigg[\int_0^{\tau^\smalltext{\varepsilon}} \rho \mathrm{e}^{-\rho s} F\big(\varepsilon Y^\varepsilon_s\big) \d s\bigg] + (\bar{a} - \varepsilon_m) \Big(1-\E^{\P^{\smalltext{\bar{u}}}}\Big[\mathrm{e}^{-\rho \tau^\smalltext{\varepsilon}}\Big]\Big).
\end{equation}
In what follows, we provide some estimates for each term to show that the sum convergences to its upper bound $\bar{a} - \varepsilon_m$ for $\varepsilon$ going to zero. First, it is worth noticing that $Y^\varepsilon$ is a a spectrally negative L\'evy process. As a result, we can apply \citeauthor*{kyprianou2014fluctuations} \cite[Theorem 8.1]{kyprianou2014fluctuations} and \citeauthor*{kyprianou2013theory} \cite[Lemma 3.3]{kyprianou2013theory} to show that
\begin{equation*}
\lim_{\varepsilon \rightarrow 0^+} \P^{\bar{u}} \big[T^\varepsilon_0 < \infty\big] = \lim_{\varepsilon \rightarrow 0^+} \bigg(1- \E^{\bar{u}}\big[Y^{\varepsilon}_1\big] W\bigg(\frac{y_0}{\sqrt{\varepsilon}}\bigg)\bigg) = 0,
\end{equation*}
where $W$ is the scale function associated to $Y^{\varepsilon}$. This, together with Step 1 of the proof in \cite[Theorem 3.4]{possamai2020there}, is sufficient to conclude that the last term in \eqref{eq:criterionCVar} converges to $\bar{a} - \varepsilon_m$. Concerning the first term, Step 2 of the mentioned proof demonstrates the existence of some $C$, $C^\varepsilon > 0$ such that $\lim_{\varepsilon \rightarrow 0^+} C^\varepsilon = 0$. Additionally, it shows that
\begin{align}\label{align:secondTermBound}
\begin{split}
0 &\leq - \E^{\P^{\smalltext{\bar{u}}}} \Big[\mathrm{e}^{-\rho \tau^{\smalltext{\varepsilon}}} F\big(Y^\varepsilon_{\tau^{\smalltext{\varepsilon}}}\big)\Big] \\
&\leq C^\varepsilon + C \mathrm{e}^{\rho \frac{\log(\varepsilon)}{\varepsilon}} \bigg(\mathrm{e}^{-\gamma(r-\varepsilon) \frac{\log(\varepsilon)}{\varepsilon}} \E^{\P^{\smalltext{\bar{u}}}}\bigg[\bigg| \int_0^{-\frac{\log(\varepsilon)}{\varepsilon}} \mathrm{e}^{-(r-\varepsilon) s} \d \tilde{N}_s^{\bar{u}} \bigg|^\gamma\bigg] \bigg) + C \bigg(\E^{\P^{\smalltext{\bar{u}}}}\bigg[\mathbf{1}_{\{T^{\smalltext{\varepsilon}}_0 < \infty\}}  \bigg| \int_0^{\tau^{\smalltext{\varepsilon}}} \mathrm{e}^{-(r-\varepsilon) s} \d \tilde{N}_s^{\bar{u}} \bigg|^\gamma \bigg] \bigg).
\end{split}
\end{align}
By applying \citeauthor*{bouchard2018unified} \cite[Remark 2.1]{bouchard2018unified}, we can bound the first term in \eqref{align:secondTermBound} as follows:
\begin{align*}
&\E^{\P^{\smalltext{\bar{u}}}}\bigg[\bigg|\int_0^{-\frac{\log(\varepsilon)}{\varepsilon}} \mathrm{e}^{-(r-\varepsilon)s} \d \tilde{N}_s^{\bar{u}}\bigg|^\gamma\bigg] \\
&\leq 2^{\gamma - 1} \bigg(\E^{\P^{\smalltext{\bar{u}}}}\bigg[\bigg|\int_0^{-\frac{\log(\varepsilon)}{\varepsilon}} \mathrm{e}^{-(r-\varepsilon)s} \d N_s \bigg|^\gamma\bigg]  + C \bigg|\int_0^{-\frac{\log(\varepsilon)}{\varepsilon}} \mathrm{e}^{-(r-\varepsilon)s} \d s \bigg|^\gamma\bigg) \\
&\leq C \bigg(\sum_{i \in \N} \E^{\P^{\smalltext{\bar{u}}}}\bigg[\bigg|\int_0^{-\frac{\log(\varepsilon)}{\varepsilon}} \mathrm{e}^{-(r-\varepsilon)s} \d N_s \bigg|^\gamma \bigg| N_{-{\log(\varepsilon)}/{\varepsilon}} = i\bigg] \P^{\bar{u}}\Big(N_{-{\log(\varepsilon)}/{\varepsilon}} = i\Big)  + \bigg(\frac{1- \mathrm{e}^{(r-\varepsilon)\frac{\log(\varepsilon)}{\varepsilon}}}{(r-\varepsilon)}\bigg)^\gamma.
\end{align*}
In this expression, the constant $C > 0$ may vary. Furthermore, we introduce the jump times $(T^{\smalltext{N}}_k)_{k \in \N}$ associated to the Poisson process $N$ to see that
\begin{align*}
\E^{\P^{\smalltext{\bar{u}}}}\bigg[\bigg|\int_0^{-\frac{\log(\varepsilon)}{\varepsilon}} \mathrm{e}^{-(r-\varepsilon)s} \d N_s \bigg|^\gamma \bigg| N_{-{\log(\varepsilon)}/{\varepsilon}} = i\bigg] &= \E^{\P^{\bar{u}}}\bigg[\bigg|\sum_{k=1}^i \mathrm{e}^{-(r-\varepsilon)T^{\smalltext{N}}_k}\bigg|^\gamma \bigg| N_{-\log(\varepsilon)/\varepsilon}= i\bigg] \\
&\leq i^{\gamma-1} \sum_{k=1}^i \E^{\P^{\bar{u}}}\bigg[ \mathrm{e}^{-\gamma (r-\varepsilon)T^N_k} \bigg| N_{-{\log(\varepsilon)}/{\varepsilon}} = i\bigg] \\
&= - i^{\gamma} \frac{\varepsilon}{\log(\varepsilon)} \int_0^{-\frac{\log(\varepsilon)}{\varepsilon}} \mathrm{e}^{-\gamma (r-\varepsilon) t} \d t= - i^{\gamma} \frac{\varepsilon}{\log(\varepsilon)} \frac{1}{\gamma(r-\varepsilon)} \Big(1-\mathrm{e}^{\gamma (r-\varepsilon) \frac{\log(\varepsilon)}{\varepsilon}}\Big),
\end{align*}
where we have once again applied \cite[Remark 2.1]{bouchard2018unified}, and we also consider the fact that for any $i \in \N$, the random variables $(T^{\smalltext{N}}_k)_{k \in \{1,\dots,i\}}$ are independent and follow a uniform distribution on the interval $(0, -\log(\varepsilon)/\varepsilon)$, given that $N_{-{\log(\varepsilon)}/{\varepsilon}} = i$. We deduce that the second term in \eqref{align:secondTermBound} is bounded by 
\begin{align*}
0 &\leq C \mathrm{e}^{\rho \frac{\log(\varepsilon)}{\varepsilon}} \bigg(\mathrm{e}^{-\gamma(r-\varepsilon) \frac{\log(\varepsilon)}{\varepsilon}} \E^{\P^{\smalltext{\bar{u}}}}\bigg[\bigg| \int_0^{-\frac{\log(\varepsilon)}{\varepsilon}} \mathrm{e}^{-(r-\varepsilon) s} \d \tilde{N}_s^{\bar{u}} \bigg|^\gamma\bigg] \bigg) \\
&\leq -C \mathrm{e}^{\rho \frac{\log(\varepsilon)}{\varepsilon}} \bigg(\mathrm{e}^{-\gamma(r-\varepsilon) \frac{\log(\varepsilon)}{\varepsilon}}  \bigg(P^\varepsilon \frac{\varepsilon}{\log(\varepsilon)} \frac{1}{\gamma(r-\varepsilon)} \Big(1-\mathrm{e}^{\gamma (r-\varepsilon) \frac{\log(\varepsilon)}{\varepsilon}}\Big) + \bigg(\frac{1- \mathrm{e}^{(r-\varepsilon)\frac{\log(\varepsilon)}{\varepsilon}}}{(r-\varepsilon)}\bigg)^\gamma\bigg).
\end{align*}
Here, $P^\varepsilon$ represents a polynomial in $-\log(\varepsilon)/\varepsilon$. It can be directly verified that this expression goes to zero as $\varepsilon$ goes to zero, given that $\gamma\delta \leq 1$. Then, we need to analyse the last term in \eqref{align:secondTermBound}:
\begin{align*}
\E^{\P^{\smalltext{\bar{u}}}}\bigg[\mathbf{1}_{\{T^{\smalltext{\varepsilon}}_0 < \infty\}}  \bigg| \int_0^{\frac{\log(\varepsilon)}{\varepsilon}} \mathrm{e}^{-(r-\varepsilon) s} \d \tilde{N}_s^{\bar{u}} \bigg|^\gamma \bigg] &\leq C \Big(\P^{\bar{u}} \big(T^\varepsilon_0 < \infty\big) \Big)^{\frac{1}{2}} 
\E^{\P^{\smalltext{\bar{u}}}}\bigg[ \bigg| \int_0^{\frac{\log(\varepsilon)}{\varepsilon}} \mathrm{e}^{-(r-\varepsilon) s} \d \tilde{N}_s^{\bar{u}} \bigg|^{2 \gamma} \bigg]^{\frac{1}{2}} \\
&\leq  C \Big(\P^{\bar{u}} \big(T^\varepsilon_0 < \infty\big) \Big)^{\frac{1}{2}} \E^{\P^{\smalltext{\bar{u}}}}\bigg[ \bigg| \int_0^{\frac{\log(\varepsilon)}{\varepsilon}} \mathrm{e}^{-2(r-\varepsilon) s} \d {N}_s \bigg|^{\gamma} \bigg]^{\frac{1}{2}},
\end{align*}
as a direct consequence of the Cauchy–Schwarz’s inequality and the Burkholder–Davis–Gundy’s inequality. Analogously to the previous computations, we can prove that this expression, and hence the first term in \eqref{eq:criterionCVar}, tends to zero as $\varepsilon$ goes to zero. To conclude the proof we only need to show the convergence of the second term in \eqref{eq:criterionCVar}, which immediately follows from Step 3 of \cite[Theorem 3.4]{possamai2020there}.
\end{proof}

\subsection{The comparison theorem}\label{compResultapp}

We start this section with a generalisation of Tietze’s extension theorem to show that the domain of viscosity sub-solutions (and consequently viscosity super-solutions and viscosity solutions) to the integro-differential variational inequality \eqref{align:hjbSB} can be extended to the whole real line. This extension is feasible because the non-local part relies exclusively on the values within the interval $[0,\infty)$, as indicated by the first condition in \eqref{equation:minPlusIntCond}. This condition also explains the $y$-dependence of the set $\mathfrak{V}_y$ over which the non-local operator optimises; additionally, this dependence is the very reason why we could not apply comparison results from the existing literature to our Hamilton--Jacobi--Bellman equation.

\begin{lemma}\label{lemma:extViscositySolutions}
Fix some $\mu \geq 1$. Let $u: [0,\infty) \rightarrow \R$ be an upper--semi-continuous viscosity sub-solution to {\rm\Cref{align:hjbSB}} such that $u \in \cG_\mu$. Then, there exists an upper--semi-continuous function $\tilde{u}: \R \rightarrow \R$ verifying $\tilde{u} \in \cG_\mu$ and $\tilde{u} = u$ on $[0,\infty)$. Moreover, it satisfies the following variational inequality:
\begin{align*}
\min\left\{\tilde{u}(y)-\bar{F}(y), F^{\star}(\delta \tilde{\phi}^\prime(y)) - \delta y \tilde{\phi}^\prime(y)+\tilde{u}(y) -\cJ^{\smalltext{\rm SB}}(y,\tilde{\phi}^\prime(y), \tilde{\phi}^{\prime\prime}(y),\tilde{\phi}(\cdot))\right\}\leq 0
\end{align*}
for any $(y, \tilde{\phi}) \in (0,  \infty) \times C^2(\R)$ with $\tilde{\phi} \in \cG_\mu$ such that $\tilde{u}-\tilde{\phi}$ attains a global maximum at $y$.
\end{lemma}
\begin{proof}
Motivated by \cite[Lemma 2.6]{hollender2016levy}, we define the function $\tilde{u}$ by 
\begin{align*}
    \tilde{u}(y) \coloneqq 
    \begin{cases}
	\displaystyle \inf_{x \geq 0} \bigg\{\frac{u(x)}{1+|x|^\mu} + \frac{x}{|y|}\bigg\} (1+|y|^\mu),  \; y \in (-\infty,0) \\
    u(y),  \; y \in [0,\infty)
	\end{cases}.
\end{align*}
According to \cite[Theorem 1.26]{hollender2016levy}, $\tilde{u}$ is upper--semi-continuous function with the required growth condition. Next, we introduce a test function $\tilde{\phi} \in C^2(\R)$ such that $\tilde{\phi} \in \cG_\mu$ ,and $\tilde{u}-\tilde{\phi}$ has a global maximum at $y \geq 0$. If we define a new function $\phi: [0,\infty) \rightarrow \R$ such that $\phi = \tilde{\phi}$ on $[0,\infty)$, then it follows that $u - {\phi}$ has a global maximum at the same point $y$. Therefore, the assertion that $u$ is a viscosity sub-solution of \eqref{align:hjbSB} implies that
\begin{align*}
\min\left\{\tilde{u}(y)-\bar{F}(y), F^{\star}(\delta \tilde{\phi}^\prime(y)) - \delta y \tilde{\phi}^\prime(y)+\tilde{u}(y) -\cJ^{\smalltext{\rm SB}}(y,\tilde{\phi}^\prime(y), \tilde{\phi}^{\prime\prime}(y),\tilde{\phi}(\cdot))\right\}\leq 0.
\end{align*}
This holds true as the non-local part depends only on $[0,\infty)$.
\end{proof}

We can now state the comparison theorem for \Cref{align:hjbSB}. We emphasise once more that the fact that our Hamilton--Jacobi--Bellman equation is characterised by a non-local operator prevents us from directly applying \cite[Lemma B.1]{possamai2020there} to our framework. However, the proof of following result is strongly inspired by the proof of the comparison theorem without accidents, since the non-singularity of the non-local part allows us to use a local maximum principle, and therefore the classical viscosity solution techniques.

\begin{proposition}\label{proposition: comparisonSecondBest}
Let $u$ be an upper--semi-continuous viscosity sub-solution and $v$ be a lower--semi-continuous viscosity super-solution to {\rm\Cref{align:hjbSB}} so that $u(0) \leq 0 \leq v(0)$ and 
\[
\bar{F}(y) \leq \phi(y) \leq C \big(1+\bar{F}(y)\big),\; y\geq 0, \; \text{\rm for some} \; C \in(0,\infty), \; \text{\rm for} \; \phi \in \{u,v\}.
\]
Then, $u \leq v$ on $[0,\infty)$.
\end{proposition}

\begin{proof}
Based on our assumption, which states that $u(0) \leq v(0)$, we can redefine $v(0) \coloneqq \limsup_{y \rightarrow 0^+} v(y)$ without any loss of generality. With this modification, we can start proving the result by contradiction, assuming the existence of some $y^{{\Delta}} \in (0,\infty)$ such that $\Delta_{{y}^{\smalltext{\Delta}}} \coloneqq (u-v)(y^{{\Delta}}) > 0$. We introduce $\mu \coloneqq (2 \vee \gamma) +\varepsilon$, for some $\varepsilon > 0$, and select $\eta >0$ sufficiently small such that
\begin{align*}
    M^\eta \coloneqq \sup_{y \in [0, \infty)} \bigg\{(u-v)(y) - \frac{\eta}{\mu} |y|^\mu\bigg\} > 0.
\end{align*}
Note that the function $y \longmapsto (u-v)(y)- {\eta}/{\mu} |y|^\mu$ is upper--semi-continuous, and as $\lim_{y \rightarrow \infty} \big\{(u-v)(y)-{\eta}/{\mu}|y|^\mu\big\} = -\infty$, the supremum in the definition of $M^\eta$ is therefore achieved at some $y^{\star,\eta} \in (0,\infty)$. Subsequently, we introduce the upper--semi-continuous functions $\tilde{u}$, $-\tilde{v}$ defined in \Cref{lemma:extViscositySolutions}, and a positive non-decreasing sequence $(\alpha_n)_{n\in\N}$ such that $\alpha_n$ goes to $\infty$ as $n$ goes to $\infty$. We then proceed to define
\[
M_{\alpha_\smalltext{n}}^{\eta} \coloneqq \sup_{(x,y) \in (-\infty,\infty)^\smalltext{2}} \Psi_{\alpha_\smalltext{n}}^{\eta}(x,y) \coloneqq \sup_{(x,y) \in (-\infty,\infty)^\smalltext{2}} \big\{\tilde{u}(x) - \tilde{v}(y) - \psi_{\alpha_n}^{\eta}(x,y)\big\}, \; \text{where}\; \psi_{\alpha_\smalltext{n}}^{\eta}(x,y) \coloneqq \frac{\alpha_n}{\mu} |x-y|^\mu + \frac{\eta}{\mu} |y|^\mu .
\]
The growth condition of $\tilde{u}$ and $\tilde{v}$ implies the existence of a maximiser $(x^\eta_n,y^\eta_n)\coloneqq (x^\eta_{\alpha_\smalltext{n}}, y^\eta_{\alpha_\smalltext{n}})$ such that 
\begin{align*}
M_{\alpha_\smalltext{n}}^\eta = \tilde{u}(x^\eta_n) - \tilde{v}(y^\eta_n) - \psi_{\alpha_\smalltext{n}}^{\eta}(x_n^\eta,y_n^\eta).
\end{align*}
Notice that we can find a compact set within which the sequence $(x^\eta_n,y^\eta_n)_{n\in\N}$ takes values. Consequently, by considering a sub-sequence if necessary, we have that $(x^\eta_n,y^\eta_n)$ converges to $(\tilde{x}^\eta,\tilde{y}^\eta)$ as $n$ goes to $\infty$, for some $(\tilde{x}^\eta,\tilde{y}^\eta) \in (-\infty,\infty)^2$. Next, we introduce 
\begin{align*}
    M^\eta_{\infty} \coloneqq \sup_{y \in (-\infty,\infty)} \bigg\{(\tilde{u}-\tilde{v})(y) - \frac{\eta}{\mu} |y|^\mu \bigg\}.
\end{align*}
It is straightforward to prove that $M^\eta_{\infty} = M^\eta$ due to the fact that
\begin{align*}
    \tilde{u}(x) - \tilde{v}(y) \leq u(0) (1+|x|^\gamma) - v(0) (1+|y|^\gamma) \leq 0, \; \text{for any} \; (x,y) \in (-\infty,0)^2.
\end{align*}
Using a technique analogous to \cite[Proposition 3.7]{crandall1992user}, we can prove that the sequence $(\alpha_n)_{n\in\N}$ is such that 
\begin{align*}
(x_n^\eta,y_n^\eta) \underset{n\to\infty}{\longrightarrow} (y^{\star,\eta},y^{\star,\eta}), \; \alpha_n |x_n - y_n|^\mu \underset{n\to\infty}{\longrightarrow} 0, \; M^\eta_{\alpha_\smalltext{n}} \underset{n\to\infty}{\longrightarrow} M^\eta_{\infty} = M^\eta.
\end{align*}

Henceforth, it is worth noting that, for sufficiently small values of $\eta$, the following inequalities hold
\begin{align*}
0 < \Delta_{{y}^{\smalltext{\Delta}}} - \frac{\eta}{\mu} |{y}^{\smalltext{\Delta}}|^\mu \leq (u- v)(y^{\star,\eta}) - \frac{\eta}{\mu} |y^{\star,\eta}|^\mu  \leq \tilde{u}(x^\eta_n) - \tilde{v}(y^\eta_n) - \phi_{\alpha_\smalltext{n}}^{\eta}(x^\eta_n,y^\eta_n) &\leq \limsup_{n \rightarrow \infty} \big\{\tilde{u}(x^\eta_n) - \tilde{v}(y^\eta_n)\big\}.
\end{align*}
We can attain the desired contradiction, and the comparison result as a consequence, by simply showing that 
\begin{align}\label{align:contradictionZ}
\limsup_{\eta \rightarrow 0^+} \limsup_{n \rightarrow \infty} \big\{\tilde{u}(x^\eta_n) - \tilde{v}(y^\eta_n)\big\} \leq 0.
\end{align}

To achieve this, we adapt the proof of \cite[Lemma B.1]{possamai2020there}. Considering a sub-sequence if necessary, we can consider that $(x^\eta_n,y^\eta_n) \in (0,\infty)^2$ for any $n \in \N$ because the sequence $(x^\eta_n,y^\eta_n)_{n\in\N}$ converges to $(y^{\star,\eta},y^{\star,\eta})\in (0,\infty)^2$. Therefore, let us fix $n \in \N$. By a direct application of \citeauthor*{crandall1990maximum} \cite[Theorem 1]{crandall1990maximum} to the function $\Psi_{\alpha_\smalltext{n}}^{\eta}$ at the point $(x^\eta_n,y^\eta_n)$, we can find a sequence $(X_n^\eta,Y^\eta_n)_{n \in \N}$ such that 
\begin{align*}
\Big(\alpha_n |x^\eta_n - y^\eta_n|^{\mu-1} \sgn(x^\eta_n - y^\eta_n), X^\eta_n\Big) \in \bar{J}^{2,+}\tilde{u}(x^\eta_n), \; \Big(\alpha_n |x^\eta_n - y^\eta_n|^{\mu-1} \sgn(x^\eta_n - y^\eta_n) -\eta |y^\eta_n|^{\mu-1}, Y^\eta_n\Big) \in \bar{J}^{2,-}\tilde{v}(y^\eta_n),
\end{align*}
and 
\begin{align}\label{align:lambdaInequalities}
-\bigg(\frac{1}{\lambda} + \|C^\eta_n\|\bigg) I_2 \leq
\begin{pmatrix}
X^\eta_n & 0\\
0 & -Y^\eta_n
\end{pmatrix}
\leq C^\eta_n (I_2 + \lambda C^\eta_n) \; \text{for any} \; \lambda > 0.
\end{align}
Here, $I_2$ denotes the identity matrix, and $C^\eta_n \coloneqq a_n^\eta A + b_n^\eta B$, where 
\begin{align*}
a_n^\eta \coloneqq \alpha_n (\mu-1) |x^\eta_n - y^\eta_n|^{\mu-2}, \; b^\eta_n \coloneqq -\eta (\mu-1) |y^\eta_n|^{\mu-2}, \; \text{and} \;
A = \begin{pmatrix}
1 & -1\\
-1 & 1
\end{pmatrix}, \;
B = \begin{pmatrix}
0 & 0\\
0 & 1
\end{pmatrix}.
\end{align*}
By choosing $\lambda$ as the inverse of the square root of the spectral norm $\|C^\eta_n\|$ of the matrix $C^\eta_n$, and multiplying the inequality \eqref{align:lambdaInequalities} by the vectors $(1,1)$ to the left and $(1,1)^\top$ to the right, we get the following
\begin{equation}\label{eq:ineqBound}
X^\eta_n - Y^\eta_n \leq b^\eta_n + \frac{(b^\eta_n)^2}{\sqrt{\|C^\eta_n\|}}.
\end{equation}

\medskip
Subsequently, the definition of the closure of the second-order superjet $\bar{J}^{2,+}\tilde{u}(x^\eta_n)$ implies that
\begin{align}\label{align:uSubSecond}
\begin{split}
\min\Big\{\tilde{u}(x^\eta_n) - \bar{F}(x^\eta_n),&\; F^{\star}\big(\delta \alpha_n  |x^\eta_n - y^\eta_n|^{\mu-1} \sgn(x^\eta_n - y^\eta_n)\big) - \delta x^\eta_n \alpha_n |x^\eta_n - y^\eta_n|^{\mu-1} \sgn(x^\eta_n - y^\eta_n) \\
&+ \tilde{u}(x^\eta_n) - \cJ^{\smalltext{\rm SB}}\big(x^\eta_n,\alpha_n |x^\eta_n - y^\eta_n|^{\mu-1} \sgn(x^\eta_n - y^\eta_n), X^\eta_n,\tilde{u}(\cdot)\big) \Big\} \leq 0,
\end{split}
\end{align}
and, similarly for the closure of the subjet $\bar{J}^{2,-}\tilde{v}(y^\eta_n)$, 
\begin{align}\label{align:vSuperSecond}
\begin{split}
\min\Big\{\tilde{v}(y^\eta_n) - \bar{F}(y^\eta_n), &\; F^{\star}\big(\delta \alpha_n |x^\eta_n - y^\eta_n|^{\mu-1} \sgn(x^\eta_n - y^\eta_n) - \delta \eta |y^\eta_n|^{\mu-1}\big) - \delta y^\eta_n \alpha_n |x^\eta_n - y^\eta_n|^{\mu-1} \sgn(x^\eta_n - y^\eta_n) + \delta \eta |y^\eta_n|^\mu \\
&+ \tilde{v}(y^\eta_n) -\cJ^{\smalltext{\rm SB}}\big(y^\eta_n,\alpha_n |x^\eta_n - y^\eta_n|^{\mu-1} \sgn(x^\eta_n - y^\eta_n) - \eta |y^\eta_n|^{\mu-1}, Y^\eta_n,\tilde{v}(\cdot)\big) \Big\} \geq 0.
\end{split}
\end{align}
We divide our analysis into two cases. First, let us assume that $\tilde{u}(x^\eta_n) - \bar{F}(x^\eta_n) \leq 0$ along some subsequence. Here, \Cref{align:vSuperSecond} implies the required contradiction in \Cref{align:contradictionZ} since
\begin{align*}
   \limsup_{n \rightarrow  \infty} \big\{\tilde{u}(x^\eta_n) - \tilde{v}(y^\eta_n)\big\} \leq \limsup_{n \rightarrow  \infty} \big\{\bar{F}(x^\eta_n) - \bar{F}(y^\eta_n)\big\} = 0.
\end{align*}
Then, we need to examine the second scenario, specifically the one in which there exists a sub-sequence $(x^\eta_n,y^\eta_n)_{n \in \N}$ that satisfies the following
\begin{align*}
&F^{\star}\big(\delta \alpha_n |x^\eta_n - y^\eta_n|^{\mu-1} \sgn(x^\eta_n - y^\eta_n)\big) - \delta x^\eta_n \alpha_n |x^\eta_n - y^\eta_n|^{\mu-1} \sgn(x^\eta_n - y^\eta_n) + \tilde{u}(x^\eta_n)\\
&\quad - \cJ^{\smalltext{\rm SB}}\big(x^\eta_n,\alpha_n |x^\eta_n - y^\eta_n|^{\mu-1} \sgn(x^\eta_n - y^\eta_n), X^\eta_n,\tilde{u}(\cdot)\big) \leq 0.
\end{align*}
By combining the previous equation with \eqref{align:vSuperSecond}, we obtain that
\begin{align}\label{align:inequalityTo0}
\begin{split}
&\tilde{u}(x^\eta_n) - \tilde{v}(y^\eta_n) \\
&\leq F^{\star}\big(\delta \alpha_n |x^\eta_n - y^\eta_n|^{\mu-1} \sgn(x^\eta_n - y^\eta_n) - \delta \eta |y^\eta_n|^{\mu-1} \big) - F^{\star}\big(\delta \alpha_n  |x^\eta_n - y^\eta_n|^{\mu-1} \sgn(x^\eta_n - y^\eta_n)\big) \\
&\quad +\delta x^\eta_n \alpha_n |x^\eta_n - y^\eta_n|^{\mu-1} \sgn(x^\eta_n - y^\eta_n) - \delta y^\eta_n \alpha_n |x^\eta_n - y^\eta_n|^{\mu-1} \sgn(x^\eta_n - y^\eta_n) + \delta \eta |y^\eta_n|^\mu \\
&\quad +\cJ^{\smalltext{\rm SB}}\big(x^\eta_n,\alpha_n |x^\eta_n - y^\eta_n|^{\mu-1} \sgn(x^\eta_n - y^\eta_n), X^\eta_n,\tilde{u}(\cdot)\big) -\cJ^{\smalltext{\rm SB}}\big(y^\eta_n,\alpha_n |x^\eta_n - y^\eta_n|^{\mu-1} \sgn(x^\eta_n - y^\eta_n) - \eta |y^\eta_n|^{\mu-1}, Y^\eta_n,\tilde{v}(\cdot)\big) \\
&\leq\delta \alpha_n |x^\eta_n - y^\eta_n|^\mu + \delta \eta |y^\eta_n|^\mu \\
&\quad +\cJ^{\smalltext{\rm SB}}\big(x^\eta_n,\alpha_n |x^\eta_n - y^\eta_n|^{\mu-1} \sgn(x^\eta_n - y^\eta_n), X^\eta_n,\tilde{u}(\cdot)\big) -\cJ^{\smalltext{\rm SB}}\big(y^\eta_n,\alpha_n |x^\eta_n - y^\eta_n|^{\mu-1} \sgn(x^\eta_n - y^\eta_n) - \eta |y^\eta_n|^{\mu-1}, Y^\eta_n,\tilde{v}(\cdot)\big),
\end{split}
\end{align}
as the function $F^\star$ is non-decreasing. Before proceeding with the analysis of the aforementioned expression to derive a contradiction in \Cref{align:contradictionZ}, it is important to note that the sequence $(Y^\eta_n)_{n \in \N}$ is non-positive due to the fact \Cref{align:vSuperSecond} holds. Furthermore, the inequality in \eqref{eq:ineqBound}, along with the same reasoning that concludes the proof of \cite[Lemma B.1]{possamai2020there}, allows us to deduce that, along some subsequence and for sufficiently small values of $\eta$, it holds that $(X^\eta_n)_{n \in \N}$ is also non-positive.

\medskip
In the next step, we observe that we can assume that the sequence $(x^\eta_n - y^\eta_n)_{n \in \N}$ is monotone without any loss of generality. Initially, let us suppose that $(x^\eta_n - y^\eta_n)_{n \in \N}$ is non-decreasing, implying that $x^\eta_n \leq y^\eta_n$ for any $n \in \N$ since $(x^\eta_n - y^\eta_n)_{n \in \N}$ converges to 0. Given this assumption and the previous observation, if we fix some $n \in \N$, then we can conclude that there exists some $C > 0$ such that
\begin{align*}
&\cJ^{\smalltext{\rm SB}}\big(x^\eta_n,\alpha_n |x^\eta_n - y^\eta_n|^{\mu-1} \sgn(x^\eta_n - y^\eta_n), X^\eta_n,\tilde{u}(\cdot)\big) -\cJ^{\smalltext{\rm SB}}\big(y^\eta_n,\alpha_n |x^\eta_n - y^\eta_n|^{\mu-1} \sgn(x^\eta_n - y^\eta_n) - \eta |y^\eta_n|^{\mu-1}, Y^\eta_n,\tilde{v}(\cdot)\big)\\
&\leq \sup_{(z^{\smalltext{\rm A}}, u^{\smalltext{\rm A}}) \in \mathfrak{V}_{\smalltext{x}^{\tinytext{\eta}}_{\tinytext{n}}}}
\sup_{(a,b) \in U^\smalltext{\star}(z^{\smalltext{\rm A}}, u^{\smalltext{\rm A}})} \bigg\{\delta h(a,b) \eta |y^\eta_n|^{\mu-1} + \frac{r\delta \sigma^2}{2} (z^{\rm A})^2 (X^\eta_n - Y^\eta_n) \\
&\quad+ \frac{b}{m \rho} \int_{\R} \Big(\tilde{u}(x^\eta_n + r u^{\rm A}(\ell)) - \tilde{v}(y^\eta_n + r u^{\rm A}(\ell)) + \tilde{v}(y^\eta_n) - \tilde{u}(x^\eta_n) - r \eta |y_n^\eta|^{\mu-1} u^{\rm A}(\ell) \Big)  \Phi(\d \ell) \bigg\}\\
&\leq C \bigg(\eta |y^\eta_n|^{\mu-1} + b^\eta_n + \frac{(b^\eta_n)^2}{\sqrt{\|C^\eta_n\|}}\bigg) \\
&\quad+ C \sup_{(z^{\smalltext{\rm A}}, u^{\smalltext{\rm A}}) \in \mathfrak{V}_{\smalltext{x}^{\tinytext{\eta}}_{\tinytext{n}}}} \int_{\R} \bigg(\Psi_{\alpha_\smalltext{n}}^{\eta}(x^\eta_n + r u^{\rm A}(\ell),y^\eta_n + r u^{\rm A}(\ell)) - \Psi_{\alpha_\smalltext{n}}^{\eta}(x^\eta_n,y^\eta_n)  + \frac{\eta}{\mu} \bigg(|y^\eta_n + r u^{\rm A}(\ell)|^\mu - |y^\eta_n|^\mu - r \mu |y^\eta_n|^{\mu-1} \bigg)\bigg) \Phi(\d \ell) \\
&\leq C \bigg(\eta |y^\eta_n|^{\mu-1} + b^\eta_n + \frac{(b^\eta_n)^2}{\sqrt{\|C^\eta_n\|}}\bigg) + C \sup_{(z^{\smalltext{\rm A}}, u^{\smalltext{\rm A}}) \in \mathfrak{V}_{\smalltext{x}^{\tinytext{\eta}}_{\tinytext{n}}}} \int_{\R} \frac{\eta}{\mu} \bigg(|y^\eta_n + r u^{\rm A}(\ell)|^\mu - |y^\eta_n|^\mu - r \mu |y^\eta_n|^{\mu-1} \bigg) \Phi(\d \ell),
\end{align*}
where we have used the fact that $(x^\eta_n,y^\eta_n)$ maximises the function $\Psi_{\alpha_\smalltext{n}}^{\eta}$. We can derive from \Cref{align:inequalityTo0} that
\begin{align*}
\tilde{u}(x^\eta_n) - \tilde{v}(y^\eta_n) &\leq \delta \alpha_n |x^\eta_n - y^\eta_n|^{\mu} + \delta \eta |y^\eta_n|^\mu  \\
&\quad+ C \bigg(\eta |y^\eta_n|^{\mu-1} + b^\eta_n + \frac{(b^\eta_n)^2}{\sqrt{\|C^\eta_n\|}} + \frac{\eta}{\mu}  \sup_{u^{\smalltext{\rm A}} \in \mathfrak{U}} \int_{\R} \big(|y^\eta_n + r u^{\rm A}(\ell)|^\mu - |y^\eta_n|^\mu - \mu r |y^\eta_n|^{\mu-1} u^{\rm A}(\ell)\big) \Phi(\d \ell) \bigg),
\end{align*}
where we denote by $\mathfrak{U}$ the collection of $u^{{\rm A}} \in \cB_{{\R}}$ satisfying the condition $\int_{\R} |u^{{\rm A}}(\ell)|^\mu \Phi(\d \ell) <  \infty$. Hence, let us fix some $\tilde{\varepsilon}>0$. Given that the supremum in the above expression is finite, there exists some $u^{\rm A}_{{\tilde{\varepsilon}},{\eta},{n}} \in \mathfrak{U}$ such that 
\begin{align*}
   & \limsup_{n \rightarrow \infty}\sup_{u^{\smalltext{\rm A}} \in \mathfrak{U}} \int_{\R} \big(|y^\eta_n + r u^{\rm A}(\ell)|^\mu - |y^\eta_n|^\mu - \mu r |y^\eta_n|^{\mu-1} u^{\rm A}(\ell)\big) \Phi(\d \ell) \\
    &< \limsup_{n \rightarrow  \infty} \int_{\R} \big(|y^\eta_n + r u^{\rm A}_{\smalltext{\tilde{\varepsilon}},\smalltext{\eta},\smalltext{n}}(\ell)|^\mu - |y^\eta_n|^\mu - \mu r |y^\eta_n|^{\mu-1} u^{\rm A}_{\smalltext{\tilde{\varepsilon}},\smalltext{\eta},\smalltext{n}}(\ell)\big) \Phi(\d \ell) + \tilde{\varepsilon} \\
    &\leq \sup_{u^{\smalltext{\rm A}} \in \mathfrak{U}} \limsup_{n \rightarrow  \infty} \int_{\R} \big(|y^\eta_n + r u^{\rm A}(\ell)|^\mu - |y^\eta_n|^\mu - \mu r |y^\eta_n|^{\mu-1} u^{\rm A}(\ell)\big) \Phi(\d \ell) + \tilde{\varepsilon} \\    
    &= \sup_{u^{\smalltext{\rm A}} \in \mathfrak{U}} \int_{\R} \big(|y^{\star,\eta} + r u^{\rm A}(\ell)|^\mu - |y^{\star,\eta}|^\mu - \mu r |y^{\star,\eta}|^{\mu-1} u^{\rm A}(\ell)\big) \Phi(\d \ell) + \tilde{\varepsilon} <  \infty,
\end{align*}
where the last equality is a direct consequence of the dominated convergence theorem. Once more, we can apply the same reasoning that concludes the proof of \cite[Lemma B.1]{possamai2020there} to get the desired contradiction since
\begin{align*}
   \limsup_{\tilde{\varepsilon} \rightarrow 0^\smalltext{+}} \limsup_{\eta \rightarrow 0^\smalltext{+}} \limsup_{n \rightarrow  \infty} \big\{\tilde{u}(x^\eta_n) - \tilde{v}(y^\eta_n)\big\} = C \limsup_{\eta \rightarrow 0^\smalltext{+}} \limsup_{n \rightarrow  \infty} \bigg\{\alpha_n |x^\eta_n - y^\eta_n|^{\mu} + \eta |y^\eta_n|^{\mu-1} + b^\eta_n + \frac{(b^\eta_n)^2}{\sqrt{\|C^\eta_n\|}} + \eta\bigg\} = 0.
\end{align*}
This concludes the analysis of the first case.

\medskip
Conversely, let us now assume that the sequence $(x^\eta_n - y^\eta_n)_{n \in \N}$ is non-increasing. It holds that $x^\eta_n \geq y^\eta_n$ for any $n \in \N$. To arrive at the necessary contradiction, it is sufficient to estimate the difference between the two non-local operators in \Cref{align:inequalityTo0}, as the other terms are identical to those in the previous scenario. Hence, we proceed as follows
\begin{align}\label{align:differenceNonLocalOp}
\begin{split}
&\cJ^{\smalltext{\rm SB}}\big(x^\eta_n,\alpha_n |x^\eta_n - y^\eta_n|^{\mu-1} \sgn(x^\eta_n - y^\eta_n), X^\eta_n,\tilde{u}(\cdot)\big) -\cJ^{\smalltext{\rm SB}}\big(y^\eta_n,\alpha_n |x^\eta_n - y^\eta_n|^{\mu-1} \sgn(x^\eta_n - y^\eta_n) - \eta |y^\eta_n|^{\mu-1}, Y^\eta_n,\tilde{v}(\cdot)\big)\\
&=\cJ^{\smalltext{\rm SB}}\big(x^\eta_n,\alpha_n |x^\eta_n - y^\eta_n|^{\mu-1}, X^\eta_n,\tilde{u}(\cdot)\big) - j(x^\eta_n,y^\eta_n,\alpha_n |x^\eta_n - y^\eta_n|^{\mu-1} - \eta |y^\eta_n|^{\mu-1},Y^\eta_n,\tilde{v}(\cdot)) \\
&\quad+j(x^\eta_n,y^\eta_n,\alpha_n |x^\eta_n - y^\eta_n|^{\mu-1} - \eta |y^\eta_n|^{\mu-1},Y^\eta_n,\tilde{v}(\cdot)) -\cJ^{\smalltext{\rm SB}}\big(y^\eta_n,\alpha_n |x^\eta_n - y^\eta_n|^{\mu-1} - \eta |y^\eta_n|^{\mu-1}, Y^\eta_n,\tilde{v}(\cdot)\big),
\end{split}
\end{align}
where
\begin{align*}
    j(x,y,p,q,v(\cdot)) &\coloneqq \sup_{(z^{\smalltext{\rm A}}, u^{\smalltext{\rm A}}) \in \mathfrak{V}_{x}} \iota(y,p,q,v(\cdot),z^{\rm A},u^{\rm A}) \\
    &= \sup_{(z^{\smalltext{\rm A}}, u^{\smalltext{\rm A}}) \in \mathfrak{V}_{\smalltext{x}}} \sup_{(a,b) \in U^\smalltext{\star}(z^{\smalltext{\rm A}}, u^{\smalltext{\rm A}})} \bigg\{a-b + \delta h(a,b) p + \frac{r\delta \sigma^2}{2} (z^{\rm A})^2 q \\
    &\qquad\qquad\qquad\qquad\qquad\qquad\qquad\qquad\qquad+ \frac{b}{m \rho} \int_{\R} \big(v(y + r u^{\rm A}(\ell)) - v(y) - r p u^{\rm A}(\ell) \big)  \Phi(\d \ell)\bigg\}.
\end{align*}
Analogously to the previous case, we can prove that the difference of the first two terms in \eqref{align:differenceNonLocalOp} is bounded from above by zero by taking the limit for $n$ going to $\infty$ and $\eta$ going to $0$. To prove that this property also holds for the difference of the other two terms, we first claim that the function $\tilde{v}$ is continuous on the whole real line $(-\infty,\infty)$, and that the sequence $(\alpha_n |x^\eta_n - y^\eta_n|^{\mu-1})_{n \in \N}$ is bounded from above. Consequently, considering a subsequence if necessary, we have that there exists some $L^\eta \in [0,\infty)$ such that
$\lim_{n \rightarrow  \infty} \alpha_n |x^\eta_n - y^\eta_n|^{\mu-1} = L^\eta$. If we fix some $n \in \N$, this allows us to deduce that the supremum $j(x,y^\eta_n,\alpha_n |x^\eta_n - y^\eta_n|^{\mu-1} - \eta |y^\eta_n|^{\mu-1},Y^\eta_n,\tilde{v}(\cdot))$ is bounded from above for any $x \in (0,\infty)$.

\medskip
Now, we need to distinguish two further cases. First, we assume that $\lim_{n \rightarrow  \infty} Y^\eta_n = -\infty$. In this scenario, for sufficiently large values of $n$, we have the following
\begin{align*}
&j(x^\eta_n,y^\eta_n,\alpha_n |x^\eta_n - y^\eta_n|^{\mu-1} - \eta |y^\eta_n|^{\mu-1},Y^\eta_n,\tilde{v}(\cdot)) -\cJ^{\smalltext{\rm SB}}\big(y^\eta_n,\alpha_n |x^\eta_n - y^\eta_n|^{\mu-1} - \eta |y^\eta_n|^{\mu-1}, Y^\eta_n,\tilde{v}(\cdot)\big) \\
&=\underbrace{ \sup_{u^{\smalltext{\rm A}} \in \mathfrak{V}_{\smalltext{x}^{\tinytext{\eta}}_{\tinytext{n}}}} \iota(y^\eta_n,\alpha_n |x^\eta_n - y^\eta_n|^{\mu-1} - \eta |y^\eta_n|^{\mu-1},Y^\eta_n,\tilde{v}(\cdot),0,u^{{\rm A}}) -\sup_{u^{\smalltext{\rm A}} \in \mathfrak{V}_{\smalltext{y}^{\tinytext{\eta}}_{\tinytext{n}}}} \iota(y^\eta_n,\alpha_n |x^\eta_n - y^\eta_n|^{\mu-1} - \eta |y^\eta_n|^{\mu-1},Y^\eta_n,\tilde{v}(\cdot),0,u^{{\rm A}})}_{\eqqcolon \ell_\smalltext{n}}. 
\end{align*}
Taking the limit, we obtain that 
\begin{align*}
&\limsup_{n \rightarrow \infty} \ell_n\\
&\leq \limsup_{n \rightarrow  \infty} \sup_{u^{\smalltext{\rm A}} \in \mathfrak{U}} \Big\{\iota(y^\eta_n,\alpha_n |x^\eta_n - y^\eta_n|^{\mu-1} - \eta |y^\eta_n|^{\mu-1},Y^\eta_n,\tilde{v}(\cdot),0,u^{{\rm A}}) \; \mathbf{1}_{\{\min_{\ell \in \R_\smalltext{+}} \{r u^{\smalltext{\rm A}}(\ell)\}\geq -x^\smalltext{\eta}_\smalltext{n}\}} - \infty  \mathbf{1}_{\{\min_{\ell \in \R_\smalltext{+}} \{r u^{\smalltext{\rm A}}(\ell)\}< -x^\smalltext{\eta}_\smalltext{n}\}} \Big\} \\
&\quad-\liminf_{n \rightarrow  \infty}\sup_{u^{\smalltext{\rm A}} \in \mathfrak{U}} \Big\{\iota(y^\eta_n,\alpha_n |x^\eta_n - y^\eta_n|^{\mu-1} - \eta |y^\eta_n|^{\mu-1},Y^\eta_n,\tilde{v}(\cdot),0,u^{{\rm A}}) \; \mathbf{1}_{\{\min_{\ell \in \R_\smalltext{+}} \{r u^{\smalltext{\rm A}}(\ell)\}\geq -y^\smalltext{\eta}_\smalltext{n}\}} - \infty  \mathbf{1}_{\{\min_{\ell \in \R_\smalltext{+}} \{r u^{\smalltext{\rm A}}(\ell)\}< -y^\smalltext{\eta}_\smalltext{n}\}} \Big\} \\
&\leq \limsup_{n \rightarrow  \infty} \sup_{u^{\smalltext{\rm A}} \in \mathfrak{U}} \Big\{\iota(y^\eta_n,\alpha_n |x^\eta_n - y^\eta_n|^{\mu-1} - \eta |y^\eta_n|^{\mu-1},Y^\eta_n,\tilde{v}(\cdot),0,u^{{\rm A}}) \; \mathbf{1}_{\{\min_{\ell \in \R_\smalltext{+}} \{r u^{\smalltext{\rm A}}(\ell)\}\geq -x^\smalltext{\eta}_\smalltext{n}\}} - \infty  \mathbf{1}_{\{\min_{\ell \in \R_\smalltext{+}} \{r u^{\smalltext{\rm A}}(\ell)\}< -x^\smalltext{\eta}_\smalltext{n}\}} \Big\} \\
&\quad-\sup_{u^{\smalltext{\rm A}} \in \mathfrak{U}} \liminf_{n \rightarrow  \infty} \Big\{\iota(y^\eta_n,\alpha_n |x^\eta_n - y^\eta_n|^{\mu-1} - \eta |y^\eta_n|^{\mu-1},Y^\eta_n,\tilde{v}(\cdot),0,u^{{\rm A}}) \; \mathbf{1}_{\{\min_{\ell \in \R_\smalltext{+}} \{r u^{\smalltext{\rm A}}(\ell)\}\geq -y^\smalltext{\eta}_\smalltext{n}\}} - \infty  \mathbf{1}_{\{\min_{\ell \in \R_\smalltext{+}} \{r u^{\smalltext{\rm A}}(\ell)\}< -y^\smalltext{\eta}_\smalltext{n}\}} \Big\} \\
&< \limsup_{n \rightarrow  \infty} \Big\{\iota(y^\eta_n,\alpha_n |x^\eta_n - y^\eta_n|^{\mu-1} - \eta |y^\eta_n|^{\mu-1},Y^\eta_n,\tilde{v}(\cdot),0,u^{\rm A}_{\smalltext{\tilde{\varepsilon}},\smalltext{\eta},\smalltext{n}}) \; \mathbf{1}_{\{\min_{\ell \in \R_\smalltext{+}} \{r u^{\rm A}_{\smalltext{\tilde{\varepsilon}},\smalltext{\eta},\smalltext{n}}(\ell)\}\geq -x^\smalltext{\eta}_\smalltext{n}\}} - \infty  \mathbf{1}_{\{\min_{\ell \in \R_\smalltext{+}} \{r u^{\rm A}_{\smalltext{\tilde{\varepsilon}},\smalltext{\eta},\smalltext{n}}(\ell)\}< -x^\smalltext{\eta}_\smalltext{n}\}} \Big\} + \tilde{\varepsilon} \\
&\quad-\sup_{u^{\smalltext{\rm A}} \in \mathfrak{U}} \liminf_{n \rightarrow  \infty} \Big\{\iota(y^\eta_n,\alpha_n |x^\eta_n - y^\eta_n|^{\mu-1} - \eta |y^\eta_n|^{\mu-1},Y^\eta_n,\tilde{v}(\cdot),0,u^{{\rm A}}) \; \mathbf{1}_{\{\min_{\ell \in \R_\smalltext{+}} \{r u^{\smalltext{\rm A}}(\ell)\}\geq -y^\smalltext{\eta}_\smalltext{n}\}} - \infty \mathbf{1}_{\{\min_{\ell \in \R_\smalltext{+}} \{r u^{\smalltext{\rm A}}(\ell)\}< -y^\smalltext{\eta}_\smalltext{n}\}} \Big\},
\end{align*}
for some $\tilde{\varepsilon} > 0$ and $u^{\rm A}_{\smalltext{\tilde{\varepsilon}},\smalltext{\eta},\smalltext{n}} \in \mathfrak{U}$. Therefore
\begin{align*}
&\limsup_{n \rightarrow  \infty} \Big\{\iota(y^\eta_n,\alpha_n |x^\eta_n - y^\eta_n|^{\mu-1} - \eta |y^\eta_n|^{\mu-1},Y^\eta_n,\tilde{v}(\cdot),0,u^{\rm A}_{{\tilde{\varepsilon}},{\eta},{n}}) \; \mathbf{1}_{\{\min_{\ell \in \R_\smalltext{+}} \{r u^{\rm \smalltext{A}}_{\smalltext{\tilde{\varepsilon}}\smalltext{,}\smalltext{\eta}\smalltext{,}\smalltext{n}}(\ell)\}\geq -x^\smalltext{\eta}_\smalltext{n}\}} - \infty  \mathbf{1}_{\{\min_{\ell \in \R_\smalltext{+}} \{r u^{\rm \smalltext{A}}_{\smalltext{\tilde{\varepsilon}}\smalltext{,}\smalltext{\eta}\smalltext{,}\smalltext{n}}(\ell)\}< -x^\smalltext{\eta}_\smalltext{n}\}} \Big\} + \tilde{\varepsilon} \\
&\quad-\sup_{u^{\smalltext{\rm A}} \in \mathfrak{U}} \liminf_{n \rightarrow  \infty} \Big\{\iota(y^\eta_n,\alpha_n |x^\eta_n - y^\eta_n|^{\mu-1} - \eta |y^\eta_n|^{\mu-1},Y^\eta_n,\tilde{v}(\cdot),0,u^{{\rm A}})  \mathbf{1}_{\{\min_{\ell \in \R_\smalltext{+}} \{r u^{\smalltext{\rm A}}(\ell)\}\geq -y^\smalltext{\eta}_\smalltext{n}\}} - \infty  \mathbf{1}_{\{\min_{\ell \in \R_\smalltext{+}} \{r u^{\smalltext{\rm A}}(\ell)\}< -y^\smalltext{\eta}_\smalltext{n}\}} \Big\}\\
&\leq\sup_{u^{\smalltext{\rm A}} \in \mathfrak{U}} \Big\{\iota(y^{\star,\eta},L^\eta - \eta |y^{\star,\eta}|^{\mu-1},-\infty,\tilde{v}(\cdot),0,u^{\rm A}) \mathbf{1}_{\{\min_{\ell \in \R_\smalltext{+}} \{r u^{\smalltext{\rm A}}(\ell)\}\geq -y^{\smalltext{\star}\smalltext{,}\smalltext{\eta}}\}} - \infty \mathbf{1}_{\{\min_{\ell \in \R_\smalltext{+}} \{r u^{\smalltext{\rm A}}(\ell)\}< -y^{\smalltext{\star}\smalltext{,}\smalltext{\eta}}\}} \Big\} + \tilde{\varepsilon} \\
&\quad-\sup_{u^{\smalltext{\rm A}} \in \mathfrak{U}} \Big\{\iota(y^{\star,\eta},L^\eta - \eta |y^{\star,\eta}|^{\mu-1},-\infty,\tilde{v}(\cdot),0,u^{{\rm A}})  \mathbf{1}_{\{\min_{\ell \in \R_\smalltext{+}} \{r u^{{\rm \smalltext{A}}}(\ell)\}> -y^{\smalltext{\star}\smalltext{,}\smalltext{\eta}}\}} - \infty  \mathbf{1}_{\{\min_{\ell \in \R_\smalltext{+}} \{r u^{\smalltext{\rm A}}(\ell)\}\leq -y^{\smalltext{\star}\smalltext{,}\smalltext{\eta}}\}} \Big\}\\
&=\sup_{u^{\smalltext{\rm A}} \in \mathfrak{U},\; \min_{\smalltext{\ell}\smalltext{\in}\smalltext{\R}_\tinytext{+}} \{r u^{\smalltext{\rm A}}(\ell) \}\geq -y^{\smalltext{\star}\smalltext{,}\smalltext{\eta}}} \iota(y^{\star,\eta},L^\eta - \eta |y^{\star,\eta}|^{\mu-1},-\infty,\tilde{v}(\cdot),0,u^{\rm A})  + \tilde{\varepsilon} \\
&\quad-\sup_{u^{\smalltext{\rm A}} \in \mathfrak{U},\; \min_{\smalltext{\ell} \smalltext{\in} \smalltext{\R}_\tinytext{+}} \{r u^{\smalltext{\rm A}}(\ell) \}> -y^{\smalltext{\star}\smalltext{,}\smalltext{\eta}}} \iota(y^{\star,\eta},L^\eta - \eta |y^{\star,\eta}|^{\mu-1},-\infty,\tilde{v}(\cdot),0,u^{{\rm A}}) =\tilde{\varepsilon}.
\end{align*}
The last equality is due to the fact that the function that is maximised is continuous and the supremum is finite. The arbitrariness of $\tilde{\varepsilon}$ implies that
\begin{align}\label{align:aimBound0}
\limsup_{n \rightarrow  \infty} \big\{j&(x^\eta_n,y^\eta_n,\alpha_n |x^\eta_n - y^\eta_n|^{\mu-1} - \eta |y^\eta_n|^{\mu-1},Y^\eta_n,\tilde{v}(\cdot)) -\cJ^{\smalltext{\rm SB}}\big(y^\eta_n,\alpha_n |x^\eta_n - y^\eta_n|^{\mu-1} - \eta |y^\eta_n|^{\mu-1}, Y^\eta_n,\tilde{v}(\cdot)\big)\big\} \leq 0. 
\end{align}
This observation leads to the required contradiction in \Cref{align:contradictionZ} in the case $\lim_{n \rightarrow  \infty} Y^\eta_n = -\infty$. The computations necessary to prove \eqref{align:aimBound0} in the other case are analogous to the previous ones. This is because the condition $\limsup_{n \rightarrow  \infty} Y^\eta_n > -\infty$ implies the existence of  a convergent sub-sequence $(Y^\eta_n)_{n \in \N}$.

\medskip
In order to complete the proof, we need to verify the assertions we have previously made. Specifically, we need to prove that $\tilde{v}$ is a continuous function on the whole real line, and that the sequence $(\alpha_n |x^\eta_n - y^\eta_n|^{\mu-1})_{n \in \N}$ is bounded from above. To this aim, we first observe that $\tilde{v}$ is concave on the interval $(0,\infty)$. In fact, if we suppose, to the contrary, that $\tilde{v}$ is strictly convex on some non-empty open interval $\cI \subset (0,\infty)$, then we would have that $\cJ^{\smalltext{\rm SB}}(y,\tilde{v}^\prime(y),\tilde{v}^{\prime\prime}(y), v(\cdot)) = \infty$ for any $y \in \cI$ in viscosity sense, contradicting the hypothesis that $\tilde{v}$ is a viscosity super-solution of \Cref{align:hjbSB}. The concavity of $\tilde{v}$ on $(0,\infty)$ implies its continuity on the same interval (see for instance \cite[Theorem 10.1.1]{rockafellar1970convex}). Consequently, $\tilde{v}$ is continuous on $(-\infty,\infty)$ due to its construction since $\tilde{v}(0) = \limsup_{y \rightarrow 0^+} v(y) = \liminf_{y \rightarrow 0^+} v(y)$, as established in \cite[Lemma 1.26]{hollender2016levy}.
 
Subsequently, the continuity of $\tilde{v}$ ensures the existence of $\tilde{v}^\prime_-(y)$ and $\tilde{v}^\prime_+(y)$ for all $y \in (0,\infty)$. Furthermore, due to its concavity, the following inequalities hold:
\begin{align}\label{align:inequalitiesDerConcaveV}
    \tilde{v}^\prime_+(y_1) \geq \tilde{v}^\prime_-(y) \geq \tilde{v}^\prime_+(y) \geq \tilde{v}^\prime_-(y_2), \; \text{for} \; (y_1,y,y_2) \in (0,  \infty) \times [y_1, \infty) \times [y,  \infty).
\end{align}
It is evident that $\tilde{v}^\prime_-(y) < \infty$ for any $y \in (0,\infty)$. Consequently, we can deduce that $\lim_{n \rightarrow  \infty} \alpha_n |x^\eta_n - y^\eta_n|^{\mu-1} <  \infty$. In fact, $\alpha_n |x^\eta_n - y^\eta_n|^{\mu-1}$ is such that $\big(\alpha_n |x^\eta_n - y^\eta_n|^{\mu-1} -\eta |y^\eta_n|^{\mu-1}, Y^\eta_n\big) \in \bar{J}^{2,-}\tilde{v}(y^\eta_n)$. This condition implies that, for any fixed $n \in \N$, there exists a sequence $((y^\eta_{n})_m, (p^\eta_{n})_m, (Y^\eta_{n})_m)_{m \in \N}$ such that $\big((p^\eta_{n})_m, (Y^\eta_{n})_m\big) \in {J}^{2,-}\tilde{v}((y^\eta_{n})_m)$, and the following convergences holds
\begin{align}\label{align:convergencesM}
    (y^\eta_{n})_m \underset{m\to\infty}{\longrightarrow} y^\eta_n, \; \tilde{v}((y^\eta_{n})_m) \underset{m\to\infty}{\longrightarrow} \tilde{v}(y^\eta_n), \; (p^\eta_{n})_m \underset{m\to\infty}{\longrightarrow} \alpha_n |x^\eta_n - y^\eta_n|^{\mu-1} - \eta |y^\eta_n|^{\mu-1}, \; (Y^\eta_{n})_m \underset{m\to\infty}{\longrightarrow} Y^\eta_n.
\end{align}
For any fixed $(m,n) \in \N^2$, the condition $\big((p^\eta_{{n}})_m, (Y^\eta_{{n}})_m\big) \in {J}^{2,-}\tilde{v}((y^\eta_{{n}})_m)$ is equivalent to the existence of a twice-continuously differentiable function $w_m$ defined on $(-\infty,\infty)$ (see for instance \cite[Lemma V.4.1]{fleming2006controlled}) such that 
\begin{align*}
    (p^\eta_{{n}})_m = w_m^\prime((y^\eta_{{n}})_m), \; (Y^\eta_{{n}})_m = w_m^{\prime\prime}((y^\eta_{{n}})_m),
\end{align*}
and $(y^\eta_{{n}})_m$ is a local minimiser of the difference $\tilde{v} - w_m$. Additionally, we have that $(p^\eta_{{n}})_m \in D^-\tilde{v}((y^\eta_{{n}})_m)$, as proved in \cite[Lemma II.1.7]{bardi1997optimal}. Not only that, we know that $(p^\eta_{{n}})_m = \tilde{v}^\prime((y^\eta_{{n}})_m)$ since the concavity of the function $\tilde{v}$ on $(0,\infty)$ implies that $D^+\tilde{v}((y^\eta_{\smalltext{n}})_m)$ is not empty (see for instance \cite[Lemma II.1.8]{fleming2006controlled}), and thus $\tilde{v}$ is differentiable at $(y^\eta_n)_m$ because of \cite[Lemma II.1.8]{bardi1997optimal}. It follows that $(p^\eta_{n})_m = \tilde{v}^\prime((y^\eta_{n})_m)$. Then, given the first convergence in \eqref{align:convergencesM}, it is straightforward that $\lim_{n \rightarrow  \infty} \lim_{m \rightarrow  \infty} (y^\eta_{{n}})_m = y^{\star,\eta} \in (0, \infty)$, and therefore we can assume the existence of some $\hat{y}^\eta \in (0, \infty)$ such that $(y^\eta_{{n}})_m \in [\hat{y}^\eta, \infty)$ for all $m$, $n \in \N$. This fact, along with the inequalities provided in \eqref{align:inequalitiesDerConcaveV}, imply that $(p^\eta_{n})_m = \tilde{v}^\prime((y^\eta_{n})_m) \leq \tilde{v}^\prime_{-}(\hat{y}^\eta)$ for any $(m,n) \in \N^2$, allowing us to conclude that
\begin{align*}
    -\eta |y^{\star,\eta}|^{\mu-1} \leq \lim_{n \rightarrow  \infty} \big\{\alpha_n |x^\eta_n - y^\eta_n|^{\mu-1} -\eta |y^\eta_n|^{\mu-1}\big\} = \lim_{n \rightarrow  \infty} \lim_{m \rightarrow  \infty} (p^\eta_{n})_m  = \lim_{n \rightarrow  \infty} \lim_{m \rightarrow  \infty} \tilde{v}^\prime((y^\eta_{n})_m) \leq \tilde{v}^\prime_{-}(\hat{y}^\eta) <  \infty.
\end{align*}
This completes the proof.
\end{proof}

\end{appendix}

\end{document}